\documentclass[12pt]{scrartcl} 

\usepackage{array}
\usepackage{supertabular}
\usepackage{amsmath,amsfonts,amssymb,amsthm} 
\usepackage{bm}
\usepackage{graphicx}

\newcommand{\diag}{\text{diag}}

\newcommand{\xx}{\bm{X}}

\newcommand{\wai}{\sum_{i=1}^p}
\newcommand{\waj}{\sum_{j=1}^p}
\newcommand{\wak}{\sum_{k=1}^p}
\newcommand{\wal}{\sum_{l=1}^p}
\newcommand{\wat}{\sum_{t=1}^p}
\newcommand{\was}{\sum_{s=1}^p}

\newcommand{\alphad}{\alpha '}
\newcommand{\bareu}[1]{{\bar{e}}_{#1}}
\newcommand{\bareo}[1]{{\bar{e}}^{#1}}
\newcommand{\barteta}[1]{{\bar{\theta}}^{#1}}
\newcommand{\aaa}[2]{A_{#1}^{#2}}
\newcommand{\bee}[2]{B_{#1}^{#2}}
\newcommand{\bbar}[2]{\bar{B}_{#1}^{#2}}
\newcommand{\cee}[2]{C_{#1}^{#2}}
\newcommand{\cbar}[2]{\bar{C}_{#1}^{#2}}
\newcommand{\dee}[2]{D_{#1}^{#2}}

\newcommand{\secffe}[2]{\overset{\;e}{A}\hspace{-2pt}\begin{smallmatrix}
{#2}\\{#1}\end{smallmatrix}}
\newcommand{\secfse}[1]{\overset{\;e}{A}\hspace{-2pt}\begin{smallmatrix}
{\rule{0pt}{4pt}}\\{#1}\end{smallmatrix}}
\newcommand{\secffm}[2]{\overset{\;\scalebox{0.6}{$m$}}{A}\hspace{-2pt}\begin{smallmatrix}
{#2}\\{#1}\end{smallmatrix}}
\newcommand{\secfsm}[1]{\overset{\;\scalebox{0.6}{$m$}}{A}\hspace{-2pt}\begin{smallmatrix}
{\rule{0pt}{4pt}}\\{#1}\end{smallmatrix}}
\newcommand{\secffa}[2]{\overset{\;\alpha}{A}\hspace{-2pt}\begin{smallmatrix}
{#2}\\{#1}\end{smallmatrix}}
\newcommand{\secfsa}[1]{\overset{\;\alpha}{A}\hspace{-2pt}\begin{smallmatrix}
{\rule{0pt}{4pt}}\\{#1}\end{smallmatrix}}
\newcommand{\secfsma}[1]{\overset{\;-\alpha}{A}\hspace{-2pt}\begin{smallmatrix}
{\rule{0pt}{4pt}}\\{#1}\end{smallmatrix}}
\newcommand{\cristoffe}[2]{\overset{\:e}{\varGamma}\hspace{-2pt}\begin{smallmatrix}{#2}\\{#1}\end{smallmatrix}}
\newcommand{\cristofse}[2]{\overset{\:e}{\varGamma}\hspace{-2pt}\begin{smallmatrix}{\rule{0pt}{4pt}}\\{#1,#2}\end{smallmatrix}}
\newcommand{\cristofte}[2]{\overset{\:e}{\varGamma}\hspace{-2pt}\begin{smallmatrix}{#1,#2}\\{\rule{0pt}{5pt}}\end{smallmatrix}}
\newcommand{\cristoffm}[2]{\overset{\:\scalebox{0.6}{$m$}}{\varGamma}\hspace{-2pt}\begin{smallmatrix}{#2}\\{#1}\end{smallmatrix}}
\newcommand{\cristofsm}[2]{\overset{\:\scalebox{0.6}{$m$}}{\varGamma}\hspace{-2pt}\begin{smallmatrix}{\rule{0pt}{4pt}}\\{#1,#2}\end{smallmatrix}}
\newcommand{\cristoftm}[2]{\overset{\:\scalebox{0.6}{$m$}}{\varGamma}\hspace{-2pt}\begin{smallmatrix}{#1,#2}\\{\rule{0pt}{5pt}}\end{smallmatrix}}
\newcommand{\cristoffa}[2]{\overset{\:\alpha}{\varGamma}\hspace{-2pt}\begin{smallmatrix}{#2}\\{#1}\end{smallmatrix}}
\newcommand{\cristofsa}[2]{\overset{\:\alpha}{\varGamma}\hspace{-2pt}\begin{smallmatrix}{\rule{0pt}{4pt}}\\{#1,#2}\end{smallmatrix}}
\newcommand{\cristofta}[2]{\overset{\:\alpha}{\varGamma}\hspace{-2pt}\begin{smallmatrix}{#1,#2}\\{\rule{0pt}{5pt}}\end{smallmatrix}}
\newcommand{\cristoffma}[2]{\overset{\:-\alpha}{\varGamma}\hspace{-2pt}\begin{smallmatrix}{#2}\\{#1}\end{smallmatrix}}
\newcommand{\cristofsma}[2]{\overset{\:-\alpha}{\varGamma}\hspace{-2pt}\begin{smallmatrix}{\rule{0pt}{4pt}}\\{#1,#2}\end{smallmatrix}}

\allowdisplaybreaks[4]

\newtheorem{lemma}{Lemma}
\newtheorem{theo}{Theorem}

\newtheorem{coro}{Corollary}

\makeatletter
  \def\mathcomposite{%
     \@ifstar
        {\def\@mathcomposite@option{%
            \baselineskip\z@skip\lineskiplimit-\maxdimen}%
         \@mathcomposite}%
        {\let\@mathcomposite@option\offinterlineskip
         \@mathcomposite}}
  \def\@mathcomposite{%
     \@ifnextchar[\@@mathcomposite{\@@mathcomposite[0]}}
  \def\@@mathcomposite[#1]#2#3#4{%
     #2{\mathchoice
        {\@mathcomposite@{#1}{#3}{#4}\displaystyle{1}}%
        {\@mathcomposite@{#1}{#3}{#4}\textstyle{1}}%
        {\@mathcomposite@{#1}{#3}{#4}%
         \scriptstyle\defaultscriptratio}%
        {\@mathcomposite@{#1}{#3}{#4}%
         \scriptscriptstyle\defaultscriptscriptratio}}}
  \def\@mathcomposite@#1#2#3#4#5{%
     \vcenter{\m@th\@mathcomposite@option
        \dimen@\f@size\p@\dimen@#1\dimen@\dimen@#5\dimen@
        \divide\dimen@ 18
        \edef\@mathcomposite@skipamount{\the\dimen@}%
        \ialign{\hfil$#4##$\hfil\cr
           #2\crcr
           \noalign{\vskip\@mathcomposite@skipamount}%
           #3\crcr}}}
  \makeatother

\makeatletter
\newcommand{\oset}[2]{%
  {\mathop{#2}\limits^{\vbox to -.5\ex@{\kern-\tw@\ex@
   \hbox{\scriptsize #1}\vss}}}}
\makeatother

\usepackage{microtype} 

\usepackage[utf8]{inputenc} 
\usepackage[T1]{fontenc} 

\title {Asymptotic expansion of the risk of maximum likelihood estimator with respect to $\alpha$-divergence as a measure of the difficulty of specifying a parametric model --with detailed proof--}
\author{Yo Sheena\thanks{Faculty of Economics, Shinshu University; Faculty of Data Science, Shiga University}}

\date{Sep. 2017}

\begin{document}
\maketitle

\begin{abstract}
For a given parametric probability model, we consider the risk of the maximum likelihood estimator with respect to $\alpha$-divergence, which includes the special cases of Kullback--Leibler divergence, the Hellinger distance and  $\chi^2$ divergence. The asymptotic expansion of the risk is given with respect to sample sizes of up to order $n^{-2}$. Each term in the expansion is expressed with the geometrical properties of the Riemannian manifold formed by the parametric probability model. We attempt to measure the difficulty of specifying a model through this expansion. 
\end{abstract}
\noindent
MSC(2010) \textit{Subject Classification}: Primary 60F99; Secondary 62F12
\section{Introduction}
\label{section:int}
For parametric models of probability distributions, we are naturally concerned with the following comparisons:\\
1. Comparison of the risks of estimation among different parameters within a same model. \\
2. Overall comparison of the risks between different models. 

To carry out these comparisons, particularly the second, we need some way of measuring the risk of estimation that is common to all parametric probability models. The maximum likelihood method, in which the maximum likelihood estimator (m.l.e.) is plugged into the unknown parameters, is the most common approach, and is applicable to any parametric model. In this paper, we choose the m.l.e., and use its risk with respect to a certain loss function.

For comparisons above, particularly the first, the loss function should be independent of the choice of the parameter (coordinate). Consider the binomial distribution model $B(n, p)$, where $n$ is known and we wish to estimate $p$. Take the quadratic loss function as an example of a parameter-dependent loss function:
$$
L(\hat{p},p)=(\hat{p}-p)^2.
$$
The m.l.e. $\hat{p}$ is the sample ratio. The risk of the m.l.e. with respect to this quadratic loss function is given by 
$$
E_p[(\hat{p}-p)^2]=V(\hat{p})=p(1-p)/n,
$$
which says that the estimation of this model reaches the highest risk point when $p=1/2$. In contrast, if we use $p^{-1}$ as a parameter for the model, the risk of the m.l.e. is given by 
$$
E_{p^{-1}}[(\hat{p}^{-1}-p^{-1})^2]=p^{-3}(1-p)/n+o(n^{-1}).
$$
For a large sample size, the highest risk point is $p=0$.

Another consideration for the loss function is the invariance with respect to the transformation between the random variables. If $Y(X)$ is a sufficient statistic for the parametric model of a random object $X$, then the risk of estimation should be measured independently of the choice of $Y$ and $X$. In particular, when $X  \leftrightarrow Y$ is a one-to-one transformation, the parametric models for the distribution of $X$ and $Y$ are essentially equivalent. We wish to measure the risk of estimating the model without being concerned by the form in which the observations were acquired. 

Taking these considerations into account, \textit{f-divergence} is a natural loss function, as it satisfies both parameter independence and transformation invariance. It also has other favorable properties (see, e.g., Chapter 9 of Vajda \cite{Vajda} ), and has been widely used in engineering problems (e.g.  \cite{Maji}, \cite{Qiao&Minematsu}). Because f-divergence is quite a general class of divergence, we need more specific forms for a concrete result. Here, we focus on \textit{$\alpha$-divergence}. This is a subclass of f-divergence, but is still a general class of divergence, and includes the well-known Kullback--Leibler divergence ($\alpha=-1$), the Hellinger distance ($\alpha=0$). It is also equivalent, when $\alpha=-3$, to $\chi^2$-divergence in a small neighborhood of the point where the two probabilities coincide. 
More importantly, from the perspective of information geometry, $\alpha$-divergence gives rise to a ``dually flat'' structure for the manifold of the given parametric model (see Eguchi \cite{Eguchi1}, Amari \cite{Amari3}, and Amari and Cichocki \cite{Amari&Cichocki}). 

Consequently, in this paper, we consider the risk of the m.l.e. with respect to $\alpha$-divergence in estimating the parameter. However, an exact calculation of the risk of the m.l.e is often quite difficult, even if we choose a mathematically easy-to-handle quantity such as Kullback--Leibler divergence. Hence, it would be useful if we had an asymptotic expansion of the risk with respect to the sample size $n$. The main concern of this paper is to give this expansion for a general $\alpha$ up to order $n^{-2}$. The result is expressed with the geometrical properties of the Riemannian manifold formed by the parametric model. The geometric expression of the risk expansion provides an insight into which geometrical property of the model affects the estimation risk. As examples, we calculate the asymptotic risk expansion for some specific families of distributions such as an exponential family and a mixture family. 

The most relevant work to this paper is that of Komaki \cite{Komaki} and Corcuera and Giummol\`e  \cite{Corcuera&Giummole}. Both of these papers gave asymptotic expansions of the risk of \textit{predictive distributions} with respect to Kullback--Leibler divergence (Theorem 1 of \cite{Komaki}) and $\alpha$-divergence (Theorem 3.1 of  \cite{Corcuera&Giummole}). They considered a curved exponential family as the presupposed model, and developed a general method for the asymptotic improvement of an \textit{estimative distribution} (that is, a distribution within the model gained by parameter estimation) using a predictive distribution that belongs to a full exponential family but lies outside of the model. Note that the whole class of predictive distributions includes the estimative distributions as special cases, and the distribution with the estimated parameter by the m.l.e. is a typical estimative distribution. Hence, \cite{Komaki} and  \cite{Corcuera&Giummole} treat a more general class of estimation procedure than m.l.e. under the framework of curved exponential families. This paper differs in two points from these papers. First, we do not presuppose the exponential families. We gain the main result, Theorem 1 in Section 2,  for a general parametric family. Second, we present all the results through geometrical terms, which enables us to understand how the geometrical structure of the model is related to the risk of model estimation.

We now state the formal framework of the problem. 
First, we consider a parametric family of probability distributions on a space $\mathfrak{X}$ (say, $\mathcal P$), which is given by a family of positive-valued densities  $f(x; \theta)$ on $\mathfrak{X}$ with respect to a measure $\mu$:
\begin{equation}
\mathcal P=\{f(x; \theta)\: | \: \theta=(\theta^1, \dots, \theta^p), \theta \in \Theta \},
\end{equation}
where $\Theta$ is an open set in $R^p$, and $f(x;\theta_1)=f(x;\theta_2)$ almost everywhere if and only if $\theta_1=\theta_2$. We will treat  $\mathcal{P}$ as a Riemannian manifold, and define several geometrical properties on it.

$\alpha$-divergence $(-\infty < \alpha < \infty)$  between $f(x; \theta_1)$ and $f(x; \theta_2)$ is defined as 
\begin{equation}
\label{alphadive}
\overset{\alpha}{D}[\theta_1: \theta_2]=
\begin{cases}
\frac{4}{1-\alpha^2}\Bigl\{ 1- \int_{\mathfrak X}f^{(1-\alpha)/2}(x; \theta_1)  f^{(1+\alpha)/2}(x; \theta_2) d\mu \Bigr\}, & \text{ if $\alpha \ne \pm 1,$}\\
\int_{\mathfrak X}f(x; \theta_2) \log \Bigl(f(x; \theta_2)/f(x; \theta_1) \Bigr)d\mu, & \text{ if  $\alpha=1$,}\\
\int_{\mathfrak X}f(x; \theta_1) \log \Bigl(f(x; \theta_1)/f(x; \theta_2) \Bigr) d\mu, & \text{ if $\alpha=-1$.}
\end{cases}
\end{equation}
As a general divergence property, this satisfies $\overset{\alpha}{D}[\theta_1: \theta_2] \geq 0$, where the equality holds if and only if $\theta_1=\theta_2$. It is not symmetric between $\theta_1$ and $\theta_2$. Actually, the following relation holds:
$$
\overset{\alpha}{D}[\theta_1: \theta_2]=\overset{-\alpha}{D}[\theta_2: \theta_1].
$$

If we adopt $\alpha$-divergence as a loss function between the m.l.e. $\hat\theta$ and the true parameter $\theta$, the risk of the m.l.e. is given by
$$
\overset{\alpha}{E\!D}(\theta)\triangleq E_\theta\bigl[\overset{\alpha}{D}[\hat\theta(\bm{X}):\theta]\bigr],
$$
where $\bm{X}=(X_1,\ldots,X_n)$ are $n$ independent random samples from the distribution given by $f(x;\theta)$. 

The remainder of this paper is organized as follows. In Section 2, we give the main result (Theorem 1) and its corollaries for some distribution families and specific values of $\alpha$. In Section 3, we give examples of the adaptation of the main result to some specific distribution. The minimum knowledge of the information geometry and some preliminary results used for the derivation for Theorem 1 are presented in Appendix 1. The lengthy calculation of geometrical properties and the proof of some results are stated Appendix 2. 

Before closing this section, we state some technical conditions and notation used throughout this paper. 
We assume that $f(x;\theta)$ is differentiable at least five times with respect to $\theta$, and that differentiation and integration on $\mathfrak{X}$ is always exchangeable. Every expectation that appears in this context is assumed to be finite. (We refer to these conditions as "C.1".)
The following notation is used:
\begin{align*}
&\partial_i \triangleq \frac{\partial\ \ }{\partial \theta^i}\\
&f_i\triangleq f_i(x; \theta) \triangleq \partial_i f(x; \theta), \quad f_{ij}\triangleq f_{ij}(x; \theta)\triangleq \partial_i \partial_j f(x;\theta),\quad \cdots \\
&l_i \triangleq l_i(x; \theta) \triangleq \partial_i \log f(x ;\theta), \quad l_{ij}\triangleq l_{ij}(x; \theta) \triangleq \partial_i \partial_j \log f(x ;\theta), \quad \cdots \\
&E_\theta[h(x;\theta)] \triangleq \int_{\mathfrak X} h(x; \theta) f(x; \theta) d\mu, \\
&\sum_{i}\triangleq \sum_{1\leq i \leq p},\quad \sum_{i,j}\triangleq \sum_{1 \leq i, j \leq p},\quad \cdots
\end{align*}
\section{Asymptotic Expansion of the Risk of M.L.E.}
\label{section:largesam}
\quad In this section we consider the asymptotic expansion of the risk of the maximum likelihood estimator (m.l.e.) with respect to $\alpha$-divergence. 

Let $\bm{X}=(X_1,\ldots,X_n)$ be $n$ independent random sample from the distribution given by $f(x;\theta)$. The m.l.e. of $\theta$ based on $X$ is denoted by $\hat\theta(X)$. The expected divergence at $\theta$ between $f(x;\hat\theta(\bm{X}))$ and $f(x;\theta)$ is given by
$$
\overset{\alpha}{E\!D}(\theta)\triangleq E_\theta\bigl[\overset{\alpha}{D}[\hat\theta(\bm{X}):\theta]\bigr].
$$
Using the expansion of divergence \eqref{expansion_divergence}  in Section \ref{section:expan} , then we have
\begin{align}
&\overset{\alpha}{E\!D}(\theta)\nonumber\\
&=\wai \bigl(\epsilon_i \overset{\alpha}{D}[\theta: \theta]\bigr)E_\theta[(\hat{\theta}^i-\theta^i)] +
\frac{1}{2}\wai \waj \bigl(\epsilon_i \epsilon_j \overset{\alpha}{D}[\theta: \theta]\bigr)E_\theta[(\hat{\theta}^i-\theta^i) (\hat{\theta}^j-\theta^j)]\nonumber\\
&\ +\frac{1}{6}\wai \waj \wak \bigl(\epsilon_i \epsilon_j \epsilon_k \overset{\alpha}{D}[\theta: \theta]\bigr)E_\theta[(\hat{\theta}^i-\theta^i) (\hat{\theta}^j-\theta^j)(\hat{\theta}^k-\theta^k)]\nonumber\\
&\ +\frac{1}{24}\wai \waj \wak \wal \bigl(\epsilon_i \epsilon_j \epsilon_k \epsilon_l \overset{\alpha}{D}[\theta: \theta]\bigr)E_\theta[(\hat{\theta}^i-\theta^i) (\hat{\theta}^j-\theta^j)(\hat{\theta}^k-\theta^k)(\hat{\theta}^l-\theta^l)]\nonumber \\
&\ +E_\theta[O(||\hat{\theta}-\theta||^5)]. \label{expan_ED}
\end{align}
As we see in Section \ref{section:expan}, the terms
$$
\epsilon_i \overset{\alpha}{D}[\theta: \theta],\  \epsilon_i \epsilon_j \overset{\alpha}{D}[\theta: \theta],\ \epsilon_i \epsilon_j \epsilon_k \overset{\alpha}{D}[\theta: \theta],\ \epsilon_i \epsilon_j \epsilon_k \epsilon_l \overset{\alpha}{D}[\theta: \theta]
$$
could be  geometrically interpreted. In this section, we will show that the terms
\begin{equation}
\label{target_E}
\begin{split}
&E_\theta[(\hat{\theta}^i-\theta^i) (\hat{\theta}^j-\theta^j)],\quad
E_\theta[(\hat{\theta}^i-\theta^i) (\hat{\theta}^j-\theta^j)(\hat{\theta}^k-\theta^k)],\\
&E_\theta[(\hat{\theta}^i-\theta^i) (\hat{\theta}^j-\theta^j)(\hat{\theta}^k-\theta^k)(\hat{\theta}^l-\theta^l)]
\end{split}
\end{equation}
are also interpretable with the geometrical properties of $\mathcal{P}$.

First we evaluate  $\barteta{i}\triangleq\hat{\theta}^i-\theta^i$ following the way of Eguchi and Yanagimoto \cite{Eguchi&Yanagimoto}.
Let
$$
\bareu{i}(\xx;\theta)\triangleq \frac{1}{n}\sum_{a=1}^n \frac{\partial }{\partial\theta^i}\log f(X_a;\theta),\qquad \bareo{i}(\xx;\theta)\triangleq \waj g^{ij} \bareu{j}(\xx;\theta)
$$
for $1\leq i \leq p$.
Since m.l.e. $\hat\theta$ maximizes $\log$-likelihood $\sum_{a=1}^n \log f(x_a;\theta)$
\begin{equation}
\label{bareu=0}
\bareu{i}(\xx; \hat{\theta})=0.
\end{equation}
Taylor expansion of $\bareu{i}(\xx; \hat\theta)$ around $\theta$ is given by
\begin{align*}
&\bareu{i}(\xx;\hat\theta)\\
&=\bareu{i}(\xx;\theta)+\sum_{j}\frac{\partial\bareu{i}(\xx;\theta)}{\partial\theta^j}\barteta{j}+\frac{1}{2!}\sum_{j,k}\frac{\partial^2 \bareu{i}(\xx;\theta)}{\partial \theta^j \partial\theta^k}\barteta{j}\barteta{k}\\
&\qquad+\frac{1}{3!}\sum_{j,k,l}\frac{\partial^3 \bareu{i}(\xx;\theta)}{\partial \theta^j \partial\theta^k \partial\theta^l}\barteta{j}\barteta{k}\barteta{l}+\frac{1}{4!}\sum_{j,k,l,m}\frac{\partial^4 \bareu{i}(\xx;\theta_i^*)}{\partial \theta^j \partial\theta^k \partial\theta^l \partial\theta^m}
\barteta{j}\barteta{k}\barteta{l}\barteta{m},
\end{align*}
where $\theta_i^*$ is on the segment $\overline{\theta \hat{\theta}}$.  If we add $\sum_j g_{ij}(\theta)\barteta{j}$ on the both sides of the above expansion and use \eqref{bareu=0}, then we have
\begin{align*}
&\sum_j g_{ij}(\theta)\barteta{j}\\
&=\bareu{i}(\xx;\theta)+\sum_{j}\biggl(\frac{\partial\bareu{i}(\xx;\theta)}{\partial\theta^j}+g_{ij}(\theta)\biggr)\barteta{j}+\frac{1}{2!}\sum_{j,k}\frac{\partial^2 \bareu{i}(\xx;\theta)}{\partial \theta^j \partial\theta^k}\barteta{j}\barteta{k}\\
&\qquad+\frac{1}{3!}\sum_{j,k,l}\frac{\partial^3 \bareu{i}(\xx;\theta)}{\partial \theta^j \partial\theta^k \partial\theta^l}\barteta{j}\barteta{k}\barteta{l}+\frac{1}{4!}\sum_{j,k,l,m}\frac{\partial^4 \bareu{i}(\xx;\theta_i^*)}{\partial \theta^j \partial\theta^k \partial\theta^l \partial\theta^m}
\barteta{j}\barteta{k}\barteta{l}\barteta{m}.
\end{align*}
Furthermore if we multiply the both sides with $g^{is}$ and sum them up over $i$ from 1 to $p$, then  we have
\begin{align}
\barteta{s}&=\bareo{s}+\sum_j \aaa{j}{s}\barteta{j}+\sum_{j,k}\bee{jk}{s}\barteta{j}\barteta{k}+\sum_{j,k}\bbar{jk}{s}\barteta{j}\barteta{k}+
\sum_{j,k,l}\cee{jkl}{s}\barteta{j}\barteta{k}\barteta{l}+\sum_{jkl}\cbar{jkl}{s}\barteta{j}\barteta{k}\barteta{l}\nonumber\\
&\qquad+\sum_{j,k,l,m}\dee{jklm}{s}\barteta{j}\barteta{k}\barteta{l}\barteta{m},
\label{expand_barteta}
\end{align}
where we used the following notations: For $1\leq j, k, l, m, s \leq p,$
\begin{align*}
\aaa{j}{s}(\xx;\theta)&\triangleq \wai g^{is}(\theta) \biggl(\frac{\partial \bareu{i}(\xx;\theta)}{\partial \theta^j}+g_{ij}(\theta)\biggr)\\
\bee{jk}{s}(\xx;\theta)&\triangleq \frac{1}{2}\wai g^{is}(\theta)\biggl(\frac{\partial^2 \bareu{i}(\xx;\theta)}{\partial \theta^j \partial\theta^k}-E_\theta \biggl[\frac{\partial^2 \bareu{i}(\xx;\theta)}{\partial \theta^j \partial\theta^k}\biggr]\biggr)\\
\bbar{jk}{s}(\theta)&\triangleq \frac{1}{2}\wai  g^{is}(\theta) E_\theta \biggl[\frac{\partial^2 \bareu{i}(\xx;\theta)}{\partial \theta^j \partial\theta^k}\biggr]\\
\cbar{jkl}{s}(\theta)&\triangleq \frac{1}{3!}\wai  g^{is}(\theta) E_\theta \biggl[\frac{\partial^3 \bareu{i}(\xx;\theta)}{\partial \theta^j \partial\theta^k \partial\theta^l}\biggr]\\
\cee{jkl}{s}(\xx;\theta)&\triangleq \frac{1}{3!}\sum_i g^{is}(\theta)\biggl(\frac{\partial^3 \bareu{i}(\xx;\theta)}{\partial \theta^j \partial\theta^k \partial\theta^l}-E_\theta \biggl[\frac{\partial^3 \bareu{i}(\xx;\theta)}{\partial \theta^j \partial\theta^k \partial\theta^l}\biggr]\biggr),\\
\dee{jklm}{s}(\xx;\theta)&\triangleq \frac{1}{4!}\sum_i  g^{is}(\theta) \frac{\partial^4 \bareu{i}(\xx;\theta_i^*)}{\partial \theta^j \partial\theta^k \partial\theta^l \partial\theta^m}.
\end{align*}
Here we impose the moment conditions as follows. The suitably higher-order joint moments composed of the following variables are bounded with respect to $n$;
\begin{equation}
\label{comp_of_residuals}
\begin{split}
&\sqrt{n}\,\barteta{s}(\xx;\theta),\quad\sqrt{n}\,\bareo{s}(\xx;\theta),\quad\sqrt{n}\,\aaa{j}{s}(\xx;\theta),\quad \sqrt{n}\,\bee{jk}{s}(\xx;\theta), \\ &\sqrt{n}\,\cee{jkl}{s}(\xx;\theta),\quad \dee{jklm}{s}(\xx;\theta),
\end{split}
\end{equation}
where $1\leq j,  k,  l,  m,  s \leq p.$ (We refer to this condition as "C.2".)
We obtain the following results (for the proof, see Appendix 2). For $1\leq s \leq p$, 
\begin{equation}
\label{expan_barteta}
\begin{split}
\barteta{s}=&\bareo{s}+\sum_{j}\aaa{j}{s}\bareo{j}+\sum_{j,k}\bbar{jk}{s}\bareo{j}\bareo{k}+\sum_{i,j}\aaa{j}{s}\aaa{i}{j}\bareo{i}+\sum_{i,j,k}\aaa{j}{s}\bbar{ik}{j}\bareo{i}\bareo{k}\\
&\ +\sum_{j,k}\bee{jk}{s}\bareo{j}\bareo{k}+2\sum_{i,j,k}\bbar{jk}{s}\aaa{i}{k}\bareo{i}\bareo{j}+2\sum_{i,j,k,l}\bbar{jk}{s}\bbar{il}{k}\bareo{i}\bareo{j}\bareo{l}\\
&\ +\sum_{j,k,l}\cbar{jkl}{s}\bareo{j}\bareo{k}\bareo{l}+Re(4),
\end{split}
\end{equation}
where $Re(4)$ is the polynomial  with respect to the variables $\barteta{s}$, $\bareo{s}$, $\aaa{j}{s}$, $\bee{jk}{s}$, $\cee{jkl}{s}$, $\dee{jklm}{s}$ $(1\leq j, k, l, m, s \leq p)$, and each term is of at least fourth order with respect to $\barteta{s}$, $\bareo{s}$, $\aaa{j}{s}$, $\bee{jk}{s}$, $\cee{jkl}{s}$ $(1\leq j, k, l, s \leq p)$.

Combining this evaluation with Lemma \ref{int_log_der} in Section \ref{interpre_log}, we can express three expectations in \eqref{target_E} with geometrical terms (for the detailed calculation, see Appendix 2 ). The results are given as follows. For achieving notational brevity, we use Einstein summation convention (the summation is carried out as every pair of upper and lower index moves from 1 to $p$).
\begin{equation}
\label{expec_barteta^2}
\begin{split}
&E_\theta[(\hat{\theta}^i-\theta^i) (\hat{\theta}^j-\theta^j)]\\
&=n^{-1}g^{ij}+n^{-2}\times \\
&\biggl[
g^{sj}g^{it}g^{lm}\langle \secfse{sl}, \secfsm{tm}-\secfse{tm}\rangle
+g^{si}g^{jt}g^{lm}\langle \secfse{sl}, \secfsm{tm}-\secfse{tm}\rangle\\
&\ +g^{sj}g^{it}g^{lm}\cristofse{sl}{u}\bigl(\cristoffm{tm}{u}-\cristoffe{tm}{u}\bigr)
+g^{si}g^{jt}g^{lm}\cristofse{sl}{u}\bigl(\cristoffm{tm}{u}-\cristoffe{tm}{u}\bigr)\\
&\ +\bbar{lm}{j}g^{ik}g^{ls}g^{mt}\bigl(\cristofsm{ks}{t}-\cristofse{ks}{t}\bigr)
+\bbar{lm}{i}g^{jk}g^{ls}g^{mt}\bigl(\cristofsm{ks}{t}-\cristofse{ks}{t}\bigr)\\
&\ +g^{jk}g^{lu}g^{is}g^{mt}\bigl(\langle \secfse{kl}, \secfse{um}\rangle g_{st}
+\cristoffe{kl}{v}\cristoffe{um}{w}g_{vw}g_{st}+\cristofse{kl}{s}\cristofse{um}{t}+\cristofse{kl}{t}\cristofse{um}{s}\bigr)\\
&\ +g^{ik}g^{lu}g^{js}g^{mt}\bigl(\langle \secfse{kl}, \secfse{um}\rangle g_{st}
+\cristoffe{kl}{v}\cristoffe{um}{w}g_{vw}g_{st}+\cristofse{kl}{s}\cristofse{um}{t}+\cristofse{kl}{t}\cristofse{um}{s}\bigr)\\
&\ +\bbar{ms}{l}g^{jk}g^{it}g^{mu}g^{sv}\bigl(\cristofse{kl}{t}g_{uv}+\cristofse{kl}{u}g_{tv}+\cristofse{kl}{v}g_{tu}\bigr)\\
&\ +\bbar{ms}{l}g^{ik}g^{jt}g^{mu}g^{sv}\bigl(\cristofse{kl}{t}g_{uv}+\cristofse{kl}{u}g_{tv}+\cristofse{kl}{v}g_{tu}\bigr)\\
&\ +\frac{1}{2}g^{js}g^{it}g^{lu}g^{mv}\bigl\{
\bigl((\partial_m\cristoffe{sl}{k})g_{tk}+\cristoffe{sl}{k}\cristofse{km}{t}-\langle \secfse{sl}, \secfsm{mt}\rangle \bigr)g_{uv}\\
&\hspace{30mm}+\bigl((\partial_m\cristoffe{sl}{k})g_{uk}+\cristoffe{sl}{k}\cristofse{km}{u}-\langle \secfse{sl}, \secfsm{mu}\rangle \bigr)g_{tv}\\
&\hspace{30mm}+\bigl((\partial_m\cristoffe{sl}{k})g_{vk}+\cristoffe{sl}{k}\cristofse{km}{v}-\langle \secfse{sl}, \secfsm{mv}\rangle \bigr)g_{ut}\bigr\}\\
&\ +\frac{1}{2}g^{is}g^{jt}g^{lu}g^{mv}\bigl\{
\bigl((\partial_m\cristoffe{sl}{k})g_{tk}+\cristoffe{sl}{k}\cristofse{km}{t}-\langle \secfse{sl}, \secfsm{mt}\rangle \bigr)g_{uv}\\
&\hspace{30mm}+\bigl((\partial_m\cristoffe{sl}{k})g_{uk}+\cristoffe{sl}{k}\cristofse{km}{u}-\langle \secfse{sl}, \secfsm{mu}\rangle \bigr)g_{tv}\\
&\hspace{30mm}+\bigl((\partial_m\cristoffe{sl}{k})g_{vk}+\cristoffe{sl}{k}\cristofse{km}{v}-\langle \secfse{sl}, \secfsm{mv}\rangle \bigr)g_{ut}\bigr\}\\
&\ +2\bbar{lm}{j}g^{mt}g^{iu}g^{lv}g^{sw}\bigl(\cristofse{st}{u}g_{vw}+\cristofse{st}{v}g_{uw}+\cristofse{st}{w}g_{uv}\bigr)\\
&\ +2\bbar{lm}{i}g^{mt}g^{ju}g^{lv}g^{sw}\bigl(\cristofse{st}{u}g_{vw}+\cristofse{st}{v}g_{uw}+\cristofse{st}{w}g_{uv}\bigr)\\
&\ +2\bbar{lm}{j}\bbar{st}{m}g^{ik}g^{lu}g^{sv}g^{tw}\bigl(g_{ku}g_{vw}+g_{kv}g_{uw}+g_{kw}g_{uv}\bigr)\\
&\ +2\bbar{lm}{i}\bbar{st}{m}g^{jk}g^{lu}g^{sv}g^{tw}\bigl(g_{ku}g_{vw}+g_{kv}g_{uw}+g_{kw}g_{uv}\bigr)\\
&\ +\cbar{lmt}{j} g^{ik} g^{ls} g^{mu} g^{tv}\bigl( g_{ks}g_{uv}+g_{ku}g_{sv}+g_{kv}g_{su}\bigr)\\
&\ +\cbar{lmt}{i} g^{jk} g^{ls} g^{mu} g^{tv}\bigl( g_{ks}g_{uv}+g_{ku}g_{sv}+g_{kv}g_{su}\bigr)\\
&\ +g^{ik} g^{js} g^{lt} g^{mu}\bigl( \langle \secfse{kl}, \secfse{sm}\rangle g_{tu}+
\cristoffe{kl}{v}\cristoffe{sm}{w}g_{vw}g_{tu}+\cristofse{kl}{t}\cristofse{sm}{u}+\cristofse{kl}{u}\cristofse{sm}{t}\bigr)\\
&\ +\bbar{st}{j} g^{ik}g^{lu}g^{sv}g^{tw}\bigl(\cristofse{kl}{u}g_{vw}+\cristofse{kl}{v}g_{uw}+\cristofse{kl}{w}g_{uv}\bigr)\\
&\ +\bbar{st}{i} g^{jk}g^{lu}g^{sv}g^{tw}\bigl(\cristofse{kl}{u}g_{vw}+\cristofse{kl}{v}g_{uw}+\cristofse{kl}{w}g_{uv}\bigr)\\
&\ +\bbar{lm}{i}\bbar{st}{j}g^{lk}g^{mu}g^{sv}g^{tw}\bigl(g_{ku}g_{vw}+g_{kv}g_{uw}+g_{kw}g_{uv}\bigr)\biggr]+O(n^{-5/2}).
\end{split}
\end{equation}
\begin{equation}
\label{expec_barteta^3}
\begin{split}
&E_\theta[(\hat{\theta}^i-\theta^i) (\hat{\theta}^j-\theta^j)(\hat{\theta}^k-\theta^k)]\\
&=n^{-2}\bigl\{g^{is}g^{jt}g^{ku}\bigl(\cristofsm{st}{u}-\cristofse{st}{u}\bigr)\\
&\ +\cristoffe{st}{s}g^{it}g^{jk}+\cristoffe{st}{j}g^{sk}g^{it}+\cristoffe{st}{k}g^{sj}g^{it}\\
&\ +\cristoffe{st}{s}g^{jt}g^{ik}+\cristoffe{st}{i}g^{sk}g^{jt}+\cristoffe{st}{k}g^{si}g^{jt}\\
&\ +\cristoffe{st}{s}g^{kt}g^{ji}+\cristoffe{st}{j}g^{si}g^{kt}+\cristoffe{st}{i}g^{sj}g^{kt}\\
&\ +\bbar{st}{i}\bigl(g^{st} g^{jk}+g^{js}g^{kt}+g^{ks}g^{jt}\bigr)\\
&\ +\bbar{st}{j}\bigl(g^{st} g^{ik}+g^{is}g^{kt}+g^{ks}g^{it}\bigr)\\
&\ +\bbar{st}{k}\bigl(g^{st} g^{ji}+g^{js}g^{it}+g^{is}g^{jt}\bigr)
\bigr\}+O(n^{-5/2}).
\end{split}
\end{equation}
\begin{equation}
\label{expec_barteta^4}
E_\theta[(\hat{\theta}^i-\theta^i) (\hat{\theta}^j-\theta^j)(\hat{\theta}^k-\theta^k)(\hat{\theta}^l-\theta^l)]
=n^{-2}\bigl(g^{ij}g^{kl}+g^{ik}g^{jl}+g^{il}g^{jk}\bigr)+O(n^{-5/2}).
\end{equation}
From \eqref{Egur1}--\eqref{Egur3}, \eqref{epsi^4_Diverge} in Section \ref{section:expan} and \eqref{expec_barteta^2}--\eqref{expec_barteta^4}, we obtain the following result (for the detailed calculation, see  Appendix 2).
\begin{theo}
Under the conditions C.1 and C.2, the following expansion holds.
\begin{align}
&\overset{\alpha}{E\!D} \nonumber\\
&=\frac{p}{2n}+\frac{1}{24n^{^2}}\nonumber\\
&\times
\biggl[ 
(\alpha')^2\bigl\{3\overset{\:e}{F}+3T^{ijk}T_{ijk}-6\langle \secffe{i}{j}, (\secffm{j}{i}-\secffe{j}{i})\rangle-3\langle \secffe{i}{i}, (\secffm{j}{j}-\secffe{j}{j})\rangle+3p^2+6p\bigr\}\nonumber\\
&\hspace{8mm}+
\alpha'\bigl\{3\overset{\:e}{F}-5T^{ijk}T_{ijk}-6T_{is}^iT_j^{js}+6\langle \secffe{i}{j}, (\secffm{j}{i}-\secffe{j}{i})\rangle+3\langle \secffe{i}{i}, (\secffm{j}{j}-\secffe{j}{j})\rangle \nonumber\\
&\hspace{20mm}-3p^2-6p\bigr\}\nonumber\\
&\hspace{8mm}
+12\langle \secffe{j}{i}, \secffe{i}{j} \rangle -2\langle \secffe{j}{i}, \secffm{i}{j} \rangle -\langle \secffe{i}{i}, \secffm{j}{j} \rangle+T_{ijk}T^{ijk}+9T_{is}^i T^{js}_j+8\overset{\;e}{R}\hspace{-10pt}\begin{smallmatrix}{\ \  \ \ ij}\\{ij}\end{smallmatrix}-9\overset{\:e}{F}
\biggr]\nonumber\\
&+o(n^{-2}), \label{expan_ED_final}
\end{align}
where  $\alpha'=(1-\alpha)/2$, and for $1 \leq i, j, k, l \leq p$,
\begin{align}
&\Bigl(g^{ij}(\theta)\Bigr)\triangleq \Bigl(g_{ij}(\theta)\Bigr)^{-1},\label{notation_Fisher_info_and_its_inverse}\\
&T_{ijk}(\theta)\triangleq\cristofsm{ij}{k}(\theta)-\cristofse{ij}{k}(\theta)\label{notation_T_ijk}\\
&T_{ij}^k(\theta)\triangleq \sum_jT_{ijl}(\theta)g^{lk}(\theta), \label{def_T_{ij}^k}\\
&T_i^{jk}(\theta)\triangleq \sum_{l,m}T_{iml}(\theta)g^{mj}(\theta)g^{lk}(\theta), \label{def_T_i^{jk}}\\
&T^{ijk}(\theta)\triangleq \sum_{t,m,l}T_{tml}(\theta)g^{ti}(\theta)g^{mj}(\theta)g^{lk}(\theta),\label{def_T^{ijk}}\\
&\overset{\;e}{R}\hspace{-10pt}\begin{smallmatrix}{\ \  \ \ kl}\\{ij}\end{smallmatrix}(\theta)\triangleq \sum_s\overset{\;e}{R}\hspace{-8pt}\begin{smallmatrix}{\ \  \ \ \ l}\\{ijs}\end{smallmatrix}(\theta)g^{sk}(\theta),\label{def_R_{ij}^{kl}}\\
&\secffe{i}{j}\triangleq \sum_k \secfse{ik}g^{kj},\qquad \secffm{i}{j}\triangleq \sum_k \secfsm{ik}g^{kj}, \label{notation_secffe&m}\\
&\overset{\:\alpha}{F}(\theta)\triangleq \sum_{i,k,s}g^{ks}(\theta)\partial_s T_{ik}^i(\theta)-\sum_{i,j,s,t}g^{ti}(\theta)\cristofsa{it}{s}(\theta) T_j^{js}(\theta),\quad \overset{\:e}{F}(\theta)\triangleq \overset{\;\scalebox{0.6}{$1$}}{F}(\theta),\label{def_F}
\end{align}
which are the variations from the following fundamental geometric properties of Riemannian manifold $\mathcal{P}$ (for their formal definitions, see Section \ref{basic_con_info_geo});
\begin{align*}
&g_{ij}(\theta): \text{Fisher information metric, }\\
&\cristofse{ij}{k}(\theta),\ \cristofsm{ij}{k}(\theta): \text{ Christoffel's second symbol for $e$-connection and $m$-connection, }\\
&\overset{\;e}{R}\hspace{-8pt}\begin{smallmatrix}{\ \  \ \ \ l}\\{ijk}\end{smallmatrix}(\theta): \text{ Riemann curvature for $e$-connection }\\
&\secfse{ij}(\theta), \ \secfsm{ij}(\theta): \text{ Second fundamental form for $e$-connection and $m$-connection.}
\end{align*}
\end{theo}

Roughly speaking, $\cristofse{ij}{k},\ \cristofsm{ij}{k}$ and $\overset{\;e}{R}\hspace{-8pt}\begin{smallmatrix}{\ \  \ \ \ l}\\{ijk}\end{smallmatrix}$ give us information on the \textit{intrinsic curvature} of $\mathcal{P}$, while $\secfse{ij}, \ \secfsm{ij}$ tell us how the manifold $\mathcal{P}$ is located in the ambient space (\textit{extrinsic curvature}). Since $\mathcal{P}$ is torsion-free, the following equivalence holds;
\begin{align*}
\text{$\mathcal{P}$ is intrinsically $e$-flat} &\Longleftrightarrow \text{$\mathcal{P}$ is intrinsically $m$-flat} \\
&\Longleftrightarrow \text{$\cristofse{ij}{k}(\theta)=0$ for some coordinate system} \\
&\Longleftrightarrow \text{$\cristofsm{ij}{k}(\theta)=0$ for some coordinate system}\\
&\Longleftrightarrow \text{$\overset{\;e}{R}\hspace{-8pt}\begin{smallmatrix}{\ \  \ \ \ l}\\{ijk}\end{smallmatrix}(\theta)=0$ for any(some) coordinate system}.
\end{align*}
If $\secfse{ij}(\theta)=0\ \bigl(\secfsm{ij}(\theta)=0 \bigr)$ for any(some) coordinate system, it means $\mathcal{P}$ is extrinsically $e$-flat ($m$-flat).

All properties from \eqref{notation_Fisher_info_and_its_inverse} to \eqref{notation_secffe&m} are tensors. $\overset{\:e}{F}$ is parameter invariant (see Appendix 2 for the proof ). Consequently every term in the right-hand side of  \eqref{expan_ED_final} is parameter invariant as is expected from the parameter independence of $\alpha$-divergence. This means that 
we can choose any coordinate system with which we can easily calculate the terms in  \eqref{expan_ED_final}. For example, if we have another coordinate system $\eta\triangleq (\eta_\alpha, \eta_\beta, \eta_\gamma,\ldots )$ for $\mathcal{P}$, we can choose to calculate such a term as $T_{\alpha \beta \gamma}T^{\alpha \beta \gamma}$ instead of $T_{ijk}T^{ijk}$.

We easily notice that $T_{ijk}T^{ijk}$ and $T_{is}^i T_j^{js}$ is nonnegative (see Appendix 2 for the proof), but other terms in the bracket in \eqref{expan_ED_final} could be negative. As we will see in Section \ref{examples}, the $n^{-2}$ term could be negative.

Note that the $n^{-1}$ term equals $p/2n$, hence the risk in estimating a model by m.l.e. is primarily determined by the  number of the parameters, in other words, ``model complexity''. The number of the parameters $p$ also appears in A.I.C. as the penalty to the model complexity. This is natural since A.I.C. (and some other information criteria for model selection) is considered to be an estimator of $\overset{-1}{E\!D}$.

A geometrical expression \eqref{expan_ED_final} immediately leads us to the simplified form for an exponential family or a mixture family. The canonical form of an exponential family is given by
\begin{equation}
\label{expo_family}
f(x;\theta)=\exp\bigl(\theta^1h_1(x)+\cdots +\theta^p h_p(x)-\psi(\theta)\bigr),
\end{equation}
where $h_i(x) (i=1,\ldots,p)$ is a measurable function on $\mathfrak{X}$,
while the one for a mixture family is given by
\begin{equation}
\label{mix_family}
f(x;\theta)=\theta^1 g_1(x)+\cdots+\theta^p g_p(x)+(1-\sum_{i=1}^p\theta^i) g_0(x),
\end{equation}
where $g_i(x) (i=0,\ldots,p)$ is a probability density function.
These families are characterized respectively as being extrinsically $e$-flat and $m$-flat.  Namely ${\oset{\scalebox{0.8}{$e$}}{a}}_{ij}(x; \theta)=0$ for an exponential family, and 
${\oset{\scalebox{0.8}{$m$}}{a}}_{ij}(x; \theta)=0$ for a mixture family.   Furthermore an exponential family is intrinsically $e$ and $m$-flat. A mixture family is also $e$ and  $m$-flat. This means $\overset{\;e}{R}\hspace{-10pt}\begin{smallmatrix}{\ \  \ \ ij}\\{ij}\end{smallmatrix}$ vanishes for both families.

Consequently we have the following corollaries. 
\begin{coro}
If the model $\mathcal P$ is an exponential family,
\begin{align}
&\overset{\alpha}{E\!D} \nonumber\\
&=\frac{p}{2n}+\frac{1}{24n^{^2}}\nonumber\\
&\times
\biggl[ 
(\alpha')^2\bigl\{3\overset{\:e}{F}+3T^{ijk}T_{ijk}+3p^2+6p\bigr\}\nonumber\\
&\hspace{8mm}+
\alpha'\bigl\{3\overset{\:e}{F}-5T^{ijk}T_{ijk}-6T_{is}^iT_j^{js}-3p^2-6p\bigr\}\nonumber\\
&\hspace{8mm}
+T_{ijk}T^{ijk}+9T_{is}^i T^{js}_j-9\overset{\:e}{F}\biggr]+o(n^{-2}).\label{expan_ED_expo_fam}
\end{align}
\end{coro}
\begin{proof}
If the model $\mathcal P$ is an exponential family, the terms
$$
\langle \secffe{i}{j}, \secffe{j}{i}\rangle,\qquad \langle \secffe{i}{i}, \secffe{j}{j}\rangle,\qquad
\langle \secffe{i}{j}, \secffm{j}{i}\rangle,\qquad \langle \secffe{i}{i}, \secffm{j}{j}\rangle,\qquad  
\overset{\;e}{R}\hspace{-10pt}\begin{smallmatrix}{\ \  \ \ ij}\\{ij}\end{smallmatrix}
$$
vanish. 
\end{proof}
\begin{coro}
If $\mathcal P$ is a mixture family,
\begin{align}
&\overset{\alpha}{E\!D} \nonumber\\
&=\frac{p}{2n}+\frac{1}{24n^{^2}}\nonumber\\
&\times
\biggl[ 
(\alpha')^2\bigl\{3\overset{\:e}{F}+3T^{ijk}T_{ijk}+6\langle \secffe{j}{i}, \secffe{i}{j}\rangle+3\langle \secffe{i}{i}, \secffe{j}{j}\rangle+3p^2+6p\bigr\}\nonumber\\
&\hspace{8mm}+
\alpha'\bigl\{3\overset{\:e}{F}-5T^{ijk}T_{ijk}-6T_{is}^iT_j^{js}-6\langle \secffe{j}{i}, \secffe{i}{j}\rangle-3\langle \secffe{i}{i}, \secffe{j}{j}\rangle-3p^2-6p\bigr\}\nonumber\\
&\hspace{8mm}
+12\langle \secffe{j}{i}, \secffe{i}{j} \rangle+T_{ijk}T^{ijk}+9T_{is}^i T^{js}_j-9\overset{\:e}{F}
\biggr]+o(n^{-2}) \label{expan_ED_m_fam}.
\end{align}
\end{coro}
\begin{proof}
If $\mathcal P$ is a mixture family, the terms
$$
\langle \secffe{i}{j}, \secffm{j}{i}\rangle,\qquad \langle \secffe{i}{i}, \secffm{j}{j}\rangle,\qquad  
\overset{\;e}{R}\hspace{-10pt}\begin{smallmatrix}{\ \  \ \ ij}\\{ij}\end{smallmatrix}
$$
vanish. 
\end{proof}
It is notable that the asymptotic risk for an exponential family depends only on the intrinsic properties of the family. 

For a specific $\alpha$, the following result holds:
If $\alpha=-1$, then $\overset{\alpha}{D}[\theta_1: \theta_2]$ is Kullback-Leibler divergence, and 
\begin{align}
&\overset{-1}{E\!D} \nonumber\\
&=\frac{p}{2n}+\frac{1}{24n^{^2}}\nonumber\\
&\times
\biggl[ 
-3\overset{\:e}{F}-T^{ijk}T_{ijk}+3T_{is}^iT_j^{js}
+12\langle \secffe{j}{i}, \secffe{i}{j} \rangle -2\langle \secffe{j}{i}, \secffm{i}{j} \rangle -\langle \secffe{i}{i}, \secffm{j}{j} \rangle+8\overset{\;e}{R}\hspace{-10pt}\begin{smallmatrix}{\ \  \ \ ij}\\{ij}\end{smallmatrix}
\bigg]\nonumber\\
&+o(n^{-2});\label{expan_ED_KL}
\end{align}
if $\alpha=0$, then $\overset{\alpha}{D}[\theta_1: \theta_2]$ is equivalent to Hellinger-distance, and 
\begin{align}
&\overset{0}{E\!D} \nonumber\\
&=\frac{p}{2n}+\frac{1}{24n^{^2}}\nonumber\\
&\times
\biggl[ 
-(27/4)\overset{\:e}{F}-(3/4)T^{ijk}T_{ijk}+6T_{is}^iT_j^{js}
+(21/2)\langle \secffe{j}{i}, \secffe{i}{j} \rangle- (3/4)\langle \secffe{i}{i}, \secffe{j}{j} \rangle \nonumber \\
&\qquad-(1/2)\langle \secffe{j}{i}, \secffm{i}{j} \rangle -(1/4)\langle \secffe{i}{i}, \secffm{j}{j} \rangle+8\overset{\;e}{R}\hspace{-10pt}\begin{smallmatrix}{\ \  \ \ ij}\\{ij}\end{smallmatrix}
-(3/4)p^2-(3/2)p\bigg]\nonumber\\
&+o(n^{-2});\label{expan_ED_Hellinger}
\end{align}
if $\alpha=-3$, then $\overset{\alpha}{D}[\theta_1: \theta_2]$ is asymptotically equivalent to $\chi^2$-divergence, and 
\begin{align}
&\overset{-3}{E\!D} \nonumber\\
&=\frac{p}{2n}+\frac{1}{24n^{^2}}\nonumber\\
&\times
\biggl[ 
9\overset{\:e}{F}+3T^{ijk}T_{ijk}-3T_{is}^iT_j^{js}
+24\langle \secffe{j}{i}, \secffe{i}{j} \rangle+6\langle \secffe{i}{i}, \secffe{j}{j} \rangle -14\langle \secffe{j}{i}, \secffm{i}{j} \rangle -7\langle \secffe{i}{i}, \secffm{j}{j} \rangle\nonumber\\
&\qquad+8\overset{\;e}{R}\hspace{-10pt}\begin{smallmatrix}{\ \  \ \ ij}\\{ij}\end{smallmatrix}+6p^2+12p
\bigg]\nonumber\\
&+o(n^{-2}).\label{expan_ED_KL*}
\end{align}
Generally speaking, the components in the $n^{-2}$ term are not explicitly gained. For the feasibility of the calculation, the expression of these components by the expectation of the derivatives of log-likelihood is useful. 
Define the following notations; for $1\leq i, j, k, l \leq p$, 
\begin{align*}
&L_{(ij)}\triangleq E_\theta[l_{ij}],\quad L_{ij}\triangleq E_\theta[l_i l_j], \\
&L_{(ij)k}\triangleq E_\theta[l_{ij} l_k], \quad L_{ijk}\triangleq E_\theta[l_i l_j l_k] \\
&L_{(ij)(kl)}\triangleq E_\theta[l_{ij} l_{kl}], \quad L_{(ijk)l}\triangleq E_\theta[l_{ijk}l_l], \quad 
L_{(ij)kl}\triangleq E_\theta[l_{ij}l_k l_l],\quad L_{ijkl}\triangleq E_\theta[l_i l_j l_k l_l], \\
&L11=g^{ij}g^{kl}L_{(il)jk},\quad L12=g^{ij}g^{kl}L_{(ij)kl},\quad L13=g^{ij}g^{kl}L_{ijkl},\\
&L14=g^{ij}g^{kl}L_{(ik)(jl)},\quad L15=g^{ij}g^{kl}L_{(ij)(kl)},\\
&L21=g^{ij}g^{kl}g^{su}L_{(ik)s}L_{jlu},\quad L22=g^{ij}g^{kl}g^{su}L_{(ij)k}L_{lsu},\quad L23=g^{ij}g^{kl}g^{su}L_{iks}L_{jlu},\\
& L24=g^{ij}g^{kl}g^{su}L_{ijk}L_{lsu},\quad L25=g^{ij}g^{kl}g^{su}L_{(ik)s}L_{(jl)u},\quad L26=g^{ij}g^{kl}g^{su}L_{(ij)k}L_{(su)l}.
\end{align*}

Then we have the following equations (see Appendix 2 for the proof ).
\begin{align}
g_{ij}&=L_{ij}=-L_{(ij)},\label{int_exp_g}\\
\overset{\:\alpha}{F}&=g^{ij}g^{ks}\bigl(2L_{(is)jk}+L_{(ks)ij}+L_{ijks}\bigr)\nonumber\\
&\quad-g^{ks}g^{uj}g^{li}L_{ijk}\bigl(2L_{(su)l}+L_{sul}\bigr)\nonumber\\
&\quad -g^{ti}g^{uj}g^{ks}\bigl(L_{(it)s}+((1-\alpha)/2)L_{its}\bigr)L_{juk}\nonumber\\
&=2L11+L12+L13-2L21-L23-L22-\alpha'L24,\label{int_exp_F}\\
T_{ijk}T^{ijk}&= L_{ijk} L_{stu} g^{is}g^{jt} g^{ku}=L23, \label{int_exp_T_ikkT^ijk}\\
T_{is}^i T_j^{js}&=L_{ijk} L_{stu} g^{ij} g^{st} g^{uk}=L24, \label{int_exp_T_is^iT_j^js}\\
\overset{\;e}{R}\hspace{-10pt}\begin{smallmatrix}{\ \  \ \ ij}\\{ij}\end{smallmatrix}
&=g^{ij}g^{sk}\bigl(L_{(ki)(js)}-L_{(ij)(ks)}+L_{(ki)js}-L_{(ij)ks}\bigr)\nonumber\\
&\quad +g^{sk}g^{ti}g^{uj}\bigl(-L_{(ki)j}L_{(st)u}+L_{(it)s}L_{(uj)k}+L_{sit}L_{(uj)k}-L_{stu}L_{(ij)k}\bigr)\nonumber\\
&=L14-L15+L11-L12-L25+L26+L22-L21,\label{int_exp_R}\\
\langle \secffe{i}{j}, \secffe{j}{i} \rangle &= g^{jk}g^{li}L_{(ik)(jl)}-g^{jk}g^{li}g^{st}L_{(ik)s}L_{(jl)t}-p\nonumber\\
&=L14-L25-p, \label{int_exp_secffe{i}{j}_secffe{i}{j}}\\
\langle \secffe{i}{i}, \secffe{j}{j} \rangle &= g^{ik}g^{jl}L_{(ik)(jl)}-g^{ik}g^{jl}g^{st}L_{(ik)s}L_{(jl)t}-p^2\nonumber\\
&=L15-L26-p^2,\label{int_exp_secffe{i}{i}_secffe{j}{j}}\\
\langle \secffe{i}{j}, \secffm{j}{i} \rangle &= g^{jk}g^{li}L_{(ik)jl}+g^{jk}g^{li}L_{(ik)(jl)}\nonumber\\
&\hspace{20mm}-g^{jk}g^{li}g^{st}L_{(ik)s}L_{(jl)t}-g^{jk}g^{li}g^{st}L_{(ik)s}L_{jlt}\nonumber\\
&=L11+L14-L25-L21, \label{int_exp_secffm{i}{j}_secffm{i}{j}}\\
\langle \secffe{i}{i}, \secffm{j}{j} \rangle &= g^{ik}g^{jl}L_{(ik)jl}+g^{ik}g^{jl}L_{(ik)(jl)}\nonumber\\
&\hspace{20mm}-g^{ik}g^{jl}g^{st}L_{(ik)s}L_{(jl)t}-g^{ik}g^{jl}g^{st}L_{(ik)s}L_{jlt}\nonumber\\
&=L12+L15-L26-L22.\label{int_exp_secffm{i}{i}_secffm{j}{j}}
\end{align}
Using these expressions, we could find the value of the components of the $n^{-2}$ term by simulation.
\section{Examples}
\label{examples}
In this section, we take three parametric models as the examples and investigate the concrete form of $\overset{\alpha}{E\!D}$ up to the $n^{-2}$ term. 
\\
\\
--\textit{Example 1}--

First we consider a discrete model, that is, a multinomial distribution. 
Consider the family consisting of $p+1$ dimensional multinomial distributions given by
$$
P(X=x_i)=m_i,\quad i=0,1,\ldots,p,
$$
where $\sum_{i=0}^p m_i=1$. We use $m\triangleq (m_1,\ldots, m_p)$ as a free parameter.

The multinomial distribution is an exponential family, then from \eqref{expan_ED_expo_fam}, we notice that we only need to calculate three terms $T_{ijk}T^{ijk}$, $T_{is}^i T^{js}_j$, $\overset{\:e}{F}$. Because of the following relation (for the proof,  see Appendix 2),
\begin{equation}
\label{convert_eF_mF}
\overset{\:e}{F}=\overset{\;\scalebox{0.6}{$m$}}{F}+T_{is}^i T_j^{js}, \qquad \overset{\;\scalebox{0.6}{$m$}}{F}\triangleq \overset{\;\scalebox{0.6}{$-1$}}{F},
\end{equation}
we only have to calculate $\overset{\;\scalebox{0.6}{$m$}}{F}$ instead of $\overset{\:e}{F}$.
Actually we have the following results (for the proof, see Appendix 2);
\begin{align}
T_{ijk}T^{ijk}&=M-3p-1 \label{T_ijkT^ijk_multinomi}\\
T_{is}^i T^{js}_j&=M-(p+1)^2 \label{T_is^iT^js_j_multinomi}\\
\overset{\;\scalebox{0.6}{$m$}}{F}&=-M+p+1, \label{mF_multinomi}
\end{align}
where $M\triangleq \sum_{t=0}^p m_t^{-1}$. Consequently we have 
\begin{align}
&\overset{\alpha}{E\!D} (m)\nonumber\\
&=\frac{p}{2n}+\frac{1}{24n^{2}}\bigl\{ 
(\alpha')^2(3M-6p-3)+\alpha'(-11M+18p+11)+10M-12p-10\bigr\}
+o(n^{-2})\nonumber\\
&=\frac{p}{2n}+\frac{1}{24n^2}\biggl\{\frac{(1-\alpha)^2}{4}(3M-6p-3)+\frac{1-\alpha}{2}(-11M+18p+11)+10M-12p-10\biggr\}\nonumber\\
&\quad+o(n^{-2})\nonumber\\
&=\frac{p}{2n}+\frac{1}{96n^2}\biggl\{(3\alpha^2+16\alpha+21)(M-1)+(-6\alpha^2-24\alpha-18)p\biggr\}+o(n^{-2})\nonumber\\
&=\frac{p}{2n}+\frac{1}{96n^2}\Bigl\{(3+\alpha)(7+3\alpha)(M-1)-6(\alpha+3)(\alpha+1)p\Bigr\}+o(n^{-2}).
\label{expan_ED_multinom}
\end{align}
Note that this result could be gained in a more straightforward way, since the risk of m.l.e w.r.t.$\alpha$-divergence for the multinomial distribution could be expressed in a simple form (see Appendix 2 for the straightforward derivation of \eqref{expan_ED_multinom}). 

The $n^{-1}$ term is determined by the dimension of the multinomial distribution, while the  $n^{-2}$ order term depends only on $M$ once $p$ is fixed. If $(3+\alpha)(7+3\alpha)>0$,  then $n^{-2}$ order term is a monotonically increasing function of $M$. When $M$ is minimized, that is when $m_0=m_1=\cdots=m_p$,  $\overset{\alpha}{E\!D} (m)$ is minimized. The asymptotically lowest risk among the possible distributions is attained by the equi-probable distribution. The estimation becomes harder as $M$ increases.  The term $M$ could be very large when some $m_i$ is close to zero. This  justifies the treatment of merging a category of a possibly very low probability with another category for a better inference. The $\alpha$ that is statistically often used such as $\alpha=\pm 1, 0, 3$ satisfies the condition  $(3+\alpha)(7+3\alpha)>0$. However when $-3 < \alpha < -7/3$, these phenomena are vice versa. The equi-probable distribution is the asymptotically highest risk point. 

When $\alpha=-1, 0, -3$,  we have the following results;
\begin{align}
\overset{\scalebox{0.6}{\!\!$-1$}}{E\!D} (m)&=\frac{p}{2n}+\frac{1}{24n^{2}}(2M-2)+o(n^{-2}),\label{expan_KL_multinom}\\
\overset{\scalebox{0.6}{$0$}}{E\!D} (m)&=\frac{p}{2n}+\frac{1}{24n^{2}}((21/4)M-(9/2)p-21/4)+o(n^{-2}).\label{expan_Hell_multinom}\\
\overset{\scalebox{0.6}{\!\!$-3$}}{E\!D} (m)&=\frac{p}{2n}+o(n^{-2}),\label{expan_chi_multinom}
\end{align}
Rather surprisingly, the $n^{-2}$ term for $\chi^2$-divergence vanishes, hence the asymptotic risk up to $n^{-2}$ order is uniform in $m$. 

Figure \ref{fig:EML_binomi} and Table \ref{table:EML_binomi} show the approximated value of $\overset{\scalebox{0.6}{\!\!$-1$}}{E\!D} (m_1)$ 
up to the $n^{-2}$ term for the case $p=1, n=10$, i.e. $B(10, m_1)$ as $m_1$ varies. 
\begin{figure}
\centering
\includegraphics[width=10cm,clip]{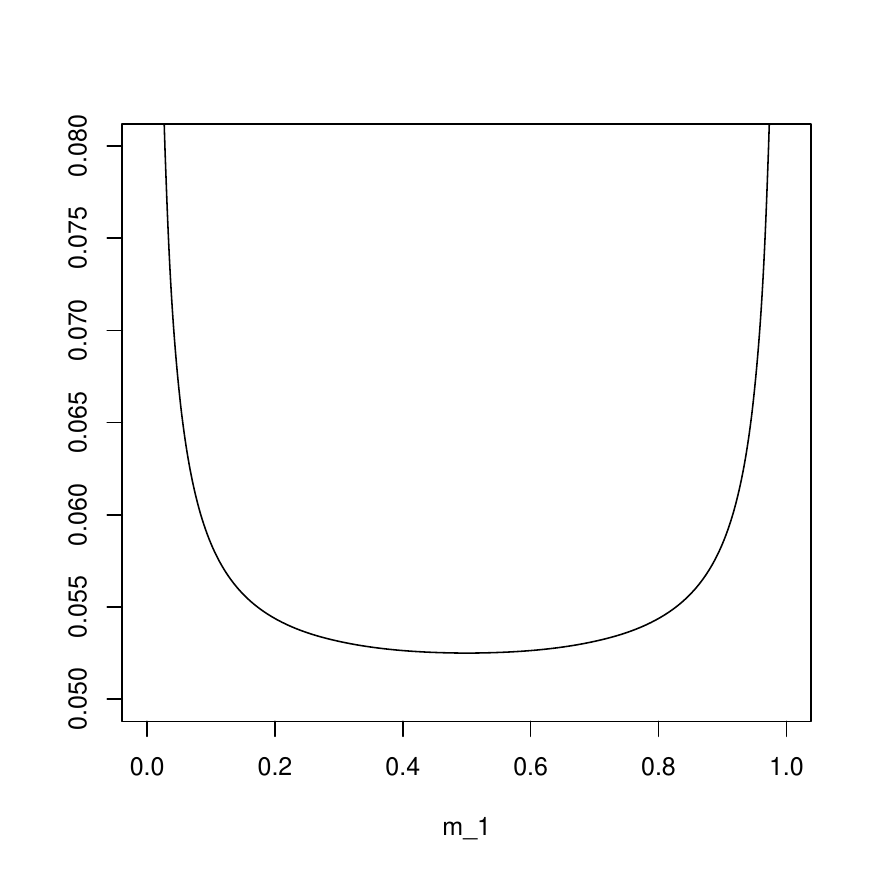}
\caption{ Approximated $\stackrel{-1}{E\!D} (m_1)$ for $B(10,m_1)$}
\label{fig:EML_binomi}
\end{figure}
\begin{table}
\caption{ Approximated $\stackrel{-1}{E\!D} (m_1)$ for $B(10,m_1)$}
\label{table:EML_binomi}
\centering
\begin{tabular}{c|c|c|c|c|c|c|c}
$m_1$ & 0.5 & 0.4 & 0.3 & 0.2 & 0.1 & 0.01 & 0.001 \\
\hline
$\overset{-1}{E\!D} (m_1)$ & 0.0525 & 0.0526 &0.0531 &0.0544 &0.0584& 0.1333 &0.8833
\end{tabular}
\end{table}
We observe that the risk of the estimation rapidly increases outside of the interval $(0.1, 0.9)$. It is really a hard task to estimate the probability which is less than 1/10 based on just 10 observations.
\\
\\
--\textit{Example 2}--

Second example is the $p$-dimensional multivariate normal distribution with zero means, that is, 
$$
X \sim N_p(0, \Sigma),  \quad \Sigma=(\sigma_{ij}).
$$
The m.l.e. is  the sample variance-covariance matrix. Note that if $\alpha$ equals $\pm 1$, the divergence is explicitly given (so called Stein's loss function), hence  we can derive the expansion of $\overset{\pm 1}{E\!D} (\Sigma)$ in a more straightforward way (e.g., for the case $\alpha=-1$, see Appendix 2) .

For this model, we can use the parameter $\sigma_{ij}(\ 1\leq i \leq  j \leq p)$ or 
$\sigma^{ij}(\ 1 \leq i \leq j \leq p)$, where
$$
\sigma^{ij}=(\Sigma^{-1})_{ij},\quad 1\leq i, j \leq p.
$$
We use the notation $(i,j),\ 1\leq i \leq j \leq p$ to specify the element of the parameters.

Since this model is also an exponential family, from \eqref{expan_ED_expo_fam} and \eqref{convert_eF_mF}, we only have to calculate  $T_{ijk}T^{ijk}$, $T_{is}^i T^{js}_j$, $\overset{\;\scalebox{0.6}{$m$}}{F}$. 
These turn out to be as follows (see Appendix 2 for the proof) ;
\begin{align}
&T_{(i,j)(k,l)(s,t)}T^{(i,j)(k,l)(s,t)}= T_{\sigma_{ij}\sigma_{kl}\sigma_{st}}T^{\sigma_{ij}\sigma_{kl}\sigma_{st}}=T_{\sigma^{ij}\sigma^{kl}\sigma^{st}}T^{\sigma^{ij}\sigma^{kl}\sigma^{st}}\nonumber\\
&=p^3+3p^2+4p, \label{Normal Model T^ijk T_ijk}\\
&T_{(i,j)(s,t)}^{(i,j)}T_{(k,l)}^{(k,l)(s,t)}=T_{\sigma_{ij}\sigma_{st}}^{\sigma_{ij}}T_{\sigma_{kl}}^{\sigma_{kl}\sigma_{st}}=T_{\sigma^{ij}\sigma^{st}}^{\sigma^{ij}}T_{\sigma^{kl}}^{\sigma^{kl}\sigma^{st}}\nonumber\\
&=2p^3+8p^2+8p, \label{Normal Model T_ij^i T_s^sj}\\
&\overset{\;\scalebox{0.6}{$m$}}{F}=-p^3-2p^2-p. \label{Normal Model F}
\end{align}
Therefore we have 
\begin{align}
&\overset{\alpha}{E\!D} (\Sigma)\nonumber\\
&=\frac{p(p+1)}{4n}\nonumber\\
&\quad+\frac{1}{24n^{2}}\biggl[ 
(\alpha')^2(6p^3+30p^2+39p)-\alpha'(14p^3+48p^2+53p)+10p^3+21p^2+13p
\biggr]\nonumber\\
&\quad+o(n^{-2}). \label{expan_ED_nomal_vacova}
\end{align}
Especially when $\alpha=-1, 0, -3$,
\begin{align}
\overset{\scalebox{0.6}{\!\!$-1$}}{E\!D} (\Sigma)&=\frac{p(p+1)}{4n}+\frac{1}{24n^{2}}(2p^3+3p^2-p)+o(n^{-2}),\label{expan_KL_normal_vacova}\\
\overset{\scalebox{0.6}{$0$}}{E\!D} (\Sigma)&=\frac{p(p+1)}{4n}+\frac{1}{32n^{2}}(6p^3+6p^2-5p)+o(n^{-2}).\label{expan_Hell_normalvacova}\\
\overset{\scalebox{0.6}{\!\!$-3$}}{E\!D} (\Sigma)&=\frac{p(p+1)}{4n}+\frac{1}{8n^{2}}(2p^3+15p^2+21p)+o(n^{-2}),\label{expan_chi_normal_vacova}
\end{align}
Notably $\overset{\alpha}{E\!D} (\Sigma)$ is not only parameter-invariant but also constant. The risk in estimating the true parameter $\Sigma$ by m.l.e.  is independent of $\Sigma.$ Actually we have the following lemma.
\begin{lemma}
Let $x | \theta $ denote the probability distribution on $\mathfrak{X}$ under the parameter $\theta (\in \Theta)$. Suppose that  there exists one to one transformations, 
$$
G(x): \mathfrak{X} \rightarrow \mathfrak{X},\qquad \tilde{G}(\theta): \Theta \rightarrow \Theta
$$
satisfying the relation
\begin{equation}
\label{condition_lem}
x | \tilde{G}(\theta) \stackrel{d}{=} G(x) | \theta \text{ or equivalently } G^{-1}(x) | \tilde{G}(\theta) \stackrel{d}{=} x | \theta.
\end{equation}
Then $\overset{\alpha}{E\!D}(\theta)\triangleq E_{\theta}\bigl[\overset{\alpha}{D}[\hat\theta(\bm{X}):\theta]\bigr]$ is equal to 
$\overset{\alpha}{E\!D}(\tilde{G}(\theta))\triangleq E_{\tilde{G}(\theta)}\bigl[\overset{\alpha}{D}[\hat{\theta}(\bm{X}):\tilde{G}(\theta)]\bigr]$.
\end{lemma}
\begin{proof}
We use the notation  $\overset{\alpha}{D}[x |\theta_1: x |\theta_2]$ instead of  $\overset{\alpha}{D}[\theta_1: \theta_2]$ for the divergence between the two distributions $x | \theta_1$ and $x | \theta_2$. 
\begin{align*}
&E_{\theta}\bigl[\overset{\alpha}{D}[x|\hat{\theta}(\bm{X}):x|\theta]\bigr]\\
&=E_{\theta}\bigl[\overset{\alpha}{D}[G^{-1}(x)|\tilde{G}(\hat{\theta}(\bm{X})):G^{-1}(x)|\tilde{G}(\theta)]\bigr] \text{ (because of \eqref{condition_lem})}\\
&=E_{\theta}\bigl[\overset{\alpha}{D}[G^{-1}(x)|\hat{\theta}(G(\bm{X})):G^{-1}(x)|\tilde{G}(\theta)]\bigr]
\text{ (note that $\hat{\theta}(G(\bm{X}))\stackrel{d}{=}\tilde{G}(\hat{\theta}(\bm{X}))$ because of \eqref{condition_lem})}\\
&=E_{\tilde{G}(\theta)}\bigl[\overset{\alpha}{D}[G^{-1}(x)|\hat{\theta}(\bm{X}):G^{-1}(x)|\tilde{G}(\theta)]\bigr]\text{ (because of \eqref{condition_lem})}\\
&=E_{\tilde{G}(\theta)}\bigl[\overset{\alpha}{D}[x|\hat{\theta}(\bm{X}):x|\tilde{G}(\theta)]\bigr] \text{ (because of the invariance property of $\alpha$-divergence).}
\end{align*}
\end{proof}
For arbitrary $\Sigma_1$ and $\Sigma_2$,  if we define 
$$
G(X)=\Sigma_2^{1/2}\Sigma_1^{-1/2}X,\qquad 
\tilde{G}(\Sigma)=\Sigma_2^{1/2}\Sigma_1^{-1/2}\Sigma\Sigma_1^{-1/2}\Sigma_2^{1/2},
$$
then we have $X|\tilde{G}(\Sigma)\stackrel{d}{=}G(X)|\Sigma$, hence 
$$
\overset{\alpha}{E\!D}(\Sigma_1)=\overset{\alpha}{E\!D}(\tilde{G}(\Sigma_1))=\overset{\alpha}{E\!D}(\Sigma_2).
$$

We observe ``the curse of dimension'' in this example. Notice that $n^{-1}$ and $n^{-2}$ terms increase with the second and third power of $p$ respectively.  If we increases $n$ and $p$ with the constant ratio $p/n$, the both terms explode. Figure \ref{fig:EML_multinorm} and Table \ref{table:EML_multinorm}  show the $n^{-2}$-order-approximated values of $\overset{\scalebox{0.6}{\!\!$-1$}}{E\!D} $ as $n$ varies, when $p$ =10.
\begin{figure}
\centering
\includegraphics[width=10cm,clip]{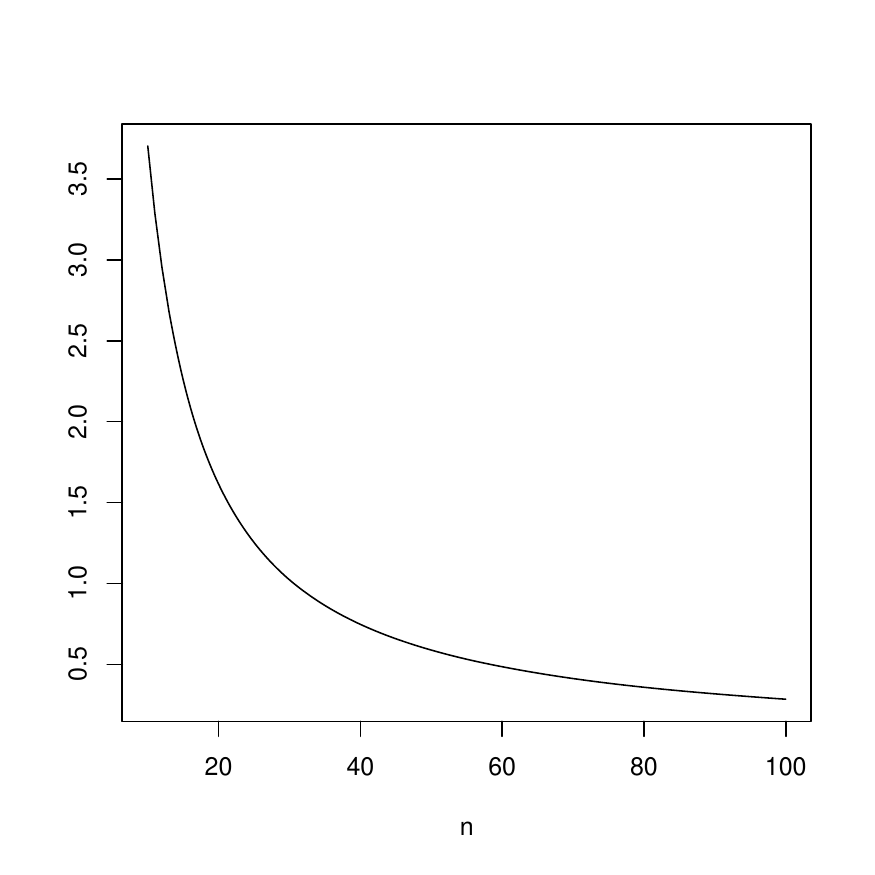}
\caption{ Approximated $\stackrel{-1}{E\!D}$ for $N_{10}(0, \Sigma)$}
\label{fig:EML_multinorm}
\end{figure}
\begin{table}
\caption{ Approximated $\stackrel{-1}{E\!D} $ for $N_{10}(0, \Sigma)$}
\label{table:EML_multinorm}
\centering
\begin{tabular}{c|c|c|c|c|c|c|c}
$n$ & 100 & 200 & 300 & 400 & 500 &  800 &1000 \\
\hline
$\overset{-1}{E\!D} $ & 0.2845 & 0.1399 & 0.0927 &0.0693 &0.0554& 0.0345 &0.0276
\end{tabular}
\end{table}
When the dimension is 10,  we need approximately 500 observations for the same risk as $B(10, 0.5)$ (that is,  a 10-times coin toss problem) in estimating the true parameter by m.l.e. 
\\
\\
--\textit{Example 3}--

As a last example, we take a mixture family.  For most cases of the mixture family, it is difficult to gain the components of \eqref{expan_ED_m_fam} explicitly so that we need to calculate them numerically. 

If we use the canonical form \eqref{mix_family} with the notation $h_i(x)\triangleq g_i(x)-g_0(x)$, we have
\begin{align*}
&l_i =\frac{h_i}{f},\quad l_{ij}=-\frac{h_i h_j}{f^2}=-l_i l_j, \\
& l_{ijk}=2\frac{h_i h_j h_k}{f^3}=2l_i l_j l_k,\quad l_{ijkl}=-6\frac{h_i h_j h_k h_l}{f^4}=-6l_i l_j l_k l_l.
\end{align*}
Using these relations, the components of \eqref{expan_ED_m_fam} are expressed as follows;
\begin{align*}
\overset{\:e}{F}&=-2g^{ij}g^{ks}L_{ijks}+g^{ks}g^{ui}g^{jk}L_{ijk}L_{sul}+g^{ti}g^{uj}g^{ks}L_{ist}L_{jku},\\
T^{ijk}T_{ijk}&=L_{ijk}L_{stu} g^{is}g^{jt}g^{ku},\\
T_{is}^iT_j^{js}&=L_{ijk}L_{stu}g^{ij}g^{st}g^{ku},\\
\langle \secffe{j}{i}, \secffe{i}{j}\rangle&=g^{jk}g^{li}L_{ikjl}-g^{jk}g^{li}g^{st}L_{iks}L_{jlt}-p,\\
\langle \secffe{i}{i}, \secffe{j}{j}\rangle&=g^{ik}g^{jl}L_{ikjl}-g^{ik}g^{jl}g^{st}L_{iks}L_{jlt}-p^2.
\end{align*}
As a more specific example, we consider the mixture of two normal distributions. Let 
\begin{equation}
\label{mix_normals}
X \sim (1-\theta_1)*N(0, \sigma^2)+\theta_1*N(1, \sigma^2),
\end{equation}
where $\sigma^2$ is a known parameter.

We numerically calculated $\overset{\:e}{F}$, $T^{ijk}T_{ijk}$, $T_{is}^iT_j^{js}$ and $\langle \secffe{j}{i}, \secffe{i}{j}\rangle$ from the above expression by Monte Carlo simulation. We actually calculated those components by generating $10^5$ random variables following the mixture distribution \eqref{mix_normals} under the values of $\theta_1$ from 0.1 to 0.99 by 0.01 increment, from which the $n^{-2}$-order-approximation of $\overset{\scalebox{0.6}{\!\!$-1$}}{E\!D}$ was gained. From \eqref{expan_ED_m_fam}, we notice that
$$
\overset{\scalebox{0.6}{\!\!$-1$}}{E\!D}=\frac{p}{2n}+\frac{1}{24n^{^2}}( -3\overset{\:e}{F}-T^{ijk}T_{ijk}+3T_{is}^iT_j^{js}+12\langle \secffe{j}{i}, \secffe{i}{j}\rangle)+o(n^{-2}).
$$

 In Figure \ref{fig:EKL_mix_fam}, we can see four U-curves  each of which corresponds to  the approximated $\overset{\scalebox{0.6}{\!\!$-1$}}{E\!D}$ of the model \eqref{mix_normals} with $\sigma^2=1/2, 1/5, 1/10$ from the top where $n=10$ is fixed. The U-curve at the lowest position corresponds to the approximated $\overset{\scalebox{0.6}{\!\!$-1$}}{E\!D}$  for $B(10,\theta_1)$.
\begin{figure}
\centering
\includegraphics[width=10cm,clip]{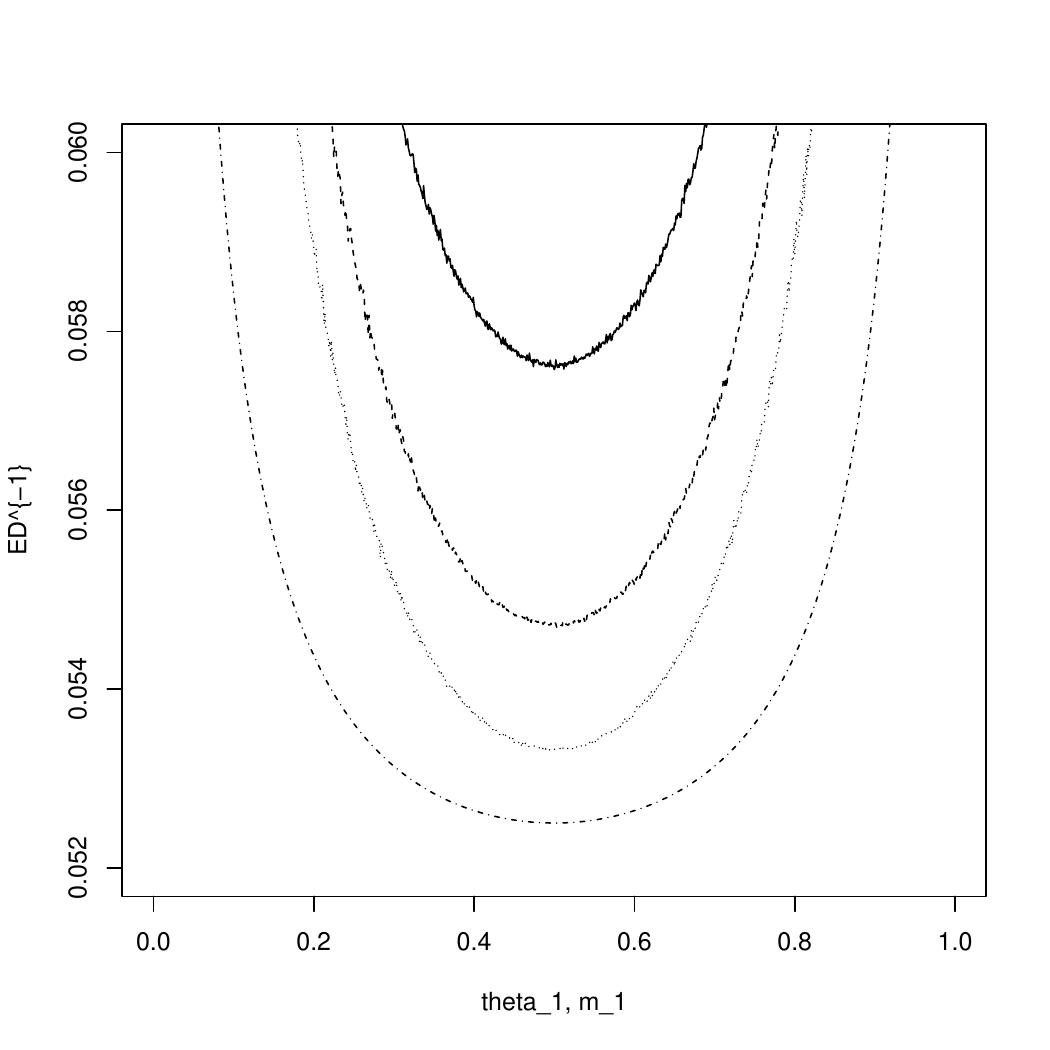}
\caption{ Approximated $\stackrel{-1}{E\!D}$ for the mixture model \eqref{mix_normals}}
\label{fig:EKL_mix_fam}
\end{figure}
It is much harder to specify the model \eqref{mix_normals} with a large variance  compared to the model 
$B(10,\theta_1)$, since many observations from the model \eqref{mix_normals} have no trace on whether it comes from $N(0, \sigma^2)$ or $N(1, \sigma^2)$. In contrast, if the variance is small, we can judge from the value of the observation which normal distribution  it came from. This information helps the inference on $\theta_1$ and the risk in estimating the parameter gets closer to the one for $B(10,\theta_1)$ where head or tail is completely clear. 
\section{Appendix 1}
\label{append1}
\subsection{Basic concepts of Information Geometry}
\label{basic_con_info_geo}
Amari \cite{Amari2}, Amari and Nagoka \cite{Amari&Nagaoka}, Murray and Rice\cite{Murray&Rice} and Calin and Udri\c{s}te \cite{Calin&Udriste} serve as a  a general guidance to the information geometry. We only briefly introduce the basic concepts of differential geometry and their concrete forms in the case of statistical manifolds.

We consider an ambient space 
$$
\mathcal{M}\triangleq \{f(x)\ | \text{ $f(x)$ is a measurable function on $\mathfrak{X}$.}\}
$$
and a scale extension model of $\mathcal{P}$
$$
\tilde{\mathcal{P}}\triangleq \{ \tilde{f}(x;\tilde\theta)\triangleq e^{\theta^0} f(x;\theta)\ |\  \tilde{\theta}=(\theta^0, \theta)=(\theta^i)_{i=0,1,\ldots,p},\quad \theta \in \Theta,\quad -\infty < \theta^0 <\infty \}.
$$
Then $\mathcal{P} \subset \tilde{\mathcal{P}} \subset \mathcal{M}$.
We will explain  how to construct a Riemmanian manifold structure in  $\tilde{\mathcal{P}}$ and $\mathcal{P}$ following the way of Amari (\cite{Amari2}). 

We start with $\mathcal{M}$. Consider the variation $g(x; u),\ -\epsilon < u < \epsilon $ in $\mathcal{M}$ and the corresponding tangent vector $\partial_u$ at $g(x;0)$. The $\alpha$-representation $(-\infty < \alpha < \infty)$ of $\partial_u$ at $g(x;0)$ is defined as 
$$
\{g(x;0) \}^{-(1+\alpha)/2} \dot{g}(x; 0), 
$$
where
$$
\dot{g}(x;0)\triangleq\frac{\partial}{\partial u} g(x; u) \biggl|_{u=0}.
$$
Suppose that another variation $h(x; t),\ -\epsilon < t < \epsilon $ such as $g(x; 0)=h(x; 0)$ is given. The inner product  between $\partial_u$ and $\partial_t$ at $g(x ;0)$ is defined as 
\begin{align*}
\langle \partial_u, \partial_t \rangle &\triangleq \int_{\mathfrak{X}} (\text{$\alpha$-representation of $\partial_u$})\times  (\text{$\alpha$-representation of $\partial_t$})\times
g^\alpha(x;0) d\mu \\
&=\int_{\mathfrak{X}}\{g(x;0)\}^{-1} \dot{g}(x;0) \dot{h}(x;0) d\mu \\
&=\int_{\mathfrak{X}} (\text{$\alpha$-representation of $\partial_u$})\times  (\text{$-\alpha$-representation of $\partial_t$}) d\mu \\
&=\int_{\mathfrak{X}}g(x;0) \dot{lg}(x;0) \dot{lh}(x;0) d\mu,
\end{align*}
where
$$
\dot{lg}(x;0) \triangleq \frac{\partial}{\partial u} \log g(x; u) \biggl|_{u=0},\qquad \dot{lh}(x;0) \triangleq \frac{\partial}{\partial t} \log h(x; t) \biggl|_{t=0}.
$$
Let $A(\tilde{\theta})$ be a vector field in $\mathcal{M}$ along $\tilde{\mathcal{P}}$. Its $\alpha$-representation at $\tilde{\theta}$ is denoted by $\oset{\scalebox{0.8}{$\alpha$}}{a}(x; \tilde{\theta})$. 
Consider a vector $\partial_i\triangleq \partial/\partial\theta^i (i=0,\ldots,p).$
The $\alpha$-covariant-derivative ($\alpha$-connection) $(-\infty < \alpha < \infty)$ of $A$ in the space of $\mathcal{M}$ in the direction of  $\partial_i$, ${\oset{\scalebox{0.8}{$\alpha$}}{\overline \nabla}}_{\partial_i} A$, is defined  as the vector field along $\tilde{\mathcal{P}}$ so that its $\alpha$-representation at $\tilde{\theta}$ is given by $\partial_i \oset{\scalebox{0.8}{$\alpha$}}{a}(x;\tilde{\theta})$. 

Now we introduce the geometrical properties of $\tilde{\mathcal{P}}$. First a base of tangent vectors are given by $\partial_i  (i=0,\dots,p)$.
The variation of $\tilde{f}(x ;\tilde{\theta})$ in $\tilde{\mathcal{P}}$ when $\theta^i$ changes gives rise to the tangent vector $\partial_i$ at each $\tilde{\theta}$, and its $\alpha$-representation is given by 
$$
\{\tilde{f}(x;\tilde{\theta}) \}^{-(1+\alpha)/2} \tilde{f}_i(x; \tilde{\theta}),
$$
where $\tilde{f}_i(x; \tilde{\theta})=\partial_i \tilde{f}(x; \tilde{\theta})$.
Components of a Riemannian metric on $\tilde{\mathcal{P}}$ are defined by
\begin{align*}
\tilde{g}_{ij}(\tilde{\theta})\triangleq \langle \partial_i, \partial_j \rangle_{\tilde{\theta}} &\triangleq 
\int_\mathfrak{X}  \{\tilde{f}(x;\tilde{\theta}) \}^{-1} \tilde{f}_i(x; \tilde{\theta}) \tilde{f}_j(x; \tilde{\theta}) d\mu \\
&=\int_\mathfrak{X}  \tilde{f}(x;\tilde{\theta})\:\tilde{l}_i(x; \tilde{\theta})\: \tilde{l}_j(x; \tilde{\theta}) d\mu \\
&=e^{\theta^0} E_\theta[\tilde{l}_i(x; \tilde{\theta})\: \tilde{l}_j(x; \tilde{\theta}) ],
\end{align*}
where $\tilde{l}_i(x; \tilde{\theta})\triangleq \partial_i \log \tilde{f}(x; \tilde{\theta})$.
Actually $\tilde{g}_{ij}(\tilde{\theta})$ is given by
\begin{equation}
\label{tildeg_ij}
\tilde{g}_{ij}(\tilde{\theta})=
\begin{cases}
e^{\theta^0}g_{ij} (\theta)& \text{if} \ 1\leq i,j \leq p, \\
0    &  \text{if} \ (i,j)=(0,1),\dots,(0,p), (1,0),\dots, (p,0) \\
e^{\theta^0}   &   \text{if} \ i=j=0,
\end{cases}
\end{equation}
where $g_{ij}(\theta)\triangleq \langle \partial_i , \partial_j \rangle_{\theta}(1\leq i,j \leq p)$ is the components of the metric (Fisher information metric) on $\mathcal{P}$ defined by  
\begin{equation}
\label{metric1}
g_{ij}(\theta)\triangleq E_\theta[l_i(x; \theta)\: l_j(x; \theta) ].
\end{equation}
The second case of \eqref{tildeg_ij} indicates $\partial_i \perp \partial_0\ (i=1,\dots,p)$, which is derived from
\begin{align}
\label{orthog_i_0}
&\int_\mathfrak{X} \{\tilde{f}(x;\tilde{\theta})\}^{-1}\:\tilde{f}_i(x; \tilde{\theta})\: \tilde{f}_0(x; \tilde{\theta}) d\mu \nonumber \\
&=\int_\mathfrak{X}  \tilde{f}_i(x; \tilde{\theta})\:  d\mu
=e^{\theta^0}\int_{\mathfrak X} f_i(x;\theta)  d\mu =
e^{\theta^0}\partial_i \int_{\mathfrak X} f(x;\theta) d\mu=0.
\end{align}
Another expression of the metric on $\mathcal{P}$
\begin{equation}
\label{metric2}
g_{ij}(\theta)=-E_\theta[l_{ij}(x; \theta)]
\end{equation}
is obtained from the relationship 
$$
E_\theta[l_i(x; \theta)\: l_j(x; \theta)+l_{ij}(x; \theta)]=\int_{\mathfrak X} f_{ij}(x;\theta) d\mu=\partial_{i}\partial_{j}\int_{\mathfrak X} f(x;\theta) d\mu=0.
$$
We use the notation $\tilde{g}^{ij}(\tilde{\theta})$, $g^{ij}(\theta)$ respectively for the components of the inverse matrix of $(\tilde{g}_{ij}({\tilde{\theta}}))$, $(g_{ij}(\theta))$.

$\partial_i \ (0\leq i \leq p)$ is the vector field along  $\tilde{\mathcal{P}}$, hence its $\alpha$-covariant-derivative ($\alpha$-connection)  in the space of $\mathcal{M}$ in the direction of $\partial_j\ (0\leq j \leq p)$,  ${\oset{\scalebox{0.8}{$\alpha$}}{\overline \nabla}}_{\partial_j} \partial_i$, could be  considered.
Its $\alpha$-representation at $\tilde\theta$ is given by 
\begin{align*}
&\partial_j[\{\tilde{f}(x;\tilde{\theta}) \}^{-(1+\alpha)/2} \tilde{f}_i(x; \tilde{\theta})]\\&=
-\frac{1+\alpha}{2} \{\tilde{f}(x;\tilde{\theta}) \}^{-(3+\alpha)/2} \tilde{f}_i(x; \tilde{\theta})\tilde{f}_j(x; \tilde{\theta})+\{\tilde{f}(x;\tilde{\theta}) \}^{-(1+\alpha)/2}
\tilde{f}_{ij}(x; \tilde{\theta}).
\end{align*}
We are mainly concerned with the case $\alpha =\pm 1$. In those cases, more familiar names exist. $e$-representation and $e$-covariant-derivative ($e$-connection) for the case $\alpha=1$; $m$-representation and  $m$-covariant-derivative ($m$-connection) for the case $\alpha=-1$. It turns out that 
\begin{align*}
&\text{$e$-representation of ${\oset{\scalebox{0.8}{$e$}}{\overline \nabla}}_{\partial_j} \partial_i$ at $\tilde\theta$}\\
&=\tilde{l}_{ij}(x;\theta)=
\begin{cases}
l_{ij}(x;\theta), &\text{if $1\leq i,j \leq p$,} \\
0, &\text{otherwise,}
\end{cases}
\end{align*}
and
\begin{align*}
&\text{ $m$-representation of ${\oset{\scalebox{0.8}{$m$}}{\overline \nabla}}_{\partial_j} \partial_i $ at $\tilde\theta$}\\
&=\tilde{f}_{ij}(x;\theta)=
\begin{cases}
e^{\theta^0}f_{ij}(x;\theta), &\text{if $1\leq i,j \leq p$,} \\
e^{\theta^0}f_{i}(x;\theta), &\text{if $(i, j)=(1,0),\dots,(p,0)$,}\\
e^{\theta^0}f_{j}(x;\theta), &\text{if $(i, j)=(0,1),\dots,(0,p)$,}\\
e^{\theta^0}f(x;\theta), &\text{if $(i, j)=(0,0)$.}\\
\end{cases}
\end{align*}
Consider the two variations in $\mathcal{M}$,
\begin{align}
&\tilde{f}(x;\tilde{\theta}) \exp(u l_{ij}(x;\theta) ), \quad |u| < \epsilon,\label{evari}\\
&\tilde{f}(x;\tilde{\theta}) +t e^{\theta^0}f_{ij}(x;\theta), \quad |t| < \label{mvari}\epsilon,
\end{align}
for $1\leq i, j \leq p$. $\partial_u$ and $\partial_t$ respectively equals  ${\oset{\scalebox{0.8}{$e$}}{\overline \nabla}}_{\partial_i} \partial_j$ and ${\oset{\scalebox{0.8}{$m$}}{\overline \nabla}}_{\partial_i} \partial_j$ since the representations coincide. 

Let $T_{\tilde{\theta}}\tilde{\mathcal{P}}$ denote the tangent space of $\tilde{\mathcal{P}}$ at $\tilde{\theta}$ (i.e. $\tilde{f}(x;\tilde{\theta}$)). 
Suppose that  $A$ is a tangent vector of $\mathcal{M}$ at $\tilde{f}(x;\tilde{\theta})$.
The orthogonal projection $\tilde{\pi}$ of $A$ onto $T_{\tilde{\theta}}\tilde{\mathcal{P}}$ is given by
$$
\tilde{\pi}(A)=\sum_{0\leq i,j \leq p} \langle A, \partial_i \rangle \: \tilde{g}^{ij}\: \partial_j.
$$

$\alpha$-covariant-derivative ($\alpha$-connection) in the space of $\tilde{\mathcal{P}}$ (denoted by ${\oset{\scalebox{0.8}{$\alpha$}}{\widetilde\nabla}}_{\partial_i} \partial_j, 0\leq i,j\leq p$) is defined as the orthogonal projection of  ${\oset{\scalebox{0.8}{$\alpha$}}{\overline \nabla}}_{\partial_i} \partial_j$. Actually the following equations hold;
\begin{align*}
{\oset{\scalebox{0.8}{$\alpha$}}{\widetilde\nabla}}_{\partial_i} \partial_j 
&\triangleq \tilde{\pi}\bigl( {\oset{\scalebox{0.8}{$\alpha$}}{\overline \nabla}}_{\partial_i} \partial_j  \bigr)\\
&= \sum_{0\leq s,t \leq p} \langle  {\oset{\scalebox{0.8}{$\alpha$}}{\overline \nabla}}_{\partial_i} \partial_j, \partial_s \rangle \: \tilde{g}^{st}\: \partial_t \\
&=\sum_{0\leq s,t \leq p} \cristofsa{ij}{s} \: \tilde{g}^{st}\: \partial_t\\
&=e^{-\theta^0}\sum_{1 \leq s, t \leq p} \cristofsa{ij}{s}  \: {g}^{st}\: \partial_t + e^{-\theta^0}\cristofsa{ij}{0}\:\partial_0
\end{align*}
where for $0 \leq i,j,k \leq p$,  $-\infty < \alpha < \infty,$
\begin{align}
&\cristofsa{ij}{k}(\tilde{\theta})\triangleq\langle  {\oset{\scalebox{0.8}{$\alpha$}}{\overline \nabla}}_{\partial_i} \partial_j, \partial_k \rangle_{\tilde{\theta}} \nonumber\\
&= \int_{\mathfrak X} (\text{$\alpha$-representation of $ {\oset{\scalebox{0.8}{$\alpha$}}{\overline \nabla}}_{\partial_i} \partial_j$ at $\tilde{\theta}$})\times  (\text{$\alpha$-representation of $\partial_k$ at $\tilde{\theta}$})\nonumber \\
&\qquad\times
\tilde{f}^\alpha(x;\tilde{\theta}) d\mu \nonumber\\
&=-\frac{1+\alpha}{2}\int_{\mathfrak X} \tilde{f}_i(x;\tilde{\theta})\tilde{f}_j(x;\tilde{\theta})\tilde{f}_k(x;\tilde{\theta}) \tilde{f}^{-2}(x;\tilde{\theta}) d\mu \nonumber\\
&\qquad+\int_{\mathfrak X} \tilde{f}_{ij}(x;\tilde{\theta})\tilde{f}_k(x;\tilde{\theta})\tilde{f}^{-1}(x;\tilde{\theta}) d\mu \nonumber\\
&=\frac{1-\alpha}{2}\int_{\mathfrak X} \tilde{l}_i(x;\tilde{\theta})\tilde{l}_j(x;\tilde{\theta})\tilde{l}_k(x;\tilde{\theta}) \tilde{f}(x;\tilde{\theta}) d\mu \nonumber\\
&\qquad+\int_{\mathfrak X} \tilde{l}_{ij}(x;\tilde{\theta})\tilde{l}_k(x;\tilde{\theta})\tilde{f}(x;\tilde{\theta}) d\mu, \label{def_cristofa}
\end{align}
where $\tilde{l}_{ij}(x;\theta)\triangleq \partial_i \partial_j \log \tilde{f}(x;\tilde{\theta}).$

The notation $\cristofsa{ij}{k}$ is called Christoffel's second symbol. We also use Christoffel's first symbol $\cristoffa{ij}{k}(0\leq i,j,k \leq p)$ defined by
\begin{align*}
\cristoffa{ij}{k}(\tilde{\theta})&\triangleq\sum_{t=0}^{p} \cristofsa{ij}{t} (\tilde{\theta})\: \tilde{g}^{kt}(\tilde{\theta}),  \\
&=
\begin{cases}
e^{-\theta^0}\sum_{t=1}^p \cristofsa{ij}{t}(\tilde{\theta})g^{kt}(\theta),  \text{ if $1\leq k \leq p$,}\\
e^{-\theta^0}\cristofsa{ij}{0}(\tilde{\theta}), \text{ if $k=0$.}
\end{cases}
\end{align*}
Cristoffel's symbols are not tensors. For example, when the coordinates are changed from $(i, j, k, \ldots )$ to $(\alpha, \beta, \gamma, \ldots ) $, the following exchange rule holds.
\begin{equation}
\label{exchange_crisotffel}
\cristofsa{ij}{k}=\cristofsa{\alpha\beta}{\gamma}B_i^\alpha B_j^\beta B_k^\gamma+B_k^\alpha B_{ij}^\gamma g_{\alpha \gamma}
\end{equation}
When $\alpha=\pm 1$, Christoffel's symbols are  denoted by $\cristofse{ij}{k}$, $\cristoffe{ij}{k}$ ($\alpha=1$) and $\cristofsm{ij}{k}$, $\cristoffm{ij}{k}$ ($\alpha=-1$). The concrete forms of $\cristofse{ij}{k}$ and $\cristofsm{ij}{k}$ are given by 
\begin{align}
\cristofse{ij}{k}(\tilde{\theta})&=\int_{\mathfrak X} \tilde{l}_{ij}(x;\tilde{\theta}) \tilde{l}_{k}(x;\tilde{\theta}) \tilde{f}(x;\tilde{\theta}) d\mu \label{Cristofse},\\
\cristofsm{ij}{k}(\tilde{\theta})&=\int_{\mathfrak X} \bigl(\tilde{l}_{ij}(x;\tilde{\theta}) \tilde{l}_{k}(x;\tilde{\theta})+\tilde{l}_i(x;\tilde{\theta})\tilde{l}_j(x;\tilde{\theta})\tilde{l}_k(x;\tilde{\theta})\bigr)\tilde{f}(x;\tilde{\theta}) d\mu .\label{Cristofsm}
\end{align}
From \eqref{Cristofse} and \eqref{Cristofsm}, we notice that 
\begin{equation}
\label{another_def_Cristofsa}
\cristofsa{ij}{k}(\tilde{\theta})=\cristofse{ij}{k}(\tilde{\theta})+\frac{1-\alpha}{2}\bigl(\cristofsm{ij}{k}(\tilde{\theta})-\cristofse{ij}{k}(\tilde{\theta})\bigr)
\end{equation}
The following relation is an important property for information geometry. It shows the duality between $\alpha$-covariant derivative and $-\alpha$-covariant derivative. Let $A, B, C$ be vector fields on $\tilde{\mathcal P}$, then
\begin{equation}
\label{duality_+-alpha}
A\bigl(\langle B, \:C \rangle\bigr)=
\langle {\oset{\scalebox{0.8}{$\alpha$}}{\widetilde\nabla}}_{A} B, \:C\rangle +\langle {\oset{\scalebox{0.8}{-$\alpha$}}{\widetilde\nabla}}_{A} C, \:B\rangle.
\end{equation}
Especially for $0 \leq i, j, k \leq p$, it turns out that
\begin{align}
\label{derivative_metric}
\partial_k\tilde{g}_{ij}&=\langle {\oset{\scalebox{0.8}{$\alpha$}}{\widetilde\nabla}}_{\partial_k} \partial_i, \:\partial_j\rangle+\langle {\oset{\scalebox{0.8}{$-\alpha$}}{\widetilde\nabla}}_{\partial_k} \partial_j, \:\partial_i\rangle \nonumber \\
&=\langle {\oset{\scalebox{0.8}{$\alpha$}}{\overline\nabla}}_{\partial_k} \partial_i, \:\partial_j\rangle+\langle {\oset{\scalebox{0.8}{$-\alpha$}}{\overline\nabla}}_{\partial_k} \partial_j, \:\partial_i\rangle \nonumber \\
&=\cristofsa{ik}{j}+\cristofsma{jk}{i}.
\end{align}

We define $\alpha$-covariant derivative in the space of $\mathcal{P}$ (denoted by  ${\oset{\scalebox{0.8}{$\alpha$}}{\nabla}}_{\partial_i} \partial_j, 1\leq i,j\leq p$) as the tangent vector field whose value at $\theta$ is the orthogonal projection of ${\oset{\scalebox{0.8}{$\alpha$}}{\widetilde\nabla}}_{\partial_i} \partial_j$ at $(0, \theta)$ on the tangent space of $\mathcal{P}$ at $\theta$. Since $\partial_i \perp \partial_0\ (1\leq i \leq p)$, we easily notice that 
$$
{\oset{\scalebox{0.8}{$\alpha$}}{\nabla}}_{\partial_i} \partial_j =\sum_{1\leq s, t \leq p} \cristofsa{ij}{s}(\theta) g^{st}(\theta) \partial_t= \wat \cristoffa{ij}{t}(\theta) \partial_t,
$$
where $\cristofsa{ij}{s}(\theta)\triangleq \cristofsa{ij}{s}((0,\theta))$ and $\cristoffa{ij}{t}(\theta)\triangleq \cristoffa{ij}{t}((0,\theta))$.

Now we define the second fundamental forms $\secfsa{ij}(\tilde\theta) ( 0 \leq i, j \leq p, -\infty < \alpha < \infty)$ of $\tilde{\mathcal{P}}$ as 
\begin{equation}
\label{def_2nd_fund_1}
\secfsa{ij}(\tilde\theta)\triangleq{\oset{\scalebox{0.8}{$\alpha$}}{\overline \nabla}}_{\partial_i} \partial_j -{\oset{\scalebox{0.8}{$\alpha$}}{\widetilde\nabla}}_{\partial_i} \partial_j.
\end{equation}
Its $\alpha$-representation, ${\oset{\scalebox{0.8}{$\alpha$}}{a}}_{ij}(x; \tilde{\theta})$, is given by 
\begin{align*}
&{\oset{\scalebox{0.8}{$\alpha$}}{a}}_{ij}(x; \tilde{\theta})\\
&=\text{($\alpha$-representation of ${\oset{\scalebox{0.8}{$\alpha$}}{\overline \nabla}}_{\partial_i} \partial_j $)}-\text{($\alpha$-representation of ${\oset{\scalebox{0.8}{$\alpha$}}{\widetilde\nabla}}_{\partial_i} \partial_j$)}\\
&=-\frac{1+\alpha}{2} \{\tilde{f}(x;\tilde{\theta}) \}^{-(3+\alpha)/2} \tilde{f}_i(x; \tilde{\theta})\tilde{f}_j(x; \tilde{\theta})+\{\tilde{f}(x;\tilde{\theta}) \}^{-(1+\alpha)/2}
\tilde{f}_{ij}(x; \tilde{\theta})\\
&\qquad -e^{-\theta^0}\sum_{1 \leq s, t \leq p} \cristofsa{ij}{s}(\tilde{\theta})  \: {g}^{st}(\theta)
\{\tilde{f}(x;\tilde{\theta}) \}^{-(1+\alpha)/2}\tilde{f}_t(x; \tilde{\theta}) \\
&\qquad- e^{-\theta^0}\cristofsa{ij}{0}(\tilde{\theta})\{\tilde{f}(x;\tilde{\theta}) \}^{-(1+\alpha)/2}\tilde{f}_0(x; \tilde{\theta}).
\end{align*}
For the case $\alpha=\pm 1$, we use the notations $\secfse{ij}(\tilde\theta)$, ${\oset{\scalebox{0.8}{$e$}}{a}}_{ij}(x; \tilde{\theta})$ ($\alpha=1$) and $\secfsm{ij}(\tilde\theta)$, ${\oset{\scalebox{0.8}{$m$}}{a}}_{ij}(x; \tilde{\theta})$ ($\alpha=-1$). The concrete forms of ${\oset{\scalebox{0.8}{$e$}}{a}}_{ij}(x; \tilde{\theta})$ and ${\oset{\scalebox{0.8}{$m$}}{a}}_{ij}(x; \tilde{\theta})$ are given as follows;
\begin{align*}
{\oset{\scalebox{0.8}{$e$}}{a}}_{ij}(x; \tilde{\theta})&=\tilde{l}_{ij}(x;\tilde{\theta})-e^{-\theta^0}\sum_{1 \leq  t \leq p} \cristoffe{ij}{t}(\tilde{\theta})  \:  \tilde{l}_{t}(x;\tilde{\theta})
+ e^{-\theta^0}\tilde{g}_{ij}(\tilde{\theta}),\\
{\oset{\scalebox{0.8}{$m$}}{a}}_{ij}(x; \tilde{\theta})&=\tilde{f}_{ij}(x; \tilde{\theta})-e^{-\theta^0}\sum_{1 \leq  t \leq p} \cristoffm{ij}{t}(\tilde{\theta})  \:  \tilde{f}_{t}(x;\tilde{\theta})-e^{-\theta^0}\tilde{f}(x;\tilde{\theta})\:\int_{\mathfrak X} \tilde{f}_{ij}(x;\tilde{\theta})d\mu.
\end{align*}
Especially when $1\leq i, j \leq p$
\begin{align}
{\oset{\scalebox{0.8}{$e$}}{a}}_{ij}(x; \tilde{\theta})&=l_{ij}(x;\theta)-e^{-\theta^0}\sum_{1 \leq  t \leq p} \cristoffe{ij}{t}(\tilde{\theta})  \: l_{t}(x;\theta)
+g_{ij}(\theta), \label{secfunde}\\
{\oset{\scalebox{0.8}{$m$}}{a}}_{ij}(x; \tilde{\theta})&=\tilde{f}_{ij}(x; \tilde{\theta})-e^{-\theta^0}\sum_{1 \leq  t \leq p} \cristoffm{ij}{t}(\tilde{\theta}) \:  \tilde{f}_{t}(x;\tilde{\theta}),
\label{secfundm}
\end{align}
since
$$
\int_{\mathfrak X} \tilde{f}_{ij}(x;\tilde{\theta})d\mu=e^{\theta^0}\int_{\mathfrak X} f_{ij}(x;\theta)d\mu=e^{\theta^0}\partial_i \partial_j\int_{\mathfrak X} f(x;\theta)d\mu=0.
$$

From the definition of  the covariant derivative, the $\alpha$-representation of ${\oset{\scalebox{0.8}{$\alpha$}}{\overline \nabla}}_{\partial_h}\secfsa{ij} $ is given by $\partial_h{\oset{\scalebox{0.8}{$\alpha$}}{a}}_{ij}(x; \tilde{\theta})$. The following equation holds;
\begin{align}
\int_{\mathfrak X} \bigl(\partial_h{\oset{\scalebox{0.8}{$\alpha$}}{a}}_{ij}(x; \tilde{\theta})\bigr)\: \tilde{f}_{k}(x; \tilde{\theta}) \{\tilde{f}(x;
\tilde{\theta})\}^{(\alpha-1)/2} d\mu&=\bigl\langle 
{\oset{\scalebox{0.8}{$\alpha$}}{\overline \nabla}}_{\partial_h}\secfsa{ij}, \partial_k \bigr\rangle_{\tilde{\theta}}\nonumber\\
&=-\bigl\langle \secfsa{ij}, {\oset{\scalebox{0.8}{$-\alpha$}}{\overline \nabla}}_{\partial_h} \partial_k \bigr\rangle_{\tilde{\theta}}\nonumber\\
&=-\bigr\langle \secfsa{ij}, {\oset{\scalebox{0.8}{$-\alpha$}}{\overline \nabla}}_{\partial_h} \partial_k -\tilde{\pi}\bigl({\oset{\scalebox{0.8}{$-\alpha$}}{\overline \nabla}}_{\partial_h}  \partial_k\bigr)\bigr\rangle_{\tilde{\theta}}\nonumber\\
&=-\bigl\langle \secfsa{ij}, \secfsma{hk} \bigr\rangle_{\tilde{\theta}},\label{deriva_secfe}
\end{align}
where the second and third equations hold since $\langle \secfsa{ij}, \partial_k \rangle_{\tilde{\theta}}=0$. 

We often use another type of second fundamental form 
\begin{equation}
\label{def_2nd_fund_2}
\secffa{i}{j}(\tilde\theta)\triangleq \secfsa{ik}(\tilde\theta)g^{kj}(\tilde\theta)=\secfsa{ki}(\tilde\theta)g^{kj}(\tilde\theta).
\end{equation}
The simplified notations $\secfsa{ij}(\theta), \secffa{i}{j}(\theta)$ are also used instead of $\secfsa{ij}((0,\theta))$ and $\secffa{i}{j}((0,\theta))$.

Lastly, we refer to Riemannian curvature tensor of $\mathcal{P}$. For vector fields $A, B, C$ on $\mathcal{P}$, Riemannian curvature tensor with respect to $\oset{\scalebox{0.8}{$\alpha$}}{\nabla}$ is defined as 
$$
\overset{\;\scalebox{0.7}{$\alpha$}}{R}(A, B)\:C\triangleq {\oset{\scalebox{0.8}{$\alpha$}}{\nabla}}_A\bigl(
{\oset{\scalebox{0.8}{$\alpha$}}{\nabla}}_B C\bigr)-{\oset{\scalebox{0.8}{$\alpha$}}{\nabla}}_B\bigl({\oset{\scalebox{0.8}{$\alpha$}}{\nabla}}_A C\bigr)-{\oset{\scalebox{0.8}{$\alpha$}}{\nabla}}_{AB-BA} C.
$$
The components of Riemannian curvature tensor, 
$\overset{\;\alpha}{R}\hspace{-5pt}\begin{smallmatrix}{\ \  \ \  l}\\{ijk}\end{smallmatrix}(\theta)$
, are defined as the unique function on $\mathcal{P}$ which satisfies
$$
\overset{\;\alpha}{R}(\partial_i, \partial_j)\:\partial_k = \sum_{l=1}^p 
\overset{\;\alpha}{R}\hspace{-5pt}\begin{smallmatrix}{\ \  \ \  l}\\{ijk}\end{smallmatrix}(\theta)\partial_l.
$$
More specifically it is given by
\begin{equation}
\label{specific_form_R}
\overset{\;\alpha}{R}\hspace{-5pt}\begin{smallmatrix}{\ \  \ \  l}\\{ijk}\end{smallmatrix}
(\theta)=\partial_i \cristoffa{jk}{l}(\theta)-\partial_j\cristoffa{ik}{l}(\theta)+\cristoffa{ir}{l}(\theta)\cristoffa{jk}{r}(\theta)-\cristoffa{jr}{l}(\theta)\cristoffa{ik}{r}(\theta).
\end{equation} 
\subsection{Geometric Interpretation of Derivatives of Log-likelihood}
\label{interpre_log}
The expectation of the derivatives of the log-likelihood $l(x;\theta)$ can be expressed in the geometric terms introduced in the previous subsection.

\begin{lemma}
\label{int_log_der}
For $1\leq i,\ j,\ h,\ k \leq p$, The following relations hold.
\begin{align}
E_\theta[l_i(x;\theta)]&=0, \label{log_i}\\
E_\theta[l_i(x;\theta) l_j(x;\theta)]&=-E_\theta[l_{ij}(x;\theta)]=g_{ij}(\theta),\label{log_ij}\\
E_\theta[l_{ij}(x;\theta) l_k(x;\theta)]&=\cristofse{ij}{k}(\theta),\label{log_ij_k}\\
E_\theta[l_i(x;\theta) l_j(x;\theta) l_k(x;\theta)]&=\cristofsm{ij}{k}(\theta)-\cristofse{ij}{k}(\theta),\label{log_i_j_k}\\
E_\theta[l_{ijk}(x;\theta)]&=-\bigl(\cristofse{ij}{k}(\theta)+\cristofse{ik}{j}(\theta)+\cristofsm{jk}{i}(\theta)\bigr),\label{log_ijk}\\
E_\theta[l_{ijh}(x;\theta)l_k(x;\theta)]&=\wat \bigl(\partial_h \cristoffe{ij}{t}(\theta)\bigr)g_{tk}(\theta)+
\wat \cristoffe{ij}{t}(\theta)\cristofse{th}{k}(\theta)\label{log_ijh_k}\nonumber\\
&\quad-\langle\secfse{ij}(\theta), \secfsm{hk}(\theta)\rangle,\\
E_\theta[l_{ij}(x;\theta) l_{kh}(x;\theta)]&=\langle\secfse{ij}(\theta), \secfse{kh}(\theta)\rangle+\wat\cristoffe{ij}{t}(\theta) \cristofse{kh}{t}(\theta) \nonumber\\
&\quad+g_{ij}(\theta)g_{kh}(\theta),\label{log_ij_kh}\\
E_\theta[l_{ij}(x;\theta) l_k(x;\theta) l_h(x;\theta)]&=\langle\secfse{ij}(\theta), \bigl(\secfsm{kh}(\theta)-\secfse{kh}(\theta)\bigr)\rangle-g_{ij}(\theta)g_{kh}(\theta)\nonumber\\
&\quad+\was\cristofse{ij}{s}(\theta)\bigr(\cristoffm{kh}{s}(\theta)-\cristoffe{kh}{s}(\theta)\bigr),\label{log_ij_k_h}\\
E_\theta[l_{ijkh}(x;\theta)]&=-\bigl(\partial_k \cristofse{ij}{h}(\theta)+\partial_k\cristofse{ih}{j}(\theta)+\partial_k\cristofsm{jh}{i}(\theta)\bigr)\nonumber\\
&\quad-\wat\bigl(\partial_h \cristoffe{ij}{t}(\theta)\bigr)g_{tk}(\theta)-\wat \cristoffe{ij}{t}(\theta) \cristofse{th}{k}(\theta)\nonumber\\
&\quad+\langle\secfse{ij}(\theta), \secfsm{hk}(\theta)\rangle\label{log_ijkh}.
\end{align}
\end{lemma}
\noindent
-\textit{ Proof }-\\
Equation \eqref{log_ij} is the result already mentioned in \eqref{metric1} and \eqref{metric2}. \eqref{log_ij_k} and \eqref{log_i_j_k} are special cases of \eqref{Cristofse} and \eqref{Cristofsm} when $\tilde{\theta}=(0, \theta)$ and $ 1 \leq i, j, k,  \leq p$.
In the following proof, we will abbreviate $f(x;\theta)$,  $l_i(x;\theta)$, $l_{ij}(x;\theta),\cdots$ respectively to $f$, $l_i$, $l_{ij},\cdots.$\\
-\textit{ Proof of \eqref{log_i}}-
$$
E_\theta[l_i]=\int_{\mathfrak X} f_i(x;\theta) d\mu =\partial_i \int_{\mathfrak X} f(x;\theta) d\mu=0.
$$
-\textit{ Proof of \eqref{log_ijk}}-\\
By differentiating the both sides of  $l_{ij}=-l_i l_j +f_{ij}/f$, we get the following equation.
\begin{align*}
l_{ijk}&=-l_{ik}l_j-l_il_{jk}+\frac{f_{ijk}}{f} -\frac{f_{ij}f_k}{f^2}\\
&=-l_{ik}l_j-l_il_{jk}+\frac{f_{ijk}}{f}\\
&\qquad -\biggl(\frac{f_{ij}}{f}-\frac{f_i}{f}\frac{f_j}{f}+\frac{f_i}{f}\frac{f_j}{f}\biggr)\biggl(\frac{f_k}{f}\biggr)\\
&=-(l_{ij} l_k+l_{ik}l_j+l_{jk}l_i)-l_i l_j l_k+\frac{f_{ijk}}{f}.
\end{align*}
Take expectation of both sides of the above equation, and use \eqref{log_ij_k}, \eqref{log_i_j_k} and the relation
$$
E_\theta[f_{ijk}/f]=\int_{\mathfrak X} f_{ijk} d\mu=\partial_i\partial_j\partial_k \int_{\mathfrak X} f d\mu=0,
$$
then we obtain the results.\\
-\textit{ Proof of \eqref{log_ijh_k}}-\\
From \eqref{secfunde} for the case $\tilde{\theta}=(0, \theta)$, 
\begin{equation}
\label{form_l_ij}
l_{ij}(x;\theta)={\oset{\scalebox{0.8}{$e$}}{a}}_{ij}(x;\theta)+\sum_{t=1}^p \cristoffe{ij}{t}(\theta)  l_{t}(x;\theta)-g_{ij}(\theta).
\end{equation}
By differentiating both sides of this equation,  we have 
$$
l_{ijh}(x;\theta)=\partial_h {\oset{\scalebox{0.8}{$e$}}{a}}_{ij}(x;\theta)+\sum_{t=1}^p (\partial_h\cristoffe{ij}{t}(\theta))  l_{t}(x;\theta)+\sum_{t=1}^p \cristoffe{ij}{t}(\theta)  l_{th}(x;\theta)-\partial_h g_{ij}(\theta).
$$
Consequently 
\begin{align*}
&E_{\theta}[l_{ijh} l_k]\\
&=E_{\theta}[(\partial_h {\oset{\scalebox{0.8}{$e$}}{a}}_{ij}) l_k]+\sum_{t=1}^p (\partial_h\cristoffe{ij}{t}(\theta)) E_{\theta}[ l_t l_k]+
\sum_{t=1}^p \cristoffe{ij}{t}(\theta) E_\theta[ l_{th} l_k]-(\partial_h g_{ij}(\theta))E_\theta[l_k] \\
&=-\langle \secfse{ij}, \secfsm{hk} \rangle_\theta+\wat (\partial_h \cristoffe{ij}{t}(\theta))g_{tk}(\theta)+
\wat \cristoffe{ij}{t}(\theta)\cristofse{th}{k}(\theta),
\end{align*}
where the second equation comes from \eqref{deriva_secfe} (for the case $\tilde\theta=(0, \theta)$ and $\alpha=1$), \eqref{log_i}, \eqref{log_ij} and \eqref{log_ij_k}.\\
-\textit{ Proof of \eqref{log_ij_kh}}-\\
From the definition of the second fundamental form, it turns out that $\langle \secfse{ij}, \partial_t \rangle=0 \ (0\leq i,j,t \leq p)$, hence, for $1\leq i,j,t \leq p$,
$$
E_\theta[{\oset{\scalebox{0.8}{$e$}}{a}}_{ij} l_t]=E_\theta[{\oset{\scalebox{0.8}{$e$}}{a}}_{ij}]=0.
$$
Using \eqref{form_l_ij} for $l_{ij}$, $l_{kh}$ and the relation \eqref{log_i}, \eqref{log_ij}, we notice the following equation holds.
\begin{align*}
E_\theta[l_{ij} l_{kh}]&=E_\theta[{\oset{\scalebox{0.8}{$e$}}{a}}_{ij} {\oset{\scalebox{0.8}{$e$}}{a}}_{kh}]+\sum_{1\leq s,t \leq p}
\cristoffe{ij}{t}\cristoffe{kh}{s}E_\theta[l_t l_s]+g_{ij} g_{kh}\\
&=\langle \secfse{ij}, \secfse{kh} \rangle +\wat
\cristoffe{ij}{t}\cristofse{kh}{t}+g_{ij} g_{kh}.
\end{align*}
-\textit{ Proof of \eqref{log_ij_k_h}}-\\
Since $l_k l_h=f_{kh}/f-l_{kh}$, 
\begin{equation}
\label{E_l_ij_k_h}
E_\theta[l_{ij} l_k l_h]=\int_{\mathfrak X} l_{ij} (f_{kh}/f-l_{kh}) f d\mu=\int_{\mathfrak X} l_{ij} f_{kh} d\mu- E_{\theta}[l_{ij}l_{kh}].
\end{equation}
From \eqref{secfundm} for the case $\tilde{\theta}=(0,\theta)$, we have
\begin{equation}
\label{form_f_kh}
f_{kh}(x;\theta)={\oset{\scalebox{0.8}{$m$}}{a}}_{kh}(x; \theta)+\wat \cristoffm{kh}{t}(\theta) \: f_{t}(x;\theta).
\end{equation}
If we substitute  $l_{ij}$ and $f_{kh}$ in the integrand of \eqref{E_l_ij_k_h} with the right-hand sides of \eqref{form_l_ij} and \eqref{form_f_kh} respectively,  we have
\begin{align*}
&\int_{\mathfrak X} l_{ij} f_{kh} d\mu \\
&=\int_{\mathfrak X} \biggl({\oset{\scalebox{0.8}{$e$}}{a}}_{ij} {\oset{\scalebox{0.8}{$m$}}{a}}_{kh} +{\oset{\scalebox{0.8}{$e$}}{a}}_{ij}\wat \cristoffm{kh}{t} f_{t}+{\oset{\scalebox{0.8}{$m$}}{a}}_{kh}\wat\cristoffe{ij}{t}  l_{t}\\
&\qquad\qquad+\wat\cristoffe{ij}{t}  l_{t} \was\cristoffm{kh}{s} f_{s}-g_{ij}{\oset{\scalebox{0.8}{$m$}}{a}}_{kh}-g_{ij}\wat\cristoffm{kh}{t} f_{t}\biggr) d\mu \\
&=\langle \secfse{ij}, \secfsm{kh} \rangle+\wat \cristoffm{kh}{t}\langle \secfse{ij}, \partial_t \rangle+ \wat\cristoffe{ij}{t} \langle \partial_t, \secfsm{kh}\rangle \\
&\qquad +\sum_{1\leq s, t \leq p}  \cristoffm{kh}{s} \cristoffe{ij}{t} g_{st}- g_{ij}\langle \secfsm{kh}, \partial_0 \rangle -g_{ij} \wat \cristoffm{kh}{t}\langle \partial_t, \partial_0 \rangle\\
&=\langle \secfse{ij}, \secfsm{kh} \rangle+\was  \cristoffm{kh}{s} \cristofse{ij}{s}.
\end{align*}
since $\langle \secfse{ij}, \partial_k \rangle=\langle \secfsm{ij}, \partial_k \rangle=0$ for $0 \leq i,j,k \leq p$, and $\langle \partial_t, \partial_0 \rangle=0$ for $1\leq t \leq p.$ Combine this equation with \eqref{log_ij_kh} and \eqref{E_l_ij_k_h}, then we obtain the results.\\
-\textit{ Proof of \eqref{log_ijkh}}-\\
If we differentiate both sides of the equation 
$$
\int_{\mathfrak X} l_{ijh}(x;\theta) f(x;\theta) d\mu=-(\cristofse{ij}{h}(\theta)+\cristofse{ih}{j}(\theta)+\cristofsm{jh}{i}(\theta))\quad (\text{see \eqref{log_ijk}}),
$$
then we have 
$$
\partial_k \int_{\mathfrak X} l_{ijh}(x;\theta) f(x;\theta) d\mu=-(\partial_k\cristofse{ij}{h}(\theta)+\partial_k\cristofse{ih}{j}(\theta)+\partial_k\cristofsm{jh}{i}(\theta))
$$
The left hand side equals $E_\theta[l_{ijkh}]+E_\theta[l_{ijh}l_k]$. From \eqref{log_ijh_k}, we have
\begin{align*}
E_\theta[l_{ijkh}]&=-(\partial_k\cristofse{ij}{h}+\partial_k\cristofse{ih}{j}+\partial_k\cristofsm{jh}{i})-E_\theta[l_{ijh} l_k]\\
&=-(\partial_k\cristofse{ij}{h}+\partial_k\cristofse{ih}{j}+\partial_k\cristofsm{jh}{i})\\
&\qquad-\wat (\partial_h \cristoffe{ij}{t})g_{tk}-\wat \cristoffe{ij}{t}\cristofse{th}{k}+\langle \secfse{ij}, \secfsm{hk} \rangle.
\end{align*}
\subsection{Expansion of Divergence}
\label{section:expan}
Substitute $\theta_1$ and $\theta_2$ in \eqref{alphadive} respectively with $\theta$ and $\theta_0$.  Fix $\theta_0$ in $\overset{\alpha}{D}[\theta: \theta_0]$ and treat it as the function of $\theta.$ Then Taylor expansion of $\overset{\alpha}{D}(\theta)=\overset{\alpha}{D}[\theta: \theta_0]$ around $\theta_0$ is given by 
\begin{align}
&\overset{\alpha}{D}(\theta)\nonumber\\
&=\wai \bigl(\epsilon_i \overset{\alpha}{D}[\theta_0: \theta_0]\bigr)(\theta^i-\theta_0^i) +
\frac{1}{2}\wai \waj \bigl(\epsilon_i \epsilon_j \overset{\alpha}{D}[\theta_0: \theta_0]\bigr)(\theta^i-\theta_0^i) (\theta^j-\theta_0^j)\nonumber\\
&\ +\frac{1}{6}\wai \waj \wak \bigl(\epsilon_i \epsilon_j \epsilon_k \overset{\alpha}{D}[\theta_0: \theta_0]\bigr)(\theta^i-\theta_0^i) (\theta^j-\theta_0^j)(\theta^k-\theta_0^k) \nonumber\\
&\ +\frac{1}{24}\wai \waj \wak \wal \bigl(\epsilon_i \epsilon_j \epsilon_k \epsilon_l \overset{\alpha}{D}[\theta_0: \theta_0]\bigr)(\theta^i-\theta_0^i) (\theta^j-\theta_0^j)(\theta^k-\theta_0^k)(\theta^l-\theta_0^l)\nonumber \\
&\ +O(||\theta-\theta_0||^5), \label{expansion_divergence}
\end{align}
where $\epsilon_i$ is the partial differentiation of $\overset{\alpha}{D}[\theta_1: \theta_2]$ with respect to $\theta_1^i$.

According to Eguchi's relationship (see \cite{Eguchi1}), 
\begin{align}
\epsilon_i \overset{\alpha}{D}[\theta_0: \theta_0]&=0,\label{Egur1}\\
\epsilon_i \epsilon_j  \overset{\alpha}{D}[\theta_0: \theta_0]&=g_{ij}(\theta_0),\label{Egur2}\\
\epsilon_i \epsilon_j \epsilon_k \overset{\alpha}{D}[\theta_0: \theta_0]&=\cristofsa{ij}{k}(\theta_0)+\cristofsa{ik}{j}(\theta_0)+\cristofsma{kj}{i}(\theta_0),\label{Egur3}
\end{align}
where $g_{ij}$ and $\cristofsa{ij}{k}$ are respectively the components of Fisher information metric and $\alpha$-connection of Riemannian manifold $\mathcal{P}$ (see Section \ref{basic_con_info_geo}). 

In this section we reveal the geometrical meaning of the forth derivative term
$$
\epsilon_i \epsilon_j \epsilon_k \epsilon_l \overset{\alpha}{D}[\theta_0: \theta_0]
$$
(see Eguchi \cite{Eguchi2} as a related work ).
Let $\delta_i(i=1,\dots,p)$ denote the partial derivative w.r.t. $\theta_2^i$ of $\overset{\alpha}{D}[\theta_1: \theta_2]$.
From \eqref{Egur3} and \eqref{another_def_Cristofsa}, we have
\begin{align}
\epsilon_i \epsilon_j \epsilon_k \overset{\alpha}{D}[\theta: \theta]&=\cristofsa{ij}{k}(\theta)+\cristofsa{ik}{j}(\theta)+\cristofsma{kj}{i}(\theta)\nonumber\\
&=\cristofse{ij}{k}(\theta)+\cristofse{ik}{j}(\theta)+\cristofse{kj}{i}(\theta)\nonumber\\
&\qquad+\frac{3-\alpha}{2}\bigl(\cristofsm{ij}{k}(\theta)-\cristofse{ij}{k}\bigr(\theta)) \text{ (see \eqref{another_def_Cristofsa})}.
\end{align} 
Differentiate both sides of this equation in $\theta^l$, then we have
\begin{align}
\label{i_j_k_l_D}
 \epsilon_i \epsilon_j \epsilon_k  \epsilon_l\overset{\alpha}{D}[\theta: \theta]
&=-\delta_l \epsilon_i \epsilon_j \epsilon_k  \overset{\alpha}{D}[\theta: \theta]
+\partial_l\cristofse{ij}{k}(\theta)+\partial_l\cristofse{ik}{j}(\theta)+\partial_l\cristofse{kj}{i}(\theta)\nonumber \\
&\qquad+\frac{3-\alpha}{2}\partial_l\bigl(\cristofsm{ij}{k}-\cristofse{ij}{k}\bigr).
\end{align}
From \eqref{alphadive}, we have
\begin{align*}
&\delta_l \epsilon_i \epsilon_j \epsilon_k  \overset{\alpha}{D}[\theta_1: \theta_2]\\
&=-\frac{(1-\alpha)^2}{4}\int_{\mathfrak X}l_i(x; \theta_1) l_j(x; \theta_1) l_k(x; \theta_1) l_l(x; \theta_2) f^{(1-\alpha)/2}(x ;\theta_1)  f^{(1+\alpha)/2}(x ;\theta_2) d\mu\\
&\quad -\frac{1-\alpha}{2}\int_{\mathfrak X}l_{ik}(x; \theta_1) l_j(x; \theta_1) l_l(x; \theta_2) f^{(1-\alpha)/2}(x ;\theta_1)  f^{(1+\alpha)/2}(x ;\theta_2) d\mu\\
&\quad -\frac{1-\alpha}{2} \int_{\mathfrak X}l_{i}(x; \theta_1) l_{jk}(x; \theta_1) l_l(x; \theta_2) f^{(1-\alpha)/2}(x ;\theta_1)  f^{(1+\alpha)/2}(x ;\theta_2) d\mu \\
&\quad -\frac{1-\alpha}{2} \int_{\mathfrak X}l_{ij}(x; \theta_1) l_k(x; \theta_1) l_l(x; \theta_2) f^{(1-\alpha)/2}(x ;\theta_1)  f^{(1+\alpha)/2}(x ;\theta_2) d\mu\\
&\quad -\int_{\mathfrak X}l_{ijk}(x; \theta_1)  l_l(x; \theta_2) f^{(1-\alpha)/2}(x ;\theta_1)  f^{(1+\alpha)/2}(x ;\theta_2) d\mu.
\end{align*}
Consequently we have
\begin{align}
\delta_l \epsilon_i \epsilon_j \epsilon_k  \overset{\alpha}{D}[\theta: \theta]
&=-\frac{(1-\alpha)^2}{4}E_\theta[l_i l_j l_k l_l]\nonumber\\
&\quad -\frac{1-\alpha}{2}\Bigl\{E_\theta[l_{ik} l_j l_l] +E_\theta[l_{jk} l_i l_l] +E_\theta[l_{ij} l_k l_l] \Bigr\}\nonumber\\
&\quad- E_\theta[l_{ijk} l_l]. \label{l_i_j_k_D}
\end{align}
The relation 
\begin{align*}
\partial_l E_\theta[l_i l_j l_k]&=\partial_l \int_{\mathfrak X} l_i l_j l_k f d\mu \\
&=\int_{\mathfrak X} l_{il} l_j l_k f d\mu +\int_{\mathfrak X} l_i l_{jl} l_k f d\mu +\int_{\mathfrak X} l_i l_j l_{kl} d\mu +\int_{\mathfrak X} l_i l_j l_k l_l f d\mu \\
&=E_{\theta}[l_{il} l_j l_k]+E_\theta[l_i l_{jl} l_k]+E_\theta[l_i l_j l_{kl}]+E_\theta[l_i l_j l_k l_l],
\end{align*}
and \eqref{log_i_j_k} lead to 
\begin{equation}
\label{E_log_i_j_k_l}
E_{\theta}[l_i l_j l_k l_l]=\partial_l\bigl(\cristofsm{ij}{k}-\cristofse{ij}{k}\bigr)-\bigl(E_{\theta}[l_{il} l_j l_k]+E_\theta[l_i l_{jl} l_k]+E_\theta[l_i l_j l_{kl}]\bigr).
\end{equation}
Substitute \eqref{E_log_i_j_k_l}, \eqref{log_ijh_k} and \eqref{log_ij_k_h} into \eqref{l_i_j_k_D}, and let $\alphad$ denote $(1-\alpha)/2$, Then we have
\begin{align*}
&\delta_l \epsilon_i \epsilon_j \epsilon_k  \overset{\alpha}{D}[\theta: \theta]\\
&=-(\alphad)^2 \partial_l \cristofsm{ij}{k}+(\alphad)^2 \partial_l \cristofse{ij}{k}\\
&\qquad +(\alphad)^2\Bigl(
\langle \secfse{il}, \secfsm{jk} \rangle- \langle \secfse{il}, \secfse{jk} \rangle - g_{il}g_{jk}+\was\cristofse{il}{s}\cristoffm{jk}{s}-\was\cristofse{il}{s}\cristoffe{jk}{s}\\
&\hspace{25mm} +\langle \secfse{jl}, \secfsm{ik} \rangle- \langle \secfse{jl}, \secfse{ik} \rangle - g_{jl}g_{ik}+\was\cristofse{jl}{s}\cristoffm{ik}{s}-\was\cristofse{jl}{s}\cristoffe{ik}{s}\\
&\hspace{25mm} +\langle \secfse{kl}, \secfsm{ij} \rangle- \langle \secfse{kl}, \secfse{ij} \rangle - g_{kl}g_{ij}+\was\cristofse{kl}{s}\cristoffm{ij}{s}-\was\cristofse{kl}{s}\cristoffe{ij}{s}
\Bigr)\\
&\qquad -\alphad\Bigl(
\langle \secfse{ik}, \secfsm{jl} \rangle- \langle \secfse{ik}, \secfse{jl} \rangle - g_{ik}g_{jl}+\was\cristofse{ik}{s}\cristoffm{jl}{s}-\was\cristofse{ik}{s}\cristoffe{jl}{s}\\
&\hspace{25mm} +\langle \secfse{jk}, \secfsm{il} \rangle- \langle \secfse{jk}, \secfse{il} \rangle - g_{jk}g_{il}+\was\cristofse{jk}{s}\cristoffm{il}{s}-\was\cristofse{jk}{s}\cristoffe{il}{s}\\
&\hspace{25mm} +\langle \secfse{ij}, \secfsm{kl} \rangle- \langle \secfse{ij}, \secfse{kl} \rangle - g_{ij}g_{kl}+\was\cristofse{ij}{s}\cristoffm{kl}{s}-\was\cristofse{ij}{s}\cristoffe{kl}{s}
\Bigr)\\
&\qquad -\wat\bigl(\partial_k \cristoffe{ij}{t} \bigr)g_{tl}-\wat\cristoffe{ij}{t}\cristofse{tk}{l}+\langle \secfse{ij}, \secfsm{kl} \rangle,
\end{align*}
where the definition and the geometrical meaning of each notation are described in Section \ref{basic_con_info_geo}.
From this equation and \eqref{i_j_k_l_D}, we have 
\begin{align}
&\epsilon_i \epsilon_j \epsilon_k \epsilon_l  \overset{\alpha}{D}[\theta: \theta]\nonumber \\
&=(\alphad)^2 \partial_l \cristofsm{ij}{k}-(\alphad)^2 \partial_l \cristofse{ij}{k}\nonumber \\
&\qquad -(\alphad)^2\Bigl(
\langle \secfse{il}, \secfsm{jk} \rangle- \langle \secfse{il}, \secfse{jk} \rangle - g_{il}g_{jk}+\was\cristofse{il}{s}\cristoffm{jk}{s}-\was\cristofse{il}{s}\cristoffe{jk}{s}\nonumber \\
&\hspace{25mm} +\langle \secfse{jl}, \secfsm{ik} \rangle- \langle \secfse{jl}, \secfse{ik} \rangle - g_{jl}g_{ik}+\was\cristofse{jl}{s}\cristoffm{ik}{s}-\was\cristofse{jl}{s}\cristoffe{ik}{s}\nonumber \\
&\hspace{25mm} +\langle \secfse{kl}, \secfsm{ij} \rangle- \langle \secfse{kl}, \secfse{ij} \rangle - g_{kl}g_{ij}+\was\cristofse{kl}{s}\cristoffm{ij}{s}-\was\cristofse{kl}{s}\cristoffe{ij}{s}
\Bigr)\nonumber \\
&\qquad +\alphad\Bigl(
\langle \secfse{ik}, \secfsm{jl} \rangle- \langle \secfse{ik}, \secfse{jl} \rangle - g_{ik}g_{jl}+\was\cristofse{ik}{s}\cristoffm{jl}{s}-\was\cristofse{ik}{s}\cristoffe{jl}{s}\nonumber \\
&\hspace{20mm} +\langle \secfse{jk}, \secfsm{il} \rangle- \langle \secfse{jk}, \secfse{il} \rangle - g_{jk}g_{il}+\was\cristofse{jk}{s}\cristoffm{il}{s}-\was\cristofse{jk}{s}\cristoffe{il}{s}\nonumber \\
&\hspace{20mm} +\langle \secfse{ij}, \secfsm{kl} \rangle- \langle \secfse{ij}, \secfse{kl} \rangle - g_{ij}g_{kl}+\was\cristofse{ij}{s}\cristoffm{kl}{s}-\was\cristofse{ij}{s}\cristoffe{kl}{s}
\Bigr)\nonumber \\
&\qquad+ (\alphad+1)\partial_l \bigl(\cristofsm{ij}{k}-\cristofse{ij}{k}\bigr)+\partial_l \cristofse{ij}{k}+\partial_l\cristofse{ik}{j} +\partial_l \cristofse{kj}{i}\nonumber \\
&\qquad +\wat\bigl(\partial_k \cristoffe{ij}{t} \bigr)g_{tl}+\wat\cristoffe{ij}{t}\cristofse{tk}{l}-\langle \secfse{ij}, \secfsm{kl} \rangle. \label{epsi^4_Diverge}
\end{align}
\section{Appendix 2 -Complementary Proof-}
\subsection{Proof of \eqref{expan_barteta}}
\label{append2-1}

If we insert the right-hand side of \eqref{expand_barteta} into $\barteta{j},\barteta{k},\barteta{l},\barteta{m}$ of itself, we have the following equation. 
\begin{align}
&\barteta{s}\nonumber\\
&=\bareo{s}+\sum_j \aaa{j}{s}
\biggl(
\bareo{j}+\sum_i \aaa{i}{j}\barteta{i}+\sum_{i,k}\bee{ik}{j}\barteta{i}\barteta{k}+\sum_{i,k}\bbar{ik}{j}\barteta{i}\barteta{k}+
\sum_{i,k,l}\cee{ikl}{j}\barteta{i}\barteta{k}\barteta{l}\nonumber\\
&\hspace{30mm}+\sum_{ikl}\cbar{ikl}{j}\barteta{i}\barteta{k}\barteta{l}
+\sum_{i,k,l,m}\dee{iklm}{j}\barteta{i}\barteta{k}\barteta{l}\barteta{m}
\biggr)\nonumber\\
&+\sum_{j,k}\biggl(\bee{jk}{s}+\bbar{jk}{s}\biggr)\nonumber\\
&\hspace{20mm} \times\biggl(
\bareo{j}+\sum_i \aaa{i}{j}\barteta{i}+\sum_{i,k}\bee{ik}{j}\barteta{i}\barteta{k}+\sum_{i,k}\bbar{ik}{j}\barteta{i}\barteta{k}+
\sum_{i,k,l}\cee{ikl}{j}\barteta{i}\barteta{k}\barteta{l}\nonumber\\
&\hspace{30mm}+\sum_{ikl}\cbar{ikl}{j}\barteta{i}\barteta{k}\barteta{l}
+\sum_{i,k,l,m}\dee{iklm}{j}\barteta{i}\barteta{k}\barteta{l}\barteta{m}
\biggr)\nonumber\\
&\hspace{20mm} \times\biggl(
\bareo{k}+\sum_i \aaa{i}{k}\barteta{i}+\sum_{i,j}\bee{ij}{k}\barteta{i}\barteta{j}+\sum_{i,j}\bbar{ij}{k}\barteta{i}\barteta{j}+
\sum_{i,j,l}\cee{ijl}{k}\barteta{i}\barteta{j}\barteta{l}\nonumber\\
&\hspace{30mm}+\sum_{ijl}\cbar{ijl}{k}\barteta{i}\barteta{j}\barteta{l}
+\sum_{i,j,l,m}\dee{ijlm}{k}\barteta{i}\barteta{j}\barteta{l}\barteta{m}
\biggr)\nonumber\\
&+\sum_{j,k,l}\biggl(\cee{jkl}{s}+\cbar{jkl}{s}\biggr)\nonumber\\
&\hspace{20mm} \times\biggl(
\bareo{j}+\sum_i \aaa{i}{j}\barteta{i}+\sum_{i,k}\bee{ik}{j}\barteta{i}\barteta{k}+\sum_{i,k}\bbar{ik}{j}\barteta{i}\barteta{k}+
\sum_{i,k,l}\cee{ikl}{j}\barteta{i}\barteta{k}\barteta{l}\nonumber\\
&\hspace{30mm}+\sum_{ikl}\cbar{ikl}{j}\barteta{i}\barteta{k}\barteta{l}
+\sum_{i,k,l,m}\dee{iklm}{j}\barteta{i}\barteta{k}\barteta{l}\barteta{m}
\biggr)\nonumber\\
&\hspace{20mm} \times\biggl(
\bareo{k}+\sum_i \aaa{i}{k}\barteta{i}+\sum_{i,j}\bee{ij}{k}\barteta{i}\barteta{j}+\sum_{i,j}\bbar{ij}{k}\barteta{i}\barteta{j}+
\sum_{i,j,l}\cee{ijl}{k}\barteta{i}\barteta{j}\barteta{l}\nonumber\\
&\hspace{30mm}+\sum_{i,j,l}\cbar{ijl}{k}\barteta{i}\barteta{j}\barteta{l}+\sum_{i,j,l,m}\dee{ijlm}{k}\barteta{i}\barteta{j}\barteta{l}\barteta{m}
\biggr)\nonumber\\
&\hspace{20mm} \times\biggl(
\bareo{l}+\sum_i \aaa{i}{l}\barteta{i}+\sum_{i,j}\bee{ij}{l}\barteta{i}\barteta{j}+\sum_{i,j}\bbar{ij}{l}\barteta{i}\barteta{j}+
\sum_{i,j,k}\cee{ijk}{l}\barteta{i}\barteta{j}\barteta{k}\nonumber\\
&\hspace{30mm}+\sum_{i,j,k}\cbar{ijk}{l}\barteta{i}\barteta{j}\barteta{k}+\sum_{i,j,k,m}\dee{ijkm}{l}\barteta{i}\barteta{j}\barteta{k}\barteta{m}
\biggr)\nonumber\\
&+\sum_{j,k,l,m}\dee{jklm}{s}\nonumber\\
&\hspace{20mm} \times\biggl(
\bareo{j}+\sum_i \aaa{i}{j}\barteta{i}+\sum_{i,k}\bee{ik}{j}\barteta{i}\barteta{k}+\sum_{i,k}\bbar{ik}{j}\barteta{i}\barteta{k}+
\sum_{i,k,l}\cee{ikl}{j}\barteta{i}\barteta{k}\barteta{l}\nonumber\\
&\hspace{30mm}+\sum_{ikl}\cbar{ikl}{j}\barteta{i}\barteta{k}\barteta{l}
+\sum_{i,k,l,m}\dee{iklm}{j}\barteta{i}\barteta{k}\barteta{l}\barteta{m}
\biggr)\nonumber\\
&\hspace{20mm} \times\biggl(
\bareo{k}+\sum_i \aaa{i}{k}\barteta{i}+\sum_{i,j}\bee{ij}{k}\barteta{i}\barteta{j}+\sum_{i,j}\bbar{ij}{k}\barteta{i}\barteta{j}+
\sum_{i,j,l}\cee{ijl}{k}\barteta{i}\barteta{j}\barteta{l}\nonumber\\
&\hspace{30mm}+\sum_{i,j,l}\cbar{ijl}{k}\barteta{i}\barteta{j}\barteta{l}+\sum_{i,j,l,m}\dee{ijlm}{k}\barteta{i}\barteta{j}\barteta{l}\barteta{m}
\biggr)\nonumber\\
&\hspace{20mm} \times\biggl(
\bareo{l}+\sum_i \aaa{i}{l}\barteta{i}+\sum_{i,j}\bee{ij}{l}\barteta{i}\barteta{j}+\sum_{i,j}\bbar{ij}{l}\barteta{i}\barteta{j}+
\sum_{i,j,k}\cee{ijk}{l}\barteta{i}\barteta{j}\barteta{k}\nonumber\\
&\hspace{30mm}+\sum_{i,j,k}\cbar{ijk}{l}\barteta{i}\barteta{j}\barteta{k}+\sum_{i,j,k,m}\dee{ijkm}{l}\barteta{i}\barteta{j}\barteta{k}\barteta{m}
\biggr)\nonumber\\
&\hspace{20mm} \times\biggl(
\bareo{m}+\sum_i \aaa{i}{m}\barteta{i}+\sum_{i,j}\bee{ij}{m}\barteta{i}\barteta{j}+\sum_{i,j}\bbar{ij}{m}\barteta{i}\barteta{j}+
\sum_{i,j,k}\cee{ijk}{m}\barteta{i}\barteta{j}\barteta{k}\nonumber\\
&\hspace{30mm}+\sum_{i,j,k}\cbar{ijk}{m}\barteta{i}\barteta{j}\barteta{k}+\sum_{i,j,k,l}\dee{ijkl}{m}\barteta{i}\barteta{j}\barteta{k}\barteta{l}
\biggr) \label{expand_barteta_2}
\end{align}
Expanding the equation, counting the order of each term, we can rewrite \eqref{expand_barteta_2} as
\begin{align*}
\barteta{s}&=\bareo{s}+\sum_j \aaa{j}{s}
\biggl(
\bareo{j}+\sum_i \aaa{i}{j}\barteta{i}+\sum_{i,k}\bbar{ik}{j}\barteta{i}\barteta{k}
\biggr)\\
&\quad+\sum_{j,k}\biggl(\bee{jk}{s}+\bbar{jk}{s}\biggr)\\
&\hspace{20mm} \times\biggl(
\bareo{j}+\sum_i \aaa{i}{j}\barteta{i}+\sum_{i,k}\bbar{ik}{j}\barteta{i}\barteta{k}
\biggr)\\
&\hspace{20mm} \times\biggl(
\bareo{k}+\sum_i \aaa{i}{k}\barteta{i}+\sum_{i,j}\bbar{ij}{k}\barteta{i}\barteta{j}
\biggr)\\
&\quad+\sum_{j,k,l}\biggl(\cee{jkl}{s}+\cbar{jkl}{s}\biggr)\bareo{j}\bareo{k}\bareo{l}+Re1,
\end{align*}
where $Re1$  is the polynomial  with respect to the variables $\barteta{s}$, $\bareo{s}$, $\aaa{j}{s}$, $\bee{jk}{s}$, $\cee{jkl}{s}$, $\dee{jklm}{s}$ $(1\leq j, k, l, m, s \leq p)$, and each term is of at least fourth order with respect to $\barteta{s}$, $\bareo{s}$, $\aaa{j}{s}$, $\bee{jk}{s}$, $\cee{jkl}{s}$ $(1\leq j, k, l, s \leq p)$.
If we insert this result into the right-hand side of itself, then we  yield the result.
\begin{align*}
\barteta{s}&=\bareo{s}+\sum_j \aaa{j}{s}
\bareo{j}+\sum_{i,j}\aaa{j}{s}\aaa{i}{j}\bareo{i}+\sum_{i,j,k}\aaa{j}{s}\bbar{ik}{j}\bareo{i}\bareo{k}\\
&+\sum_{j,k}\biggl(\bee{jk}{s}+\bbar{jk}{s}\biggr)\\
&\hspace{20mm} \times\biggl(
\bareo{j}+\sum_i \aaa{i}{j}\bareo{i}+\sum_{i,l}\bbar{il}{j}\bareo{i}\bareo{l}
\biggr)\\
&\hspace{20mm} \times\biggl(
\bareo{k}+\sum_i \aaa{i}{k}\bareo{i}+\sum_{i,l}\bbar{il}{k}\bareo{i}\bareo{l}
\biggr)\\
&+\sum_{j,k,l}\biggl(\cee{jkl}{s}+\cbar{jkl}{s}\biggr)\bareo{j}\bareo{k}\bareo{l}+Re2,
\end{align*}
where $Re2$ has the same property of $Re1$.
\subsection{Proof of \eqref{expec_barteta^2}, \eqref{expec_barteta^3}, \eqref{expec_barteta^4}}
\label{append2-2}
In this subsection, for brevity, we use Einstein summation notation. \\
-\textit{ Proof of \eqref{expec_barteta^2}}- \\
From \eqref{expan_barteta}, we have
\begin{equation}
\label{barteta^2}
\begin{split}
&\barteta{i}\barteta{j}\\
&=\Bigl(\bareo{i}+\aaa{l}{i}\bareo{l}+\bbar{lm}{i}\bareo{l}\bareo{m}+\aaa{l}{i}\aaa{m}{l}\bareo{m}+\aaa{l}{i}\bbar{ms}{l}\bareo{m}\bareo{s}+\bee{lm}{i}\bareo{l}\bareo{m}+2\bbar{lm}{i}\aaa{s}{m}\bareo{l}\bareo{s}\\
&\qquad+2\bbar{lm}{i}\bbar{st}{m}\bareo{l}\bareo{s}\bareo{t}+\cbar{lmt}{i}\bareo{l}\bareo{m}\bareo{t}+Re(4)\Bigr)\\
&\times\Bigl(\bareo{j}+\aaa{l}{j}\bareo{l}+\bbar{lm}{j}\bareo{l}\bareo{m}+\aaa{l}{j}\aaa{m}{l}\bareo{m}+\aaa{l}{j}\bbar{ms}{l}\bareo{m}\bareo{s}+\bee{lm}{j}\bareo{l}\bareo{m}+2\bbar{lm}{j}\aaa{s}{m}\bareo{l}\bareo{s}\\
&\qquad+2\bbar{lm}{j}\bbar{st}{m}\bareo{l}\bareo{s}\bareo{t}+\cbar{lmt}{j}\bareo{l}\bareo{m}\bareo{t}+Re(4)\Bigr)\\
&=\bareo{i}\bareo{j}+\aaa{l}{j}\bareo{i}\bareo{l}+\aaa{l}{i}\bareo{j}\bareo{l}+\bbar{lm}{j}\bareo{i}\bareo{l}\bareo{m}+\bbar{lm}{i}\bareo{j}\bareo{l}\bareo{m}\\
&\quad+\aaa{l}{j}\aaa{m}{l}\bareo{i}\bareo{m}+\aaa{l}{j}\bbar{ms}{l}\bareo{i}\bareo{m}\bareo{s}+\bee{lm}{j}\bareo{i}\bareo{l}\bareo{m}\\
&\quad+2\bbar{lm}{j}\aaa{s}{m}\bareo{i}\bareo{l}\bareo{s}+2\bbar{lm}{j}\bbar{st}{m}\bareo{i}\bareo{l}\bareo{s}\bareo{t}+\cbar{lmt}{j}\bareo{i}\bareo{l}\bareo{m}\bareo{t}\\
&\quad+\aaa{l}{i}\aaa{m}{l}\bareo{j}\bareo{m}+\aaa{l}{i}\bbar{ms}{l}\bareo{j}\bareo{m}\bareo{s}+\bee{lm}{i}\bareo{j}\bareo{l}\bareo{m}\\
&\quad+2\bbar{lm}{i}\aaa{s}{m}\bareo{j}\bareo{l}\bareo{s}+2\bbar{lm}{i}\bbar{st}{m}\bareo{j}\bareo{l}\bareo{s}\bareo{t}+\cbar{lmt}{i}\bareo{j}\bareo{l}\bareo{m}\bareo{t}\\
&\quad+\aaa{l}{i}\aaa{m}{j}\bareo{l}\bareo{m}+\aaa{l}{i}\bbar{st}{j}\bareo{l}\bareo{s}\bareo{t}+\aaa{l}{j}\bbar{st}{i}\bareo{l}\bareo{s}\bareo{t}+\bbar{lm}{i}\bbar{st}{j}\bareo{l}\bareo{m}\bareo{s}\bareo{t}+Re(5),
\end{split}
\end{equation}
where $Re(5)$ is  the polynomial  with respect to the variables $\barteta{s}$, $\bareo{s}$, $\aaa{j}{s}$, $\bee{jk}{s}$, $\cee{jkl}{s}$, $\dee{jklm}{s}$ $(1\leq j, k, l, m, s \leq p)$, and each term is of at least fifth order with respect to $\barteta{s}$, $\bareo{s}$, $\aaa{j}{s}$, $\bee{jk}{s}$, $\cee{jkl}{s}$ $(1\leq j, k, l, s \leq p)$.
We calculate the expectation of each term on the right-hand side of this equation. We omit the subscript $\theta$ in the notation of the expectation.
\begin{align}
&E[\bareo{i}\bareo{j}]\nonumber\\
&=n^{-2}E\biggl[\biggl(\sum_{a=1}^n g^{il}\frac{\partial\hfill}{\partial \theta^l}\log f(X_a;\theta)\biggr)\biggl(\sum_{b=1}^n g^{jm}\frac{\partial\hfill}{\partial \theta^m}\log f(X_b;\theta)\biggr)
\biggr]\nonumber\\
&=n^{-1}E\biggl[g^{il}g^{jm}\frac{\partial\hfill }{\partial \theta^l}\log f(X;\theta)\frac{\partial\hfill }{\partial \theta^m}\log f(X;\theta)\biggr]\nonumber\\
&\quad+n^{-2}\sum_{a\ne b}g^{il}E\biggl[\frac{\partial\hfill }{\partial \theta^l}\log f(X_a;\theta)\biggr]g^{jm}E\biggl[ \frac{\partial\hfill}{\partial \theta^m}\log f(X_b;\theta)\biggr]\nonumber\\
&=n^{-1}g^{il}g^{jm}g_{lm}\nonumber\\
&=n^{-1}g^{ij}.\label{bareo_bareo}
\end{align}
\begin{equation}
\label{aaa_bareo_bareo}
\begin{split}
&E[\aaa{l}{j}\bareo{i}\bareo{l}]\\
&=E\biggl[
n^{-1}g^{sj} \sum_{c=1}^n\biggl(\frac{\partial^2\hfill}{\partial \theta^s\partial \theta^l}\log f(X_c;\theta)+g_{ls}\biggr)\\
&\qquad\times n^{-2}g^{it}\biggl(\sum_{a=1}^n \frac{\partial\hfill}{\partial\theta^t}\log f(X_a;\theta)\biggr)g^{lm}\biggl(\sum_{b=1}^n \frac{\partial\hfill}{\partial\theta^m}\log f(X_b;\theta)\biggr)
\biggr]\\
&=n^{-3}g^{sj}g^{it}g^{lm}\\
&\quad\times\biggl\{ n E\biggl[\biggl(\frac{\partial^2\hfill}{\partial \theta^s \partial \theta^l} \log f(X;\theta)\biggr)\biggl(\frac{\partial\hfill}{\partial\theta^t}\log f(X;\theta)\biggr)\biggl( \frac{\partial\hfill}{\partial\theta^m}\log f(X;\theta) \biggr)\biggr]\\
&\qquad\quad+ng_{ls}E\biggl[\biggl(\frac{\partial\hfill}{\partial\theta^t}\log f(X;\theta)\biggr)\biggl( \frac{\partial\hfill}{\partial\theta^m}\log f(X;\theta) \biggr)\biggr]\\
&\qquad\quad+ \sum_{a\ne c}E\biggl[
\frac{\partial^2\hfill}{\partial \theta^s\partial \theta^l}\log f(X_c;\theta)+g_{ls}
\biggr]\\
&\qquad\qquad\times E\biggl[\biggl(\frac{\partial\hfill}{\partial\theta^t}\log f(X_a;\theta)\biggr)\biggl( \frac{\partial\hfill}{\partial\theta^m}\log f(X_a;\theta) \biggr)\biggr]\\
&\qquad\quad+\sum_{a\ne b} E\biggl[\frac{\partial\hfill}{\partial\theta^m}\log f(X_b;\theta)\biggr]\\
&\qquad\qquad\times E\biggl[ \biggl(\frac{\partial^2\hfill}{\partial \theta^s\partial \theta^l}\log f(X_a;\theta)+g_{ls}\biggr)\biggl(\frac{\partial\hfill}{\partial\theta^t}\log f(X_a;\theta)\biggr)\biggr]\\
&\qquad\quad+\sum_{b\ne a}E\biggl[\frac{\partial\hfill}{\partial\theta^t}\log f(X_a;\theta)\biggr]\\
&\qquad\qquad\times E\biggl[ \biggl(\frac{\partial^2\hfill}{\partial \theta^s\partial \theta^l}\log f(X_b;\theta)+g_{ls}\biggr)\biggl(\frac{\partial\hfill}{\partial\theta^m}\log f(X_b;\theta)\biggr)\biggr]\\\
&\qquad\quad+\sum_{a\ne b, a\ne c, b\ne c}E\biggl[\frac{\partial^2\hfill}{\partial \theta^s\partial \theta^l}\log f(X_c;\theta)+g_{ls}\biggr]\\
&\qquad\qquad\times E\biggl[\frac{\partial\hfill}{\partial\theta^t}\log f(X_a;\theta)\biggr]E\biggl[\frac{\partial\hfill}{\partial\theta^m}\log f(X_b;\theta)\biggr]
\biggr\}\\
&=n^{-2}g^{sj}g^{it}g^{lm}\\
&\quad\times\biggl\{E\biggl[\biggl(\frac{\partial^2\hfill}{\partial \theta^s \partial \theta^l} \log f(X;\theta)\biggr)\biggl(\frac{\partial\hfill}{\partial\theta^t}\log f(X;\theta)\biggr)\biggl( \frac{\partial\hfill}{\partial\theta^m}\log f(X;\theta) \biggr)\biggr]+g_{ls}g_{tm}\biggr\}.
\end{split}
\end{equation}
\begin{equation}
\label{bbar_bareo^3}
\begin{split}
&E[\bbar{lm}{j}\bareo{i}\bareo{l}\bareo{m}]\\
&=\bbar{lm}{j}E[\bareo{i}\bareo{l}\bareo{m}]\\
&=n^{-3}\bbar{lm}{j}g^{ik}g^{ls}g^{mt}E\biggl[
\biggl(\sum_{a=1}^n\frac{\partial\hfill}{\partial\theta^k}\log f(X_a;\theta)\biggr)
\biggl(\sum_{b=1}^n\frac{\partial\hfill}{\partial\theta^s}\log f(X_b;\theta)\biggr)
\biggl(\sum_{c=1}^n\frac{\partial\hfill}{\partial\theta^t}\log f(X_c;\theta)\biggr)
\biggr]\\
&=n^{-3}\bbar{lm}{j}g^{ik}g^{ls}g^{mt}\\
&\quad\times\biggl\{nE\biggl[\biggl(\frac{\partial\hfill}{\partial \theta^k}\log f(X;\theta)\biggr)\biggl(\frac{\partial\hfill}{\partial \theta^s}\log f(X;\theta)\biggr)\biggl(\frac{\partial\hfill}{\partial \theta^t}\log f(X;\theta)\biggr)\biggr]\\
&\quad\qquad+\sum_{a\ne c}E\biggl[\biggl( \frac{\partial\hfill}{\partial \theta^k}\log f(X_a;\theta)\biggr)\biggl(\frac{\partial\hfill}{\partial \theta^s}\log f(X_a;\theta)\biggr)\biggr]
E\biggl[\frac{\partial\hfill}{\partial \theta^t}\log f(X_c;\theta)\biggr]\\
&\quad\qquad+\sum_{a\ne c}E\biggl[\biggl( \frac{\partial\hfill}{\partial \theta^k}\log f(X_a;\theta)\biggr)\biggl(\frac{\partial\hfill}{\partial \theta^t}\log f(X_a;\theta)\biggr)\biggr]
E\biggl[\frac{\partial\hfill}{\partial \theta^s}\log f(X_c;\theta)\biggr]\\
&\quad\qquad+\sum_{a\ne c}E\biggl[\biggl( \frac{\partial\hfill}{\partial \theta^s}\log f(X_a;\theta)\biggr)\biggl(\frac{\partial\hfill}{\partial \theta^t}\log f(X_a;\theta)\biggr)\biggr]
E\biggl[\frac{\partial\hfill}{\partial \theta^k}\log f(X_c;\theta)\biggr]\\
&\quad\qquad+\sum_{a\ne b, a\ne c, b\ne c} E\biggl[\frac{\partial\hfill}{\partial \theta^k}\log f(X_a;\theta)\biggr]E\biggl[\frac{\partial\hfill}{\partial \theta^s}\log f(X_b;\theta)\biggr]E\biggl[\frac{\partial\hfill}{\partial \theta^t}\log f(X_c;\theta)\biggr]
\biggr\}\\
&=n^{-2}\bbar{lm}{j}g^{ik}g^{ls}g^{mt}\\
&\quad\times E\biggl[\biggl(\frac{\partial\hfill}{\partial \theta^k}\log f(X;\theta)\biggr)\biggl(\frac{\partial\hfill}{\partial \theta^s}\log f(X;\theta)\biggr)\biggl(\frac{\partial\hfill}{\partial \theta^t}\log f(X;\theta)\biggr)\biggr].
\end{split}
\end{equation}
\begin{align}
&E[\aaa{l}{j}\aaa{m}{l}\bareo{i}\bareo{m}]\nonumber\\
&=n^{-4}E\biggl[\biggl\{g^{jk}\sum_{a=1}^n \biggl(\frac{\partial^2\hfill}{\partial \theta^k\partial \theta^l}\log f(X_a;\theta)+g_{kl}\biggr)\biggr\}\nonumber\\
&\qquad\qquad\times\biggl\{g^{lu} \sum_{b=1}^n \biggl(\frac{\partial^2\hfill}{\partial \theta^u\partial \theta^m}\log f(X_b;\theta)+g_{um}\biggr)\biggr\}\nonumber\\
&\qquad\qquad\times\biggl\{g^{is}\sum_{c=1}^n\frac{\partial\hfill}{\partial \theta^s}\log f(X_c;\theta)\biggr\}\biggl\{g^{mt}\sum_{d=1}^n\frac{\partial\hfill}{\partial \theta^t}\log f(X_d;\theta)\biggr\}\biggr]\nonumber\\
&=n^{-4}g^{jk}g^{lu}g^{is}g^{mt}\nonumber\\
&\quad\times\biggl\{\sum_{a=1}^nE\biggl[ \biggl(\frac{\partial^2\hfill}{\partial \theta^k\partial \theta^l}\log f(X_a;\theta)+g_{kl}\biggr)\biggl(\frac{\partial^2\hfill}{\partial \theta^u\partial \theta^m}\log f(X_a;\theta)+g_{um}\biggr)\nonumber\\
&\qquad\qquad\qquad\times\biggl(\frac{\partial\hfill}{\partial \theta^s}\log f(X_a;\theta)\biggr)\biggl(\frac{\partial\hfill}{\partial \theta^t}\log f(X_a;\theta)\biggr)\biggr]\nonumber\\
&\qquad\quad+\sum_{a\ne b} E\biggl[\frac{\partial^2\hfill}{\partial \theta^k\partial \theta^l}\log f(X_a;\theta)+g_{kl}\biggr]\nonumber\\
&\qquad\qquad\times
E\biggl[\biggl(\frac{\partial^2\hfill}{\partial \theta^u\partial \theta^m}\log f(X_b;\theta)+g_{um}\biggr)\biggl(\frac{\partial\hfill}{\partial \theta^s}\log f(X_b;\theta)\biggr)\biggl(\frac{\partial\hfill}{\partial \theta^t}\log f(X_b;\theta)\biggr)
\biggr]\nonumber\\
&\qquad\quad+\sum_{a\ne b} E\biggl[\frac{\partial^2\hfill}{\partial \theta^u\partial \theta^m}\log f(X_a;\theta)+g_{um}\biggr]\nonumber\\
&\qquad\qquad\times
E\biggl[\biggl(\frac{\partial^2\hfill}{\partial \theta^k\partial \theta^l}\log f(X_b;\theta)+g_{kl}\biggr)\biggl(\frac{\partial\hfill}{\partial \theta^s}\log f(X_b;\theta)\biggr)\biggl(\frac{\partial\hfill}{\partial \theta^t}\log f(X_b;\theta)\biggr)
\biggr]\nonumber\\
&\qquad\quad+ \sum_{a\ne c}E\biggl[\biggl(\frac{\partial^2\hfill}{\partial \theta^k\partial \theta^l}\log f(X_a;\theta)+g_{kl}\biggr)\biggl(\frac{\partial^2\hfill}{\partial \theta^u\partial \theta^m}\log f(X_a;\theta)+g_{um}\biggr)\nonumber\\
&\qquad\qquad\qquad\times\biggl(\frac{\partial\hfill}{\partial \theta^t}\log f(X_a;\theta)\biggr)\biggr] E\biggl[\frac{\partial\hfill}{\partial \theta^s}\log f(X_c;\theta)\biggr]\nonumber\\
&\qquad\quad+ \sum_{a\ne c}E\biggl[\biggl(\frac{\partial^2\hfill}{\partial \theta^k\partial \theta^l}\log f(X_a;\theta)+g_{kl}\biggr)\biggl(\frac{\partial^2\hfill}{\partial \theta^u\partial \theta^m}\log f(X_a;\theta)+g_{um}\biggr)\nonumber\\
&\qquad\qquad\qquad\times\biggl(\frac{\partial\hfill}{\partial \theta^s}\log f(X_a;\theta)\biggr)\biggr] E\biggl[\frac{\partial\hfill}{\partial \theta^t}\log f(X_c;\theta)\biggr]\nonumber\\
&\qquad\quad+\sum_{a\ne c} E\biggl[\biggl(\frac{\partial^2\hfill}{\partial \theta^k\partial \theta^l}\log f(X_a;\theta)+g_{kl}\biggr)\biggl(\frac{\partial^2\hfill}{\partial \theta^u\partial \theta^m}\log f(X_a;\theta)+g_{um}\biggr)\biggr]\nonumber\\
&\qquad\qquad\times E\biggl[\biggl(\frac{\partial\hfill}{\partial \theta^s}\log f(X_c;\theta)\biggr)\biggl(\frac{\partial\hfill}{\partial \theta^t}\log f(X_c;\theta)\biggr)\biggr]\nonumber\\
&\qquad\quad+\sum_{a\ne c} E\biggl[\biggl(\frac{\partial^2\hfill}{\partial \theta^k\partial \theta^l}\log f(X_a;\theta)+g_{kl}\biggr)\biggl(\frac{\partial\hfill}{\partial \theta^s}\log f(X_a;\theta)\biggr)\biggr]\nonumber\\
&\qquad\qquad\times E\biggl[\biggl(\frac{\partial^2\hfill}{\partial \theta^u\partial \theta^m}\log f(X_c;\theta)+g_{um}\biggr)\biggl(\frac{\partial\hfill}{\partial \theta^t}\log f(X_c;\theta)\biggr)\biggr]\nonumber\\
&\qquad\quad+\sum_{a\ne c} E\biggl[\biggl(\frac{\partial^2\hfill}{\partial \theta^k\partial \theta^l}\log f(X_a;\theta)+g_{kl}\biggr)\biggl(\frac{\partial\hfill}{\partial \theta^t}\log f(X_a;\theta)\biggr)\biggr]\nonumber\\
&\qquad\qquad\times E\biggl[\biggl(\frac{\partial^2\hfill}{\partial \theta^u\partial \theta^m}\log f(X_c;\theta)+g_{um}\biggr)\biggl(\frac{\partial\hfill}{\partial \theta^s}\log f(X_c;\theta)\biggr)\biggr]\nonumber\\
&\qquad\quad+\sum_{a\ne b, a\ne c, b\ne c}E\biggl[\biggl(\frac{\partial^2\hfill}{\partial \theta^k\partial \theta^l}\log f(X_a;\theta)+g_{kl}\biggr)\biggl(\frac{\partial^2\hfill}{\partial \theta^u\partial \theta^m}\log f(X_a;\theta)+g_{um}\biggr)
\biggr]\nonumber\\
&\qquad\qquad\times E\biggl[\biggl(\frac{\partial\hfill}{\partial \theta^s}\log f(X_b;\theta)\biggr)\biggr]E\biggl[\biggl(\frac{\partial\hfill}{\partial \theta^t}\log f(X_c;\theta)\biggr)\biggr]\nonumber\\
&\qquad\quad+\sum_{a\ne b, a\ne c, b\ne c}E\biggl[\biggl(\frac{\partial^2\hfill}{\partial \theta^k\partial \theta^l}\log f(X_a;\theta)+g_{kl}\biggr)
\biggl(\frac{\partial\hfill}{\partial \theta^s}\log f(X_a;\theta)\biggr)
\biggr]\nonumber\\
&\qquad\qquad\times E\biggl[\biggl(\frac{\partial^2\hfill}{\partial \theta^u\partial \theta^m}\log f(X_b\theta)+g_{um}\biggr)\biggr]E\biggl[\biggl(\frac{\partial\hfill}{\partial \theta^t}\log f(X_c;\theta)\biggr)\biggr]\nonumber\\
&\qquad\quad+\sum_{a\ne b, a\ne c, b\ne c}E\biggl[\biggl(\frac{\partial^2\hfill}{\partial \theta^k\partial \theta^l}\log f(X_a;\theta)+g_{kl}\biggr)
\biggl(\frac{\partial\hfill}{\partial \theta^t}\log f(X_a;\theta)\biggr)
\biggr]\nonumber\\
&\qquad\qquad\times E\biggl[\biggl(\frac{\partial^2\hfill}{\partial \theta^u\partial \theta^m}\log f(X_b\theta)+g_{um}\biggr)\biggr]E\biggl[\biggl(\frac{\partial\hfill}{\partial \theta^s}\log f(X_c;\theta)\biggr)\biggr]\nonumber\\
&\qquad\quad+\sum_{a\ne b, a\ne c, b\ne c}E\biggl[\biggl(\frac{\partial^2\hfill}{\partial \theta^u\partial \theta^m}\log f(X_a;\theta)+g_{um}\biggr)
\biggl(\frac{\partial\hfill}{\partial \theta^s}\log f(X_a;\theta)\biggr)
\biggr]\nonumber\\
&\qquad\qquad\times E\biggl[\biggl(\frac{\partial^2\hfill}{\partial \theta^k\partial \theta^l}\log f(X_b\theta)+g_{kl}\biggr)\biggr]E\biggl[\biggl(\frac{\partial\hfill}{\partial \theta^t}\log f(X_c;\theta)\biggr)\biggr]\nonumber\\
&\qquad\quad+\sum_{a\ne b, a\ne c, b\ne c}E\biggl[\biggl(\frac{\partial^2\hfill}{\partial \theta^u\partial \theta^m}\log f(X_a;\theta)+g_{um}\biggr)
\biggl(\frac{\partial\hfill}{\partial \theta^t}\log f(X_a;\theta)\biggr)
\biggr]\nonumber\\
&\qquad\qquad\times E\biggl[\biggl(\frac{\partial^2\hfill}{\partial \theta^k\partial \theta^l}\log f(X_b\theta)+g_{kl}\biggr)\biggr]E\biggl[\biggl(\frac{\partial\hfill}{\partial \theta^s}\log f(X_c;\theta)\biggr)\biggr]\nonumber\\
&\qquad\quad+\sum_{a\ne b, a\ne c, b\ne c}E\biggl[
\biggl(\frac{\partial\hfill}{\partial \theta^s}\log f(X_a;\theta)\biggr)
\biggl(\frac{\partial\hfill}{\partial \theta^t}\log f(X_a;\theta)\biggr)
\biggr]\nonumber\\
&\qquad\qquad\times E\biggl[\biggl(\frac{\partial^2\hfill}{\partial \theta^k\partial \theta^l}\log f(X_b;\theta)+g_{kl}\biggr)\biggr]E\biggl[\biggl(\frac{\partial^2\hfill}{\partial \theta^u\partial \theta^m}\log f(X_c;\theta)+g_{um}\biggr)\biggr]\nonumber\\
&\qquad\quad+\sum_{a\ne b, a\ne c, a\ne d, b\ne c, b\ne d, c\ne d}E\biggl[
\biggl(\frac{\partial\hfill}{\partial \theta^s}\log f(X_a;\theta)\biggr)\biggr]
E\biggl[\biggl(\frac{\partial\hfill}{\partial \theta^t}\log f(X_b;\theta)\biggr)
\biggr]\nonumber\\
&\qquad\qquad\times E\biggl[\biggl(\frac{\partial^2\hfill}{\partial \theta^k\partial \theta^l}\log f(X_c;\theta)+g_{kl}\biggr)\biggr]E\biggl[\biggl(\frac{\partial^2\hfill}{\partial \theta^u\partial \theta^m}\log f(X_d;\theta)+g_{um}\biggr)\biggr]
\biggr\}\nonumber\\
&=n^{-2}g^{jk}g^{lu}g^{is}g^{mt}\nonumber\\
&\qquad\times \biggl\{ E\biggl[\biggl(\frac{\partial^2\hfill}{\partial \theta^k\partial \theta^l}\log f(X;\theta)+g_{kl}\biggr)\biggl(\frac{\partial^2\hfill}{\partial \theta^u\partial \theta^m}\log f(X;\theta)+g_{um}\biggr)\biggr]\nonumber\\
&\qquad\qquad\times E\biggl[\biggl(\frac{\partial\hfill}{\partial \theta^s}\log f(X;\theta)\biggr)\biggl(\frac{\partial\hfill}{\partial \theta^t}\log f(X;\theta)\biggr)\biggr]\nonumber\\
&\qquad\quad+ E\biggl[\biggl(\frac{\partial^2\hfill}{\partial \theta^k\partial \theta^l}\log f(X;\theta)+g_{kl}\biggr)\biggl(\frac{\partial\hfill}{\partial \theta^s}\log f(X;\theta)\biggr)\biggr]\nonumber\\
&\qquad\qquad\times E\biggl[\biggl(\frac{\partial^2\hfill}{\partial \theta^u\partial \theta^m}\log f(X;\theta)+g_{um}\biggr)\biggl(\frac{\partial\hfill}{\partial \theta^t}\log f(X;\theta)\biggr)\biggr]\nonumber\\
&\qquad\quad+ E\biggl[\biggl(\frac{\partial^2\hfill}{\partial \theta^k\partial \theta^l}\log f(X;\theta)+g_{kl}\biggr)\biggl(\frac{\partial\hfill}{\partial \theta^t}\log f(X;\theta)\biggr)\biggr]\nonumber\\
&\qquad\qquad\times E\biggl[\biggl(\frac{\partial^2\hfill}{\partial \theta^u\partial \theta^m}\log f(X;\theta)+g_{um}\biggr)\biggl(\frac{\partial\hfill}{\partial \theta^s}\log f(X;\theta)\biggr)\biggr]\biggr\}+O(n^{-3}). \label{aaa^2_bareo^2}
\end{align}
\begin{equation}
\label{aaa_bbar_bareo^3}
\begin{split}
&E[\aaa{l}{j}\bbar{ms}{l}\bareo{i}\bareo{m}\bareo{s}]\\
&=n^{-4}\bbar{ms}{l}E\biggl[g^{jk}g^{it}g^{mu}g^{sv}\biggl(\sum_{a=1}^n \biggl(\frac{\partial^2\hfill}{\partial\theta^l\partial\theta^k}\log f(X_a;\theta)+g_{kl}\biggr)\biggr)\biggl(\sum_{b=1}^n\frac{\partial\hfill}{\partial\theta^t}\log f(X_b;\theta)\biggr)\\
&\hspace{25mm}\times\biggl(\sum_{c=1}^n\frac{\partial\hfill}{\partial\theta^u}\log f(X_c;\theta)\biggr)\biggl(\sum_{d=1}^n\frac{\partial\hfill}{\partial\theta^v}\log f(X_d;\theta)\biggr)\biggr]\\
&=n^{-4}\bbar{ms}{l}g^{jk}g^{it}g^{mu}g^{sv}\\
&\qquad\times\biggl\{\sum_{a\ne b} E\biggl[\biggl(\frac{\partial^2\hfill}{\partial\theta^l\partial\theta^k}\log f(X_a;\theta)+g_{kl}\biggr)\biggl(\frac{\partial\hfill}{\partial\theta^t}\log f(X_a;\theta)\biggr)\biggr]\\
&\qquad\qquad\qquad\times E\biggl[\biggl(\frac{\partial\hfill}{\partial\theta^u}\log f(X_b;\theta)\biggr)\biggl(\frac{\partial\hfill}{\partial\theta^v}\log f(X_b;\theta)\biggr)
\biggr]\\
&\qquad\qquad+\sum_{a\ne b} E\biggl[\biggl(\frac{\partial^2\hfill}{\partial\theta^l\partial\theta^k}\log f(X_a;\theta)+g_{kl}\biggr)\biggl(\frac{\partial\hfill}{\partial\theta^u}\log f(X_a;\theta)\biggr)\biggr]\\
&\qquad\qquad\qquad\times E\biggl[\biggl(\frac{\partial\hfill}{\partial\theta^t}\log f(X_b;\theta)\biggr)\biggl(\frac{\partial\hfill}{\partial\theta^v}\log f(X_b;\theta)\biggr)
\biggr]\\
&\qquad\qquad+\sum_{a\ne b} E\biggl[\biggl(\frac{\partial^2\hfill}{\partial\theta^l\partial\theta^k}\log f(X_a;\theta)+g_{kl}\biggr)\biggl(\frac{\partial\hfill}{\partial\theta^v}\log f(X_a;\theta)\biggr)\biggr]\\
&\qquad\qquad\qquad\times E\biggl[\biggl(\frac{\partial\hfill}{\partial\theta^u}\log f(X_b;\theta)\biggr)\biggl(\frac{\partial\hfill}{\partial\theta^t}\log f(X_b;\theta)\biggr)
\biggr]
\biggr\}+O(n^{-3})\\
&=n^{-2}\bbar{ms}{l}g^{jk}g^{it}g^{mu}g^{sv}\\
&\qquad\times\biggl\{
E\biggl[\biggl(\frac{\partial^2\hfill}{\partial\theta^l\partial\theta^k}\log f(X;\theta)+g_{kl}\biggr)\biggl(\frac{\partial\hfill}{\partial\theta^t}\log f(X;\theta)\biggr)\biggr]\\
&\qquad\qquad\qquad\times E\biggl[\biggl(\frac{\partial\hfill}{\partial\theta^u}\log f(X;\theta)\biggr)\biggl(\frac{\partial\hfill}{\partial\theta^v}\log f(X;\theta)\biggr)\biggr]
\biggr\}\\
&\qquad\qquad+E\biggl[\biggl(\frac{\partial^2\hfill}{\partial\theta^l\partial\theta^k}\log f(X;\theta)+g_{kl}\biggr)\biggl(\frac{\partial\hfill}{\partial\theta^u}\log f(X;\theta)\biggr)\biggr]\\
&\qquad\qquad\qquad\times E\biggl[\biggl(\frac{\partial\hfill}{\partial\theta^t}\log f(X;\theta)\biggr)\biggl(\frac{\partial\hfill}{\partial\theta^v}\log f(X;\theta)\biggr)
\biggr]\\
&\qquad\qquad+E\biggl[\biggl(\frac{\partial^2\hfill}{\partial\theta^l\partial\theta^k}\log f(X;\theta)+g_{kl}\biggr)\biggl(\frac{\partial\hfill}{\partial\theta^v}\log f(X;\theta)\biggr)\biggr]\\
&\qquad\qquad\qquad\times E\biggl[\biggl(\frac{\partial\hfill}{\partial\theta^t}\log f(X;\theta)\biggr)\biggl(\frac{\partial\hfill}{\partial\theta^u}\log f(X;\theta)\biggr)
\biggr]\biggr\}+O(n^{-3}).
\end{split}
\end{equation}
\begin{align}
&E[\bee{lm}{j}\bareo{i}\bareo{l}\bareo{m}]\nonumber\\
&=n^{-4}2^{-1}g^{js}g^{it}g^{lu}g^{mv}\nonumber\\
&\quad\times E\biggl[\sum_{a=1}^n\biggl(\frac{\partial^3\hfill}{\partial\theta^s\partial\theta^l\partial\theta^m} \log f(X_a;\theta)-E\biggl[\frac{\partial^3\hfill}{\partial\theta^s\partial\theta^l\partial\theta^m} \log f(X_a;\theta)\biggr]\biggr)\nonumber\\
&\qquad\qquad\times \biggl(\sum_{b=1}^n\frac{\partial\hfill}{\partial\theta^t}\log f(X_b;\theta)\biggr)\biggl(\sum_{c=1}^n\frac{\partial\hfill}{\partial\theta^u}\log f(X_c;\theta)\biggr)\biggl(\sum_{d=1}^n\frac{\partial\hfill}{\partial\theta^v}\log f(X_d;\theta)\biggr)
\biggr]\nonumber\\
&=n^{-2}2^{-1}g^{js}g^{it}g^{lu}g^{mv}\nonumber\\
&\quad\times \biggl\{E\biggl[\biggl(\frac{\partial^3\hfill}{\partial\theta^s\partial\theta^l\partial\theta^m} \log f(X;\theta)-E\biggl[\frac{\partial^3\hfill}{\partial\theta^s\partial\theta^l\partial\theta^m} \log f(X;\theta)\biggr]\biggr) \biggl(\frac{\partial\hfill}{\partial\theta^t}\log f(X;\theta)\biggr)\biggr]\nonumber\\
&\qquad\qquad\times E\biggl[\biggl(\frac{\partial\hfill}{\partial\theta^u}\log f(X;\theta)\biggr)\biggl(\frac{\partial\hfill}{\partial\theta^v}\log f(X;\theta)\biggr)
\biggr]\nonumber\\
&\qquad+E\biggl[\biggl(\frac{\partial^3\hfill}{\partial\theta^s\partial\theta^l\partial\theta^m} \log f(X;\theta)-E\biggl[\frac{\partial^3\hfill}{\partial\theta^s\partial\theta^l\partial\theta^m} \log f(X;\theta)\biggr]\biggr) \biggl(\frac{\partial\hfill}{\partial\theta^u}\log f(X;\theta)\biggr)\biggr]\nonumber\\
&\qquad\qquad\times E\biggl[\biggl(\frac{\partial\hfill}{\partial\theta^t}\log f(X;\theta)\biggr)\biggl(\frac{\partial\hfill}{\partial\theta^v}\log f(X;\theta)\biggr)
\biggr]\nonumber\\
&\qquad+E\biggl[\biggl(\frac{\partial^3\hfill}{\partial\theta^s\partial\theta^l\partial\theta^m} \log f(X;\theta)-E\biggl[\frac{\partial^3\hfill}{\partial\theta^s\partial\theta^l\partial\theta^m} \log f(X;\theta)\biggr]\biggr) \biggl(\frac{\partial\hfill}{\partial\theta^v}\log f(X;\theta)\biggr)\biggr]\nonumber\\
&\qquad\qquad\times E\biggl[\biggl(\frac{\partial\hfill}{\partial\theta^t}\log f(X;\theta)\biggr)\biggl(\frac{\partial\hfill}{\partial\theta^u}\log f(X;\theta)\biggr)
\biggr]\biggr\}+O(n^{-3}).\label{bee_bareo^3}
\end{align}
\begin{align}
&E[\bbar{lm}{j}\aaa{s}{m}\bareo{i}\bareo{l}\bareo{s}]\nonumber\\
&=n^{-4}\bbar{lm}{j}g^{mt}g^{iu}g^{lv}g^{sw}\nonumber\\
&\quad\times E\biggl[\biggl(\sum_{a=1}^p\biggl(\frac{\partial^2\hfill}{\partial\theta^t\partial\theta^s}\log f(X_a;\theta)+g_{ts}\biggr)\biggr)\biggl(\sum_{b=1}^n\frac{\partial\hfill}{\partial\theta^u}\log f(X_b;\theta)\biggr)\nonumber\\
&\qquad\qquad\times\biggl(\sum_{c=1}^n\frac{\partial\hfill}{\partial\theta^v}\log f(X_c;\theta)\biggr)\biggl(\sum_{d=1}^n\frac{\partial\hfill}{\partial\theta^w}\log f(X_d;\theta)\biggr)
\biggr]\nonumber\\
&=n^{-4}\bbar{lm}{j}g^{mt}g^{iu}g^{lv}g^{sw}\nonumber\\
&\quad\times \biggl\{
\sum_{a\ne b}E\biggl[ \biggl(\frac{\partial^2\hfill}{\partial\theta^t\partial\theta^s}\log f(X_a;\theta)+g_{ts}\biggr)\biggl(\frac{\partial\hfill}{\partial\theta^u}\log f(X_a;\theta)\biggr)\biggr]\nonumber\\
&\qquad\qquad\times E\biggl[\biggl(\frac{\partial\hfill}{\partial\theta^v}\log f(X_b;\theta)\biggr)\biggl(\frac{\partial\hfill}{\partial\theta^w}\log f(X_b;\theta)\biggr)\biggr]\nonumber\\
&\qquad\quad+\sum_{a\ne b}E\biggl[ \biggl(\frac{\partial^2\hfill}{\partial\theta^t\partial\theta^s}\log f(X_a;\theta)+g_{ts}\biggr)\biggl(\frac{\partial\hfill}{\partial\theta^v}\log f(X_a;\theta)\biggr)\biggr]\nonumber\\
&\qquad\qquad\times E\biggl[\biggl(\frac{\partial\hfill}{\partial\theta^u}\log f(X_b;\theta)\biggr)\biggl(\frac{\partial\hfill}{\partial\theta^w}\log f(X_b;\theta)\biggr)\biggr]\nonumber\\
&\qquad\quad+\sum_{a\ne b}E\biggl[ \biggl(\frac{\partial^2\hfill}{\partial\theta^t\partial\theta^s}\log f(X_a;\theta)+g_{ts}\biggr)\biggl(\frac{\partial\hfill}{\partial\theta^w}\log f(X_a;\theta)\biggr)\biggr]\nonumber\\
&\qquad\qquad\times E\biggl[\biggl(\frac{\partial\hfill}{\partial\theta^u}\log f(X_b;\theta)\biggr)\biggl(\frac{\partial\hfill}{\partial\theta^v}\log f(X_b;\theta)\biggr)\biggr]
\biggr\}+O(n^{-3})\nonumber\\
&=n^{-2}\bbar{lm}{j}g^{mt}g^{iu}g^{lv}g^{sw}\nonumber\\
&\quad\times\biggl\{
E\biggl[ \biggl(\frac{\partial^2\hfill}{\partial\theta^t\partial\theta^s}\log f(X;\theta)+g_{ts}\biggr)\biggl(\frac{\partial\hfill}{\partial\theta^u}\log f(X;\theta)\biggr)\biggr]\nonumber\\
&\qquad\qquad\times E\biggl[\biggl(\frac{\partial\hfill}{\partial\theta^v}\log f(X;\theta)\biggr)\biggl(\frac{\partial\hfill}{\partial\theta^w}\log f(X;\theta)\biggr)\biggr]\nonumber\\
&\qquad\quad+E\biggl[ \biggl(\frac{\partial^2\hfill}{\partial\theta^t\partial\theta^s}\log f(X;\theta)+g_{ts}\biggr)\biggl(\frac{\partial\hfill}{\partial\theta^v}\log f(X;\theta)\biggr)\biggr]\nonumber\\
&\qquad\qquad\times E\biggl[\biggl(\frac{\partial\hfill}{\partial\theta^u}\log f(X;\theta)\biggr)\biggl(\frac{\partial\hfill}{\partial\theta^w}\log f(X;\theta)\biggr)\biggr]\nonumber\\
&\qquad\quad+E\biggl[ \biggl(\frac{\partial^2\hfill}{\partial\theta^t\partial\theta^s}\log f(X;\theta)+g_{ts}\biggr)\biggl(\frac{\partial\hfill}{\partial\theta^w}\log f(X;\theta)\biggr)\biggr]\nonumber\\
&\qquad\qquad\times E\biggl[\biggl(\frac{\partial\hfill}{\partial\theta^u}\log f(X;\theta)\biggr)\biggl(\frac{\partial\hfill}{\partial\theta^v}\log f(X;\theta)\biggr)\biggr]
\biggr\}+O(n^{-3}). \label{bbar_aaa_bareo^3}
\end{align}
\begin{align}
&E[\bbar{lm}{j}\bbar{st}{m}\bareo{i}\bareo{l}\bareo{s}\bareo{t}]\nonumber\\
&=\bbar{lm}{j}\bbar{st}{m}E[\bareo{i}\bareo{l}\bareo{s}\bareo{t}]\nonumber\\
&=n^{-4}\bbar{lm}{j}\bbar{st}{m}g^{ik}g^{lu}g^{sv}g^{tw}\nonumber\\
&\quad\times E\biggl[\biggl(\sum_{a=1}^n\frac{\partial\hfill}{\partial\theta^k}\log f(X_a;\theta)\biggr)\biggl(\sum_{b=1}^n\frac{\partial\hfill}{\partial\theta^u}\log f(X_b;\theta)\biggr)\nonumber\\
&\qquad\qquad\times\biggl(\sum_{c=1}^n\frac{\partial\hfill}{\partial\theta^v}\log f(X_c;\theta)\biggr)\biggl(\sum_{d=1}^n\frac{\partial\hfill}{\partial\theta^w}\log f(X_d;\theta)\biggr)\biggr]\nonumber\\
&=n^{-2}\bbar{lm}{j}\bbar{st}{m}g^{ik}g^{lu}g^{sv}g^{tw}\nonumber\\
&\quad\times E\biggl[\biggl(\frac{\partial\hfill}{\partial\theta^k}\log f(X;\theta)\biggr)\biggl(\frac{\partial\hfill}{\partial\theta^u}\log f(X;\theta)\biggr)\biggr]\nonumber\\
&\qquad\quad\times E\biggl[\biggl(\frac{\partial\hfill}{\partial\theta^v}\log f(X;\theta)\biggr)\biggl(\frac{\partial\hfill}{\partial\theta^w}\log f(X;\theta)\biggr)\biggr]\nonumber\\
&\quad\quad+ E\biggl[\biggl(\frac{\partial\hfill}{\partial\theta^k}\log f(X;\theta)\biggr)\biggl(\frac{\partial\hfill}{\partial\theta^v}\log f(X;\theta)\biggr)\biggr]\nonumber\\
&\qquad\quad\times E\biggl[\biggl(\frac{\partial\hfill}{\partial\theta^u}\log f(X;\theta)\biggr)\biggl(\frac{\partial\hfill}{\partial\theta^w}\log f(X;\theta)\biggr)\biggr]\nonumber\\
&\quad\quad+ E\biggl[\biggl(\frac{\partial\hfill}{\partial\theta^k}\log f(X;\theta)\biggr)\biggl(\frac{\partial\hfill}{\partial\theta^w}\log f(X;\theta)\biggr)\biggr]\nonumber\\
&\qquad\quad\times E\biggl[\biggl(\frac{\partial\hfill}{\partial\theta^u}\log f(X;\theta)\biggr)\biggl(\frac{\partial\hfill}{\partial\theta^v}\log f(X;\theta)\biggr)\biggr]\biggr\}+O(n^{-3}).\label{bbar^2_bareo^4}
\end{align}
\begin{align}
&E[\cbar{lmt}{j}\bareo{i}\bareo{l}\bareo{m}\bareo{t}]\nonumber\\
&=n^{-4}\cbar{lmt}{j}g^{ik}g^{ls}g^{mu}g^{tv}\nonumber\\
&\quad\times E\biggl[\biggl(\sum_{a=1}^n\frac{\partial\hfill}{\partial\theta^k}\log f(X_a;\theta)\biggr)\biggl(\sum_{b=1}^n\frac{\partial\hfill}{\partial\theta^s}\log f(X_b;\theta)\biggr)\nonumber\\
&\qquad\qquad\times\biggl(\sum_{c=1}^n\frac{\partial\hfill}{\partial\theta^u}\log f(X_c;\theta)\biggr)\biggl(\sum_{d=1}^n\frac{\partial\hfill}{\partial\theta^v}\log f(X_d;\theta)\biggr)\biggr]\nonumber\\
&=n^{-2}\cbar{lmt}{j}g^{ik}g^{ls}g^{mu}g^{tv}\nonumber\\
&\quad\times E\biggl[\biggl(\frac{\partial\hfill}{\partial\theta^k}\log f(X;\theta)\biggr)\biggl(\frac{\partial\hfill}{\partial\theta^s}\log f(X;\theta)\biggr)\biggr]\nonumber\\
&\qquad\quad\times E\biggl[\biggl(\frac{\partial\hfill}{\partial\theta^u}\log f(X;\theta)\biggr)\biggl(\frac{\partial\hfill}{\partial\theta^v}\log f(X;\theta)\biggr)\biggr]\nonumber\\
&\quad\quad+ E\biggl[\biggl(\frac{\partial\hfill}{\partial\theta^k}\log f(X;\theta)\biggr)\biggl(\frac{\partial\hfill}{\partial\theta^u}\log f(X;\theta)\biggr)\biggr]\nonumber\\
&\qquad\quad\times E\biggl[\biggl(\frac{\partial\hfill}{\partial\theta^s}\log f(X;\theta)\biggr)\biggl(\frac{\partial\hfill}{\partial\theta^v}\log f(X;\theta)\biggr)\biggr]\nonumber\\
&\quad\quad+ E\biggl[\biggl(\frac{\partial\hfill}{\partial\theta^k}\log f(X;\theta)\biggr)\biggl(\frac{\partial\hfill}{\partial\theta^v}\log f(X;\theta)\biggr)\biggr]\nonumber\\
&\qquad\quad\times E\biggl[\biggl(\frac{\partial\hfill}{\partial\theta^s}\log f(X;\theta)\biggr)\biggl(\frac{\partial\hfill}{\partial\theta^u}\log f(X;\theta)\biggr)\biggr]\biggr\}+O(n^{-3}).\label{cbar_bareo^4}
\end{align}
\begin{align}
&E[\aaa{l}{i}\aaa{m}{j}\bareo{l}\bareo{m}]\nonumber\\
&=n^{-4}g^{ik}g^{js}g^{lt}g^{mu}\nonumber\\
&\times E\biggl[\biggl(\sum_{a=1}^n\biggl(\frac{\partial^2\hfill}{\partial\theta^k\partial\theta^l}\log f(X_a;\theta)+g_{kl}\biggr)\biggr)\biggl(\sum_{b=1}^n\biggl(\frac{\partial^2\hfill}{\partial\theta^s\partial\theta^m}\log f(X_b;\theta)+g_{sm}\biggr)\biggr)\nonumber\\
&\qquad\times \biggl(\sum_{c=1}^n\frac{\partial\hfill}{\partial\theta^t}\log f(X_c;\theta)\biggr)\biggl(\sum_{d=1}^n\frac{\partial\hfill}{\partial\theta^u}\log f(X_d;\theta)\biggr)\biggr]\nonumber\\
&=n^{-2}g^{ik}g^{js}g^{lt}g^{mu}\nonumber\\
&\times \biggl\{
E\biggl[\biggl(\frac{\partial^2\hfill}{\partial\theta^k\partial\theta^l}\log f(X;\theta)+g_{kl}\biggr)\biggl(\frac{\partial^2\hfill}{\partial\theta^s\partial\theta^m}\log f(X;\theta)+g_{sm}\biggr)\biggr]\nonumber\\
&\qquad\times E\biggl[\biggl(\frac{\partial\hfill}{\partial\theta^t}\log f(X;\theta)\biggr)\biggl(\frac{\partial\hfill}{\partial\theta^u}\log f(X;\theta)\biggr)\biggr]\nonumber\\
&\quad+E\biggl[ \biggl(\frac{\partial^2\hfill}{\partial\theta^k\partial\theta^l}\log f(X;\theta)+g_{kl}\biggr)\biggl(\frac{\partial\hfill}{\partial\theta^t}\log f(X;\theta)\biggr)\biggr]\nonumber\\
&\qquad\times E\biggl[\biggl(\frac{\partial^2\hfill}{\partial\theta^s\partial\theta^m}\log f(X;\theta)+g_{sm}\biggr)\biggl(\frac{\partial\hfill}{\partial\theta^u}\log f(X;\theta)\biggr)\biggr]\nonumber\\
&\quad+E\biggl[ \biggl(\frac{\partial^2\hfill}{\partial\theta^k\partial\theta^l}\log f(X;\theta)+g_{kl}\biggr)\biggl(\frac{\partial\hfill}{\partial\theta^u}\log f(X;\theta)\biggr)\biggr]\nonumber\\
&\qquad\times E\biggl[\biggl(\frac{\partial^2\hfill}{\partial\theta^s\partial\theta^m}\log f(X;\theta)+g_{sm}\biggr)\biggl(\frac{\partial\hfill}{\partial\theta^t}\log f(X;\theta)\biggr)\biggr]\label{aaa^2_bareo^2_2}
\bigg\}+O(n^{-3}).
\end{align}
\begin{align}
&E[\aaa{l}{i}\bbar{st}{j}\bareo{l}\bareo{s}\bareo{t}]\nonumber\\
&=n^{-4}\bbar{st}{j}g^{ik}g^{lu}g^{sv}g^{tw}\nonumber\\
&\quad\times E\biggl[ \biggl(\sum_{a=1}^n\biggl(\frac{\partial^2\hfill}{\partial\theta^k\partial\theta^l}\log f(X_a;\theta)+g_{kl}\biggr)\biggr)
 \biggl(\sum_{b=1}^n\frac{\partial\hfill}{\partial\theta^u}\log f(X_b;\theta)\biggr)\nonumber\\
&\qquad \qquad\times\biggl(\sum_{c=1}^n\frac{\partial\hfill}{\partial\theta^v}\log f(X_c;\theta)\biggr)\biggl(\sum_{d=1}^n\frac{\partial\hfill}{\partial\theta^w}\log f(X_d;\theta)\biggr)
\biggr]\nonumber\\
&=n^{-2}\bbar{st}{j}g^{ik}g^{lu}g^{sv}g^{tw}\nonumber\\
&\quad \times\biggl\{
 E\biggl[\biggl(\frac{\partial^2\hfill}{\partial\theta^k\partial\theta^l}\log f(X;\theta)+g_{kl}\biggr)\biggl(\frac{\partial\hfill}{\partial\theta^u}\log f(X;\theta)\biggr)\biggr]\nonumber\\
&\qquad\qquad \times E\biggl[\biggl(\frac{\partial\hfill}{\partial\theta^v}\log f(X;\theta)\biggr)\biggl(\frac{\partial\hfill}{\partial\theta^w}\log f(X;\theta)\biggr)\biggr]\nonumber\\
&\qquad\quad+ E\biggl[\biggl(\frac{\partial^2\hfill}{\partial\theta^k\partial\theta^l}\log f(X;\theta)+g_{kl}\biggr)\biggl(\frac{\partial\hfill}{\partial\theta^v}\log f(X;\theta)\biggr)\biggr]\nonumber\\
&\qquad\qquad \times E\biggl[\biggl(\frac{\partial\hfill}{\partial\theta^u}\log f(X;\theta)\biggr)\biggl(\frac{\partial\hfill}{\partial\theta^w}\log f(X;\theta)\biggr)\biggr]\nonumber\\
&\qquad\quad+ E\biggl[\biggl(\frac{\partial^2\hfill}{\partial\theta^k\partial\theta^l}\log f(X;\theta)+g_{kl}\biggr)\biggl(\frac{\partial\hfill}{\partial\theta^w}\log f(X;\theta)\biggr)\biggr]\nonumber\\
&\qquad\qquad \times E\biggl[\biggl(\frac{\partial\hfill}{\partial\theta^u}\log f(X;\theta)\biggr)\biggl(\frac{\partial\hfill}{\partial\theta^v}\log f(X;\theta)\biggr)\biggr]
\biggr\}+O(n^{-3}).\label{aaa_bbar_bareo^3_2}
\end{align}
\begin{align}
&E[\bbar{lm}{i}\bbar{st}{j}\bareo{l}\bareo{m}\bareo{s}\bareo{t}]\nonumber\\
&=n^{-4}\bbar{lm}{i}\bbar{st}{j}g^{lk}g^{mu}g^{sv}g^{tw}\nonumber\\
&\quad\times 
E\biggl[\biggl(\sum_{a=1}^n\frac{\partial\hfill}{\partial\theta^k}\log f(X_a;\theta)\biggr)\biggl(\sum_{b=1}^n\frac{\partial\hfill}{\partial\theta^u}\log f(X_b;\theta)\biggr)\nonumber\\
&\qquad\qquad\times\biggl(\sum_{c=1}^n\frac{\partial\hfill}{\partial\theta^v}\log f(X_c;\theta)\biggr)\biggl(\sum_{d=1}^n\frac{\partial\hfill}{\partial\theta^w}\log f(X_d;\theta)\biggr)\biggr]\nonumber\\
&=n^{-2}\bbar{lm}{i}\bbar{st}{j}g^{lk}g^{mu}g^{sv}g^{tw}\nonumber\\
&\quad\times \biggl\{
E\biggl[\biggl(\frac{\partial\hfill}{\partial\theta^k}\log f(X;\theta)\biggr)\biggl(\frac{\partial\hfill}{\partial\theta^u}\log f(X;\theta)\biggr)\biggr]\nonumber\\
&\qquad\quad\times E\biggl[\biggl(\frac{\partial\hfill}{\partial\theta^v}\log f(X;\theta)\biggr)\biggl(\frac{\partial\hfill}{\partial\theta^w}\log f(X;\theta)\biggr)\biggr]\nonumber\\
&\quad\quad+ E\biggl[\biggl(\frac{\partial\hfill}{\partial\theta^k}\log f(X;\theta)\biggr)\biggl(\frac{\partial\hfill}{\partial\theta^v}\log f(X;\theta)\biggr)\biggr]\nonumber\\
&\qquad\quad\times E\biggl[\biggl(\frac{\partial\hfill}{\partial\theta^u}\log f(X;\theta)\biggr)\biggl(\frac{\partial\hfill}{\partial\theta^w}\log f(X;\theta)\biggr)\biggr]\nonumber\\
&\quad\quad+ E\biggl[\biggl(\frac{\partial\hfill}{\partial\theta^k}\log f(X;\theta)\biggr)\biggl(\frac{\partial\hfill}{\partial\theta^w}\log f(X;\theta)\biggr)\biggr]\nonumber\\
&\qquad\quad\times E\biggl[\biggl(\frac{\partial\hfill}{\partial\theta^u}\log f(X;\theta)\biggr)\biggl(\frac{\partial\hfill}{\partial\theta^v}\log f(X;\theta)\biggr)\biggr]\biggr\}+O(n^{-3}).\label{bbar^2_bareo^3}
\end{align}
Combining all equations from \eqref{barteta^2}
to \eqref{bbar^2_bareo^3} and using Lemma \ref{int_log_der}, we obtain the result.\\
-\textit{ Proof of \eqref{expec_barteta^3}}-
\begin{align}
&\barteta{i}\barteta{j}\barteta{k}\nonumber\\
&=\bareo{i}\bareo{j}\bareo{k}+\aaa{s}{i}\bareo{s}\bareo{j}\bareo{k}+\aaa{s}{j}\bareo{s}\bareo{i}\bareo{k}+\aaa{s}{k}\bareo{s}\bareo{i}\bareo{j}\nonumber\\
&\quad+\bbar{st}{i}\bareo{s}\bareo{t}\bareo{j}\bareo{k}+\bbar{st}{j}\bareo{s}\bareo{t}\bareo{i}\bareo{k}+\bbar{st}{k}\bareo{s}\bareo{t}\bareo{i}\bareo{j}+Re(5), \label{barteta^3}
\end{align}
where  $Re(5)$ is defined as before.
We evaluate the expectation of each term.
\begin{align}
&E[\bareo{i}\bareo{j}\bareo{k}]\nonumber\\
&=n^{-3}E \biggl[\biggl(g^{il}\sum_{a=1}^n\biggl(\frac{\partial \hfill}{\partial \theta^l} \log f(X_a;\theta)\biggr)\biggr)\biggl(g^{jl}\sum_{b=1}^n\biggl(\frac{\partial \hfill}{\partial \theta^l} \log f(X_b;\theta)\biggr)\biggr)\nonumber\\
&\qquad\times\biggl(g^{kl}\sum_{c=1}^n\biggl(\frac{\partial \hfill}{\partial \theta^l} \log f(X_c;\theta)\biggr)\biggr)
\biggr]\nonumber\\
&=n^{-3}\sum_{a=1}^n g^{is}g^{jt}g^{ku} E\biggl[\biggl(\frac{\partial \hfill}{\partial \theta^s} \log f(X_a;\theta)\biggr)\biggl(\frac{\partial \hfill}{\partial \theta^t} \log f(X_a;\theta)\biggr)\biggl(\frac{\partial \hfill}{\partial \theta^u} \log f(X_a;\theta)\biggr)
\biggr]\nonumber\\
&\quad+n^{-3}\sum_{a\ne b} g^{is}g^{jt}g^{ku}\nonumber\\
&\quad\qquad\times \biggl\{ E\biggl[\biggl(\frac{\partial \hfill}{\partial \theta^s} \log f(X_a;\theta)\biggr)\biggl(\frac{\partial \hfill}{\partial \theta^t} \log f(X_a;\theta)\biggr)\biggr]E\biggl[\frac{\partial \hfill}{\partial \theta^u} \log f(X_b;\theta)\biggr]\nonumber\\
&\quad\qquad\qquad+E\biggl[\biggl(\frac{\partial \hfill}{\partial \theta^s} \log f(X_a;\theta)\biggr)\biggl(\frac{\partial \hfill}{\partial \theta^u} \log f(X_a;\theta)\biggr)\biggr]E\biggl[\frac{\partial \hfill}{\partial \theta^t} \log f(X_b;\theta) \biggr]\nonumber\\
&\quad\qquad\qquad+E\biggl[\biggl(\frac{\partial \hfill}{\partial \theta^t} \log f(X_a;\theta)\biggr)\biggl(\frac{\partial \hfill}{\partial \theta^u} \log f(X_a;\theta)\biggr)\biggr]E\biggl[\frac{\partial \hfill}{\partial \theta^s} \log f(X_b;\theta) \biggr]\biggr\}\nonumber\\
&\quad+n^{-3}\sum_{a\ne b, a\ne c, b \ne c} g^{is}g^{jt}g^{ku}\nonumber\\
&\qquad\qquad\times E\biggl[\frac{\partial \hfill}{\partial \theta^s} \log f(X_a;\theta)\biggr]E\biggl[\frac{\partial \hfill}{\partial \theta^t} \log f(X_b;\theta)\biggr]E\biggl[\frac{\partial \hfill}{\partial \theta^u} \log f(X_c;\theta)\biggr]\nonumber\\
&=n^{-2}g^{is}g^{jt}g^{ku}E\biggl[\biggl(\frac{\partial \hfill}{\partial \theta^s} \log f(X;\theta)\biggr)\biggl(\frac{\partial \hfill}{\partial \theta^t} \log f(X;\theta)\biggr)\biggl(\frac{\partial \hfill}{\partial \theta^u} \log f(X;\theta)\biggr)\biggr]\nonumber\\
&=n^{-2}g^{is}g^{jt}g^{ku}(\cristofsm{st}{u}-\cristofse{st}{u}).\label{bareo^4}
\end{align}
\begin{align}
&E[\aaa{s}{i}\bareo{s}\bareo{j}\bareo{k}]\nonumber\\
&=n^{-4}g^{it}g^{su}g^{jv}g^{kw}\sum_{1\leq a, b, c, d \leq n} E\biggl[
\biggl(\frac{\partial^2\hfill}{\partial\theta^s\partial\theta^t}\log f(X_a;\theta)+g_{st}\biggr)\nonumber\\
&\quad\times\biggl(\frac{\partial \hfill}{\partial \theta^u} \log f(X_b;\theta)\biggr)\biggl(\frac{\partial \hfill}{\partial \theta^v} \log f(X_c;\theta)\biggr)\biggl(\frac{\partial \hfill}{\partial \theta^w} \log f(X_d;\theta)\biggr)
\biggr]\nonumber\\
&=n^{-4}g^{it}g^{su}g^{jv}g^{kw}\nonumber\\
&\quad\times \sum_{a\ne b}\biggl\{E\biggl[\biggl(\frac{\partial^2\hfill}{\partial\theta^s\partial\theta^t}\log f(X_a;\theta)+g_{st}\biggr)\biggl(\frac{\partial \hfill}{\partial \theta^u} \log f(X_a;\theta)\biggr)\biggr]\nonumber\\
&\qquad\qquad\times E\biggl[\biggl(\frac{\partial \hfill}{\partial \theta^v} \log f(X_b;\theta)\biggr)\biggl(\frac{\partial \hfill}{\partial \theta^w} \log f(X_b;\theta)\biggr)\biggr]\nonumber\\
&\qquad\qquad+E\biggl[\biggl(\frac{\partial^2\hfill}{\partial\theta^s\partial\theta^t}\log f(X_a;\theta)+g_{st}\biggr)\biggl(\frac{\partial \hfill}{\partial \theta^v} \log f(X_a;\theta)\biggr)\biggr]\nonumber\\
&\qquad\qquad\times E\biggl[\biggl(\frac{\partial \hfill}{\partial \theta^u} \log f(X_b;\theta)\biggr)\biggl(\frac{\partial \hfill}{\partial \theta^w} \log f(X_b;\theta)\biggr)\biggr]\nonumber\\
&\qquad\qquad+E\biggl[\biggl(\frac{\partial^2\hfill}{\partial\theta^s\partial\theta^t}\log f(X_a;\theta)+g_{st}\biggr)\biggl(\frac{\partial \hfill}{\partial \theta^w} \log f(X_a;\theta)\biggr)\biggr]\nonumber\\
&\qquad\qquad\times E\biggl[\biggl(\frac{\partial \hfill}{\partial \theta^u} \log f(X_b;\theta)\biggr)\biggl(\frac{\partial \hfill}{\partial \theta^v} \log f(X_b;\theta)\biggr)\biggr]\biggr\}+O(n^{-3})\nonumber\\
&=n^{-2}g^{it}g^{su}g^{jv}g^{kw}(\cristofse{st}{u}g_{vw}+\cristofse{st}{v}g_{uw}+\cristofse{st}{w}g_{uv})+O(n^{-3})\nonumber\\
&=n^{-2}(\cristoffe{st}{s}g^{it}g^{jk}+\cristoffe{st}{j}g^{sk}g^{it}+\cristoffe{st}{k}g^{sj}g^{it})+O(n^{-3}). \label{aaa_bareo^3}
\end{align}
\begin{align}
&E[\bbar{st}{i}\bareo{s}\bareo{t}\bareo{j}\bareo{k}]\nonumber\\
&=n^{-4}\bbar{st}{i}g^{su}g^{tv}g^{jw}g^{km}\nonumber\\
&\quad\times\sum_{1 \leq a,b,c,d \leq n}E\biggl[\biggl(\frac{\partial \hfill}{\partial \theta^u} \log f(X_a;\theta)\biggr)\biggl(\frac{\partial \hfill}{\partial \theta^v} \log f(X_b;\theta)\biggr)\nonumber\\
&\qquad\qquad\times\biggl(\frac{\partial \hfill}{\partial \theta^w} \log f(X_c;\theta)\biggr)\biggl(\frac{\partial \hfill}{\partial \theta^m} \log f(X_d;\theta)\biggr)
\biggr]\nonumber\\
&=n^{-2}\bbar{st}{i}g^{su}g^{tv}g^{jw}g^{km}\nonumber\\
&\quad\times\biggl\{
E\biggl[ \biggl(\frac{\partial \hfill}{\partial \theta^u} \log f(X;\theta)\biggr)\biggl(\frac{\partial \hfill}{\partial \theta^v} \log f(X;\theta)\biggr)\biggr]\nonumber\\
&\qquad\qquad\times E\biggl[ \biggl(\frac{\partial \hfill}{\partial \theta^w} \log f(X;\theta)\biggr)\biggl(\frac{\partial \hfill}{\partial \theta^m} \log f(X;\theta)\biggr)\biggr]\nonumber\\
&\quad\qquad+E\biggl[ \biggl(\frac{\partial \hfill}{\partial \theta^u} \log f(X;\theta)\biggr)\biggl(\frac{\partial \hfill}{\partial \theta^w} \log f(X;\theta)\biggr)\biggr]\nonumber\\
&\qquad\qquad\times E\biggl[ \biggl(\frac{\partial \hfill}{\partial \theta^v} \log f(X;\theta)\biggr)\biggl(\frac{\partial \hfill}{\partial \theta^m} \log f(X;\theta)\biggr)\biggr]\nonumber\\
&\quad\qquad+E\biggl[ \biggl(\frac{\partial \hfill}{\partial \theta^u} \log f(X;\theta)\biggr)\biggl(\frac{\partial \hfill}{\partial \theta^m} \log f(X;\theta)\biggr)\biggr]\nonumber\\
&\qquad\qquad\times E\biggl[ \biggl(\frac{\partial \hfill}{\partial \theta^v} \log f(X;\theta)\biggr)\biggl(\frac{\partial \hfill}{\partial \theta^w} \log f(X;\theta)\biggr)\biggr]\bigg\}+O(n^{-3})\nonumber\\
&=n^{-2}\bbar{st}{i} g^{su}g^{tv}g^{jw}g^{km}(g_{uv}g_{wm}+g_{uw}g_{vm}+g_{um}g_{vw})\nonumber\\
&=n^{-2}\bbar{st}{i}(g^{ts}g^{jk}+g^{js}g^{kt}+g^{ks}g^{jt})+O(n^{-3}).\label{bbar_bareo^4}
\end{align}
From \eqref{barteta^3}, \eqref{bareo^4}, \eqref{aaa_bareo^3}, \eqref{bbar_bareo^4}, the result is obtained.\\
-\textit{ Proof of \eqref{expec_barteta^4}}-\\
It is obtained from the following facts;
\begin{equation}
\barteta{i}\barteta{j}\barteta{k}\barteta{l}=\bareo{i}\bareo{j}\bareo{k}\bareo{l}+Re(5),
\end{equation}
where $Re(5)$ is defined as before.
\begin{align}
&E[\bareo{i}\bareo{j}\bareo{k}\bareo{l}]\nonumber\\
&=n^{-2}(g^{ij}g^{kl}+g^{ik}g^{jl}+g^{il}g^{jk})+O(n^{-3}).
\end{align}
\subsection{Proof of \eqref{expan_ED_final} and related resutls}
\label{append2-3}
-\textit{ Proof of \eqref{expan_ED_final}}-\\
In this section, we use again Einstein summation convention and following rather unconventional notations: For $1\leq i, j, k, u, s, t \leq p$, 
$$
\cristoffe{i,j}{k,}\triangleq \cristofse{is}{j} g^{sk},\quad 
\cristoffe{ij,}{,k}\triangleq \cristofse{ij}{s} g^{sk},\quad
\cristoffe{i,}{j,k}\triangleq \cristofse{is}{t} g^{sj}g^{tk}, \quad
\cristofte{ij}{k}\triangleq \cristofse{st}{u} g^{si}g^{tj}g^{uk}
$$
Though $\cristoffe{ij,}{,k}$ equals Christoffel's first symbol $\cristoffe{ij}{k}$,  we use the notation $\cristoffe{ij,}{,k}$ in this section, which clarifies the position of scripts with the commas. $\cristoffm{i,j}{k,}$, $\cristoffm{ij,}{,k}$, $\cristoffm{i,}{j,k}$ and $\cristoftm{ij}{k}$ are similarly defined. 

We also use the partial differentiation notation $\partial^i (1\leq i \leq p)$ defined as
$$
\partial^i \triangleq g^{ij}\partial_j=g^{ij} \frac{\partial \hfill}{\partial \theta^j}
$$

We use the following relations later; for $1 \leq i, j, k \leq p, $
\begin{equation}
\begin{split}
 \langle {\oset{\scalebox{0.8}{$\alpha$}}{\nabla}}_{\partial^i} \partial^j, \partial^k \rangle&= g^{ku} g^{is} \langle  {\oset{\scalebox{0.8}{$\alpha$}}{\nabla}}_{\partial_s}(g^{jt} \partial_t), \partial_u \rangle\\
&=g^{ku}g^{is}g^{jt}\cristofsa{st}{u}+g^{ku}g^{is}g_{tu}(\partial_s g^{jt})\\
&=\cristofta{ij}{k}+\partial^i g^{jk}.
\end{split}
\end{equation}
\begin{equation}
\label{partial^k g^ij}
-\partial^k g^{ij}=\cristofte{kj}{i}+\cristoftm{ki}{j},
\end{equation}
since 
\begin{equation}
\begin{split}
\partial^k g^{ij}&=\partial^k \langle \partial^j, \partial^i \rangle \\
&=\langle {\oset{\scalebox{0.8}{$e$}}{\nabla}}_{\partial^k} \partial^j, \: \partial^i \rangle +\langle {\oset{\scalebox{0.8}{$m$}}{\nabla}}_{\partial^k} \partial^i, \: \partial^j \rangle\\
&=\cristofte{kj}{i}+\partial^k g^{ij}+\cristoftm{ki}{j}+\partial^k g^{ij}.
\end{split}
\end{equation}
\begin{equation}
\label{partial_j cristoffa{i,}{i,j}}
\begin{split}
\partial_j \cristoffa{i,}{i,j}&=g_{jl}\partial^{l}\cristoffa{i,}{i,j}\\
&=\partial^{l} \cristoffa{i,l}{i,}-(\partial^{l}g_{jl})\cristoffa{i,}{i,j}\\
&=\partial^{l}\cristoffa{i,l}{i,}-(g^{lm}\partial_m g_{jl})\cristoffa{i,}{i,j}\\
&=\partial^{l}\cristoffa{i,l}{i,}-g^{lm}(\cristofsa{ml}{j}+\cristofsma{mj}{l})\cristoffa{i,}{i,j}\\
&=\partial^{l}\cristoffa{i,l}{i,}-(\cristoffa{l,j}{l,}+\cristoffma{j,l}{l,})\cristoffa{i,}{i,j}.
\end{split}
\end{equation}

We will check each term in \eqref{expan_ED} one by one.
First since $\epsilon_i \overset{\alpha}{D}[\theta: \theta]=0$ (see \eqref{Egur1}), the first term of  \eqref{expan_ED} vanishes. 

Since $\epsilon_i \epsilon_j  \overset{\alpha}{D}[\theta: \theta]=g_{ij}$ (see \eqref{Egur2}), we need to evaluate 
$$g_{ij}E_\theta[(\hat{\theta}^i-\theta^i) (\hat{\theta}^j-\theta^j)].$$ 
From \eqref{expec_barteta^2}, we have
\begin{equation}
\label{g_ij_expec_barteta^2}
\begin{split}
&g_{ij}E_\theta[(\hat{\theta}^i-\theta^i) (\hat{\theta}^j-\theta^j)]\\
&=n^{-1}g_{ij}g^{ij}+n^{-2}\times \\
&\biggl[
g_{ij}g^{sj}g^{it}g^{lm}\langle \secfse{sl}, \secfsm{tm}-\secfse{tm}\rangle
+g_{ij}g^{si}g^{jt}g^{lm}\langle \secfse{sl}, \secfsm{tm}-\secfse{tm}\rangle\\
&\ +g_{ij}g^{sj}g^{it}g^{lm}\cristofse{sl}{u}\bigl(\cristoffm{tm,}{,u}-\cristoffe{tm,}{,u}\bigr)
+g_{ij}g^{si}g^{jt}g^{lm}\cristofse{sl}{u}\bigl(\cristoffm{tm,}{,u}-\cristoffe{tm,}{,u}\bigr)\\
&\ +g_{ij}\bbar{lm}{j}g^{ik}g^{ls}g^{mt}\bigl(\cristofsm{ks}{t}-\cristofse{ks}{t}\bigr)
+g_{ij}\bbar{lm}{i}g^{jk}g^{ls}g^{mt}\bigl(\cristofsm{ks}{t}-\cristofse{ks}{t}\bigr)\\
&\ +g_{ij}g^{jk}g^{lu}g^{is}g^{mt}\bigl(\langle \secfse{kl}, \secfse{um}\rangle g_{st}
+\cristoffe{kl,}{,v}\cristoffe{um,}{,w}g_{vw}g_{st}+\cristofse{kl}{s}\cristofse{um}{t}+\cristofse{kl}{t}\cristofse{um}{s}\bigr)\\
&\ +g_{ij}g^{ik}g^{lu}g^{js}g^{mt}\bigl(\langle \secfse{kl}, \secfse{um}\rangle g_{st}
+\cristoffe{kl,}{,v}\cristoffe{um,}{,w}g_{vw}g_{st}+\cristofse{kl}{s}\cristofse{um}{t}+\cristofse{kl}{t}\cristofse{um}{s}\bigr)\\
&\ +g_{ij}\bbar{ms}{l}g^{jk}g^{it}g^{mu}g^{sv}\bigl(\cristofse{kl}{t}g_{uv}+\cristofse{kl}{u}g_{tv}+\cristofse{kl}{v}g_{tu}\bigr)\\
&\ +g_{ij}\bbar{ms}{l}g^{ik}g^{jt}g^{mu}g^{sv}\bigl(\cristofse{kl}{t}g_{uv}+\cristofse{kl}{u}g_{tv}+\cristofse{kl}{v}g_{tu}\bigr)\\
&\ +\frac{1}{2}g_{ij}g^{js}g^{it}g^{lu}g^{mv}\bigl\{
\bigl((\partial_m\cristoffe{sl,}{,k})g_{tk}+\cristoffe{sl,}{,k}\cristofse{km}{t}-\langle \secfse{sl}, \secfsm{mt}\rangle \bigr)g_{uv}\\
&\hspace{30mm}+\bigl((\partial_m\cristoffe{sl,}{,k})g_{uk}+\cristoffe{sl,}{,k}\cristofse{km}{u}-\langle \secfse{sl}, \secfsm{mu}\rangle \bigr)g_{tv}\\
&\hspace{30mm}+\bigl((\partial_m\cristoffe{sl,}{,k})g_{vk}+\cristoffe{sl,}{,k}\cristofse{km}{v}-\langle \secfse{sl}, \secfsm{mv}\rangle \bigr)g_{ut}\bigr\}\\
&\ +\frac{1}{2}g_{ij}g^{is}g^{jt}g^{lu}g^{mv}\bigl\{
\bigl((\partial_m\cristoffe{sl,}{,k})g_{tk}+\cristoffe{sl,}{,k}\cristofse{km}{t}-\langle \secfse{sl}, \secfsm{mt}\rangle \bigr)g_{uv}\\
&\hspace{30mm}+\bigl((\partial_m\cristoffe{sl,}{,k})g_{uk}+\cristoffe{sl,}{,k}\cristofse{km}{u}-\langle \secfse{sl}, \secfsm{mu}\rangle \bigr)g_{tv}\\
&\hspace{30mm}+\bigl((\partial_m\cristoffe{sl,}{,k})g_{vk}+\cristoffe{sl,}{,k}\cristofse{km}{v}-\langle \secfse{sl}, \secfsm{mv}\rangle \bigr)g_{ut}\bigr\}\\
&\ +2g_{ij}\bbar{lm}{j}g^{mt}g^{iu}g^{lv}g^{sw}\bigl(\cristofse{st}{u}g_{vw}+\cristofse{st}{v}g_{uw}+\cristofse{st}{w}g_{uv}\bigr)\\
&\ +2g_{ij}\bbar{lm}{i}g^{mt}g^{ju}g^{lv}g^{sw}\bigl(\cristofse{st}{u}g_{vw}+\cristofse{st}{v}g_{uw}+\cristofse{st}{w}g_{uv}\bigr)\\
&\ +2g_{ij}\bbar{lm}{j}\bbar{st}{m}g^{ik}g^{lu}g^{sv}g^{tw}\bigl(g_{ku}g_{vw}+g_{kv}g_{uw}+g_{kw}g_{uv}\bigr)\\
&\ +2g_{ij}\bbar{lm}{i}\bbar{st}{m}g^{jk}g^{lu}g^{sv}g^{tw}\bigl(g_{ku}g_{vw}+g_{kv}g_{uw}+g_{kw}g_{uv}\bigr)\\
&\ +g_{ij}\cbar{lmt}{j} g^{ik} g^{ls} g^{mu} g^{tv}\bigl( g_{ks}g_{uv}+g_{ku}g_{sv}+g_{kv}g_{su}\bigr)\\
&\ +g_{ij}\cbar{lmt}{i} g^{jk} g^{ls} g^{mu} g^{tv}\bigl( g_{ks}g_{uv}+g_{ku}g_{sv}+g_{kv}g_{su}\bigr)\\
&\ +g_{ij}g^{ik} g^{js} g^{lt} g^{mu}\bigl( \langle \secfse{kl}, \secfse{sm}\rangle g_{tu}+
\cristoffe{kl,}{,v}\cristoffe{sm,}{,w}g_{vw}g_{tu}+\cristofse{kl}{t}\cristofse{sm}{u}+\cristofse{kl}{u}\cristofse{sm}{t}\bigr)\\
&\ +g_{ij}\bbar{st}{j} g^{ik}g^{lu}g^{sv}g^{tw}\bigl(\cristofse{kl}{u}g_{vw}+\cristofse{kl}{v}g_{uw}+\cristofse{kl}{w}g_{uv}\bigr)\\
&\ +g_{ij}\bbar{st}{i} g^{jk}g^{lu}g^{sv}g^{tw}\bigl(\cristofse{kl}{u}g_{vw}+\cristofse{kl}{v}g_{uw}+\cristofse{kl}{w}g_{uv}\bigr)\\
&\ +g_{ij}\bbar{lm}{i}\bbar{st}{j}g^{lk}g^{mu}g^{sv}g^{tw}\bigl(g_{ku}g_{vw}+g_{kv}g_{uw}+g_{kw}g_{uv}\bigr)\biggr]+O(n^{-5/2}).
\end{split}
\end{equation}
Taking into account the fact $g_{ij}=g_{ji}$, we have the following equation.
\begin{equation}
\label{g_ij_expec_barteta^2_contruction}
\begin{split}
&g_{ij}E_\theta[(\hat{\theta}^i-\theta^i) (\hat{\theta}^j-\theta^j)]\\
&=n^{-1}g_{ij}g^{ij}\text{(\textit{ say} $T_0$)}+n^{-2}\times \\
&\biggl[
2g_{ij}g^{sj}g^{it}g^{lm}\langle \secfse{sl}, \secfsm{tm}-\secfse{tm}\rangle \text{(\textit{ say} $T_1$)}\\
&\ +2g_{ij}g^{sj}g^{it}g^{lm}\cristofse{sl}{u}\bigl(\cristoffm{tm,}{,u}-\cristoffe{tm,}{,u}\bigr)
\text{(\textit{ say} $T_2$)}\\
&\ +2g_{ij}\bbar{lm}{j}g^{ik}g^{ls}g^{mt}\bigl(\cristofsm{ks}{t}-\cristofse{ks}{t}\bigr)\text{(\textit{ say} $T_3$)}\\
&\ +2g_{ij}g^{jk}g^{lu}g^{is}g^{mt}\bigl(\langle \secfse{kl}, \secfse{um}\rangle g_{st}
+\cristoffe{kl,}{,v}\cristoffe{um,}{,w}g_{vw}g_{st}+\cristofse{kl}{s}\cristofse{um}{t}+\cristofse{kl}{t}\cristofse{um}{s}\bigr)\text{(\textit{ say} $T_4$)}\\
&\ +2g_{ij}\bbar{ms}{l}g^{jk}g^{it}g^{mu}g^{sv}\bigl(\cristofse{kl}{t}g_{uv}+\cristofse{kl}{u}g_{tv}+\cristofse{kl}{v}g_{tu}\bigr)\text{(\textit{ say} $T_5$)}\\
&\ +g_{ij}g^{js}g^{it}g^{lu}g^{mv}\bigl\{
\bigl((\partial_m\cristoffe{sl,}{,k})g_{tk}+\cristoffe{sl,}{,k}\cristofse{km}{t}-\langle \secfse{sl}, \secfsm{mt}\rangle \bigr)g_{uv}\\
&\hspace{30mm}+\bigl((\partial_m\cristoffe{sl,}{,k})g_{uk}+\cristoffe{sl,}{,k}\cristofse{km}{u}-\langle \secfse{sl}, \secfsm{mu}\rangle \bigr)g_{tv}\\
&\hspace{30mm}+\bigl((\partial_m\cristoffe{sl,}{,k})g_{vk}+\cristoffe{sl,}{,k}\cristofse{km}{v}-\langle \secfse{sl}, \secfsm{mv}\rangle \bigr)g_{ut}\bigr\}\text{(\textit{ say} $T_6$)}\\
&\ +4g_{ij}\bbar{lm}{j}g^{mt}g^{iu}g^{lv}g^{sw}\bigl(\cristofse{st}{u}g_{vw}+\cristofse{st}{v}g_{uw}+\cristofse{st}{w}g_{uv}\bigr)\text{(\textit{ say} $T_7$)}\\
&\ +4g_{ij}\bbar{lm}{j}\bbar{st}{m}g^{ik}g^{lu}g^{sv}g^{tw}\bigl(g_{ku}g_{vw}+g_{kv}g_{uw}+g_{kw}g_{uv}\bigr)\text{(\textit{ say} $T_8$)}\\
&\ +2g_{ij}\cbar{lmt}{j} g^{ik} g^{ls} g^{mu} g^{tv}\bigl( g_{ks}g_{uv}+g_{ku}g_{sv}+g_{kv}g_{su}\bigr)\text{(\textit{ say} $T_9$)}\\
&\ +g_{ij}g^{ik} g^{js} g^{lt} g^{mu}\bigl( \langle \secfse{kl}, \secfse{sm}\rangle g_{tu}+
\cristoffe{kl,}{,v}\cristoffe{sm,}{,w}g_{vw}g_{tu}+\cristofse{kl}{t}\cristofse{sm}{u}+\cristofse{kl}{u}\cristofse{sm}{t}\bigr)\text{(\textit{ say} $T_{10}$)}\\
&\ +2g_{ij}\bbar{st}{j} g^{ik}g^{lu}g^{sv}g^{tw}\bigl(\cristofse{kl}{u}g_{vw}+\cristofse{kl}{v}g_{uw}+\cristofse{kl}{w}g_{uv}\bigr)\text{(\textit{ say} $T_{11}$)}\\
&\ +g_{ij}\bbar{lm}{i}\bbar{st}{j}g^{lk}g^{mu}g^{sv}g^{tw}\bigl(g_{ku}g_{vw}+g_{kv}g_{uw}+g_{kw}g_{uv}\bigr)\text{(\textit{ say} $T_{12}$)}\biggr]+O(n^{-5/2}).
\end{split}
\end{equation}
We will check each term of the above equation one by one. We will gather the terms including $\bbar{ij}{k}$ ($T_3, T_5, T_7, T_8, T_{11}, T_{12}$) and treat them as a whole later.
\begin{equation}
\label{epec_barteta^2*g_ij_0}
T_0: n^{-1}g_{ij}g^{ij}=n^{-1}p.
\end{equation}
\begin{equation}
\label{epec_barteta^2*g_ij_1}
\begin{split}
&T_1: 2g_{ij}g^{sj}g^{it}g^{lm}\langle \secfse{sl}, \secfsm{tm}-\secfse{tm}\rangle=2g^{st}g^{lm}\langle \secfse{sl}, \secfsm{tm}-\secfse{tm}\rangle\\
&\qquad=2\langle \secffe{s}{m}, \secffm{m}{s}-\secffe{m}{s} \rangle.
\end{split}
\end{equation}
\begin{equation}
\label{epec_barteta^2*g_ij_2}
\begin{split}
&T_2:2g_{ij}g^{sj}g^{it}g^{lm}\cristofse{sl}{u}\bigl(\cristoffm{tm,}{,u}-\cristoffe{tm,}{,u}\bigr)
=2g^{st}g^{lm}\cristofse{sl}{u}T_{tm}^u=2\cristofse{sl}{u}T^{slu}.
\end{split}
\end{equation}
\begin{equation}
\label{epec_barteta^2*g_ij_4}
\begin{split}
&T_4: 2g_{ij}g^{jk}g^{lu}g^{is}g^{mt}\bigl(\langle \secfse{kl}, \secfse{um}\rangle g_{st}
+\cristoffe{kl,}{,v}\cristoffe{um,}{,w}g_{vw}g_{st}+\cristofse{kl}{s}\cristofse{um}{t}+\cristofse{kl}{t}\cristofse{um}{s}\bigr)\\
&=2\bigl\{g^{km}g^{lu}\bigl(\langle \secfse{kl}, \secfse{um}\rangle
+\cristofse{kl}{w}\cristoffe{um,}{,w}\bigr)+g^{lu}g^{ks}g^{mt}\bigl(\cristofse{kl}{s}\cristofse{um}{t}+\cristofse{kl}{t}\cristofse{um}{s}\bigr)\bigr\}\\
&=2\bigl\{\langle \secffe{k}{u}, \secffe{u}{k}\rangle+\cristoffe{l,s}{m,}\cristoffe{m,}{l,s}+\bigl(\cristoffe{k,}{u,k}\cristoffe{u,t}{t,}+\cristoffe{k,}{u,m}\cristoffe{um,}{,k}\bigr)\bigr\}.
\end{split}
\end{equation}
We break $T_6$ into parts.
\begin{equation}
\label{epec_barteta^2*g_ij_6_1}
\begin{split}
&g_{ij}g^{js}g^{it}g^{lu}g^{mv}(\partial_m\cristoffe{sl,}{,k})g_{tk}g_{uv}=g^{st}g^{ml}(\partial_m\cristoffe{sl,}{,k})g_{tk}\\
&=g^{ml}(\partial_m\cristoffe{kl,}{,k})=(\partial^l \cristoffe{kl,}{,k}).
\end{split}
\end{equation}
\begin{equation}
\label{epec_barteta^2*g_ij_6_2}
\begin{split}
&g_{ij}g^{js}g^{it}g^{lu}g^{mv}\cristoffe{sl,}{,k}\cristofse{km}{t}g_{uv}=g^{st}g^{ml}\cristoffe{sl,}{,k}\cristofse{km}{t}=\cristoffe{l,}{t,k}\cristoffe{k,t}{l,}.
\end{split}
\end{equation}
\begin{equation}
\label{epec_barteta^2*g_ij_6_3}
\begin{split}
&-g_{ij}g^{js}g^{it}g^{lu}g^{mv}\langle \secfse{sl}, \secfsm{mt}\rangle g_{uv}=-g^{st}g^{ml}\langle \secfse{sl}, \secfsm{mt}\rangle=-\langle \secffe{l}{t}, \secffm{t}{l} \rangle.
\end{split}
\end{equation}
\begin{equation}
\label{epec_barteta^2*g_ij_6_4}
\begin{split}
&g_{ij}g^{js}g^{it}g^{lu}g^{mv}(\partial_m\cristoffe{sl,}{,k})g_{uk}g_{tv}=g^{st}g^{mv}(\partial_m\cristoffe{sl,}{,l})g_{tv}\\
&=g^{ms}(\partial_m\cristoffe{sl,}{,l})=(\partial^s \cristoffe{sl,}{,l}).
\end{split}
\end{equation}
\begin{equation}
\label{epec_barteta^2*g_ij_6_5}
\begin{split}
&g_{ij}g^{js}g^{it}g^{lu}g^{mv}\cristoffe{sl,}{,k}\cristofse{km}{u}g_{tv}=g^{sm}g^{lu}\cristoffe{sl,}{,k}\cristofse{km}{u}=\cristoffe{s,}{u,k}\cristoffe{k,u}{s,}.
\end{split}
\end{equation}
\begin{equation}
\label{epec_barteta^2*g_ij_6_6}
\begin{split}
&-g_{ij}g^{js}g^{it}g^{lu}g^{mv}\langle \secfse{sl}, \secfsm{mu}\rangle g_{tv}=-g^{lu}g^{ms}\langle \secfse{sl}, \secfsm{mu}\rangle =-\langle \secffe{s}{u}, \secffm{u}{s} \rangle.
\end{split}
\end{equation}
\begin{equation}
\label{epec_barteta^2*g_ij_6_7}
\begin{split}
&g_{ij}g^{js}g^{it}g^{lu}g^{mv}(\partial_m\cristoffe{sl,}{,k})g_{vk}g_{ut}=g^{st}g^{lu}g^{mv}(\partial_m\cristoffe{sl,}{,k})g_{vk}g_{ut}\\
&=g^{ls}(\partial_k\cristoffe{sl,}{,k})=g^{ls}\bigl\{ \partial_k(\cristoffe{s,}{t,k}g_{tl})\bigr\}\\
&=g^{ls}\bigl\{ (\partial_k\cristoffe{s,}{t,k})g_{tl}+(\partial_k g_{tl})(\cristoffe{s,}{t,k})\bigr\}=\bigl(\partial_k\cristoffe{s,}{s,k}+\cristoffe{tk,}{,s}\cristoffe{s,}{t,k}+\cristoffm{k,t}{s,}\cristoffe{s,}{t,k}\bigr)\\
&=\bigl\{\partial^{k}\cristoffe{s,k}{s,}-\bigl(\cristoffe{s,t}{s,}+\cristoffm{t,s}{s,}\bigr)\cristoffe{i,}{i,t}+\cristofse{ij}{k}\cristofte{ik}{j}+\cristofsm{ij}{k}\cristofte{ik}{j}\}\text{ (see \eqref{partial_j cristoffa{i,}{i,j}})}.
\end{split}
\end{equation}
\begin{equation}
\label{epec_barteta^2*g_ij_6_8}
\begin{split}
&g_{ij}g^{js}g^{it}g^{lu}g^{mv}\cristoffe{sl,}{,k}\cristofse{km}{v}g_{ut}=g^{st}g^{mv}\cristoffe{st,}{,k}\cristofse{km}{v}=\cristoffe{s,}{s,k}\cristoffe{k,v}{v,}.
\end{split}
\end{equation}
\begin{equation}
\label{epec_barteta^2*g_ij_6_9}
\begin{split}
&-g_{ij}g^{js}g^{it}g^{lu}g^{mv}\langle \secfse{sl}, \secfsm{mv}\rangle g_{ut}=-g^{st}g^{mv}\langle \secfse{st}, \secfsm{mv}\rangle=-\langle \secffe{t}{t}, \secffm{v}{v} \rangle.
\end{split}
\end{equation}
If we sum up all the right-hand sides of the equations from \eqref{epec_barteta^2*g_ij_6_1} to \eqref{epec_barteta^2*g_ij_6_9}, we obtain
\begin{equation}
\label{epec_barteta^2*g_ij_6}
\begin{split}
&T_6: 2\partial^m \cristoffe{km,}{,k}+\partial^k\cristoffe{s,k}{s,}+3\cristofse{ij}{k}\cristofte{ik}{j}+\cristoffe{s,}{s,k}\cristoffe{k,v}{v,}+\cristoffe{s,}{t,k}\cristoffm{k,t}{s,}\\
&\qquad\quad-(\cristoffe{s,t}{s,}+\cristoffm{t,s}{s,})\cristoffe{i,}{i,t}-2\langle \secffe{l}{t},\: \secffm{t}{l}\rangle-\langle \secffe{t}{t},\: \secffm{v}{v}\rangle.
\end{split}
\end{equation}
We  will use Kronecker's delta $\delta_{ij}$ defined by $\delta_{ij}=1,\text{ if } i=j$, and $\delta_{ij}=0, \text{ if } i\ne j$.  
\begin{equation}
\label{epec_barteta^2*g_ij_9}
\begin{split}
&T_9: 2\cbar{lmt}{k}\bigl(\delta_{lk} \delta_{mv}g^{tv}+\delta_{km}\delta_{st}g^{ls}+\delta_{lu}\delta_{kt}g^{mu}\bigr)\\
&=2\bigl( \cbar{lvt}{l}g^{tv}+\cbar{lks}{k}g^{ls}+\cbar{umt}{t}g^{mu}\bigr)=6\cbar{lvt}{l}g^{vt}\\
&=-g^{li}g^{tv}\bigl(\partial_t\cristofse{il}{v}+\partial_t\cristofse{iv}{l}+\partial_t\cristofsm{lv}{i}+(\partial_v\cristoffe{il,}{,u})g_{ut}+\cristoffe{il,}{,u}\cristofse{uv}{t}-\langle \secfse{il}, \secfsm{vt}\rangle \bigr),
\end{split}
\end{equation}
where we used the relation 
\begin{equation}
\begin{split}
&\cbar{jks}{t}(\theta)\triangleq \frac{1}{3!}\wai  g^{it}(\theta) E_\theta \biggl[\frac{\partial^3 \bareu{i}(\xx;\theta)}{\partial \theta^j \partial\theta^k \partial\theta^s}\biggr]\\
&=\frac{1}{6}g^{ti}E_\theta[l_{ijks}(x;\theta)]\\
&=-\frac{1}{6}g^{ti}\biggl\{\bigl(\partial_s\cristofse{ij}{k}+\partial_s\cristofse{ik}{j}+\partial_s\cristofsm{jk}{i}\bigr)+(\partial_k \cristoffe{ij,}{,u})g_{us}+\cristoffe{ij,}{,u}\cristofse{uk}{s}-\langle \secfse{ij}, \secfsm{ks}\rangle\biggr\}\\
&\qquad \text{(see \eqref{log_ijkh}).}
\end{split}
\end{equation}
If we continue  the calculation of $T_9$, we have the following result.
\begin{equation}
\begin{split}
&T_9: -g^{li}\partial^v\cristofse{il}{v}-g^{li}\partial^v\cristofse{iv}{l}-g^{li}\partial^v\cristofsm{lv}{i}-g^{li}g^{tv}(\partial_v\cristoffe{il,}{,u})g_{ut}-\cristoffe{i,}{i,u}\cristoffe{u,t}{t,}+\langle \secffe{i}{i}, \secffm{j}{j} \rangle\\
&=-\partial^v (g^{li}\cristofse{il}{v})+(\partial^v g^{li})\cristofse{il}{v}-\partial^v(g^{li}\cristofse{iv}{l})+(\partial^v g^{li})\cristofse{iv}{l}\\
&\quad-\partial^v(g^{li}\cristofsm{lv}{i})+(\partial^v g^{li})\cristofsm{lv}{i}-\partial^u\cristoffe{i,u}{i,}+(\cristoffe{u,}{t,u}+\cristoffm{u,}{u,t})\cristoffe{i,t}{i,}\\
&\quad-(\cristofte{ul}{i}+\cristoftm{ui}{l})\cristofse{il}{u}-\cristoffe{i,}{i,u}\cristoffe{u,t}{t,}+\langle \secffe{i}{i}, \secffm{j}{j} \rangle\\
&=-\partial^v\cristoffe{i,v}{i,}-(\cristofte{vl}{i}+\cristoftm{vi}{l})\cristofse{il}{v}-\partial^v\cristoffe{v,l}{l,}-(\cristofte{vl}{i}+\cristoftm{vi}{l})\cristofse{iv}{l}\\
&\quad-\partial^v\cristoffm{v,i}{i,}-(\cristofte{vl}{i}+\cristoftm{vi}{l})\cristofsm{lv}{i}-\partial^u\cristoffe{i,u}{i,}\\
&\quad+\cristoffe{i,t}{i,}(\cristoffe{u,}{t,u}+\cristoffm{u,}{u,t})-(\cristofte{ul}{i}+\cristoftm{ui}{l})\cristofse{il}{u}-\cristoffe{i,}{i,u}\cristoffe{u,t}{t,}+\langle \secffe{i}{i}, \secffm{j}{j} \rangle \text{ (see \eqref{partial^k g^ij}),}
\end{split}
\end{equation}
where the second equation is due to the fact 
\begin{equation}
\begin{split}
&-g^{li}g^{tv}(\partial_v\cristoffe{il,}{,u})g_{ut}=-g^{li}(\partial_u\cristoffe{il,}{,u})=-\partial_u(g^{li}\cristoffe{il,}{,u})+(\partial_u g^{li})\cristoffe{il,}{,u}\\
&=-\partial_u\cristoffe{i,}{i,u}+(\partial^u g^{li})\cristofse{il}{u}=-g_{su}(\partial^s\cristoffe{i,}{i,u})+(\partial^u g^{li})\cristofse{il}{u}\\
&=-g_{su}\partial^s(\cristoffe{i,t}{i,}g^{tu})-(\cristofte{ul}{i}+\cristoftm{ui}{l})\cristofse{il}{u} \text{ (see \eqref{partial^k g^ij})} \\
&=-\partial^t\cristoffe{i,t}{i,}+g_{su}\cristoffe{i,t}{i,}(\cristofte{st}{u}+\cristoftm{su}{t})-(\cristofte{ul}{i}+\cristoftm{ui}{l})\cristofse{il}{u} \text{ (see \eqref{partial^k g^ij})}\\
&=-\partial^u\cristoffe{i,u}{i,}+\cristoffe{i,t}{i,}(\cristoffe{u,}{t,u}+\cristoffm{u,}{u,t})-(\cristofte{ul}{i}+\cristoftm{ui}{l})\cristofse{il}{u}.
\end{split}
\end{equation}
\begin{equation}
\label{epec_barteta^2*g_ij_10}
\begin{split}
&T_{10}: g^{ks}g^{lt}\bigl\{\langle \secfse{kl},\secfse{st} \rangle+\cristoffe{kl,}{,v}\cristoffe{st,}{,w}g_{vw}+\cristofse{kl}{t}\cristoffe{s,u}{u,}+\cristofse{kl}{u}\cristoffe{s,t}{u,}\bigr\}\\
&=\bigl(\langle\secffe{l}{s},\secffe{s}{l}\rangle +\cristoffe{k,w}{m,}\cristoffe{m,}{k,w}+\cristoffe{l,}{s,l}\cristoffe{s,u}{u,}+\cristoffe{l,u}{s,}\cristoffe{s,}{u,l}\bigr).
\end{split}
\end{equation}
Now we treat those terms that include $\bbar{ij}{k}$.
\begin{equation}
\begin{split}
&T_3: 2g_{ij}\bbar{lm}{j}g^{ik}g^{ls}g^{mt}\bigl(\cristofsm{ks}{t}-\cristofse{ks}{t}\bigr)
=2g_{ij}\bbar{lm}{j}g^{ik}g^{ls}g^{mt}T_{kst}\\
&=2\bbar{lm}{k}T_{k}^{lm}.
\end{split}
\end{equation}
\begin{equation}
\begin{split}
&T_5: 2g_{ij}\bbar{ms}{l}g^{jk}g^{it}g^{mu}g^{sv}\bigl(\cristofse{kl}{t}g_{uv}+\cristofse{kl}{u}g_{tv}+\cristofse{kl}{v}g_{tu}\bigr)\\
&=2\delta_{ik}\delta_{su}\bbar{ms}{l}g^{it}g^{mu}\cristofse{kl}{t}+
2\delta_{ik}\delta_{ts}\bbar{ms}{l}g^{it}g^{mu}\cristofse{kl}{u}+
2\delta_{ik}\delta_{tm}\bbar{ms}{l}g^{it}g^{sv}\cristofse{kl}{v}\\
&=2\bbar{ms}{l}g^{ms}\cristoffe{l,t}{t,}+2\bbar{mt}{l}\cristoffe{l,}{t,m}+2\bbar{ts}{l}\cristoffe{l,}{t,s}.
\end{split}
\end{equation}
\begin{equation}
\begin{split}
&T_7: 4g_{ij}\bbar{lm}{j}g^{mt}g^{iu}g^{lv}g^{sw}\bigl(\cristofse{st}{u}g_{vw}+\cristofse{st}{v}g_{uw}+\cristofse{st}{w}g_{uv}\bigr)\\
&=4\delta_{sv}\bbar{lm}{u}g^{mt}g^{lv}\cristofse{st}{u}+4\delta_{su}\bbar{lm}{u}g^{mt}g^{lv}\cristofse{st}{v}+4\delta_{ul}\bbar{lm}{u}g^{mt}g^{sw}\cristofse{st}{w}\\
&=4\bbar{lm}{u}g^{mt}\cristoffe{t,u}{l,}+4\bbar{lm}{u}g^{mt}\cristoffe{ut,}{,l}+4\bbar{um}{u}g^{mt}\cristoffe{t,w}{w,}.
\end{split}
\end{equation}
\begin{equation}
\begin{split}
&T_8: 4g_{ij}\bbar{lm}{j}\bbar{st}{m}g^{ik}g^{lu}g^{sv}g^{tw}\bigl(g_{ku}g_{vw}+g_{kv}g_{uw}+g_{kw}g_{uv}\bigr)\\
&=4  \bbar{lm}{k}\bbar{st}{m}\delta_{kl}\delta_{tv}g^{sv}+4  \bbar{lm}{k}\bbar{st}{m}\delta_{sk}\delta_{tu}g^{lu}+4  \bbar{lm}{k}\bbar{st}{m}\delta_{kt}\delta_{us}g^{lu}\\
&=4 \bbar{lm}{l}\bbar{st}{m}g^{st}+4\bbar{lm}{s}\bbar{su}{m}g^{lu}+4\bbar{lm}{t}\bbar{st}{m}g^{sl}.
\end{split}
\end{equation}
\begin{equation}
\begin{split}
&T_{11}: 2g_{ij}\bbar{st}{j} g^{ik}g^{lu}g^{sv}g^{tw}\bigl(\cristofse{kl}{u}g_{vw}+\cristofse{kl}{v}g_{uw}+\cristofse{kl}{w}g_{uv}\bigr)\\
&=2\bbar{st}{k}g^{lu}g^{sv}\delta_{tv}\cristofse{kl}{u}+2\bbar{st}{k}g^{lu}g^{sv}\delta_{tu}\cristofse{kl}{v}+2\bbar{st}{k}g^{lu}g^{tw}\delta_{us}\cristofse{kl}{w}\\
&=2\bbar{st}{k}g^{st}\cristoffe{k,u}{u,}+2\bbar{st}{k}\cristoffe{k,}{t,s}+2\bbar{st}{k}\cristoffe{k,}{s,t}.
\end{split}
\end{equation}
\begin{equation}
\begin{split}
&T_{12}: g_{ij}\bbar{lm}{i}\bbar{st}{j}g^{lk}g^{mu}g^{sv}g^{tw}\bigl(g_{ku}g_{vw}+g_{kv}g_{uw}+g_{kw}g_{uv}\bigr)\\
&=\bbar{lm}{i}\bbar{st}{j}g_{ij}g^{lk}g^{sv}\delta_{mk}\delta_{vt}+\bbar{lm}{i}\bbar{st}{j}g_{ij}g^{lk}g^{mu}\delta_{sk}\delta_{tu}+\bbar{lm}{i}\bbar{st}{j}g_{ij}g^{lk}g^{mu}\delta_{us}\delta_{tk}\\
&=\bbar{lm}{i}\bbar{st}{j}g^{lm}g^{st}g_{ij}+\bbar{lm}{i}\bbar{st}{j}g^{sl}g^{tm}g_{ij}+\bbar{lm}{i}\bbar{st}{j}g^{sm}g^{tl}g_{ij}.
\end{split}
\end{equation}
We notice from \eqref{log_ijk} that $\bbar{jk}{s}$ can be rewritten as 
\begin{align*}
&\bbar{jk}{s}(\theta)\triangleq \frac{1}{2}\wai  g^{is}(\theta) E_\theta \biggl[\frac{\partial^2 \bareu{i}(\xx;\theta)}{\partial \theta^j \partial\theta^k}\biggr]=\frac{1}{2}g^{si}E_\theta[l_{ijk}(x;\theta)]\\
&=-\frac{1}{2}g^{si}\bigl(\cristofse{ij}{k}+\cristofse{ik}{j}+\cristofsm{jk}{i}\bigr).
\end{align*}
 Further calculation using this relationship shows the following results. 
\begin{equation}
\begin{split}
&T_3: -g^{ki}\bigl(\cristofse{il}{m}+\cristofse{im}{l}+\cristofsm{lm}{i}\bigr)T_k^{lm}\\
&=-2\cristofse{ij}{k}T^{ijk}-\cristofsm{ij}{k}T^{ijk}=-2\cristofse{ij}{k}T^{ijk}-(T_{ijk}+\cristofse{ij}{k})T^{ijk}\\
&=-T_{ijk}T^{ijk}-3\cristofse{ij}{k}T^{ijk}.
\end{split}
\end{equation}
\begin{equation}
\begin{split}
&T_5+T_{11}: 4\bbar{ms}{l}g^{ms}\cristoffe{l,t}{t,}+8\bbar{mt}{l}\cristoffe{l,}{t,m}\\
&=-2g^{li}\bigl(\cristofse{im}{s}+\cristofse{is}{m}+\cristofsm{ms}{i}\bigr)g^{ms}\cristoffe{l,t}{t,}-4g^{li}\bigl(\cristofse{im}{t}+\cristofse{it}{m}+\cristofsm{tm}{i}\bigr)\cristoffe{l,}{t,m}\\
&=-2\cristoffe{m,}{l,m}\cristoffe{l,t}{t,}-2\cristoffe{s,}{l,s}\cristoffe{l,t}{t,}-2\cristoffm{s,}{s,l}\cristoffe{l,t}{t,}-4\cristofse{ij}{k}\cristofte{ik}{j}-4\cristofse{ij}{k}\cristofte{ij}{k}-4\cristofsm{ij}{k}\cristofte{ik}{j}
\end{split}
\end{equation}
\begin{equation}
\begin{split}
&T_7: -2g^{ui}\bigl(\cristofse{il}{m}+\cristofse{im}{l}+\cristofsm{ml}{i}\bigr)\cristoffe{t,u}{l,}g^{mt}-2g^{ui}\bigl(\cristofse{il}{m}+\cristofse{im}{l}+\cristofsm{ml}{i}\bigr)\cristoffe{ut,}{,l}g^{mt}\\
&\quad -2g^{ui}\bigl(\cristofse{iu}{m}+\cristofse{im}{u}+\cristofsm{um}{i}\bigr)\cristoffe{t,s}{s,}g^{mt}\\
&=-2\cristofse{ij}{k}\cristofte{ik}{j}-2\cristofse{ij}{k}\cristofte{ik}{j}-2\cristofsm{ij}{k}\cristofte{ij}{k}\\
&\quad-2\cristofse{ij}{k}\cristofte{ik}{j}-2\cristofse{ij}{k}\cristofte{ij}{k}-2\cristofsm{ij}{k}\cristofte{ik}{j}\\
&\quad-2\cristoffe{i,}{i,t}\cristoffe{t,s}{s,}-2\cristoffe{i,}{t,i}\cristoffe{t,s}{s,}-2\cristoffm{u,}{t,u}\cristoffe{t,s}{s,}
\end{split}
\end{equation}
\begin{equation}
\begin{split}
&T_8: g^{lt}\bigl(\cristofse{tl}{m}+\cristofse{tm}{l}+\cristofsm{lm}{t}\bigr)g^{mu}\bigl(\cristofse{us}{t}+\cristofse{ut}{s}+\cristofsm{st}{u}\bigr)g^{st}\\
&+g^{st}\bigl(\cristofse{tl}{m}+\cristofse{tm}{l}+\cristofsm{lm}{t}\bigr)g^{mv}\bigl(\cristofse{vs}{u}+\cristofse{vu}{s}+\cristofsm{us}{v}\bigr)g^{lu}\\
&+g^{ti}\bigl(\cristofse{il}{m}+\cristofse{im}{l}+\cristofsm{ml}{i}\biggr)g^{mu}\bigl(\cristofse{us}{t}+\cristofse{ut}{s}+\cristofsm{st}{u}\bigr)g^{sl}\\
&=\cristoffe{l,}{l,u}\cristoffe{u,t}{t,}+\cristoffe{l,}{l,u}\cristoffe{u,s}{s,}+\cristoffe{l,}{l,u}\cristoffm{t,u}{t,}+\cristoffe{m,l}{l,}\cristoffe{s,}{m,s}+\cristoffe{m,l}{l,}\cristoffe{t,}{m,t}+\cristoffe{m,l}{l,}\cristoffm{t,}{t,m}\\
&\quad+\cristoffm{m,t}{t,}\cristoffe{s,}{m,s}+\cristoffm{m,t}{t,}\cristoffe{s,}{m,s}+\cristoffm{m,t}{t,}\cristoffm{s,}{s,m}\\
&\quad+\cristofte{ij}{k}\cristofse{ik}{j}+\cristofte{ij}{k}\cristofse{ik}{j}+\cristofte{ij}{k}\cristofsm{ij}{k}+\cristofte{ij}{k}\cristofse{ij}{k}+\cristofte{ij}{k}\cristofse{ik}{j}+\cristofte{ij}{k}\cristofsm{ik}{j}
\\
&\quad+\cristoftm{ij}{k}\cristofse{ik}{j}+\cristoftm{ij}{k}\cristofse{ij}{k}+\cristoftm{ij}{k}\cristofsm{ik}{j}\\
&\quad+\cristofte{ij}{k}\cristofse{ik}{j}+\cristofte{ij}{k}\cristofse{ik}{j}+\cristofte{ij}{k}\cristofsm{ij}{k}+\cristofte{ij}{k}\cristofse{ik}{j}+\cristofte{ij}{k}\cristofse{ij}{k}+\cristofte{ij}{k}\cristofsm{ik}{j}
\\
&\quad+\cristoftm{ij}{k}\cristofse{ij}{k}+\cristoftm{ij}{k}\cristofse{ik}{j}+\cristoftm{ij}{k}\cristofsm{ik}{j}.
\end{split}
\end{equation}
\begin{equation}
\begin{split}
&T_{12}:  \frac{1}{4}g^{iu}\bigl(\cristofse{ul}{m}+\cristofse{um}{l}+\cristofsm{lm}{u}\bigr)g^{jv}\bigl(\cristofse{vs}{t}+\cristofse{vt}{s}+\cristofsm{st}{v}\bigr)g^{lm}g^{st}g_{ij}\\
&\quad+ \frac{1}{4}g^{iu}\bigl(\cristofse{ul}{m}+\cristofse{um}{l}+\cristofsm{lm}{u}\bigr)g^{jv}\bigl(\cristofse{vs}{t}+\cristofse{vt}{s}+\cristofsm{st}{v}\bigr)g^{sl}g^{tm}g_{ij}\\
&\quad+ \frac{1}{4}g^{iu}\bigl(\cristofse{ul}{m}+\cristofse{um}{l}+\cristofsm{lm}{u}\bigr)g^{jv}\bigl(\cristofse{vs}{t}+\cristofse{vt}{s}+\cristofsm{st}{v}\bigr)g^{sm}g^{tl}g_{ij}\\
&=\frac{1}{4}\bigl(\cristofse{ul}{m}+\cristofse{um}{l}+\cristofsm{lm}{u}\bigr)\bigl(\cristofse{vs}{t}+\cristofse{vt}{s}+\cristofsm{st}{v}\bigr)g^{uv}g^{lm}g^{st}\\
&\quad+\frac{1}{4}\bigl(\cristofse{ul}{m}+\cristofse{um}{l}+\cristofsm{lm}{u}\bigr)\bigl(\cristofse{vs}{t}+\cristofse{vt}{s}+\cristofsm{st}{v}\bigr)g^{uv}g^{sl}g^{tm}\\
&\quad+\frac{1}{4}\bigl(\cristofse{ul}{m}+\cristofse{um}{l}+\cristofsm{lm}{u}\bigr)\bigl(\cristofse{vs}{t}+\cristofse{vt}{s}+\cristofsm{st}{v}\bigr)g^{uv}g^{sm}g^{tl}\\
&=\frac{1}{4}\bigl(\cristoffe{l,}{v,l}\cristoffe{v,t}{t,}+\cristoffe{l,}{v,l}\cristoffe{v,s}{s,}+\cristoffe{l,}{v,l}\cristoffm{s,v}{s,}+\cristoffe{m,}{v,m}\cristoffe{v,t}{t,}+\cristoffe{m,}{v,m}\cristoffe{v,s}{s,}+\cristoffe{m,}{v,m}\cristoffm{t,v}{t,}\\
&\qquad +\cristoffm{m,}{m,v}\cristoffe{v,t}{t,}+\cristoffm{m,}{m,v}\cristoffe{v,s}{s,}+\cristoffm{m,}{m,v}\cristoffm{t,v}{t,}\\
&\qquad+\cristofse{ij}{k}\cristofte{ij}{k}+\cristofse{ij}{k}\cristofte{ik}{j}+\cristofse{ij}{k}\cristoftm{ik}{j}+\cristofse{ij}{k}\cristofte{ik}{j}+\cristofse{ij}{k}\cristofte{ij}{k}+\cristofse{ij}{k}\cristoftm{ik}{j}\\
&\qquad +\cristofsm{ij}{k}\cristofte{ik}{j}+\cristofsm{ij}{k}\cristofte{ik}{j}+\cristofsm{ij}{k}\cristoftm{ij}{k}\\
&\qquad+\cristofse{ij}{k}\cristofte{ik}{j}+\cristofse{ij}{k}\cristofte{ij}{k}+\cristofse{ij}{k}\cristoftm{ik}{j}+\cristofse{ij}{k}\cristofte{ij}{k}+\cristofse{ij}{k}\cristofte{ik}{j}+\cristofse{ij}{k}\cristoftm{ik}{j}\\
&\qquad +\cristofsm{ij}{k}\cristofte{ik}{j}+\cristofsm{ij}{k}\cristofte{ik}{j}+\cristofsm{ij}{k}\cristoftm{ij}{k}\bigr)
\end{split}
\end{equation}
Now we combine each result on $T_i,\ i=1,\ldots,12$ that we have gained so far.  We count the times of appearance through $T_i$s' for the terms consisting of geometrical properties. Counting order is  $T_1, T_2, T_4, T_6, T_9, T_{10}, T_3, T_5+T_{11}, T_7, T_8,  T_{12}$.
\begin{align*}
&\langle \secffe{j}{i}, \secffe{i}{j}\rangle: -2+2+1=1\\
&\langle \secffe{j}{i}, \secffm{i}{j}\rangle: 2-2=0\\
&\langle \secffe{i}{i}, \secffm{j}{j}\rangle: -1+1=0\\
&\cristofse{ij}{k}T^{ijk}: 2-3=-1\\
&T_{ijk}T^{ijk}: -1\\
&\partial^i\cristoffe{i,j}{j,}: 2-1=1\\
&\partial^i\cristoffe{j,i}{j,}: 1-1-1=-1\\
&\partial^i\cristoffm{i,j}{j,}: -1\\
&\cristofse{ij}{k}\cristofte{ij}{k}: 2+1-4-2+1+1+1/4+1/4+1/4+1/4=0\\
&\cristofse{ij}{k}\cristofte{ik}{j}: 2+3-1-1-1+1-4-2-2-2+1+1+1+1+1+1\\
&\qquad\qquad\quad+1/4+1/4+1/4+1/4=0\\
&\cristoffe{j,i}{i,}\cristoffe{s,}{j,s}: 2+1-2-2-2+1+1+1/4+1/4+1/4+1/4=0\\
&\cristoffe{i,j}{i,}\cristoffe{s,}{j,s}: 1+1-1-2+1+1=1\\
&\cristoffe{i,j}{i,}\cristoffe{s,}{s,j}: -1\\
&\cristofse{ij}{k}\cristoftm{ij}{k}: -1-1-2+1+1+1+1=0\\
&\cristofse{ij}{k}\cristoftm{ik}{j}: 1-1-1-4-2+1+1+1+1\\
&\qquad\qquad\quad+1/4+1/4+1/4+1/4+1/4+1/4+1/4+1/4=-1\\
&\cristoffe{j,i}{i,}\cristoffm{s,}{j,s}: -2+1+1=0\\
&\cristoffe{j,i}{i,}\cristoffm{s,}{s,j}: -2+1+1/4+1/4+1/4+1/4=0\\
&\cristoffe{i,}{i,j}\cristoffm{j,s}{s,}: -1\\
&\cristoffe{i,j}{i,}\cristoffm{s,}{s,j}: 1+1=2\\
&\cristofsm{ij}{k}\cristoftm{ij}{k}: 1/4+1/4=1/2\\
&\cristofsm{ij}{k}\cristoftm{ik}{j}: -1+1+1=1\\
&\cristoffm{j,i}{i,}\cristoffm{s,}{s,j}: 1\\
&\cristoffm{i,j}{i,}\cristoffm{s,}{s,j}: 1/4
\end{align*}
Finally it turns out that the following term in \eqref{expan_ED}
$$
\frac{1}{2}\wai \waj \bigl(\epsilon_i \epsilon_j \overset{\alpha}{D}[\theta: \theta]\bigr)E_\theta[(\hat{\theta}^i-\theta^i) (\hat{\theta}^j-\theta^j)]
$$
equals
\begin{equation}
\label{2nd_term_final_ED}
\begin{split}
&\frac{p}{2}n^{-1}+\frac{1}{2}n^{-2}\\
&\times \biggl[ \langle \secffe{j}{i}, \secffe{i}{j}\rangle-\cristofse{ij}{k}T^{ijk}-T_{ijk}T^{ijk}+\partial^i\cristoffe{i,j}{j,}-\partial^i\cristoffe{j,i}{j,}-\partial^i\cristoffm{i,j}{j,}\\
&\qquad+\cristoffe{i,j}{i,}\cristoffe{s,}{j,s}-\cristoffe{i,j}{i,}\cristoffe{s,}{s,j}-\cristofse{ij}{k}\cristoftm{ik}{j}-\cristoffe{i,}{i,j}\cristoffm{j,s}{s,}+2\cristoffe{i,j}{i,}\cristoffm{s,}{s,j}\\
&\qquad+\frac{1}{2}\cristofsm{ij}{k}\cristoftm{ij}{k}+\cristofsm{ij}{k}\cristoftm{ik}{j} +\cristoffm{j,i}{i,}\cristoffm{s,}{s,j}+\frac{1}{4}\cristoffm{i,j}{i,}\cristoffm{s,}{s,j}\biggl]
+O(n^{-5/2}).
\end{split}
\end{equation}
Now we will evaluate the third term in \eqref{expan_ED}
$$
\frac{1}{6}\wai \waj \wak \bigl(\epsilon_i \epsilon_j \epsilon_k \overset{\alpha}{D}[\theta: \theta]\bigr)E_\theta[(\hat{\theta}^i-\theta^i) (\hat{\theta}^j-\theta^j)(\hat{\theta}^k-\theta^k)]
$$
From \eqref{Egur3}, the above term equals
\begin{equation}
\label{third_term_expan_ED}
\begin{split}
&\frac{1}{6}\wai\waj\wak\bigl(2\cristofsa{ij}{k}+\cristofsma{ij}{k}\bigr)E_\theta[(\hat{\theta}^i-\theta^i) (\hat{\theta}^j-\theta^j)(\hat{\theta}^k-\theta^k)]\\
&=\frac{1}{6}\wai\waj\wak \bigl(3\cristofse{ij}{k}+\frac{3-\alpha}{2}(\cristofsm{ij}{k}-\cristofse{ij}{k})\bigr)\\
&\hspace{30mm}\times E_\theta[(\hat{\theta}^i-\theta^i) (\hat{\theta}^j-\theta^j)(\hat{\theta}^k-\theta^k)]\text{ (see \eqref{another_def_Cristofsa})}\\
&=\frac{1}{6}\wai\waj\wak \bigl(\cristofsm{ij}{k}+2\cristofse{ij}{k}+\alpha'(\cristofsm{ij}{k}-\cristofse{ij}{k})\bigr)\\
&\hspace{30mm}\times E_\theta[(\hat{\theta}^i-\theta^i) (\hat{\theta}^j-\theta^j)(\hat{\theta}^k-\theta^k)],
\end{split}
\end{equation}
where $\alpha'\triangleq(1-\alpha)/2$. Because of  \eqref{expec_barteta^3}, the left-hand side of \eqref{third_term_expan_ED} equals
\begin{align*}
&\frac{1}{6}n^{-2}\wai\waj\wak \bigl(\cristofsm{ij}{k}+2\cristofse{ij}{k}+\alpha'(\cristofsm{ij}{k}-\cristofse{ij}{k})\bigr)\\
&\times \biggl\{g^{is}g^{jt}g^{ku}\bigl(\cristofsm{st}{u}-\cristofse{st}{u}\bigr)
+\bigl(\cristoffe{st,}{,s}g^{it}g^{jk}+\cristoffe{st,}{,j}g^{sk}g^{it}+\cristoffe{st,}{,k}g^{sj}g^{it}\bigr)\\
&\quad+\bigl(\cristoffe{st,}{,s}g^{jt}g^{ik}+\cristoffe{st,}{,i}g^{sk}g^{jt}+\cristoffe{st,}{,k}g^{si}g^{jt}\bigr)+\bigl(\cristoffe{st,}{,s}g^{kt}g^{ji}+\cristoffe{st,}{,j}g^{si}g^{kt}+\cristoffe{st,}{,i}g^{sj}g^{kt}\bigr)\\
&\quad-\frac{1}{2}g^{iu}\bigl(\cristofse{us}{t}+\cristofse{ut}{s}+\cristofsm{st}{u}\bigr)\bigl(g^{ts}g^{jk}+g^{js}g^{kt}+g^{ks}g^{jt}\bigr)\\
&\quad-\frac{1}{2}g^{ju}\bigl(\cristofse{us}{t}+\cristofse{ut}{s}+\cristofsm{st}{u}\bigr)\bigl(g^{ts}g^{ik}+g^{is}g^{kt}+g^{ks}g^{it}\bigr)\\
&\quad-\frac{1}{2}g^{ku}\bigl(\cristofse{us}{t}+\cristofse{ut}{s}+\cristofsm{st}{u}\bigr)\bigl(g^{ts}g^{ji}+g^{js}g^{it}+g^{is}g^{jt}\bigr)\biggr\}+O(n^{-5/2}).
\end{align*}
Further calculation shows that the summation over $i, j, k$ equals
\begin{align*}
&\bigl(\cristofsm{ij}{k}+2\cristofse{ij}{k}\bigr)\bigl(\cristoftm{ij}{k}-\cristofte{ij}{k}\bigr)+\bigl(\cristoffm{i,k}{k,}+2\cristoffe{i,k}{k,}\bigr)\cristoffe{s,}{i,s}+\bigl(\cristofsm{ij}{k}+2\cristofse{ij}{k}\bigr)\cristofte{ki}{j}\\
&+\bigl(\cristofsm{ij}{k}+2\cristofse{ij}{k}\bigr)\cristofte{ij}{k}+\bigl(\cristoffm{j,k}{k,}+2\cristoffe{j,k}{k,}\bigr)\cristoffe{s,}{j,s}+\bigl(\cristofsm{ij}{k}+2\cristofse{ij}{k}\bigr)\cristofte{kj}{i}\\
&+\bigl(\cristofsm{ij}{k}+2\cristofse{ij}{k}\bigr)\cristofte{ij}{k}+\bigl(\cristoffm{i,k}{i,}+2\cristoffe{i,k}{i,}\bigr)\cristoffe{s,}{k,s}+\bigl(\cristofsm{ij}{k}+2\cristofse{ij}{k}\bigr)\cristofte{ik}{j}\\
&+\bigl(\cristofsm{ij}{k}+2\cristofse{ij}{k}\bigr)\cristofte{jk}{i}-\bigl(\cristoffm{i,k}{k,}+2\cristoffe{i,k}{k,}\bigr)\cristoffe{t,}{i,t}-\bigl(\cristofsm{ij}{k}+2\cristofse{ij}{k}\bigr)\cristofte{ij}{k}\\
&-\bigl(\cristofsm{ij}{k}+2\cristofse{ij}{k}\bigr)\cristofte{ik}{j}-\bigl(\cristoffm{i,k}{k,}+2\cristoffe{i,k}{k,}\bigr)\cristoffe{t,}{i,t}-\bigl(\cristofsm{ij}{k}+2\cristofse{ij}{k}\bigr)\cristofte{ik}{j}\\
&-\bigl(\cristofsm{ij}{k}+2\cristofse{ij}{k}\bigr)\cristofte{ij}{k}-\bigl(\cristoffm{i,k}{k,}+2\cristoffe{i,k}{k,}\bigr)\cristoffm{s,}{s,i}-\bigl(\cristofsm{ij}{k}+2\cristofse{ij}{k}\bigr)\cristoftm{jk}{i}\\
&-\bigl(\cristofsm{ij}{k}+2\cristofse{ij}{k}\bigr)\cristoftm{kj}{i}\\
&-\frac{1}{2}\bigl(\cristoffm{i,}{i,u}+2\cristoffe{i,}{i,u}\bigr)\cristoffe{u,t}{t,}-\frac{1}{2}\bigl(\cristofsm{ij}{k}+2\cristofse{ij}{k}\bigr)\cristofte{kj}{i}-\frac{1}{2}\bigl(\cristofsm{ij}{k}+2\cristofse{ij}{k}\bigr)\cristofte{ki}{j}\\
&-\frac{1}{2}\bigl(\cristoffm{i,}{i,u}+2\cristoffe{i,}{i,u}\bigr)\cristoffe{u,t}{t,}-\frac{1}{2}\bigl(\cristofsm{ij}{k}+2\cristofse{ij}{k}\bigr)\cristofte{ki}{j}-\frac{1}{2}\bigl(\cristofsm{ij}{k}+2\cristofse{ij}{k}\bigr)\cristofte{kj}{i}\\
&-\frac{1}{2}\bigl(\cristoffm{i,}{i,u}+2\cristoffe{i,}{i,u}\bigr)\cristoffm{t,u}{t,}-\frac{1}{2}\bigl(\cristofsm{ij}{k}+2\cristofse{ij}{k}\bigr)\cristoftm{ji}{k}-\frac{1}{2}\bigl(\cristofsm{ij}{k}+2\cristofse{ij}{k}\bigr)\cristoftm{ij}{k}\\
&+\alpha'\bigl(\cristofsm{ij}{k}-\cristofse{ij}{k}\bigr)\bigl(\cristoftm{ij}{k}-\cristofte{ij}{k}\bigr)\\
&+3\alpha'\bigl(\cristoffm{i,k}{k,}-\cristoffe{i,k}{k,}\bigr)\cristoffe{s,}{i,s}+3\alpha'\bigl(\cristofsm{ij}{k}-\cristofse{ij}{k}\bigr)\cristofte{ki}{j}+3\alpha'\bigl(\cristofsm{ij}{k}-\cristofse{ij}{k}\bigr)\cristofte{ij}{k}\\
&-\frac{3}{2}\alpha'\bigl(\cristoffm{i,k}{k,}-\cristoffe{i,k}{k,}\bigr)\cristoffe{t,}{i,t}-\frac{3}{2}\alpha'\bigl(\cristofsm{ij}{k}-\cristofse{ij}{k}\bigr)\cristofte{ij}{k}-\frac{3}{2}\alpha'\bigl(\cristofsm{ij}{k}-\cristofse{ij}{k}\bigr)\cristofte{ik}{j}\\
&-\frac{3}{2}\alpha'\bigl(\cristoffm{i,k}{k,}-\cristoffe{i,k}{k,}\bigr)\cristoffe{t,}{i,t}-\frac{3}{2}\alpha'\bigl(\cristofsm{ij}{k}-\cristofse{ij}{k}\bigr)\cristofte{ik}{j}-\frac{3}{2}\alpha'\bigl(\cristofsm{ij}{k}-\cristofse{ij}{k}\bigr)\cristofte{ij}{k}\\
&-\frac{3}{2}\alpha'\bigl(\cristoffm{i,k}{k,}-\cristoffe{i,k}{k,}\bigr)\cristoffm{s,}{s,i}-\frac{3}{2}\alpha'\bigl(\cristofsm{ij}{k}-\cristofse{ij}{k}\bigr)\cristoftm{jk}{i}-\frac{3}{2}\alpha'\bigl(\cristofsm{ij}{k}-\cristofse{ij}{k}\bigr)\cristoftm{kj}{i}\\
\end{align*}
First we consider the terms without $\alpha'$. We count the times of appearance for each geometric property there.
\begin{align*}
&\cristofse{ij}{k}\cristofte{ij}{k}: -2+2+2-2-2=-2\\
&\cristofse{ij}{k}\cristofte{ik}{j}: 2+2+2+2-2-2-1-1-1-1=0\\
&\cristoffe{j,i}{i,}\cristoffe{s,}{j,s}: 2+2-2-2=0\\
&\cristoffe{i,j}{i,}\cristoffe{s,}{j,s}: 2-1-1=0\\
&\cristofse{ij}{k}\cristoftm{ij}{k}: 2-1+1+1-1-1-1-1=-1\\
&\cristofse{ij}{k}\cristoftm{ik}{j}: 1+1+1+1-1-1-2-2-1/2-1/2-1/2-1/2=-4\\
&\cristoffe{j,i}{i,}\cristoffm{s,}{j,s}: 1+1-1-1=0\\
&\cristoffe{j,i}{i,}\cristoffm{s,}{s,j}: 1-2-1/2-1/2=-2\\
&\cristoffe{i,j}{i,}\cristoffm{s,}{s,j}: -1\\
&\cristofsm{ij}{k}\cristoftm{ij}{k}: 1-1/2-1/2=0\\
&\cristofsm{ij}{k}\cristoftm{ik}{j}: -1-1=-2\\
&\cristoffm{j,i}{i,}\cristoffm{s,}{s,j}: -1\\
&\cristoffm{i,j}{i,}\cristoffm{s,}{s,j}: -1/2
\end{align*}
Therefore the terms without $\alpha'$ is summarized as
\begin{align*}
&-2\cristofse{ij}{k}\cristofte{ij}{k}-\cristofse{ij}{k}\cristoftm{ij}{k}-4\cristofse{ij}{k}\cristoftm{ik}{j}-2\cristoffe{i,k}{k,}\cristoffm{s,}{s,i}-\cristoffe{i,}{i,u}\cristoffm{t,u}{t,}\\
&-2\cristofsm{ij}{k}\cristoftm{ik}{j}-\cristoffm{i,k}{k,}\cristoffm{s,}{s,i}-1/2\cristoffm{i,}{i,u}\cristoffm{t,u}{t,}.
\end{align*}
Next we notice that the terms with $\alpha'$ equals
\begin{align*}
&\alpha'T_{ijk}T^{ijk}+3\alpha'T_{ik}^k\bigl(\cristoffm{s,}{i,s}-T_s^{is}\bigr)+3\alpha'T_{ijk}\bigl(\cristoftm{ki}{j}-T^{kij}\bigr)+3\alpha'T_{ijk}\bigl(\cristoftm{ij}{k}-T^{ijk}\bigr)\\
&-\frac{3}{2}\alpha'T_{ik}^{k}\bigl(\cristoffm{t,}{i,t}-T_t^{it}\bigr)-\frac{3}{2}\alpha'T_{ijk}\bigl(\cristoftm{ij}{k}-T^{ijk}\bigr)-\frac{3}{2}\alpha'T_{ijk}\bigl(\cristoftm{ik}{j}-T^{ikj}\bigr)\\
&-\frac{3}{2}\alpha'T_{ik}^k\bigl(\cristoffm{t,}{i,t}-T_t^{it}\bigr)-\frac{3}{2}\alpha'T_{ijk}\bigl(\cristoftm{ik}{j}-T^{ikj}\bigr)-\frac{3}{2}\alpha'T_{ijk}\bigl(\cristoftm{ij}{k}-T^{ijk}\bigr)\\
&-\frac{3}{2}\alpha'T_{ik}^k\cristoffm{s,}{s,i}-\frac{3}{2}\alpha'T_{ijk}\cristoftm{jk}{i}-\frac{3}{2}\alpha'T_{ijk}\cristoftm{kj}{i}\\
&=\alpha'\bigl(T_{ijk}T^{ijk}-3T^{ijk}\cristofsm{ij}{k}-\frac{3}{2}T_{j}^{js}\cristoffm{i,s}{i,}\bigr)
\end{align*}
Consequently it turns out that the following term in  \eqref{expan_ED}
$$
\frac{1}{6}\wai \waj \wak \bigl(\epsilon_i \epsilon_j \epsilon_k \overset{\alpha}{D}[\theta: \theta]\bigr)E_\theta[(\hat{\theta}^i-\theta^i) (\hat{\theta}^j-\theta^j)(\hat{\theta}^k-\theta^k)]
$$
equals

\begin{align}
&\frac{1}{6}n^{-2}\nonumber \\
&\times\biggl[\alpha'\Bigl\{T_{ijk}T^{ijk}-3T^{ijk}\cristofsm{ij}{k}-\frac{3}{2}T_{j}^{js}\cristoffm{i,s}{i,}\Bigr\}\nonumber \\
&\qquad-2\cristofse{ij}{k}\cristofte{ij}{k}-\cristofse{ij}{k}\cristoftm{ij}{k}-4\cristofse{ij}{k}\cristoftm{ik}{j}-2\cristoffe{i,k}{k,}\cristoffm{s,}{s,i}-\cristoffe{i,}{i,u}\cristoffm{t,u}{t,}\nonumber \\
&\qquad-2\cristofsm{ij}{k}\cristoftm{ik}{j}-\cristoffm{i,k}{k,}\cristoffm{s,}{s,i}-\frac{1}{2}\cristoffm{i,}{i,u}\cristoffm{t,u}{t,}\biggr]+O(n^{-5/2})\nonumber \\
&=\frac{1}{6}n^{-2}\nonumber \\
&\times\biggl[\alpha'\Bigl\{-2T_{ijk}T^{ijk}-\frac{3}{2}T_{j}^{js}T^i_{is}-3T^{ijk}\cristofse{ij}{k}-\frac{3}{2}T_j^{js}\cristoffe{i,s}{i,}\Bigr\}\nonumber \\
&\qquad-2\cristofse{ij}{k}\cristofte{ij}{k}-\cristofse{ij}{k}\cristoftm{ij}{k}-4\cristofse{ij}{k}\cristoftm{ik}{j}-2\cristoffe{i,k}{k,}\cristoffm{s,}{s,i}-\cristoffe{i,}{i,u}\cristoffm{t,u}{t,}\nonumber \\
&\qquad-2\cristofsm{ij}{k}\cristoftm{ik}{j}-\cristoffm{i,k}{k,}\cristoffm{s,}{s,i}-\frac{1}{2}\cristoffm{i,}{i,u}\cristoffm{t,u}{t,}\biggr]+O(n^{-5/2})\label{3rd_term_final_ED} 
\end{align}
Now, it is the turn to focus on the forth term in \eqref{expan_ED}, 
$$
\frac{1}{24}\wai \waj \wak \wal \bigl(\epsilon_i \epsilon_j \epsilon_k \epsilon_l \overset{\alpha}{D}[\theta: \theta]\bigr)E_\theta[(\hat{\theta}^i-\theta^i) (\hat{\theta}^j-\theta^j)(\hat{\theta}^k-\theta^k)(\hat{\theta}^l-\theta^l)].
$$
From \eqref{epsi^4_Diverge} and \eqref{expec_barteta^4}, the forth term in \eqref{expan_ED} equals
\begin {align}
&\frac{1}{24} \bigl(\epsilon_i \epsilon_j \epsilon_k \epsilon_l \overset{\alpha}{D}[\theta: \theta]\bigr)E_\theta[\barteta{i}\barteta{j}\barteta{k}\barteta{k}]\nonumber \\
&=\frac{1}{24}n^{-2}\bigl(\epsilon_i \epsilon_j \epsilon_k \epsilon_l \overset{\alpha}{D}[\theta: \theta]\bigr)\bigl(g^{ij}g^{kl}+g^{ik}g^{jl}+g^{il}g^{jk}\bigr)+O(n^{-5/2})\nonumber \\
&=\frac{1}{24}n^{-2}\nonumber \\
&\quad\times\biggl[(\alpha')^2\nonumber \\
&\qquad\times\biggl(\partial^k(\cristoffm{i,k}{i,}-\cristoffe{i,k}{i,})+\partial^j(\cristoffm{j,k}{k,}-\cristoffe{j,k}{k,})+\partial^i(\cristoffm{i,k}{k,}-\cristoffe{i,k}{k,})\nonumber \\
&\qquad\qquad-(\partial^k g^{ij})T_{ijk}-(\partial^j g^{ik})T_{ijk}-(\partial^i g^{jk})T_{ijk}\nonumber \\
&\qquad\qquad -\langle \secffe{i}{k}, \secffm{k}{i}\rangle+\langle \secffe{i}{k}, \secffe{k}{i} \rangle +p -\cristoffe{l,s}{j,}\cristoffm{j,}{l,s}+\cristoffe{l,s}{j,}\cristoffe{j,}{l,s}\nonumber \\
&\qquad\qquad -\langle \secffe{l}{k}, \secffm{k}{l}\rangle+\langle \secffe{l}{k}, \secffe{k}{l} \rangle +p -\cristoffe{l,s}{k,}\cristoffm{k,}{l,s}+\cristoffe{l,s}{k,}\cristoffe{k,}{l,s}\nonumber \\
&\qquad\qquad -\langle \secffe{l}{l}, \secffm{k}{k}\rangle+\langle \secffe{l}{l}, \secffe{k}{k} \rangle +p^2 -\cristoffe{l,s}{l,}\cristoffm{k,}{k,s}+\cristoffe{l,s}{l,}\cristoffe{k,}{k,s}\nonumber \\
&\qquad\qquad -\langle \secffe{l}{i}, \secffm{i}{l}\rangle+\langle \secffe{l}{i}, \secffe{i}{l} \rangle +p -\cristoffe{l,s}{i,}\cristoffm{i,}{l,s}+\cristoffe{l,s}{i,}\cristoffe{i,}{l,s}\nonumber \\
&\qquad\qquad -\langle \secffe{j}{j}, \secffm{i}{i}\rangle+\langle \secffe{j}{j}, \secffe{i}{i} \rangle +p^2 -\cristoffe{l,s}{l,}\cristoffm{i,}{i,s}+\cristoffe{l,s}{l,}\cristoffe{i,}{i,s}\nonumber \\
&\qquad\qquad -\langle \secffe{j}{i}, \secffm{i}{j}\rangle+\langle \secffe{j}{i}, \secffe{i}{j} \rangle +p -\cristoffe{j,s}{i,}\cristoffm{i,}{j,s}+\cristoffe{j,s}{i,}\cristoffe{i,}{j,s}\nonumber \\
&\qquad\qquad -\langle \secffe{k}{k}, \secffm{i}{i}\rangle+\langle \secffe{k}{k}, \secffe{i}{i} \rangle +p^2 -\cristoffe{k,s}{k,}\cristoffm{j,}{j,s}+\cristoffe{k,s}{k,}\cristoffe{j,}{j,s}\nonumber \\
&\qquad\qquad -\langle \secffe{l}{i}, \secffm{i}{l}\rangle+\langle \secffe{l}{i}, \secffe{i}{l} \rangle +p -\cristoffe{l,s}{i,}\cristoffm{i,}{l,s}+\cristoffe{l,s}{i,}\cristoffe{i,}{l,s}\nonumber \\
&\qquad\qquad -\langle \secffe{k}{i}, \secffm{i}{k}\rangle+\langle \secffe{k}{i}, \secffe{i}{k} \rangle +p -\cristoffe{k,s}{i,}\cristoffm{i,}{k,s}+\cristoffe{k,s}{i,}\cristoffe{i,}{k,s}\biggr)\nonumber \\
&\qquad+\alpha'\nonumber \\
&\qquad\times\biggl(\langle \secffe{k}{j}, \secffm{j}{k} \rangle -\langle \secffe{k}{j}, \secffe{j}{k}\rangle-p+\cristoffe{k,s}{j,}\cristoffm{j,}{k,s}-\cristoffe{k,s}{j,}\cristoffe{j,}{k,s}\nonumber \\
&\qquad\qquad+ \langle \secffe{i}{i}, \secffm{j}{j} \rangle -\langle \secffe{i}{i}, \secffe{j}{j}\rangle-p^2+\cristoffe{i,s}{i,}\cristoffm{j,}{j,s}-\cristoffe{i,s}{i,}\cristoffe{j,}{j,s}\nonumber \\
&\qquad\qquad+ \langle \secffe{k}{l}, \secffm{l}{k} \rangle -\langle \secffe{k}{l}, \secffe{l}{k}\rangle-p+\cristoffe{k,s}{l,}\cristoffm{l,}{k,s}-\cristoffe{k,s}{l,}\cristoffe{l,}{k,s}\nonumber \\
&\qquad\qquad+ \langle \secffe{k}{i}, \secffm{i}{k} \rangle -\langle \secffe{k}{i}, \secffe{i}{k}\rangle-p+\cristoffe{k,s}{i,}\cristoffm{i,}{k,s}-\cristoffe{k,s}{i,}\cristoffe{i,}{k,s}\nonumber \\
&\qquad\qquad+ \langle \secffe{j}{i}, \secffm{i}{j} \rangle -\langle \secffe{j}{i}, \secffe{i}{j}\rangle-p+\cristoffe{j,s}{i,}\cristoffm{i,}{j,s}-\cristoffe{j,s}{i,}\cristoffe{i,}{j,s}\nonumber \\
&\qquad\qquad+ \langle \secffe{i}{i}, \secffm{j}{j} \rangle -\langle \secffe{i}{i}, \secffe{j}{j}\rangle-p^2+\cristoffe{k,s}{k,}\cristoffm{i,}{i,s}-\cristoffe{k,s}{k,}\cristoffe{i,}{i,s}\nonumber \\
&\qquad\qquad+ \langle \secffe{i}{i}, \secffm{j}{j} \rangle -\langle \secffe{i}{i}, \secffe{j}{j}\rangle-p^2+\cristoffe{i,s}{i,}\cristoffm{l,}{l,s}-\cristoffe{i,s}{i,}\cristoffe{l,}{l,s}\nonumber \\
&\qquad\qquad+\langle \secffe{j}{k}, \secffm{k}{j} \rangle -\langle \secffe{j}{k}, \secffe{k}{j}\rangle-p+\cristoffe{j,s}{k,}\cristoffm{k,}{j,s}-\cristoffe{j,s}{k,}\cristoffe{k,}{j,s}\nonumber \\
&\qquad\qquad+\langle \secffe{j}{l}, \secffm{l}{j} \rangle -\langle \secffe{j}{l}, \secffe{l}{j}\rangle-p+\cristoffe{j,s}{l,}\cristoffm{l,}{j,s}-\cristoffe{j,s}{l,}\cristoffe{l,}{j,s}\biggr)\nonumber \\
&\qquad+(1+\alpha')\nonumber \\
&\qquad\times\biggl(\partial^k \cristoffm{i,k}{i,}-\partial^k \cristoffe{i,k}{i,}-(\partial^k g^{ij})T_{ijk}\nonumber \\
&\qquad\qquad+\partial^j \cristoffm{j,k}{k,}-\partial^j \cristoffe{j,k}{k,}-(\partial^j g^{ik})T_{ijk}\nonumber \\
&\qquad\qquad+\partial^i \cristoffm{i,k}{k,}-\partial^i \cristoffe{i,k}{k,}-(\partial^i g^{jk})T_{ijk}\biggr)\nonumber \\
&\qquad+\partial^k \cristoffe{i,k}{i,}+\partial^k \cristoffe{k,j}{j,}+\partial^k \cristoffe{k,i}{i,}
+\partial^j \cristoffe{j,k}{k,}+\partial^j \cristoffe{k,j}{k,}+\partial^j \cristoffe{j,i}{i,}
+\partial^i \cristoffe{i,k}{k,}+\partial^i \cristoffe{i,j}{j,}+\partial^i \cristoffe{j,i}{j,}\nonumber \\
&\qquad+\partial_t(\cristoffe{ij,}{,t}g^{ij})-(\partial_t g^{ij})\cristoffe{ij,}{,t}+\cristoffe{i,}{i,t}\cristoffe{t,l}{l,}-\langle \secffe{i}{i}, \secffm{l}{l}\rangle \nonumber \\
&\qquad+\partial^i \cristoffe{ij,}{,j}+\cristoffe{j,}{k,t}\cristoffe{tk,}{,j}-\langle \secffe{j}{k}, \secffm{k}{j} \rangle \nonumber \\
&\qquad+\partial^j \cristoffe{ij,}{,i}+\cristoffe{j,}{l,t}\cristoffe{t,l}{j,}-\langle \secffe{j}{l}, \secffm{l}{j} \rangle \nonumber\\
&\qquad+3(\cristofte{ki}{j}+\cristoftm{kj}{i})\cristofse{ij}{k}+3(\cristofte{ji}{k}+\cristoftm{jk}{i})\cristofse{ij}{k}+3(\cristofte{ij}{k}+\cristoftm{ik}{j})\cristofse{ij}{k}\biggl]\nonumber \\
&\qquad+O(n^{-5/2}) \label{4th term of ED_alpha}
\end{align}
Note  that
$$
\partial^s g_{st}=g^{us}\partial_u g_{st}=g^{us}(\cristofse{us}{t}+\cristofsm{ut}{s})=\cristoffe{s,t}{s,}+\cristoffm{t,s}{s,},\\
$$
where we used \eqref{derivative_metric}. From this equation, we find that the following equation holds.
\begin{align*}
& \{\text {the twelfth and thirteenth terms from the last in the bracket of \eqref{4th term of ED_alpha}}\} \\
&=\partial_t(\cristoffe{ij,}{,t}g^{ij})-(\partial_t g^{ij})\cristoffe{ij,}{,t}\\
&=\partial_t\cristoffe{i,}{i,t}-(\partial^t g^{ij})\cristofse{ij}{t}\\
&=\partial_t\cristoffe{i,}{i,t}+(\cristofte{it}{j}+\cristoftm{jt}{i})\cristofse{ij}{t}\quad \text{(see \eqref{partial^k g^ij})}\\
&=g_{st}(\partial^s \cristoffe{i,}{i,t})+(\cristofte{it}{j}+\cristoftm{jt}{i})\cristofse{ij}{t}\\
&=\partial^s\cristoffe{i,s}{i,}-(\partial^s g_{st})\cristoffe{i,}{i,t}+(\cristofte{it}{j}+\cristoftm{jt}{i})\cristofse{ij}{t}\\
&=\partial^s\cristoffe{i,s}{i,}+(\cristofte{it}{j}+\cristoftm{jt}{i})\cristofse{ij}{t}-(\cristoffe{s,t}{s,}+\cristoffm{t,s}{s,})\cristoffe{i,}{i,t}
\end{align*}
Together with this result, adding and subtracting the same terms in \eqref{4th term of ED_alpha}, we have the following equation.
\begin{align}
&\frac{1}{24}\wai \waj \wak \wal \bigl(\epsilon_i \epsilon_j \epsilon_k \epsilon_l \overset{\alpha}{D}[\theta: \theta]\bigr)E_\theta[(\hat{\theta}^i-\theta^i) (\hat{\theta}^j-\theta^j)(\hat{\theta}^k-\theta^k)(\hat{\theta}^l-\theta^l)] \nonumber\\
&=\frac{1}{24}n^{-2}\nonumber\\
&\times \biggl[ (\alpha')^2\biggl\{3\partial^k(\cristoffm{i,k}{i,}-\cristoffe{i,k}{i,})-6\langle \secffe{i}{j}, (\secffm{j}{i}-\secffe{j}{i}) \rangle - 3\langle \secffe{i}{i}, (\secffm{j}{j}-\secffe{j}{j}) \rangle \nonumber\\
&\qquad \qquad +3p^2+6p-6\cristofse{ij}{k}(\cristoftm{ij}{k}-\cristofte{ij}{k})-3\cristoffe{i,s}{i,}(\cristoffm{j,}{j,s}-\cristoffe{j,}{j,s})+3\cristofte{ij}{k}T_{ijk}+3\cristoftm{ij}{k}T_{ijk}\biggr\}\nonumber\\
&\qquad+\alpha'\biggl\{3\partial^k(\cristoffm{i,k}{i,}-\cristoffe{i,k}{i,})+6\langle \secffe{i}{j}, (\secffm{j}{i}-\secffe{j}{i}) \rangle + 3\langle \secffe{i}{i}, (\secffm{j}{j}-\secffe{j}{j}) \rangle \nonumber\\
&\qquad \qquad -3p^2-6p+6\cristofse{ij}{k}(\cristoftm{ij}{k}-\cristofte{ij}{k})+3\cristoffe{i,s}{i,}(\cristoffm{j,}{j,s}-\cristoffe{j,}{j,s})+3\cristofte{ij}{k}T_{ijk}+3\cristoftm{ij}{k}T_{ijk}\biggr\}\nonumber\\
&\qquad +\partial^k\cristoffm{i,k}{i,}+2\partial^k \cristoffm{k,i}{i,}+3\partial^k \cristoffe{i,k}{i,}+6\partial^k \cristoffe{k,i}{i,}\nonumber\\
&\qquad+6\cristofte{ij}{k}\cristofse{ij}{k}+6\cristofte{ij}{k}\cristofse{ik}{j}+10\cristofte{ij}{k}\cristofsm{ik}{j}+\cristoffe{i,}{i,t}\cristoffe{t,j}{j,}\nonumber\\
&\qquad+3(\cristofte{ij}{k}T_{ijk}+\cristoftm{ij}{k}T_{ijk})-(\cristoffe{s,t}{s,}+\cristoffm{t,s}{s,})\cristoffe{i,}{i,t}-2\langle \secffe{j}{i}, \secffm{i}{j} \rangle -\langle \secffe{i}{i}, \secffm{j}{j} \rangle \biggl]\nonumber\\
&\quad+O(n^{-5/2})\nonumber\\
&=\frac{1}{24}n^{-2}\nonumber\\
&\times \biggl[ (\alpha')^2\biggl\{3(\partial^k  T_{ik}^i-\cristoffe{i,s}{i,}T_{j}^{js})+3 T^{ijk}T_{ijk}-6\langle \secffe{i}{j}, (\secffm{j}{i}-\secffe{j}{i}) \rangle - 3\langle \secffe{i}{i}, (\secffm{j}{j}-\secffe{j}{j})\rangle +3p^2+6p \biggr\}\nonumber\\
&\qquad+\alpha'\biggl\{3(\partial^k  T_{ik}^i-\cristoffe{i,s}{i,}T_{j}^{js})+6\cristoffe{i,s}{i,}T_j^{js}+12\cristofse{ij}{k}T^{ijk}+3 T^{ijk}T_{ijk}\nonumber\\
&\qquad\qquad+6\langle \secffe{i}{j}, (\secffm{j}{i}-\secffe{j}{i}) \rangle +3\langle \secffe{i}{i}, (\secffm{j}{j}-\secffe{j}{j})\rangle -3p^2-6p\biggr\}\nonumber\\
&\qquad+\partial^k\cristoffm{i,k}{i,}+2\partial^k \cristoffm{k,i}{i,}+3\partial^k \cristoffe{i,k}{i,}+6\partial^k \cristoffe{k,i}{i,}\nonumber\\
&\qquad+3\cristofte{ij}{k}\cristofse{ij}{k}+6\cristofte{ij}{k}\cristofse{ik}{j}+\cristoffe{i,}{i,t}\cristoffe{t,j}{j,}-\cristoffe{s,t}{s,}\cristoffe{i,}{i,t}+10\cristofte{ij}{k}\cristofsm{ik}{j}-\cristoffe{i,}{i,t}\cristoffm{t,s}{s,}+3\cristoftm{ij}{k}\cristofsm{ij}{k}\nonumber\\
&\qquad-2\langle \secffe{j}{i}, \secffm{i}{j} \rangle -\langle \secffe{i}{i}, \secffm{j}{j} \rangle \biggl]\nonumber\\
&\quad+O(n^{-5/2}) \label{4th_term_final_ED}
\end{align}
Finally we combine \eqref{2nd_term_final_ED}, \eqref{3rd_term_final_ED} and \eqref{4th_term_final_ED} to get the total evaluation of $\overset{\alpha}{E\!D}(\theta)$.
\begin{align}
&\overset{\alpha}{E\!D}(\theta)\nonumber\\
&=\frac{1}{2}\wai \waj \bigl(\epsilon_i \epsilon_j \overset{\alpha}{D}[\theta: \theta]\bigr)E_\theta[(\hat{\theta}^i-\theta^i) (\hat{\theta}^j-\theta^j)]\nonumber\\
&\ +\frac{1}{6}\wai \waj \wak \bigl(\epsilon_i \epsilon_j \epsilon_k \overset{\alpha}{D}[\theta: \theta]\bigr)E_\theta[(\hat{\theta}^i-\theta^i) (\hat{\theta}^j-\theta^j)(\hat{\theta}^k-\theta^k)]\nonumber\\
&\ +\frac{1}{24}\wai \waj \wak \wal \bigl(\epsilon_i \epsilon_j \epsilon_k \epsilon_l \overset{\alpha}{D}[\theta: \theta]\bigr)E_\theta[(\hat{\theta}^i-\theta^i) (\hat{\theta}^j-\theta^j)(\hat{\theta}^k-\theta^k)(\hat{\theta}^l-\theta^l)]\nonumber \\
&\ +O(n^{-5/2}) \nonumber\\
&=\frac{p}{2}n^{-1}+\frac{1}{24}n^{-2}\nonumber\\
&\  \times \biggl[ (\alpha')^2\biggl\{3(\partial^k  T_{ik}^i-\cristoffe{i,s}{i,}T_{j}^{js})+3 T^{ijk}T_{ijk}-6\langle \secffe{i}{j}, (\secffm{j}{i}-\secffe{j}{i}) \rangle - 3\langle \secffe{i}{i}, (\secffm{j}{j}-\secffe{j}{j})\rangle +3p^2+6p \biggr\}\nonumber\\
&\qquad+\alpha'\biggl\{3(\partial^k  T_{ik}^i-\cristoffe{i,s}{i,}T_{j}^{js})-6T_{is}^i T_j^{js}-5T^{ijk}T_{ijk}\nonumber\\
&\qquad\qquad+6\langle \secffe{i}{j}, (\secffm{j}{i}-\secffe{j}{i}) \rangle +3\langle \secffe{i}{i}, (\secffm{j}{j}-\secffe{j}{j})\rangle -3p^2-6p\biggr\}\nonumber \\
&\qquad+12\langle \secffe{j}{i}, \secffe{i}{j} \rangle -2 \langle \secffe{j}{i} , \secffm{i}{j} \rangle - \langle \secffe{i}{i}, \secffm{j}{j} \rangle-12T_{ijk}T^{ijk}\nonumber \\
&\qquad+18 \partial^k \cristoffe{k,i}{i,}-9 \partial^k \cristoffe{i,k}{i,}-10\partial^k \cristoffm{k,i}{i,}+\partial^k \cristoffm{i,k}{i,} \text{ (say $A$ for the summation over this line)} \nonumber \\
&\qquad-5\cristofte{ij}{k}\cristofse{ij}{k}+6\cristofte{ij}{k}\cristofse{ik}{j}+13\cristoffe{j,i}{i,}\cristoffe{s,}{s,j}-13\cristoffe{s,t}{s,}\cristoffe{i,}{i,t}\text{ (say $B1$ as above)}\nonumber \\
&\qquad-4\cristofte{ij}{k}\cristofsm{ij}{k}-18\cristofte{ij}{k}\cristofsm{ik}{j}-8\cristoffe{s,i}{i,}\cristoffm{j,}{j,s}+20\cristoffe{i,s}{i,}\cristoffm{j,}{j,s}-13\cristoffe{s,t}{s,}\cristoffm{i,}{t,i}\text{ (say $B2$ as above)}\nonumber \\
&\qquad+9\cristoftm{ij}{k}\cristofsm{ij}{k}+4\cristoftm{ij}{k}\cristofsm{ik}{j}+8\cristoffm{j,i}{i,}\cristoffm{s,}{s,j}+\cristoffm{i,s}{i,}\cristoffm{j,}{j,s}-12\cristofse{ij}{k}T^{ijk}\text{ (say $B3$ as above)}\biggr]\nonumber\\
&\ +O(n^{-5/2}) \label{near_final_form_ED}
\end{align}
We defined $\overset{\;\alpha}{R}\hspace{-10pt}\begin{smallmatrix}{\ \ \ \ ij}\\{ij}\end{smallmatrix}$
as
$$
\overset{\;\alpha}{R}\hspace{-10pt}\begin{smallmatrix}{\ \ \ \ ij}\\{ij}\end{smallmatrix}\triangleq\overset{\;\alpha}{R}\hspace{-5pt}\begin{smallmatrix}{\ \ \ \ j}\\{ijk}\end{smallmatrix}g^{ik}.
$$
From \eqref{specific_form_R}, the  specific form for $\alpha=1$ is given by the following equation.
\begin{align}
&\overset{\;e}{R}\hspace{-10pt}\begin{smallmatrix}{\ \ \ \ ij}\\{ij}\end{smallmatrix}\nonumber\\
&=g^{ik}\bigl(\partial_i \cristoffe{jk,}{,j}-\partial_j \cristoffe{ik,}{,j}+\cristoffe{ir,}{,j}\cristoffe{jk,}{,r}-\cristoffe{jr,}{,j}\cristoffe{ik,}{,r}\bigr)\nonumber\\
&=\partial^k \cristoffe{k,j}{j,}-\bigl(\partial_j \cristoffe{i,}{i,j}-(\partial_j g^{ik})\cristoffe{ik,}{,j}\bigr)+\cristofse{ij}{k}\cristofte{ik}{j}-\cristoffe{r,j}{j,}\cristoffe{i,}{i,r}\nonumber\\
&=\partial^k \cristoffe{k,j}{j,}-\partial^j \cristoffe{i,j}{i,}+(\cristoffe{s,t}{s,}+\cristoffm{t,s}{s,})\cristoffe{i,}{i,t}\nonumber\\
&\quad +(\partial^j g^{ik})\cristofse{ik}{j}+\cristofse{ij}{k}\cristofte{ik}{j}-\cristoffe{r,j}{j,}\cristoffe{i,}{i,r} \text{( see \eqref{partial_j cristoffa{i,}{i,j}})} \nonumber\\
&=\partial^k \cristoffe{k,j}{j,}-\partial^k \cristoffe{i,k}{i,}+(\cristoffe{s,t}{s,}+\cristoffm{t,s}{s,})\cristoffe{i,}{i,t}\nonumber\\
&\quad-(\cristofte{ij}{k}+\cristoftm{jk}{i})\cristofse{ik}{j}+\cristofse{ij}{k}\cristofte{ik}{j}-\cristoffe{r,j}{j,}\cristoffe{i,}{i,r}\text{ (see \eqref{partial^k g^ij}) }\nonumber\\
&=\partial^k \cristoffe{k,j}{j,}-\partial^k \cristoffe{i,k}{i,}+(\cristoffe{s,t}{s,}+\cristoffm{t,s}{s,})\cristoffe{i,}{i,t}-\cristoftm{jk}{i}\cristofse{ik}{j}-\cristoffe{r,j}{j,}\cristoffe{i,}{i,r} \label{Form_1_R_ij^ij}
\end{align}
This equation is equivalent to
\begin{equation}
\label{d^k crist{kj}{j}-d^k cristoffe{ik}{i}}
\partial^k \cristoffe{k,j}{j,}-\partial^k \cristoffe{i,k}{i,}=\overset{\;e}{R}\hspace{-10pt}\begin{smallmatrix}{\ \ \ \ ij}\\{ij}\end{smallmatrix}-(\cristoffe{s,t}{s,}+\cristoffm{t,s}{s,})\cristoffe{i,}{i,t}+\cristoftm{jk}{i}\cristofse{ik}{j}+\cristoffe{r,j}{j,}\cristoffe{i,}{i,r}
\end{equation}
Since 
$$
\cristoffe{s,j}{s,}+\cristoffm{j,s}{s,}=g^{ts}(\cristofse{ts}{j}+\cristofsm{tj}{s})=g^{ts}\partial_t g_{sj}=g^{ts}\partial_t g_{js}=g^{ts}(\cristofse{tj}{s}+\cristofsm{ts}{j})=\cristoffe{j,s}{s,}+\cristoffm{s,j}{s,},
$$
we can rewrite $\overset{\;e}{R}\hspace{-10pt}\begin{smallmatrix}{\ \  \ \ ij}\\{ij}\end{smallmatrix}$ 
as
\begin{align}
\overset{\;e}{R}\hspace{-10pt}\begin{smallmatrix}{\ \  \ \ ij}\\{ij}\end{smallmatrix}&=\partial^k \cristoffe{k,j}{j,}-\partial^k \cristoffe{i,k}{i,}+(\cristoffm{s,t}{s,}+\cristoffe{t,s}{s,})\cristoffe{i,}{i,t}-\cristoftm{jk}{i}\cristofse{ik}{j}-\cristoffe{r,j}{j,}\cristoffe{i,}{i,r}\nonumber\\
&=\partial^k \cristoffe{k,j}{j,}-\partial^k \cristoffe{i,k}{i,}+\cristoffm{s,t}{s,}\cristoffe{i,}{i,t}-\cristoftm{jk}{i}\cristofse{ik}{j}
\label{Form_2_R_ij^ij}
\end{align}
From the definition of $\overset{\:e}{F}$ (see  \eqref{def_F}) , we have
\begin{equation}
\label{d^k T_ik^i}
\partial^k T_{ik}^i=\overset{\:e}{F}+\cristoffe{i,s}{i,}T_j^{js}
\end{equation}
Taking into account the relations $\cristoffm{k,i}{i,}=T_{ik}^i+\cristoffe{k,i}{i,},\quad \cristoffm{i,k}{i,}=T_{ik}^i+\cristoffe{i,k}{i,}$ , we find that 
\begin{equation}
\label{A_ver1}
A=8(\partial^k \cristoffe{k,i}{i,}-\partial^k \cristoffe{i,k}{i,})-9\partial^k T_{ik}^i
\end{equation}
If we substitute the right-hand sides of \eqref{d^k crist{kj}{j}-d^k cristoffe{ik}{i}} and \eqref{d^k T_ik^i} into the above equation, we have 
\begin{align}
A&=8\overset{\;e}{R}\hspace{-10pt}\begin{smallmatrix}{\ \ \ \ ij}\\{ij}\end{smallmatrix}-8(\cristoffe{s,t}{s,}+\cristoffm{t,s}{s,})\cristoffe{i,}{i,t}+8\cristoftm{ij}{k}\cristofse{ik}{j}+8\cristoffe{r,j}{j,}\cristoffe{i,}{i,r}-9\overset{\:e}{F}-9\cristoffe{i,s}{i,}T_j^{js}\nonumber\\
&=8\overset{\;e}{R}\hspace{-10pt}\begin{smallmatrix}{\ \ \ \ ij}\\{ij}\end{smallmatrix}-8(\cristoffe{s,t}{s,}+\cristoffe{t,s}{s,})\cristoffe{i,}{i,t}+8\cristofte{ij}{k}\cristofse{ik}{j}\nonumber\\
&\quad+8\cristoffe{r,j}{j,}\cristoffe{i,}{i,r}-9\overset{\:e}{F}
-8\cristoffe{i,}{i,t}T_{ts}^s+8\cristofse{ij}{k}T^{ijk}-9\cristoffe{i,s}{i,}T_j^{js}\\
&=8\overset{\;e}{R}\hspace{-10pt}\begin{smallmatrix}{\ \ \ \ ij}\\{ij}\end{smallmatrix}-8\cristoffe{s,t}{s,}\cristoffe{i,}{i,t}+8\cristofte{ij}{k}\cristofse{ik}{j}-9\overset{\:e}{F}-17\cristoffe{i,}{i,t}T_{ts}^s+8\cristofse{ij}{k}T^{ijk}
\label{A_ver2}\end{align}
On the other hand 
\begin{align*}
B2&=-4\cristofte{ij}{k}(\cristofse{ij}{k}+T_{ijk})-18\cristofte{ij}{k}(\cristofse{ik}{j}+T_{ijk})
-8\cristoffe{s,i}{i,}(\cristoffe{j,}{j,s}+T_j^{js})\\
&\quad+20\cristoffe{i,s}{i,}(\cristoffe{j,}{j,s}+T_j^{js})
-13\cristoffe{s,t}{s,}(\cristoffe{i,}{t,i}+T_i^{ti})\\
&=-4\cristofte{ij}{k}\cristofse{ij}{k}-18\cristofte{ij}{k}\cristofse{ik}{j}-21\cristoffe{s,i}{i,}\cristoffe{j,}{j,s}+20\cristoffe{i,s}{i,}\cristoffe{j,}{j,s}\\
&\quad-22\cristofte{ij}{k}T_{ijk}-8\cristoffe{s,i}{i,}T_j^{js}+7\cristoffe{i,s}{i,}T_j^{js}
\end{align*}
\begin{align*}
B3&=9(\cristofte{ij}{k}+T^{ijk})(\cristofse{ij}{k}+T_{ijk})+4(\cristofte{ij}{k}+T^{ijk})(\cristofse{ik}{j}+T_{ijk})\\
&\quad+8(\cristoffe{j,i}{i,}+T_{ij}^i)(\cristoffe{s,}{s,j}+T_s^{sj})
+(\cristoffe{i,s}{i,}+T_{is}^i)(\cristoffe{j,}{j,s}+T_j^{js})-12\cristofse{ij}{k}T^{ijk}\\
&=9\cristofte{ij}{k}\cristofse{ij}{k}+4\cristofte{ij}{k}\cristofse{ik}{j}+8\cristoffe{j,i}{i,}\cristoffe{s,}{s,j}+\cristoffe{i,s}{i,}\cristoffe{j,}{j,s}\\
&\quad+14\cristofse{ij}{k}T^{ijk}+10\cristoffe{i,s}{i,}T_j^{js}+8\cristoffe{j,i}{i,}T_s^{sj}+13T_{ijk}T^{ijk}+9T_{is}^i T_j^{js}
\end{align*}
\begin{align*}
B&\triangleq B1+B2+B3\\
&=-8\cristofte{ij}{k}\cristofse{ik}{j}+8\cristoffe{i,s}{i,}\cristoffe{j,}{j,s}-8\cristofte{ij}{k}T_{ijk}+17\cristoffe{i,s}{i,}T_j^{js}+13T_{ijk}T^{ijk}+9T_{is}^i T_j^{js}
\end{align*}
Consequently 
\begin{equation}
A+B=8\overset{\;e}{R}\hspace{-10pt}\begin{smallmatrix}{\ \ \ \ ij}\\{ij}\end{smallmatrix}-9\overset{\:e}{F}+13T^{ijk}T_{ijk}+9T_{is}^iT_j^{js} \label{A+B}
\end{equation}
If we substitute the right-hand side of \eqref{A+B} into \eqref{near_final_form_ED}, we get the equation
\eqref{expan_ED_final}.
\\
\\
-\textit{ Proof of the parameter invariance of $\overset{\:\alpha}{F}$}-\\
We prove that 
$$
\overset{\:\alpha}{F}\triangleq \partial^k T_{ik}^i-\cristoffa{i,j}{i,}T_k^{kj}
$$ 
is parameter invariant. $\overset{\:e}{F}$ of \eqref{def_F} is the special case for $\alpha=1$. Consider the change of coordinates
$$
(i, j, k, \ldots) \leftrightarrow (\alpha, \beta, \gamma, \ldots)
$$
Let
$$
B_\alpha^i \triangleq \frac{\partial i}{\partial \alpha},  \quad B_i^\alpha\triangleq \frac{\partial \alpha}{\partial i}, \quad 
B_{ik}^\beta \triangleq \frac{\partial^2 \beta}{\partial i \: \partial k}.
$$
Then we have
\begin{align*}
&\partial^k T_{ik}^i=B_\beta^k \partial^\beta \Bigl( T_{\tau \gamma}^\alpha B_\alpha^i B^\tau_i B_k^\gamma\Bigr)=B_\beta^k \partial^\beta \Bigl( T_{\alpha \gamma}^\alpha  B_k^\gamma\Bigr)\\
&=B_\beta^k B_k^\gamma\Bigl(\partial^\beta T_{\alpha \gamma}^\alpha \Bigr)+B_\beta^k T_{\alpha \gamma}^\alpha \Bigl(\partial^\beta B_k^\gamma \Bigr)=\partial^\beta T_{\alpha \beta}^\alpha +\Bigl(\partial^k B_k^\gamma\Bigr) T_{\alpha \gamma}^\alpha,
\end{align*}
while using \eqref{exchange_crisotffel}, we have
\begin{align*}
\cristoffa{i,j}{i,}&=\cristofsa{ik}{j}g^{ik}=\Bigl(\cristofsa{\alpha\gamma}{\beta}B_i^\alpha B_k^\gamma B_j^\beta+B_j^\alpha B^\beta_{ik} g_{\alpha \beta}\Bigr) g^{\tau \xi} B_\tau^i B^k_\xi \\
&=\cristoffa{\tau, \beta}{\tau,}B_j^\beta+B_j^\alpha g_{\alpha \beta}g^{\tau \xi}(\partial_\xi B_i^\beta )B_\tau^i\\
&=\cristoffa{\tau, \beta}{\tau,}B_j^\beta+B_j^\alpha g_{\alpha \beta}(\partial^i B_i^\beta ).
\end{align*}
Furthermore, 
\begin{align*}
\cristoffa{i,j}{i,}T_k^{k,j}&=\Bigl(\cristoffa{\tau, \beta}{\tau,}B_j^\beta+B_j^\alpha g_{\alpha \beta}(\partial^i B_i^\beta )\Bigr) T_\rho^{\rho \phi} B_\phi^j \\
&=\cristoffa{\tau, \beta}{\tau,}T_\rho^{\rho \beta}+(\partial^i B_i^\beta) T_{\rho \beta}^\rho.
\end{align*}
Consequently 
$$
\partial^k T_{ik}^i-\cristoffa{i,j}{i,}T_k^{k,j}=\partial^\beta T_{\alpha \beta}^\alpha-\cristoffa{\tau,\beta}{\tau,}T_\rho^{\rho \beta}.
$$
\\
\\
-\textit{ Proof of the positivity of $T_{ijk}T^{ijk}$ and $T_{i}^{is}T_{js}^j$}-\\
 Note that 
$$
T_{ijk}T^{ijk}=T_{ijk}T_{stu}g^{is}g^{jt}g^{ku}.
$$
For a general $p$-dimensional real symmetric matrices $A(\geq 0)$ and $B$, 
\begin{equation}
\textrm{tr}(BABA)\geq 0,
\label{positiv_trBABA}
\end{equation}
since 
$$
\textrm{tr}(BABA)=\textrm{tr}\Bigl(A^{\frac{1}{2}}BA^{\frac{1}{2}}A^{\frac{1}{2}}BA^{\frac{1}{2}}\Bigr)=\textrm{tr}\Bigl(A^{\frac{1}{2}}BA^{\frac{1}{2}}\Bigr)^2=\sum_{1\leq i \leq p}\lambda_i^2\geq 0, 
$$
where $\lambda_i,\ i=1,\ldots, p$ are  eigenvalues of $A^{\frac{1}{2}}BA^{\frac{1}{2}}$.

Now let 
$$
A\triangleq (a_{ku}), \quad a_{ku}\triangleq T_{ijk}T^{stu}g^{is}g^{jt}.
$$
From \eqref{positiv_trBABA}, we notice $A$ is nonnegative, because for any vector $(x^1, \ldots, x^p)$, 
$$
a_{ku}x^k x^u=(T_{ijk}x^k)(T_{stu}x^u)g^{is}g^{jt}=\textrm{ tr}(CG^{-1}CG^{-1}),
$$
where $C$ is a symmetric matrix defined by $(C)_{ij}\triangleq T_{ijk}x^k$. Since $A\geq 0$ and $G>0$, 
$G^{-1/2}AG^{-1/2}\geq 0.$ Consequently 
$$
T_{ijk}T^{ijk}=\textrm{ tr}(AG^{-1})=\textrm{ tr}(G^{-1/2}AG^{-1/2})\geq 0.
$$
If we define $A=(a_{ku})$ as 
$$
a_{ku}\triangleq T_{ijk}T_{stu}g^{ij}g^{st},
$$
we notice $A$ is nonnegative in a similar way to the above argument. Therefore we have
$$
T_{is}^i T_j^{js}=\textrm{ tr}(AG^{-1})=\textrm{ tr}(G^{-1/2}AG^{-1/2})\geq 0.
$$
\subsection{Other Proofs}
\label{other_proofs}
-\textit{ Proof of \eqref{convert_eF_mF}}-
\begin{align*}
\overset{\:e}{F}(\theta)&\triangleq\sum_{i,k,s}g^{ks}(\theta)\partial_s T_{ik}^i(\theta)-\sum_{i,j,s,t}g^{ti}(\theta)\cristofse{it}{s}(\theta) T_j^{js}(\theta) \\
&=\sum_{i,k,s}g^{ks}(\theta)\partial_s T_{ik}^i(\theta)-\sum_{i,j,s,t}g^{ti}(\theta)(\cristofsm{it}{s}(\theta) -T_{its}(\theta))T_j^{js}(\theta)\\
&=\overset{\:m}{F}(\theta)+\sum_{i,j,s}T^{i}_{is}(\theta)T_{j}^{js}(\theta)
\end{align*}
-\textit{ Proof of \eqref{T_ijkT^ijk_multinomi}, \eqref{T_is^iT^js_j_multinomi}, \eqref{mF_multinomi},
\eqref{expan_ED_multinom}}-\\
Let $\mu$ be the  counting measure on ${\mathfrak X}=\{x_0, x_1,\ldots, x_p\}$, then the density function is given by
$$
f(x; m)=\prod_{i=0}^p m_i^{I_i(x)},
$$
where $I_i(x)$ is the  indicator function for the point $x_i$.   We have
\begin{align*}
\log f(x; m) &=\sum_{i=0}^p I_i(x) \log m_i \\
&=\sum_{i=1}^p I_i(x) \log m_i + I_0(x) \log m_0\\
&=\sum_{i=1}^p I_i(x) \log m_i + (1-\sum_{i=1}^p I_i(x)) \log m_0\\
&=\sum_{i=1}^p I_i(x) \log(m_i/m_0)+\log m_0 \\
&=\sum_{i=1}^p I_i(x)\theta^i -\psi(\theta),
\end{align*}
where $\theta\triangleq (\theta^1,\ldots,\theta^p) $ with $\theta^i \triangleq  \log(m_i/m_0),\ i=0,\ldots,p$, and $\psi(\theta)\triangleq \log (\sum_{i=0}^p e^{\theta^i})=-\log m_0$. $\theta$ is an $e$-affine coordinate system, while $m=(m_1,\ldots, m_p)$ is a $m$-affine coordinate system, since $E_m[I_i(X)]=m_i$. Using the basic facts on an exponential family (see, e.g. \cite{Amari2} or \cite{Amari&Nagaoka}.) we can calculate the geometrical properties as follows;
\begin{align*}
&T_{\theta^i \theta^j \theta^k}\\
&=\frac{\partial }{\partial \theta^i}\frac{\partial }{\partial \theta^j}\frac{\partial }{\partial \theta^k} \psi(\theta)\\
&=\frac{\partial }{\partial \theta^i}\frac{\partial }{\partial \theta^j}\frac{e^{\theta^k}}{\sum_{s=0}^p e^{\theta^s}}\\
&=\frac{\partial }{\partial \theta^i}\biggl(\delta_{kj}\frac{e^{\theta^k}}{\sum_{s=0}^p e^{\theta^s}}-\frac{e^{\theta^k}e^{\theta^j}}{(\sum_{s=0}^k e^{\theta^s})^2}\biggr)\\
&=\delta_{kj}\delta_{ij}\frac{e^{\theta^k}}{\sum_{s=0}^p e^{\theta^s}}-\delta_{kj}\frac{e^{\theta^k}e^{\theta^i}}{(\sum_{s=0}^k e^{\theta^s})^2}
-\delta_{ki}\frac{e^{\theta^k}e^{\theta^j}}{(\sum_{s=0}^k e^{\theta^s})^2}
-\delta_{ij}\frac{e^{\theta^k}e^{\theta^i}}{(\sum_{s=0}^k e^{\theta^s})^2}
+2\frac{e^{\theta^k}e^{\theta^i}e^{\theta^j}}{(\sum_{s=0}^k e^{\theta^s})^3}\\
&=\delta_{ij}\delta_{kj}m_k-\delta_{kj}m_k m_i-\delta_{ki} m_k m_j-\delta_{ij}m_i m_k+2m_i m_k m_j
\end{align*}
\begin{align*}
g_{\theta^i \theta^j}
&=\frac{\partial }{\partial \theta^i}\frac{\partial }{\partial \theta^j}\psi(\theta)\\
&=\delta_{ij}m_i-m_i m_j
\end{align*}
, which means 
$$
(g_{\theta^i \theta^j})=\textrm{ diag}(m_1,\ldots,m_p)-{\bm{m}}{\bm{m}}^t,\qquad {\bm{m}}=(m_1,\ldots,m_p)^t.
$$
Since 
\begin{align*}
\frac{\partial \log f(x; m)}{\partial m_i}&=\frac{1}{m_0}\bigl(\sum_{s=1}^p I_s(x)\bigr)+\frac{1}{m_i}I_i(x)-\frac{1}{m_0},\\
-\frac{\partial^2 \log f(x;m)}{\partial m_j \partial m_i}&=-\frac{1}{m_0^2}\bigl(\sum_{s=1}^p I_s(x)\bigr)+\delta_{ij}\frac{1}{m_i^2}I_i(x)+\frac{1}{m_0^2},
\end{align*}
\begin{align*}
g^{\theta^i \theta^j}&=g_{m_i m_j}\\
&=E_m\biggl[-\frac{\partial^2 \log f(x;m)}{\partial m_j \partial m_i}\biggr]\\
&=(m_0-1)\frac{1}{m_0^2}+\frac{1}{m_i}\delta_{ij}+\frac{1}{m_0^2}\\
&=\frac{1}{m_0}+\delta_{ij}\frac{1}{m_i},
\end{align*}
which means
$$
(g^{\theta^i \theta^j})=m_0^{-1}\bm{1}{\bm{1}}^t+\diag(m_1^{-1},\ldots,m_p^{-1}), \qquad \bm{1}\triangleq (1,\ldots,1)^t.
$$
\begin{align*}
&\sum_{i=1}^p T_{\theta^i \theta^j \theta^k}g^{\theta^i \theta^s}\\
&=m_0^{-1}\biggl(\delta_{kj}m_k-\delta_{kj}m_k\sum_{i=1}^p m_i-m_j m_k -m_k m_j+2m_j m_k \sum_{i=1}^p m_i\biggr)\\
&\quad +\delta_{sk}\delta_{kj} m_k m_s^{-1}-\delta_{kj} m_k-m_k m_j m_s^{-1}\delta_{ks}-m_k\delta_{sj}+2m_j m_k,\\
&\sum_{1\leq i, j \leq p}T_{\theta^i \theta^j \theta^k}g^{\theta^i \theta^s}g^{\theta^j \theta^t}\\
&=m_0^{-2}\bigl\{m_k-m_k(1-m_0)-2m_k(1-m_0)+2m_k(1-m_0)^2\bigr\}\\
&\quad +m_0^{-1}\bigl\{\delta_{sk} m_k m_s^{-1}-m_k-m_k(1-m_0)m_s^{-1}\delta_{ks}-m_k+2m_k(1-m_0)\bigr\}\\
&\quad +m_0^{-1}\bigl\{\delta_{kt}m_k m_t^{-1}-\delta_{kt}m_k m_t^{-1}(1-m_0)-m_k-m_k+2m_k(1-m_0)\bigr\}\\
&\quad +\delta_{sk}\delta_{kt} m_k m_s^{-1} m_t^{-1}-\delta_{kt}m_k m_t^{-1}-\delta_{ks}m_k m_s^{-1}-\delta_{st}m_k m_t^{-1}+2m_k
\end{align*}
\begin{align*}
&T^{\theta^s \theta^t \theta^u}\\
&=\sum_{1\leq i,j,k \leq p} T_{\theta^i \theta^j \theta^k} g^{\theta^i \theta^s} g^{\theta^j \theta^t} g^{\theta^k \theta^u}\\
&=m_0^{-3}\bigl\{(1-m_0)-3(1-m_0)^2+2(1-m_0)^3\bigr\}\\
&\quad +m_0^{-2}\bigl\{1-3(1-m_0)+2(1-m_0)^2\bigr\}+m_0^{-2}\bigl\{1-(1-m_0)-2(1-m_0)m_0\bigr\}\\
&\quad +m_0^{-1}\bigl\{\delta_{st}m_s^{-1}-1-1-(1-m_0)\delta_{st}m_t^{-1}+2(1-m_0)\bigr\}\\
&\quad +m_0^{-2}\bigl\{1-3(1-m_0)+2(1-m_0)^2\bigr\}\\
&\quad +m_0^{-1}\bigl\{\delta_{su}m_s^{-1}-\delta_{su}m_s^{-1}(1-m_0)-2m_0\bigr\}+m_0^{-1}\bigl\{\delta_{ut}m_u^{-1}-\delta_{ut}m_u^{-1}(1-m_0)-2m_0\bigr\}\\
&\quad +\delta_{st}\delta_{tu}m_s^{-1}m_t^{-1}-\delta_{ut}m_t^{-1}+\delta_{su}m_u^{-1}-\delta_{st}m_t^{-1}+2\\
&=m_0^{-3}\bigl\{2(1-m_0)^3-3(1-m_0)^2+(1-m_0)\bigr\}\\
&\quad +m_0^{-2}(6(1-m_0)^2-9(1-m_0)+3)\\
&\quad +(\delta_{st}m_s^{-1}+\delta_{us}m_u^{-1}+\delta_{ut}m_t^{-1}-6)\\
&\quad +\delta_{st}m_s^{-1}\delta_{us}m_u^{-1}-\bigl\{\delta_{tu}m_t^{-1}+\delta_{us}m_u^{-1}+\delta_{st}m_s^{-1}\bigr\}+2\\
&=m_0^{-3}(-2m_0^3+3m_0^2-m_0)\\
&\quad+m_0^{-2}(6m_0^2-3m_0)+\delta_{st}m_s^{-1}\delta_{ut}m_t^{-1}-4\\
&=-m_0^{-2}+\delta_{st}m_s^{-1}\delta_{ut}m_t^{-1}.
\end{align*}
Since 
\begin{align*}
T_{\theta^s \theta^t \theta^u}&=\delta_{st}\delta_{tu} m_s-\delta_{tu}m_s m_u-\delta_{ts}m_t m_u-\delta_{su}m_s m_t +2m_s m_t m_u\\
&=m_s m_t m_u \bigl\{\delta_{st}m_s^{-1} \delta_{us}m_u^{-1}-(\delta_{st}m_s^{-1}+\delta_{us}m_u^{-1}+\delta_{tu} m_t^{-1})+2\bigr\},
\end{align*}
the following relations hold;
\begin{align*}
\sum_{1\leq s,t,u \leq p} T_{\theta^s \theta^t \theta^u}&=(1-m_0)-3(1-m_0)^2+2(1-m_0)^3,\\
\sum_{1\leq s \leq p} T_{\theta^s \theta^t \theta^u} \delta_{st} m_s^{-1}&=\delta_{tu}-\delta_{tu}m_u-m_u-\delta_{tu}m_t+2m_t m_u,\\
\sum_{1\leq s,u, t \leq p}T_{\theta^s \theta^t \theta^u}\delta_{st}m_s^{-1}\delta_{ut}m_t^{-1}&=\sum_{1\leq t \leq p}m_t^{-1}-3p+2(1-m_0).
\end{align*}
We now have the relation \eqref{T_ijkT^ijk_multinomi} as follows;
\begin{align*}
&T^{\theta^s \theta^t \theta^u}T_{\theta^s \theta^t \theta^u}\\
&=-m_0^{-2}\{2(1-m_0)^3-3(1-m_0)^2+(1-m_0)\}+\sum_{1\leq t \leq p}m_t^{-1}-3p+2(1-m_0)\\
&=-(-2m_0+3-m_0^{-1})+\sum_{1\leq t \leq p}m_t^{-1}-3p+2-2m_0\\
&=-3p-1+\sum_{t=0}^pm_t^{-1}.
\end{align*}
Because
$$
T_{m_s m_t m_u}=T^{\theta^s \theta^t \theta^u}=-m_0^{-2}+\delta_{st}m_s^{-1}\delta_{ut}m_t^{-1},\qquad g^{m_s m_u}=g_{\theta^s \theta^u}=\delta_{su}m_s-m_s m_u,
$$
we get
\begin{align*}
\sum_{1\leq u \leq p}T_{m_s m_t m_u}g^{m_s m_u}&=-m_0^{-2}(m_s-(1-m_0)m_s)+\delta_{st}m_t^{-1}-\delta_{st}\\
&=-m_0^{-1}m_s+\delta_{st}(m_t^{-1}-1)\\
T_{m_s m_t}^{m_s}=\sum_{1\leq s,u \leq p}T_{m_s m_t m_u}g^{m_s m_u}&=\sum_{1\leq s \leq p}(-m_0^{-1}m_s+\delta_{st}(m_t^{-1}-1))\\
&=-m_0^{-1}(1-m_0)+(m_t^{-1}-1)\\
&=-\Bigl(\sum_{i=0}^p e^{\theta^i}\Bigr)+e^{-\theta^t}\Bigl(\sum_{i=0}^p e^{\theta^i}\Bigr)\\
\overset{\:m}{F}=\sum_{1\leq t \leq p}\partial^{m_t}T_{m_s m_t}^{m_s}&=\sum_{1\leq t \leq p}
\partial_{\theta^t}T_{m_s m_t}^{m_s}\\
&=\sum_{1\leq t \leq p}\Bigl\{-e^{\theta^t}-e^{-\theta^t}\Bigl( \sum_{i=0}^p e^{\theta^i}\Bigr) +e^{-\theta^t}e^{\theta^t}\Bigr\}\\
&=-\sum_{1\leq t \leq p}e^{\theta^t}-\Bigl(\sum_{1\leq t \leq p}e^{-\theta^t}\Bigr)\Bigl(\sum_{i=0}^p e^{\theta^i}\Bigr)+p\\
&=-\Bigl(\sum_{t=0}^p e^{\theta^t}-1\Bigr)-\Bigl(\sum_{1\leq t \leq p}m_0 m_t^{-1}\Bigr)\Bigl(\sum_{i=0}^p e^{\theta^i}\Bigr)+p\\
&=-m_0^{-1}+1-\Bigl(\sum_{1 \leq t \leq p}m_t^{-1}\Bigr)+p\\
&=-\Bigl(\sum_{t=0}^pm_t^{-1}\Bigr)+p+1 \quad \text{( \eqref{mF_multinomi} holds )}.
\end{align*}
We also have
\begin{align*}
&T_{m_s m_t}^{m_s} T_{m_v}^{m_v m_t}\\
&=T_{m_s m_t}^{m_s} T_{m_v m_u}^{m_v} g^{m_u m_t}\\
&=\sum_{1 \leq u, t \leq p}(-m_0^{-1}+m_t^{-1})(-m_0^{-1}+m_u^{-1})(\delta_{ut}m_u-m_um_t)\\
&=\sum_{1 \leq u, t \leq p}(m_0^{-2}-m_0^{-1}m_t^{-1}-m_0^{-1}m_u^{-1}+m_u^{-1}m_t^{-1})(\delta_{ut}m_u-m_u m_t)\\
&=\sum_{1\leq t \leq p}(m_0^{-2}m_t-m_0^{-1}-m_0^{-1}+m_t^{-1})-\sum_{1\leq t \leq p}m_0^{-2}m_t(1-m_0)+\sum_{1\leq t \leq p}m_0^{-1}(1-m_0)\\
&\qquad+\sum_{1\leq t \leq p}m_0^{-1}m_t p- \sum_{1\leq t \leq p}p\\
&=m_0^{-2}(1-m_0)-2m_0^{-1}p+\sum_{1\leq t \leq p}m_t^{-1}-(1-m_0)^2m_0^{-2}+(1-m_0)m_0^{-1}p\\
&\qquad+m_0^{-1}p(1-m_0)-p^2\\
&=m_0^{-2}-m_0^{-1}-2pm_0^{-1}+\sum_{1\leq t \leq p}m_t^{-1}-1+2m_0^{-1}-m_0^{-2}+m_0^{-1}p-p\\
&\qquad+m_0^{-1}p-p-p^2\\
&=\sum_{t=0}^p m_t^{-1}-(p+1)^2 \quad \text{(\eqref{T_is^iT^js_j_multinomi} holds)}.
\end{align*}
Finally we find the concrete form of $\overset{\alpha}{E\!D}$ as follows;
\begin{align*}
&\overset{\alpha}{E\!D} (m)\\
&=\frac{p}{2n}\\
&\quad+\frac{1}{24n^{2}}\biggl[(\alpha')^2\{3\overset{\:m}{F}+3T^{ijk}T_{ijk}+3T^i_{is}T^{js}_j+3p^2+6p\}\\
&\qquad \qquad+\alpha'\{3\overset{\:m}{F}-5T^{ijk}T_{ijk}-3T^i_{is}T^{js}_j-3p^2-6p\}\\
&\qquad\qquad+T^{ijk}T_{ijk}-9\overset{\:m}{F}\biggr]+o(n^{-2})\\
&=\frac{p}{2n}\\
&\quad+\frac{1}{24n^{2}}\biggl[(\alpha')^2\{3(-M+p+1)+3(-3p+M-1)+3(M-p^2-2p-1)+3p^2+6p\}\\
&\qquad \qquad+\alpha'\{3(-M+p+1)-5(-3p+M-1)-3(M-p^2-2p-1)-3p^2-6p\}\\
&\qquad\qquad+(-3p+M-1)-9(-M+p+1)\biggr]+o(n^{-2})\\
&=\frac{p}{2n}\\
&\quad+\frac{1}{24n^{2}}\biggl[(\alpha')^2\{-3M+3p+3-9p+3M-3+3M-3p^2-6p-3+3p^2+6p\}\\
&\qquad \qquad+\alpha'\{-3M+3p+3+15p-5M+5-3M+3p^2+6p+3-3p^2-6p\}\\
&\qquad\qquad+10M-12p-10\biggr]+o(n^{-2})\\
&=\frac{p}{2n}\\
&\quad+\frac{1}{24n^{2}}\biggl[(\alpha')^2(3M-6p-3)+\alpha'(-11M+18p+11)+10M-12p-10\biggr]+o(n^{-2}).\\
\end{align*}

As we mentioned in the main text, we can derive straightforwardly the asymptotic expansion of $\overset{\alpha}{E\!D} (m)$ not using the knowledge of information geometry because the risk of the m.l.e. could be expressed in a simple form for the multinomial distribution. We will demonstrate this.

Let $x_j, j=1,\ldots,n$ be the i.i.d. samples. The maximum likelihood estimator for $m$, $\bar{m}\triangleq(\bar{m}_1,\ldots, \bar{m}_p)$, is given by 
$$
\bar{m}_i\triangleq \sum_{j=1}^n I_i(x_j)/n,\quad i=0,\ldots,p.
$$

First suppose that $\alpha \ne \pm 1$.
The function $(1+x/a)^{(1-\alpha)/2}$ could be  expanded around $x=0$ as follows;
\begin{align*}
\Bigl(1+\frac{x}{a}\Bigr)^{(1-\alpha)/2}&=1+\frac{1-\alpha}{2}\frac{x}{a}-\frac{(1-\alpha^2)}{8}\frac{x^2}{a^2}+\frac{(1-\alpha^2)(3+\alpha)}{48}\frac{x^3}{a^3}-\frac{(1-\alpha^2)(3+\alpha)(5+\alpha)}{384}\frac{x^4}{a^4}\\
&\quad+\frac{(1-\alpha^2)(3+\alpha)(5+\alpha)(7+\alpha)}{3840}\frac{x^5}{a^5}\Bigl(1+\frac{\xi_a(x)}{a}\Bigr)^{-(9+\alpha)/2}\quad(0<\xi_a(x)<x).
\end{align*}
If we substitute $x$ and $a$ respectively with $\bar{m}_i-m_i$ and $m_i$, we have
\begin{align*}
\Bigl(\frac{\bar{m}_i}{m_i}\Bigr)^{(1-\alpha)/2}&=1+\frac{1-\alpha}{2}\frac{\bar{m}_i-m_i}{m_i}-\frac{(1-\alpha^2)}{8}\frac{(\bar{m}_i-m_i)^2}{m_i^2}\\
&\quad+\frac{(1-\alpha^2)(3+\alpha)}{48}\frac{(\bar{m}_i-m_i)^3}{m_i^3}-\frac{(1-\alpha^2)(3+\alpha)(5+\alpha)}{384}\frac{(\bar{m}_i-m_i)^4}{m_i^4}\\
&\quad+\frac{(1-\alpha^2)(3+\alpha)(5+\alpha)(7+\alpha)}{3840}\frac{(\bar{m}_i-m_i)^5}{m_i^5}\Bigl(1+\frac{\xi_{m_i}(\bar{m}_i-m_i)}{m_i}\Bigr)^{-(9+\alpha)/2}.
\end{align*}
If we use the basic relations for an exponential family (see, e.g., p.121 of Amari \cite{Amari2}), we have the following equations; 
\begin{align*}
E_m[\bar{m}_i-m_i]&=0, \\
E_m[(\bar{m}_i-m_i)^2]&=\frac{1}{n}\frac{\partial^2}{\partial^2 \theta^i}\psi(\theta)\\
&=\frac{1}{n}\frac{\partial}{\partial \theta^i}\biggl(\frac{e^{\theta^i}}{\sum_{j=0}^p e^{\theta^j}}\biggr)\\
&=\frac{1}{n}\biggl(\frac{e^{\theta^i}}{\sum_{j=0}^p e^{\theta^j}}-\frac{e^{2\theta^i}}{\bigl(\sum_{j=0}^p e^{\theta^j}\bigr)^2}\biggr)\\
&=\frac{1}{n}(m_i-m_i^2),\\
E_m[(\bar{m}_i-m_i)^3]&=\frac{1}{n^2}\frac{\partial^3}{\partial^3 \theta^i}\psi(\theta)\\
&=\frac{1}{n^2}\biggl(\frac{e^{\theta^i}}{\sum_{j=0}^p e^{\theta^j}}-\frac{e^{2\theta^i}}{\bigl(\sum_{j=0}^p e^{\theta^j}\bigr)^2}-\frac{2e^{2\theta^i}}{\bigl(\sum_{j=0}^p e^{\theta^j}\bigr)^2})+\frac{2e^{3\theta^i}}{\bigl(\sum_{j=0}^p e^{\theta^j}\bigr)^3}\biggr)\\
&=\frac{1}{n^2}(m_i-3m_i^2+2m_i^3),\\
E_m[(\bar{m}_i-m_i)^4]&=\frac{3}{n^2}\biggl(\frac{\partial^2}{\partial^2\theta^i}\psi(\theta)\biggr)^2+o(n^{-2})\\
&=\frac{3}{n^2}(m_i-m_i^2)^2+o(n^{-2}).
\end{align*}
Using this result, we have
\begin{align*}
&E_m\Bigl[\sum_{i=0}^p m_i\Bigl(\frac{\bar{m}_i}{m_i}\Bigr)^{(1-\alpha)/2} \Bigr]\\
&=1-\frac{1-\alpha^2}{8}\frac{1}{n}\sum_{i=0}^p \frac{1}{m_i}(m_i-m_i^2)\\
&\qquad+\frac{1-\alpha^2}{48}(3+\alpha)\frac{1}{n^2}\sum_{i=0}^p \frac{1}{m_i^2}(m_i-3m_i^2+2m_i^3)\\
&\qquad-\frac{1}{384}(1-\alpha^2)(3+\alpha)(5+\alpha)\frac{3}{n^2}\sum_{i=0}^p\frac{1}{m_i^3}(m_i-m_i^2)^2+o(n^{-2})\\
&=1-(1-\alpha^2)\frac{p}{8n}\\
&\qquad+(1-\alpha^2)\frac{1}{384n^2}\Bigl\{8(3+\alpha)\bigl(M-3(p+1)+2\bigr)-3(3+\alpha)(5+\alpha)\bigl(M-2(p+1)+1\bigr)\Bigr\}\\
&\qquad+o(n^{-2})\\
&=1-(1-\alpha^2)\frac{p}{8n}\\
&\qquad+(1-\alpha^2)\frac{1}{384n^2}\Bigl\{-(3+\alpha)(7+3\alpha)M+(3+\alpha)(6+6\alpha)p+(3+\alpha)(7+3\alpha))\Bigr\}\\
&\qquad+o(n^{-2})\\
&=1-(1-\alpha^2)\frac{p}{8n}\\
&\qquad+(1-\alpha^2)\frac{1}{384n^2}\Bigl\{(3+\alpha)(7+3\alpha)(1-M)+6(\alpha+3)(1+\alpha)p\Bigr\}+o(n^{-2}).
\end{align*}
From this, we have 
\begin{align}
\overset{\alpha}{E\!D} (m)&=\frac{4}{1-\alpha^2}\Bigl\{1-E_m\Bigl[\sum_{i=0}^p m_i\Bigl(\frac{\bar{m}_i}{m_i}\Bigr)^{(1-\alpha)/2} \Bigr]\Bigr\}\nonumber\\
&=\frac{p}{2n}+\frac{1}{96n^2}\Bigl\{(3+\alpha)(7+3\alpha)(M-1)-6(\alpha+3)(\alpha+1)p\Bigr\}+o(n^{-2}),\label{alte_expan_ED_multinom}
\end{align}
which is equal to \eqref{expan_ED_multinom}.

Now we consider the case $\alpha=-1$. The $-1$-divergence, i.e. Kullback-Leibler divergence between $f(x;\tilde{m})$ and $f(x; m)$ is given by
\begin{align*}
\overset{-1}{D}[\tilde{m}: m]&=E_{\tilde{m}}[\log f(x;\tilde{m})-\log f(x; m)]\\
&=E_{\tilde{m}}\Bigl[\sum_{i=0}^p I_i(x)(\log \tilde{m}_i-\log m_i)\Bigr]\\
&=\sum_{i=0}^p \tilde{m}_i\log (\tilde{m}_i/m_i).
\end{align*}
We have
\begin{align*}
\overset{-1}{D}[\bar{m}: m]&=\sum_{i=0}^p \bar{m}_i\log (\bar{m}_i/m_i)\\
&=\sum_{i=0}^p(\bar{m}_i\log \bar{m}_i-\bar{m}_i\log m_i),\\
\overset{-1}{E\!D} (m)&=E_m\Bigl[\overset{-1}{D}[\bar{m}: m]\Bigr]\\
&=\sum_{i=0}^p\Bigl(E_m[\bar{m}_i \log \bar{m}_i]-(\log m_i)E_m[\bar{m}_i]\Bigr)\\
&=\sum_{i=0}^p\Bigl(E_m[\bar{m}_i \log \bar{m}_i]-m_i (\log m_i) \Bigr).\\
\end{align*}
If we use Taylor expansion of $x\log x$ around $x_0$
\begin{align*}
x\log x &=x_0 \log x_0 +(1+\log x_0)(x-x_0)\\
&\quad+\frac{1}{2}\frac{1}{x_0}(x-x_0)^2-\frac{1}{6}\frac{1}{x_0^2}(x-x_0)^3+\frac{2}{24}\frac{1}{x_0^3}(x-x_0)^4\\
&\quad-\frac{6}{120}\frac{1}{\xi^4(x)} (x-x_0)^5,\qquad |\xi(x)-x_0| < |x-x_0|,
\end{align*}
then we have
\begin{align*}
\bar{m}_i \log\bar{m}_i&=m_i \log m_i+(1+\log m_i)(\bar{m}_i-m_i)\\
&\quad +\frac{1}{2}\frac{1}{m_i}(\bar{m}_i-m_i)^2-\frac{1}{6}\frac{1}{m_i^2}(\bar{m}_i-m_i)^3+\frac{1}{12}\frac{1}{m_i^3}(\bar{m}_i-m_i)^4\\
&\quad-\frac{1}{20}\frac{1}{\xi^4(\bar{m}_i)}(\bar{m}_i-m_i)^5, \qquad |\xi(\bar{m}_i)-m_i|<|\bar{m}_i-m_i|.
\end{align*}

Thus, we have
\begin{align*}
E_m[\bar{m}_i \log \bar{m}_i]&=m_i \log m_i +\frac{1}{2n}(1-m_i)-\frac{1}{6n^2}(m_i^{-1}-3+2m_i)\\
&\quad+\frac{1}{4n^2}(m_i-m_i^2)^2 m_i^{-3}+o(n^{-2})\\
&=m_i \log m_i+\frac{1}{2n}(1-m_i)\\
&\quad+\frac{1}{12n^2}\{(-2m_i^{-1}+6-4m_i)+3(m_i^{-1}-2+m_i)\}+o(n^{-2})\\
&=m_i \log m_i+\frac{1}{2n}(1-m_i)+\frac{1}{12n^2}(m_i^{-1}-m_i)+o(n^{-2}).
\end{align*}
Consequently we have the same result as \eqref{expan_KL_multinom};
\begin{align*}
\overset{-1}{E\!D} (m)&=\frac{1}{2n}\sum_{i=0}^p(1-m_i)+\frac{1}{12n^2}\sum_{i=0}^p(m_i^{-1}-m_i)+o(n^{-2})\\
&=\frac{p}{2n}+\frac{1}{12n^2}\sum_{i=0}^p (m_i^{-1}-m_i)+o(n^{-2})\\
&=\frac{p}{2n}+\frac{1}{12n^2}\Bigl(\sum_{i=0}^p m_i^{-1}-1 \Bigr)+o(n^{-2}).
\end{align*}
When $\alpha=1$, we can similarly derive the result. \\
-\textit{ Proof of \eqref{Normal Model T^ijk T_ijk}, \eqref{Normal Model T_ij^i T_s^sj}, 
\eqref{Normal Model F}, \eqref{expan_ED_nomal_vacova}}-\\
For brevity, we use the following notations;
\begin{align*}
&\sum_{(a,b)}\triangleq\sum_{1\leq a \leq b \leq p},\quad \sum_{(a,b)(c,d)}\triangleq\sum_{1\leq a \leq b \leq p, 1\leq c \leq d \leq p},\cdots
\end{align*}
From the density function of $f(x; \Sigma)$ of $p$-variate normal distribution $N_p(0, \Sigma)$, we have
$$
\log f(x;\Sigma)= \sum_{(i,j)} y_{ij}(x)\theta^{ij}-\psi(\Sigma),
$$
where for $1 \leq i \leq j \leq p$,
$$
y_{ij}(x)\triangleq x_i x_j, \qquad \theta^{ij}=-\frac{1}{1+\delta_{ij}}\sigma^{ij},
$$
and 
$$
\psi(\Sigma)=\frac{p}{2}\log 2\pi+\frac{1}{2}\log |\Sigma|.
$$
Since $E[y_{ij}]=\sigma_{ij}$,  $\theta^{ij}(1\leq i \leq j \leq p)$ is $e$-affine coordinate and $\sigma_{ij}(1\leq i \leq j \leq p)$ is $m$-affine coordinate. We use the notations
\begin{align*}
&\partial_{(i,j)}\triangleq \frac{\partial}{\partial \theta^{ij}}, \qquad g_{(i,j)(k,l)}\triangleq <\partial_{(i,j)}, \partial_{(k,l)}>=g_{\theta^{ij}\theta^{kl}},\\
&\partial^{(i,j)}\triangleq \frac{\partial}{\partial \sigma_{ij}}, \qquad g^{(i,j)(k,l)}\triangleq <\partial^{(i,j)}, \partial^{(k,l)}>=g_{\sigma_{ij}\sigma_{kl}}\\
\end{align*}
for $1\leq i \leq j \leq p,\ 1\leq k \leq l \leq p.$
Actually the following relations hold;
\begin{align}
&\frac{\partial \sigma_{kl}}{\partial \theta^{ij}}=g_{(i,j)(k,l)},\qquad\partial_{(i,j)}=\sum_{(k,l)}g_{(i,j)(k,l)}\partial^{(k,l)}, \label{exp_partial^(ij)}\\
&\frac{\partial \theta^{kl}}{\partial \sigma_{ij}}=g^{(i,j)(k,l)},\qquad\partial^{(i,j)}=\sum_{(k,l)}g^{(i,j)(k,l)}\partial_{(k,l)}, \label{another_exp_partial^(ij)}\\
&\sum_{(k,l)}g_{(i,j)(k,l)}g^{(k,l)(s,t)}=
\begin{cases}
1, & \text{if $(i,j)=(s,t)$,} \\
0, & \text{otherwise.}
\end{cases}\label{orthogo_g^{(ij)(kl)}_and_g_{(ij)(kl)}}
\end{align}
First the metric  of $\mathcal{P}$ is given as follows (see e.g., \cite{Skovgaard} or \cite{Eguchi1});
\begin{align}
g_{(i,j)(k,l)}&=\sigma_{ik}\sigma_{jl}+\sigma_{il}\sigma_{jk},\label{metric_e}\\
g^{(i,j)(k,l)}&=(1+\delta_{ij})^{-1}(1+\delta_{kl})^{-1}(\sigma^{ik}\sigma^{jl}+\sigma^{il}\sigma^{jk}). \label{metric_m}
\end{align}
\eqref{metric_e} is easily obtained if we use the fact 
\begin{align*}
g_{(i,j)(k,l)}&=\frac{\partial^2\ }{\partial \theta^{ij} \partial \theta^{kl}}\psi(\Sigma)\\
&=(1+\delta_{ij})(1+\delta_{kl})\frac{\partial^2\ }{\partial \sigma^{ij} \partial \sigma^{kl}}\psi(\Sigma)\\
&=-2^{-1}(1+\delta_{ij})(1+\delta_{kl})\frac{\partial^2\ }{\partial \sigma^{ij} \partial \sigma^{kl}}
\log |\Sigma^{-1}|
\end{align*}
and the following formula; For a symmetric matrix $X=(x_{ij})$ and its inverse $X^{-1}=(x^{ij})$, the following relationship holds.
\begin{align}
&\frac{\partial\ }{\partial x_{ij}}\log |X|=\frac{2}{1+\delta_{ij}}x^{ij},\label{del_log|x|}\\
&\frac{\partial x^{ij}}{\partial x_{kl}}=-(1+\delta_{kl})^{-1}(x^{ik}x^{jl}+x^{il}x^{jk}).\label{del_x^{-1}}
\end{align}
\eqref{metric_m} is proved by showing as follows that \eqref{orthogo_g^{(ij)(kl)}_and_g_{(ij)(kl)}} holds true for $g_{(i,j)(k,l)} $ and $g^{(i,j)(k,l)}$ given respectively by \eqref{metric_e} and \eqref{metric_m};
If $s<t$, then we have
\begin{align*}
\sum_{(k,l)}g_{(i,j)(k,l)}g^{(k,l)(s,t)}&=(1+\delta_{kl})^{-1}\sum_{(k,l)}(\sigma_{ik}\sigma_{jl}+\sigma_{il}\sigma_{jk})(\sigma^{sk}\sigma^{tl}+\sigma^{sl}\sigma^{tk})\\
&=2^{-1}\sum_{k,l}(\sigma_{ik}\sigma_{jl}+\sigma_{il}\sigma_{jk})(\sigma^{sk}\sigma^{tl}+\sigma^{sl}\sigma^{tk})\\
&=2^{-1}(\delta_{is}\delta_{jt}+\delta_{is}\delta_{jt}+\delta_{it}\delta_{js}+\delta_{it}\delta_{js})\\
&=
\begin{cases}
1, & \text{if $(i,j)=(s,t)$,}\\
0, & \text{otherwise,}
\end{cases}
\end{align*}
while if $s=t$, then we have 
\begin{align*}
\sum_{(k,l)}g_{(i,j)(k,l)}g^{(k,l)(s,t)}&=4^{-1}\sum_{k,l}(\sigma_{ik}\sigma_{jl}+\sigma_{il}\sigma_{jk})(\sigma^{sk}\sigma^{tl}+\sigma^{sl}\sigma^{tk})\\
&=2^{-1}(\delta_{is}\delta_{jt}+\delta_{it}\delta_{js})\\
&=
\begin{cases}
1, & \text{if $(i,j)=(s,t)$,}\\
0, & \text{otherwise.}
\end{cases}
\end{align*}
Christoffel's symbol w.r.t. $\alpha$-connection is given as follows;
$$
\cristofsa{\sigma_{ij}\sigma_{kl}}{\sigma_{st}}=-\frac{1+\alpha}{2}(1+\delta_{ij})^{-1}(1+\delta_{kl})^{-1}(1+\delta_{st})^{-1} \sum_{1-(ij)(kl)(st)}\sigma^{ab}\sigma^{cd}\sigma^{ef}, 
$$
where $\sum_{1-(ij)(kl)(st)}\sigma^{ab}\sigma^{cd}\sigma^{ef}$ means the summation over all $a, b(1\leq a \leq b \leq p), c, d(1\leq c \leq d \leq p), e, f(1\leq e \leq f \leq p)$ such  that the union of parings $\{(ab), (cd), (ef), (ij), (kl), (st)\} $ makes one ring. More specifically 
\begin{align}
\cristofsa{\sigma_{ij}\sigma_{kl}}{\sigma_{st}}&=-\frac{1+\alpha}{2}(1+\delta_{ij})^{-1}(1+\delta_{kl})^{-1}(1+\delta_{st})^{-1} \nonumber\\
&\qquad \times \{\sigma^{ik}\sigma^{js}\sigma^{lt}+\sigma^{ik}\sigma^{jt}\sigma^{ls}+\sigma^{il}\sigma^{js}\sigma^{kt}+\sigma^{il}\sigma^{jt}\sigma^{ks}
\nonumber\\
&\qquad\quad+\sigma^{is}\sigma^{jk}\sigma^{lt}+\sigma^{is}\sigma^{jl}\sigma^{kt}+\sigma^{it}\sigma^{jk}\sigma^{ls}+\sigma^{it}\sigma^{jl}\sigma^{ks}\}
\nonumber\\
&=-\frac{1+\alpha}{2}(1+\delta_{ij})^{-1}(1+\delta_{kl})^{-1}(1+\delta_{st})^{-1} \sum_{i \leftrightarrow j, k\leftrightarrow l, s\leftrightarrow t} \sigma^{ik}\sigma^{js}\sigma^{lt},\label{Chris_for_sigma}
\end{align}
where  $\sum_{i \leftrightarrow j, k\leftrightarrow l, s\leftrightarrow t}$ means the summation over all the cases where each of the three interchange $i \leftrightarrow j, k\leftrightarrow l, s\leftrightarrow t$ happens or not. 

For the coordinate $\theta^{ij}$, we have
\begin{align}
\cristofsa{\theta^{ij}\theta^{kl}}{\theta^{st}}&=\frac{1-\alpha}{2} \sum_{1-(ij)(kl)(st)}\sigma_{ab}\sigma_{cd}\sigma_{ef}\nonumber\\
&=\frac{1-\alpha}{2}\sum_{i \leftrightarrow j, k\leftrightarrow l, s\leftrightarrow t}\sigma_{ik}\sigma_{js}\sigma_{lt}.\label{Chris_for_theta}
\end{align}
We will prove \eqref{Chris_for_sigma} and \eqref{Chris_for_theta}. First we prove that
\begin{equation}
T_{\sigma_{ij}\sigma_{kl}\sigma_{st}}=(1+\delta_{ij})^{-1}(1+\delta_{kl})^{-1}(1+\delta_{st})^{-1} \sum_{i \leftrightarrow j, k\leftrightarrow l, s\leftrightarrow t} \sigma^{ik}\sigma^{js}\sigma^{lt}. \label{T_sigma}
\end{equation}
Because $\mathcal{P}=\{N(0,\Sigma)\}$ is an exponential family,  it is $m$-flat, hence $\cristofsm{\sigma_{ij}\sigma_{kl}}{\sigma_{st}}=0$.
Once \eqref{T_sigma} is proved, from the relationship
\begin{align*}
\cristofsm{\sigma_{ij}\sigma_{kl}}{\sigma_{st}}&=\cristofse{\sigma_{ij}\sigma_{kl}}{\sigma_{st}}+T_{\sigma_{ij}\sigma_{kl}\sigma_{st}},\\
\cristofsa{\sigma_{ij}\sigma_{kl}}{\sigma_{st}}&=\cristofse{\sigma_{ij}\sigma_{kl}}{\sigma_{st}}+\frac{1-\alpha}{2}T_{\sigma_{ij}\sigma_{kl}\sigma_{st}},
\end{align*}
\eqref{Chris_for_sigma} immediately follows. 

From the definition, we have
\begin{equation}
T_{\sigma_{ij}\sigma_{kl}\sigma_{st}}=E[(\partial^{(i,j)}\log f )(\partial^{(k,l)}\log f)(\partial^{(s,t)}\log f)]. \label{def_T_sigma}
\end{equation}
Using \eqref{del_log|x|} and \eqref{del_x^{-1}}, we have
\begin{align*}
\partial^{(i,j)}\log f &= \sum_{(a,b)} y_{ab}(\partial^{(i,j)}\theta^{ab})-2^{-1}\partial^{(i,j)}\log |\Sigma|\\
&=-\sum_{(a,b)}  y_{ab} (1+\delta_{ab})^{-1} (\partial^{(i,j)}\sigma^{ab})-
(1+\sigma_{ij})^{-1}\sigma^{ij}\\
&=(1+\sigma_{ij})^{-1}\sum_{(a,b)}y_{ab}(1+\delta_{ab})^{-1}(\sigma^{ai}\sigma^{bj}+\sigma^{aj}\sigma^{bi})-(1+\sigma_{ij})^{-1}\sigma^{ij}\\
&=(1+\sigma_{ij})^{-1}2^{-1}\sum_{a,b}y_{ab}(\sigma^{ai}\sigma^{bj}+\sigma^{aj}\sigma^{bi})-(1+\sigma_{ij})^{-1}\sigma^{ij}.
\end{align*}
If we substitute this result into \eqref{def_T_sigma}, then
\begin{equation}
\begin{split}
&T_{\sigma_{ij}\sigma_{kl}\sigma_{st}}\\
&=(1+\sigma_{ij})^{-1}(1+\sigma_{kl})^{-1}(1+\sigma_{st})^{-1}\\
&\qquad\times  E\Bigl[\bigl(2^{-1}\sum_{a,b}y_{ab}(\sigma^{ai}\sigma^{bj}+\sigma^{aj}\sigma^{bi})-\sigma^{ij}\bigr)\\
&\qquad\qquad\times\bigl(2^{-1}\sum_{c,d}y_{cd}(\sigma^{ck}\sigma^{dl}+\sigma^{cl}\sigma^{dk})-\sigma^{kl}\bigr)\\
&\qquad\qquad\times\bigl(2^{-1}\sum_{e,f}y_{ef}(\sigma^{es}\sigma^{ft}+\sigma^{et}\sigma^{fs})-\sigma^{st}\bigr)\Bigr]\\
&=(1+\sigma_{ij})^{-1}(1+\sigma_{kl})^{-1}(1+\sigma_{st})^{-1}\\
&\qquad\times\sum_{a,b,c,d,e,f}\\
&\qquad\quad\Bigl\{8^{-1}\sum_{a,b,c,d,e,f}E[y_{ab}y_{cd}y_{ef}](\sigma^{ai}\sigma^{bj}+\sigma^{aj}\sigma^{bi})(\sigma^{ck}\sigma^{dl}+\sigma^{cl}\sigma^{dk})(\sigma^{es}\sigma^{ft}+\sigma^{et}\sigma^{fs})\\
&\qquad\qquad-4^{-1}\sum_{a,b,c,d}E[y_{ab}y_{cd}](\sigma^{ai}\sigma^{bj}+\sigma^{aj}\sigma^{bi})(\sigma^{ck}\sigma^{dl}+\sigma^{cl}\sigma^{dk})\sigma^{st}\\
&\qquad\qquad-4^{-1}\sum_{a,b,e,f}E[y_{ab}y_{ef}](\sigma^{ai}\sigma^{bj}+\sigma^{aj}\sigma^{bi})(\sigma^{es}\sigma^{ft}+\sigma^{et}\sigma^{fs})\sigma^{kl}\\
&\qquad\qquad-4^{-1}\sum_{c,d,e,f}E[y_{cd}y_{ef}](\sigma^{ck}\sigma^{dl}+\sigma^{cl}\sigma^{dk})(\sigma^{es}\sigma^{ft}+\sigma^{et}\sigma^{fs})\sigma^{ij}\\
&\qquad\qquad+2^{-1}\sum_{a,b}E[y_{ab}](\sigma^{ai}\sigma^{bj}+\sigma^{aj}\sigma^{bi})\sigma^{kl}\sigma^{st}\\
&\qquad\qquad+2^{-1}\sum_{c,d}E[y_{cd}](\sigma^{ck}\sigma^{dl}+\sigma^{cl}\sigma^{dk})\sigma^{ij}\sigma^{st}\\
&\qquad\qquad+2^{-1}\sum_{e,f}E[y_{ef}](\sigma^{es}\sigma^{ft}+\sigma^{et}\sigma^{fs})\sigma^{ij}\sigma^{kl}\\
&\qquad\qquad-\sigma^{ij}\sigma^{kl}\sigma^{st}\Bigr\}.
\end{split}
\label{T_{sigma_ij sigma_kl sigma_st}}
\end{equation}
Since 
\begin{align*}
E[y_{ab}y_{cd}y_{ef}]&=E[x_a x_b x_c x_d x_e x_f]\\
&=\sigma_{ab}(\sigma_{cd}\sigma_{ef}+\sigma_{ce}\sigma_{df}+\sigma_{cf}\sigma_{de})\\
&\quad+\sigma_{ac}(\sigma_{bd}\sigma_{ef}+\sigma_{be}\sigma_{df}+\sigma_{bf}\sigma_{de})\\
&\quad+\sigma_{ad}
(\sigma_{bc}\sigma_{ef}+\sigma_{be}\sigma_{cf}+\sigma_{bf}\sigma_{ce})\\
&\quad+\sigma_{ae}
(\sigma_{bc}\sigma_{df}+\sigma_{bd}\sigma_{cf}+\sigma_{bf}\sigma_{cd})\\
&\quad+\sigma_{af}
(\sigma_{bc}\sigma_{de}+\sigma_{bd}\sigma_{ce}+\sigma_{be}\sigma_{cd}),
\end{align*}
we have
\begin{equation}
\begin{split}
&\sum_{a,b,c,d,e,f}E[y_{ab}y_{cd}y_{ef}](\sigma^{ai}\sigma^{bj}+\sigma^{aj}\sigma^{bi})(\sigma^{ck}\sigma^{dl}+\sigma^{cl}\sigma^{dk})(\sigma^{es}\sigma^{ft}+\sigma^{et}\sigma^{fs})\\
&=8\sum_{a,b,c,d,e,f}E[y_{ab}y_{cd}y_{ef}](\sigma^{ai}\sigma^{bj}\sigma^{ck}\sigma^{dl}\sigma^{es}\sigma^{ft})\\
&=8(\sigma^{ij}\sigma^{kl}\sigma^{st}+\sigma^{ij}\sigma^{ks}\sigma^{lt}+\sigma^{ij}\sigma^{kt}\sigma^{ls}\\
&\qquad+\sigma^{ik}\sigma^{jl}\sigma^{st}+\sigma^{ik}\sigma^{js}\sigma^{lt}+\sigma^{ik}\sigma^{jt}\sigma^{ls}\\
&\qquad+\sigma^{il}\sigma^{jk}\sigma^{st}+\sigma^{il}\sigma^{js}\sigma^{kt}+\sigma^{il}\sigma^{jt}\sigma^{ks}\\
&\qquad+\sigma^{is}\sigma^{jk}\sigma^{lt}+\sigma^{is}\sigma^{jl}\sigma^{kt}+\sigma^{is}\sigma^{jt}\sigma^{kl}\\
&\qquad+\sigma^{it}\sigma^{jk}\sigma^{ls}+\sigma^{it}\sigma^{jl}\sigma^{ks}+\sigma^{it}\sigma^{js}\sigma^{kl})\\
&=8\sum_{\text{pairing of $\{i,j,k,l,s,t\}$}}\sigma^{ab}\sigma^{cd}\sigma^{ef}, \label{E[y_ab y_cd y_ef]}
\end{split}
\end{equation}
where the summation in the last line is carried out over all possible parings $(ab), (cd), (ef)$ from $\{i, j, k, l, s, t\}$.
We also have 
\begin{equation}
\begin{split}
&\sum_{a,b,c,d}E[y_{ab}y_{cd}](\sigma^{ai}\sigma^{bj}+\sigma^{aj}\sigma^{bi})(\sigma^{ck}\sigma^{dl}+\sigma^{cl}\sigma^{dk})\sigma^{st}\\
&=4\sum_{a,b,c,d}E[y_{ab}y_{cd}]\sigma^{ai}\sigma^{bj}\sigma^{ck}\sigma^{dl}\sigma^{st}\\
&=4\sum_{a,b,c,d}(\sigma^{ab}\sigma^{cd}+\sigma^{ac}\sigma^{bd}+\sigma^{ad}\sigma^{bc})(\sigma^{ai}\sigma^{bj}\sigma^{ck}\sigma^{dl}\sigma^{st})\\
&=4(\sigma^{ij}\sigma^{kl}+\sigma^{ik}\sigma^{jl}+\sigma^{il}\sigma^{jk})\sigma^{st}. \label{E[y_ab y_cd]}
\end{split}
\end{equation}
Similarly we have
\begin{align}
&\sum_{a,b,e,f}E[y_{ab}y_{ef}](\sigma^{ai}\sigma^{bj}+\sigma^{aj}\sigma^{bi})(\sigma^{es}\sigma^{ft}+\sigma^{et}\sigma^{fs})\sigma^{kl}\nonumber\\
&=4(\sigma^{ij}\sigma^{st}+\sigma^{is}\sigma^{jt}+\sigma^{it}\sigma^{js})\sigma^{kl}, \label{E[y_ab y_ef]}\\
&\sum_{c,d,e,f}E[y_{cd}y_{ef}](\sigma^{ck}\sigma^{dl}+\sigma^{cl}\sigma^{dk})(\sigma^{es}\sigma^{ft}+\sigma^{et}\sigma^{fs})\sigma^{ij}\nonumber\\
&=4(\sigma^{kl}\sigma^{st}+\sigma^{ks}\sigma^{lt}+\sigma^{kt}\sigma^{ls})\sigma^{ij}.
\label{E[y_cd y_ef]}
\end{align}
Furthermore we have
\begin{align}
\sum_{a,b}E[y_{ab}](\sigma^{ai}\sigma^{bj}+\sigma^{aj}\sigma^{bi})\sigma^{kl}\sigma^{st}
=2\sum_{a,b}E[y_{ab}]\sigma^{ai}\sigma^{bj}\sigma^{kl}\sigma^{st}=2\sigma^{ij}\sigma^{kl}\sigma^{st},
\label{E[y_ab]}\\
\sum_{c,d}E[y_{cd}](\sigma^{ck}\sigma^{dl}+\sigma^{cl}\sigma^{dk})\sigma^{ij}\sigma^{st}
=2\sum_{c,d}E[y_{cd}]\sigma^{ck}\sigma^{dl}\sigma^{ij}\sigma^{st}=2\sigma^{ij}\sigma^{kl}\sigma^{st},
\label{E[y_cd]}\\
\sum_{e,f}E[y_{ef}](\sigma^{es}\sigma^{ft}+\sigma^{et}\sigma^{fs})\sigma^{ij}\sigma^{kl}
=2\sum_{e,f}E[y_{ef}]\sigma^{es}\sigma^{ft}\sigma^{ij}\sigma^{kl}=2\sigma^{ij}\sigma^{kl}\sigma^{st}.
\label{E[y_ef]}
\end{align}
If we substitute the result from \eqref{E[y_ab y_cd y_ef]} to \eqref{E[y_ef]} into the right-hand side of \eqref{T_{sigma_ij sigma_kl sigma_st}}, we get the equation \eqref{T_sigma}.

\eqref{Chris_for_theta} is proved as follows; Because $\theta^{ij}$ is $e$-affine coordinator, we have
$$
\cristofsa{\theta^{ij}\theta^{kl}}{\theta^{st}}=\frac{(1-\alpha)}{2}T_{\theta^{ij}\theta^{kl}\theta^{st}},
$$
while, using \eqref{del_log|x|} and \eqref{del_x^{-1}}, we have
\begin{align*}
T_{\theta^{ij}\theta^{kl}\theta^{st}}&=\partial_{(i,j)}\partial_{(k,l)}\partial_{(s,t)}\psi(\Sigma)\\
&=2^{-1}(1+\delta_{ij})(1+\delta_{kl})(1+\delta_{st})\frac{\partial\ }{\partial \sigma^{ij}}\frac{\partial\ }{\partial \sigma^{kl}}\frac{\partial\ }{\partial \sigma^{st}}\log|\Sigma^{-1}|\\
&=(1+\delta_{ij})(1+\delta_{kl})\frac{\partial\ }{\partial \sigma^{ij}}\frac{\partial\ }{\partial \sigma^{kl}}\sigma_{st}\\
&=-(1+\delta_{ij})\frac{\partial\ }{\partial \sigma^{ij}}(\sigma_{ks}\sigma_{lt}+\sigma_{kt}\sigma_{ls})\\
&=(\sigma_{ik}\sigma_{js}+\sigma_{is}\sigma_{jk})\sigma_{lt}+(\sigma_{il}\sigma_{jt}+\sigma_{it}\sigma_{jl})\sigma_{ks}\\
&\qquad+(\sigma_{ik}\sigma_{jt}+\sigma_{it}\sigma_{jk})\sigma_{ls}+(\sigma_{il}\sigma_{js}+\sigma_{is}\sigma_{jl})\sigma_{kt}.
\end{align*}
Now we are ready to prove \eqref{Normal Model T^ijk T_ijk}, \eqref{Normal Model T_ij^i T_s^sj}, \eqref{Normal Model F} and \eqref{expan_ED_nomal_vacova}.
From \eqref{T_sigma}, we have
\begin{align*}
T_{\sigma_{ij}\sigma_{kl}\sigma_{st}}&=(1+\delta_{ij})^{-1}(1+\delta_{kl})^{-1}(1+\delta_{st})^{-1}\sum_{i\leftrightarrow j, k\leftrightarrow l, s \leftrightarrow t} \sigma^{ik}\sigma^{ls}\sigma^{tj}\\
T_{\sigma_{ab}\sigma_{cd}\sigma_{ef}}&=(1+\delta_{ab})^{-1}(1+\delta_{cd})^{-1}(1+\delta_{ef})^{-1}\sum_{a\leftrightarrow b, c\leftrightarrow d, e\leftrightarrow f} \sigma^{ac}\sigma^{de}\sigma^{bf}.
\end{align*}
From \eqref{metric_e}, we have
\begin{align*}
&g^{\sigma_{ab}\sigma_{ij}}\triangleq\Bigl((g_{\sigma_{cd}\sigma_{ef}})^{-1}\Bigr)_{(ab)(ij)}=\Bigl((g^{(cd)(ef)})^{-1}\Bigr)_{(ab)(ij)}=\Bigl((g_{(cd)(ef)})\Bigr)_{(ab)(ij)}\\
&=g_{(ab)(ij)}=\sigma_{ai}\sigma_{bj}+\sigma_{aj}\sigma_{bi}.
\end{align*}
From above three equations, \eqref{Normal Model T^ijk T_ijk} is proved as follows;
\begin{align*}
&\sum_{(i,j)(k,l)(s,t)}T_{\sigma_{ij}\sigma_{kl}\sigma_{st}}T^{\sigma_{ij}\sigma_{kl}\sigma_{st}}\\
&=\sum_{(i,j)(k,l)(s,t)(a,b)(c,d)(e,f)}T_{\sigma_{ij}\sigma_{kl}\sigma_{st}}T_{\sigma_{ab}\sigma_{cd}\sigma_{ef}}
g^{\sigma_{ab}\sigma_{ij}}g^{\sigma_{cd}\sigma_{kl}}g^{\sigma_{ef}\sigma_{st}}\\
&=\sum_{(i,j)(k,l)(s,t)(a,b)(c,d)(e,f)}(1+\delta_{ij})^{-1}(1+\delta_{kl})^{-1}(1+\delta_{st})^{-1}
(1+\delta_{ab})^{-1}(1+\delta_{cd})^{-1}(1+\delta_{ef})^{-1}\\
&\hspace{30mm}\times\Bigl(\sum_{i\leftrightarrow j, k\leftrightarrow l, s \leftrightarrow t}\sigma^{ik}\sigma^{ls}\sigma^{tj}\Bigr)\Bigl(\sum_{a\leftrightarrow b, c\leftrightarrow d, e\leftrightarrow f} \sigma^{ac}\sigma^{de}\sigma^{fb}\Bigr)\\
&\hspace{30mm}\times (\sigma_{ai}\sigma_{bj}+\sigma_{aj}\sigma_{bi})(\sigma_{ck}\sigma_{dl}+\sigma_{cl}\sigma_{kd})(\sigma_{es}\sigma_{ft}+\sigma_{et}\sigma_{fs})\\
&=2^{-9}\sum_{i,j,k,l,s,t,a,b,c,d,e,f}\Bigl(\sum_{i\leftrightarrow j, k\leftrightarrow l, s \leftrightarrow t}\sigma^{ik}\sigma^{ls}\sigma^{tj}\Bigr)\Bigl(\sum_{a\leftrightarrow b, c\leftrightarrow d, e\leftrightarrow f} \sigma^{ac}\sigma^{de}\sigma^{fb}\Bigr)\\
&\hspace{30mm}\times\Bigl(\sum_{a \leftrightarrow b, i\leftrightarrow j}\sigma_{ai}\sigma_{bj}\Bigr)\Bigl(\sum_{c \leftrightarrow d, k\leftrightarrow l}\sigma_{ck}\sigma_{dl}\Bigr)\Bigl(\sum_{e \leftrightarrow f, s\leftrightarrow t}\sigma_{es}\sigma_{ft}\Bigr)\\
&=2^{-9}(512p^3+1536p^2+2048p) \\
&=p^3+3p^2+4p.
\end{align*}
We used the R program "T\_ijkT\^{}ijk.R" in Appendix \ref{r programs} for the last but one equation.

Now we will prove \eqref{Normal Model T_ij^i T_s^sj}.  Notice that
\begin{align*}
&\sum_{(i,j)}T_{\sigma_{ij}\sigma_{st}}^{\sigma_{ij}}\\
&=\sum_{(k,l)(i,j)}T_{\sigma_{ij}\sigma_{kl}\sigma_{st}}g^{\sigma_{ij}\sigma_{kl}}\\
&=\sum_{(k,l)(i,j)}(1+\sigma_{ij})^{-1}(1+\sigma_{kl})^{-1}(1+\sigma_{st})^{-1}\Bigl(\sum_{i \leftrightarrow j, k \leftrightarrow l, s \leftrightarrow t}\sigma^{ik}\sigma^{ls}\sigma^{tj}\Bigr)\\
&\qquad\times(\sigma_{ik}\sigma_{jl}+\sigma_{il}\sigma_{jk})\\
&=4^{-1}\sum_{i,j,k,l}(1+\delta_{st})^{-1}\\
&\qquad\quad \times (\sigma^{ik}\sigma^{ls}\sigma^{tj}+\sigma^{ik}\sigma^{js}\sigma^{lt}+\sigma^{il}\sigma^{js}\sigma^{kt}+\sigma^{il}\sigma^{jt}\sigma^{ks}\\
&\qquad\qquad+\sigma^{is}\sigma^{jk}\sigma^{lt}+\sigma^{is}\sigma^{jl}\sigma^{kt}+\sigma^{it}\sigma^{jk}\sigma^{ls}+\sigma^{it}\sigma^{jl}\sigma^{ks})(\sigma_{ik}\sigma_{jl}+\sigma_{il}\sigma_{jk})\\
&=4^{-1}(1+\delta_{st})^{-1}\\
&\qquad\times(p\sigma^{st}+\sigma^{st}+p\sigma^{st}+\sigma^{st}+\sigma^{st}+p\sigma^{st}+\sigma^{st}+p\sigma^{st}\\
&\qquad\qquad+\sigma^{st}+p\sigma^{st}+p\sigma^{st}+\sigma^{st}+\sigma^{st}+p\sigma^{st}+p\sigma^{st}+\sigma^{st})\\
&=4^{-1}(1+\sigma_{st})^{-1}(8p+8)\sigma^{st}.
\end{align*}
Hence we have
\begin{align*}
&\sum_{(i,j)(k,l)(s,t)}T_{\sigma_{ij}\sigma_{st}}^{\sigma_{ij}}T_{\sigma_{kl}}^{\sigma_{kl}\sigma_{st}}\\
&=\sum_{(i,j)(k,l)(s,t)(u,v)}T_{\sigma_{ij}\sigma_{st}}^{\sigma_{ij}}T_{\sigma_{uv}\sigma_{kl}}^{\sigma_{uv}}g^{\sigma_{st}\sigma_{kl}}\\
&=\sum_{(s,t)(k,l)}\{2(1+\sigma_{st})^{-1}(p+2)\sigma^{st}\}\{2(1+\sigma_{kl})^{-1}(p+2)\sigma^{kl}\}(\sigma_{sk}\sigma_{tl}+\sigma_{sl}\sigma_{kt})\\
&=(p+2)^2\sum_{s,t}\sum_{k,l}\sigma^{st}\sigma^{kl}(\sigma_{sk}\sigma_{tl}+\sigma_{sl}\sigma_{kt})\\
&=2(p+2)^2\sum_{s,t}\sigma^{st}\sigma_{st}\\
&=2(p+2)^2p\\
&=2p^3+8p^2+8p.
\end{align*}
Since $\sum_{(i,j)}\cristoffm{\sigma_{ij}\sigma_{st}}{\sigma_{ij}}=0$, $
\overset{\;\scalebox{0.6}{$m$}}{F}=\sum_{(i,j)(s,t)}\partial^{\sigma_{st}}T_{\sigma_{ij}\sigma_{st}}^{\sigma_{ij}}.
$
We also have
$$
\partial^{\sigma_{st}}\triangleq\sum_{(i,j)}g^{\sigma_{ij}\sigma_{st}}\partial_{\sigma_{ij}}
=\sum_{(i,j)}g_{\theta^{ij}\theta^{st}}\partial_{\sigma_{ij}}=\partial_{\theta^{st}}=-(1+\sigma_{st})\frac{\partial\ }{\partial \sigma^{st}}.
$$
\eqref{Normal Model F} follows from the next equations.
\begin{align*}
\sum_{(i,j)(s,t)}\partial^{\sigma_{st}}T_{\sigma_{ij}\sigma_{st}}^{\sigma_{ij}}&=-4^{-1}(8p+8)\sum_{(s,t)}\frac{\partial\sigma^{st}}{\partial\sigma^{st}}\\
&=-8^{-1}(8p+8)p(p+1)\\
&=-(p+1)^2p\\
&=-p^3-2p^2-p.
\end{align*}
Finally we gain \eqref{expan_ED_nomal_vacova} as follows;
\begin{align}
&\overset{\alpha}{E\!D} \nonumber\\
&=\frac{p(p+1)}{4n}+\frac{1}{24n^{^2}}\nonumber\\
&\times
\biggl[ 
(\alpha')^2\bigl\{3\overset{\:m}{F}+3T^{\sigma_{ij}\sigma_{kl}\sigma_{st}}T_{\sigma_{ij}\sigma_{kl}\sigma_{st}}+3T_{\sigma_{ij}\sigma_{st}}^{\sigma^{ij}}T_{\sigma_{kl}}^{\sigma_{kl}\sigma_{st}}+3p^2+6p\bigr\}\nonumber\\
&\hspace{8mm}+
\alpha'\bigl\{3\overset{\:m}{F}-5T^{\sigma_{ij}\sigma_{kl}\sigma_{st}}T_{\sigma_{ij}\sigma_{kl}\sigma_{st}}-3T_{\sigma_{ij}\sigma_{st}}^{\sigma^{ij}}T_{\sigma_{kl}}^{\sigma_{kl}\sigma_{st}}-3p^2-6p\bigr\}\nonumber\\
&\hspace{8mm}
+T^{\sigma_{ij}\sigma_{kl}\sigma_{st}}T_{\sigma_{ij}\sigma_{kl}\sigma_{st}}-9\overset{\:m}{F}\biggr]+o(n^{-2})\nonumber\\
&=\frac{p(p+1)}{4n}+\frac{1}{24n^{^2}}\nonumber\\
&\times\{(\alpha')^2(6p^3+30p^2+39p)-\alpha'(14p^3+48p^2+53p)+10p^3+21p^2+13p\}+o(n^{-2}).\nonumber
\end{align}
As we mentioned in the main text,  Kullback-Leibler divergence $\overset{-1}{D}[\Sigma_1: \Sigma_2]$ is explicitly given, hence we can derive the asymptotic expansion of $\overset{-1}{E\!D}$ in a more straightforward way. We will see this. First $\overset{-1}{D}[\Sigma_1: \Sigma_2]$ is given as follows;
\begin{align*}
&\overset{-1}{D}[\Sigma_1: \Sigma_2]\\
&=E_{\Sigma_1}[\log \bigl(f(X;\Sigma_1)/f(X;\Sigma_2)\bigr)]\\
&=E_{\Sigma_1}\Bigl[2^{-1}(X^t \Sigma_{2}^{-1} X-X^t \Sigma_1^{-1} X)+2^{-1}\log|\Sigma_2 \Sigma_1^{-1}|\Bigr]\\
&=E_{\Sigma_1}\Bigl[2^{-1}\textrm{ tr}(\Sigma_2^{-1}-\Sigma_1^{-1})X X^t+2^{-1}\log|\Sigma_2 \Sigma_1^{-1}|\Bigr]\\
&=2^{-1}\textrm{ tr}(\Sigma_2^{-1}\Sigma_1)-2^{-1}\log|\Sigma_2^{-1}\Sigma_1|-2^{-1}p.
\end{align*}
For i.i.d. sample $X_i=(X_{i1},\ldots,X_{ip})^t,\ i=1,\ldots,n$, m.l.e. is given by the sample variance-covariance matrix $\hat{\Sigma}=n^{-1} \sum_{i=1}^n X_i (X_i)^t$. Therefore we have
\begin{align}
\overset{-1}{E\!D}(\Sigma)&=E_{\Sigma}\Bigl[\overset{-1}{D}[\hat{\Sigma}: \Sigma]\Bigr]\nonumber\\
&=2^{-1}E_{\Sigma}[\textrm{ tr}\Sigma^{-1}\hat{\Sigma}-\log|\Sigma^{-1}\hat{\Sigma}|-p]\nonumber\\
&=2^{-1}\Bigl(p \log n-\sum_{i=1}^pE[\log\chi^2_{n-i+1}]\Bigr), \quad \text{ (see e.g., 3.2.15. of Muirhead \cite{Muirhead} )}\nonumber \\
&=2^{-1}\Bigl(p \log n -p \log 2- \sum_{i}^p \psi((n-i+1)/2)\Bigr), \label{E^{-1}(Sigma)}\\
& \Bigl(\because E[\log (\chi^2_k)]=\log 2+\psi(k/2)\Bigr) \nonumber
\end{align}
where $\psi$ is the di-gamma function. The expansion of the di-gamma function for large $n$ is given as follows;
\begin{align*}
\psi\Bigl(\frac{n-i+1}{2}\Bigr)&=\log\Big(\frac{n-i+1}{2}\Bigr)-\frac{1}{n-i+1}-\frac{1}{3(n-i-1)^2}+o(n^{-2}).\\
\sum_{i=1}^p\psi\Bigl(\frac{n-i+1}{2}\Bigr)&=\log \prod_{i=1}^p \frac{n-i+1}{2}-\sum_{i=1}^p\frac{1}{n-i+1}-\frac{p}{3n^2}+o(n^{-2})\\
&=-p\log 2+p \log n+ \log \prod_{i=1}^p \Bigl(1-\frac{i-1}{n}\Bigr)-\sum_{i=1}^p\frac{1}{n-i+1}\\
&\qquad-\frac{p}{3n^2}+o(n^{-2})\\
&=-p\log 2+p \log n+  \sum_{i=1}^p \log\Bigl(1-\frac{i-1}{n}\Bigr)-\sum_{i=1}^p\frac{1}{n-i+1}\\
&\qquad-\frac{p}{3n^2}+o(n^{-2})\\
&=-p\log 2+p \log n-\sum_{i=1}^p \Bigl(\frac{i-1}{n}+\frac{(i-1)^2}{2n^2}\Bigr) -\sum_{i=1}^n\frac{1}{n-i+1}\\
&\qquad-\frac{p}{3n^2}+o(n^{-2}).
\end{align*}
Substitute this result into \eqref{E^{-1}(Sigma)}, then we have
\begin{align*}
&\overset{-1}{E\!D}(\Sigma)\\
&=\frac{1}{2}\Bigl\{\sum_{i=1}^p \Bigl(\frac{i-1}{n}+\frac{1}{n-i+1}\Bigr)+\frac{1}{2n^2}\sum_{i=1}^p (i-1)^2 +\frac{p}{3n^2}\Bigr\}+o(n^{-2})\\
&=\frac{1}{2}\Bigl\{\frac{1}{n}\sum_{i=1}^p(i-1)+\frac{p}{n}+\frac{1}{n^2}\sum_{i=1}^p(i-1)+\frac{1}{2n^2}\sum_{i=1}^p(i-1)^2+\frac{p}{3n^2}\Bigr\}+o(n^{-2})\\
&\qquad\biggl(\because \frac{1}{n-(i-1)}=\frac{1}{n}\Bigl(\frac{1}{1-(i-1)/n}\Bigr)\\
&\qquad\qquad=\frac{1}{n}\Bigl(1+\frac{i-1}{n}+\frac{(i-1)^2}{n^2}+\cdots\Bigr)=\frac{1}{n}+\frac{i-1}{n^2}+o(n^{-2})\biggr)\\
&=\frac{1}{2}\Bigl\{-\frac{p}{n}+\frac{p(p+1)}{2n}+\frac{p}{n}+\frac{1}{n^2}\Bigl(\frac{p(p+1)}{2}-p\Bigr)\\
&\qquad+\frac{1}{2n^2}\Bigl(\frac{1}{6}(p-1)p(2p-1)\Bigr)+\frac{p}{3n^2}\Bigr\}+o(n^{-2})\\
&=\frac{p(p+1)}{4n}+\frac{1}{4n^2}(p^2-p)+\frac{1}{4n^2}\Bigl(\frac{1}{6}(2p^3-3p^2+p)\Bigr)+\frac{p}{6n^2}+o(n^{-2})\\
&=\frac{p(p+1)}{4n}+\frac{1}{24n^2}(6p^2-6p+2p^3-3p^2+p+4p)+o(n^{-2})\\
&=\frac{p(p+1)}{4n}+\frac{1}{24n^2}(2p^3+3p^2-p)+o(n^{-2}),
\end{align*}
which is equal to the right-hand side of \eqref{expan_KL_normal_vacova}. 
\\
\\
-\textit{ Proof of \eqref{int_exp_g}, \eqref{int_exp_F}, \eqref{int_exp_T_ikkT^ijk}, \eqref{int_exp_T_is^iT_j^js}, \eqref{int_exp_R}, \eqref{int_exp_secffe{i}{j}_secffe{i}{j}}, \eqref{int_exp_secffe{i}{i}_secffe{j}{j}}, \eqref{int_exp_secffm{i}{j}_secffm{i}{j}}, \eqref{int_exp_secffm{i}{i}_secffm{j}{j}}}-\\
We use Einstein's convention for summation again.

\eqref{int_exp_g} is obvious from \eqref{log_ij}. Fort the proof of \eqref{int_exp_F}, we use the following equations; 
\begin{align}
\partial_s g^{ij}&=g_{us}(\partial^{u}g^{ij})\nonumber\\
&=-g_{us}(\cristofte{ui}{j}+\cristoftm{uj}{i})\ (\because \eqref{partial^k g^ij})\nonumber\\
&=-(\cristoffe{s,}{i,j}+\cristoffm{s,}{j,i})\nonumber\\
&=-\cristofse{su}{l}g^{ui}g^{lj}-\cristofsm{su}{l}g^{uj}g^{li},\\
\partial_s T_{ijk}&=\partial_s E_\theta[l_i l_j l_k]\nonumber\\
&=\partial_s \int_{\mathcal{X}} l_i l_j l_k f(x; \theta) d\mu \nonumber\\
&=E_\theta[l_{is} l_j l_k]+E_\theta[l_i l_{js} l_k]+E_\theta[l_i l_j l_{ks}]+E_\theta[l_i l_j l_k l_s].
\end{align}
From these relations, we have
\begin{align*}
\overset{\:\alpha}{F}&=g^{ks}(\partial_s T_{ijk})g^{ij}+g^{ks}T_{ijk}(\partial_s g^{ij})-g^{ti}\cristofsa{it}{s}T_{juk}g^{uj}g^{ks}\\
&=g^{ij}g^{ks}(L_{(is)jk}+L_{(js)ik}+L_{(ks)ij}+L_{ijks})\\
&\qquad-g^{ks}g^{ui}g^{lj}L_
{ijk}L_{(su)l}-g^{ks}g^{uj}g^{li}L_{ijk}(L_{(su)l}+L_{sul})\\
&\qquad-g^{ti}g^{uj}g^{ks}(L_{(it)s}+((1-\alpha)/2)L_{its})L_{juk}\\
&=g^{ij}g^{ks}(2L_{(is)jk}+L_{(ks)ij}+L_{ijks})\\
&\qquad-g^{ks}g^{uj}g^{li}L_{ijk}(2L_{(su)l}+L_{sul})\\
&\qquad-g^{ti}g^{uj}g^{ks}(L_{(it)s}+((1-\alpha)/2)L_{its})L_{juk}.\\
\end{align*}

\eqref{int_exp_T_ikkT^ijk} and \eqref{int_exp_T_is^iT_j^js} immediately follow from the definitions 
\eqref{notation_T_ijk}, \eqref{def_T_{ij}^k}, \eqref{def_T_i^{jk}}, \eqref{def_T^{ijk}} and the relation \eqref{log_i_j_k}.

In order to prove \eqref{int_exp_R}, we use the following relations;
\begin{align*}
\partial^k \cristoffe{k,j}{j,}
&=\partial^k(\cristofse{ki}{j}g^{ij})\\
&=(\partial^k \cristofse{ki}{j})g^{ij}+\cristofse{ki}{j}(\partial^k g^{ij})\\
&=g^{ij}g^{sk}(\partial_s \cristofse{ki}{j})-\cristofse{ki}{j}(\cristofte{ki}{j}+\cristoftm{kj}{i})\\
&=g^{ij}g^{sk}\partial_s\int_{\mathcal{X}} l_{ki} l_j f(x; \theta) d\mu-\cristofse{ki}{j}\cristofse{st}{u}g^{sk}g^{ti}g^{uj}-\cristofse{ki}{j}\cristofsm{st}{u}g^{sk}g^{tj}g^{ui}\\
&=g^{ij}g^{sk}(L_{(kis)j}+L_{(ki)(sj)}+L_{(ki)js})-L_{(ki)j}L_{(st)u}g^{sk}g^{ti}g^{uj}\\
&\qquad-L_{(ki)j}(L_{(st)u}+L_{stu})g^{sk}g^{tj}g^{ui},\\
\partial^k \cristoffe{i,k}{i,}&=\partial^k(\cristofse{ij}{k}g^{ij})\\
&=(\partial^k \cristofse{ij}{k})g^{ij}+\cristofse{ij}{k}(\partial^k g^{ij})\\
&=g^{sk}(\partial_s L_{(ij)k})g^{ij}-L_{(ij)k}(\cristofte{ik}{j}+\cristoftm{jk}{i})\\
&=g^{ij}g^{sk}(L_{(ijs)k}+L_{(ij)(ks)}+L_{(ij)ks})-L_{(ij)k}L_{(st)u}g^{si}g^{tk}g^{uj}\\
&\qquad-L_{(ij)k}(L_{(st)u}+L_{stu})g^{sj}g^{tk}g^{ui},\\
\cristoffm{s,t}{s,}\cristoffe{i,}{i,t}&=(L_{(su)t}+L_{sut})g^{us}L_{(ij)k}g^{ij}g^{kt},\\
\cristoftm{jk}{i}\cristofse{ik}{j}&=(L_{(st)u}+L_{stu})g^{sj}g^{tk}g^{ui}L_{(ik)j}.
\end{align*}
Using these relations, we have
\begin{align*}
\overset{\;e}{R}\hspace{-10pt}\begin{smallmatrix}{\ \  \ \ ij}\\{ij}\end{smallmatrix}&=\partial^k \cristoffe{k,j}{j,}-\partial^k \cristoffe{i,k}{i,}+\cristoffm{s,t}{s,}\cristoffe{i,}{i,t}-\cristoftm{jk}{i}\cristofse{ik}{j}\ (\because \eqref{Form_2_R_ij^ij})\\
&=g^{ij}g^{sk}(L_{(kis)j}+L_{(ki)(sj)}+L_{(ki)js}-L_{(ijs)k}-L_{(ij)(ks)}-L_{(ij)ks})\\
&\qquad+g^{sk}g^{ti}g^{uj}\{L_{(kj)i}L_{(st)u}+L_{(jk)i}(L_{(st)u}+L_{stu})-L_{(ki)j}L_{(st)u}\\
&\qquad-L_{(kj)i}(L_{(st)u}+L_{stu})+(L_{(it)s}+L_{sit})L_{(uj)k}-(L_{(st)u}+L_{stu})L_{(ji)k}\}\\
&=g^{ij}g^{sk}(L_{(kis)j}-L_{(ijs)k}+L_{(ki)(sj)}-L_{(ij)(ks)}+L_{(ki)js}-L_{(ij)ks})\\
&\qquad+g^{sk}g^{ti}g^{uj}(L_{(kj)i}L_{(st)u}-L_{(ki)j}L_{(st)u}+L_{(it)s}L_{(uj)k}+L_{sit}L_{(uj)k}\\
&\hspace{30mm}-L_{(st)u}L_{(ij)k}-L_{stu}L_{(ij)k})\\
&=g^{ij}g^{sk}(L_{(ki)(sj)}-L_{(ij)(ks)}+L_{(ki)js}-L_{(ij)ks})\\
&\qquad+g^{sk}g^{ti}g^{uj}(-L_{(ki)j}L_{(st)u}+L_{(it)s}L_{(uj)k}+L_{sit}L_{(uj)k}-L_{stu}L_{(ij)k}).\\
\end{align*}

From \eqref{log_ij_kh}, we have 
$$
\langle\secfse{ik}, \secfse{jl}\rangle = L_{(ik)(jl)}-g^{st}L_{(ik)s}L_{(jl)t}-g_{ik}g_{jl}.
$$
\eqref{int_exp_secffe{i}{j}_secffe{i}{j}}, \eqref{int_exp_secffe{i}{i}_secffe{j}{j}} is calculated from this equation as follows; 
\begin{align*}
\langle \secffe{i}{j}, \secffe{j}{i} \rangle &= g^{jk}g^{li}\langle\secfse{ik}, \secfse{jl}\rangle\\
&=g^{jk}g^{li}L_{(ik)(jl)}-g^{jk}g^{li}g^{st}L_{(ik)s}L_{(jl)t}-p,\\
\langle \secffe{i}{i}, \secffe{j}{j} \rangle &= g^{ik}g^{jl}\langle\secfse{ik}, \secfse{jl}\rangle\\
&=g^{ik}g^{jl}L_{(ik)(jl)}-g^{ik}g^{jl}g^{st}L_{(ik)s}L_{(jl)t}-p^2.
\end{align*}
From \eqref{log_ij_k_h}, we have
\begin{align*}
\langle\secfse{ik}, \secfsm{jl}\rangle &= L_{(ik)jl}+\langle\secfse{ik}, \secfse{jl}\rangle+g_{ik}g_{jl}-\cristofse{ik}{s}T_{jl}^s\\
&=L_{(ik)jl}+L_{(ik)(jl)}-g^{st}L_{(ik)s}L_{(jl)t}-L_{(ik)s}L_{jlt}g^{st}.
\end{align*}
From this equation, \eqref{int_exp_secffm{i}{j}_secffm{i}{j}} and \eqref{int_exp_secffm{i}{i}_secffm{j}{j}} hold as follows;
\begin{align*}
\langle \secffe{i}{j}, \secffm{j}{i} \rangle &= g^{jk}g^{li}\langle\secfse{ik}, \secfsm{jl}\rangle\\
&=g^{jk}g^{li}L_{(ik)jl}+g^{jk}g^{li}L_{(ik)(jl)}-g^{jk}g^{li}g^{st}L_{(ik)s}L_{(jl)t}-g^{jk}g^{li}g^{st}L_{(ik)s}L_{jlt},\\
\langle \secffe{i}{i}, \secffm{j}{j} \rangle &= g^{ik}g^{jl}\langle\secfse{ik}, \secfsm{jl}\rangle\\
&=g^{ik}g^{jl}L_{(ik)jl}+g^{ik}g^{jl}L_{(ik)(jl)}-g^{ik}g^{jl}g^{st}L_{(ik)s}L_{(jl)t}-g^{ik}g^{jl}g^{st}L_{(ik)s}L_{jlt}.
\end{align*}
\subsection{R Programs}
\label{r programs}
-\textit{ T\_ijkT\^{}ijk.R }-\\
This program is based on the result of Takeuchi  and Takemura \cite{Takeuchi&Takemura}.
\begin{verbatim}
####################### Description ##########################
#The summation of $\sigma^{ik}\simga^{ls}\sigma^{kj}\sigma^{ac}
#\sigma^{de}\sigma^{fb}\simga_{ai}\sigma_{bj}\sigma_{ck}
#\sigma_{dl}\sigma_{es}\sigma_{ft} 
#over $1\leq i,j,k,l,s,t,a,b,c,d,e,d,e,f \leq p$.
#This program calculates the power of p (the matrix dimension) 

########## Def of function "Combinesegmens" ###################
#This function combines two segments F1 and F2
# (each of which is a 2-dimensional vector). 
#If they have a common end, then this function returns two vectors 
#(the combined one and "NA"). 
#IF F1 and F2 do not have a common end, 
#it returns them as they are.

Combinesegments <- function(F1, F2)
{
for (i in 1:2) 
{
 for (j in 1:2)
 {
  if (F1[i]==F2[j]) 
  {
  AF1 <- c(min(F1[-i],F2[-j]), max(F1[-i],F2[-j]))
  AF2 <- NA
  return(list(AF1,AF2))
  }
 }
}
return(list(F1,F2))
}
##########################################################

########### Def of "nloop"function ###########################
#This function tells the numbers of loops
# after the following procedure;
#1) Make a pairing of segments from LOS1 and LOS2 with a common 
#end, and change the used segments into "NA".
#2) If the combined one has common ends, that is,  a loop is made, 
#then raise the counter "Lcounter" by one, 
#otherwise put the combined one into "delta".
#3) Combine the segment in "delta" with the one in LOS1
#that has a common end, 
#and alter the used one in LOS1 with the combined one.
#4) Repeat 1) to 3) until all the segments in LOS1 become NA.
#Input 
#k: the number of the segments
#LOS1, LOS2: a list whose k objects are 2-dimensional vectors, 
#where each vector corresponds to a segment.
#Segments are the parings of integers from 1:2k, 
#e.g., (2,3),(4,6),(1,5) for k=3.
#LOS1 and LOS2 corresponds respectively to
# $\sigma^{ij}$ and $\sigma_{ij}$.
#Output
#Lcounter:  the number of loops


nloop <- function(k,LOS1,LOS2)
{
 Lcounter <- 0
  
  #Pair each segment in LOS2 with a segment in LOS1 
  #under the condition the paired segments have a common end. 
  #Put the combined segments into "delta", 
  #and put "NA" into LOS1 & LOS2

  
  for (i in 1:k)
  {
   if (identical(is.na(LOS1[[i]]),TRUE)) {next} #neglect "NA"
   for (j in 1:k)
   {
    if (identical(is.na(LOS2[[j]]),TRUE)) {next} #neglect "NA"
    AF <- Combinesegments(LOS1[[i]],LOS2[[j]])
    if (identical(is.na(AF[[2]]),TRUE)) 
    #is.na(AF[[2]])=TRUE means the condition AF[[2]]=NA, 
    #which means a paring has occurred 
    {
    #1) the paired segments are changed into NA
    #2) if the combined segments has a common end, 
    #then raise "Lcounter" by one,
    #3) otherwise put the combined segments into "delta".
    #4) combine the segment in "delta" with  a segment in LO1 
    #with a common end, and alter the used segment in LO1 
    #with the combined one.
    # "Lcounter" corresponds to the number of loops, 
    #that is the power of p (the given dimension of the matrix)

     LOS1[[i]] <- NA; LOS2[[j]] <- NA
    
     if (AF[[1]][1]==AF[[1]][2]) {Lcounter <- Lcounter+1} else 
     # "}" and else must be on the same line !
     {
      delta <- AF[[1]] ; 
      for (t in 1:k)
      {
       if (identical(is.na(LOS1[[t]]),TRUE)) {next}
       #neglect "NA"
       AK<- Combinesegments(delta,LOS1[[t]])
       if (identical(is.na(AK[[2]]),TRUE))
       {LOS1[[t]] <- AK[[1]]}
      }
     }
     break
    }
   }
  }
 Lcounter
}
###########################################################

#We count the power of p for every possible set 
#of LO1 and LO2 for T_{ijk}T^{ijk}

k <- 6 #the number of segments
#Original input for segments are made by 2k-dim vector
#e.g., for k=3,  (1,3,4,6,2,5) means 
#segments(1,3), (4,6), (2,5)


#the exchange pattern for T
#(1,2,3,4,5,6,7) means 
#(i,k,l,s,t,j) or (a,c,d,e,f,b) in the text
kristorder1<-c(6,2,3,4,5,1) #e.g., exchange between i and j
kristorder2<-c(1,3,2,4,5,6) #e.g., exchange between k and l
kristorder3<-c(1,2,3,5,4,6) #e.g., exchange between s and t
#the exchange pattern for g
#(1,2,3,4) means (a,i,b,j), (c,k,d,l) (e,s,f,t)
metricorder1<-c(3,2,1,4)  #e.g., exchange between a and b
metricorder2<-c(1,4,3,2)  #e.g., exchange between i and j

kekka <- c()

#fk1, fk2, fk3 is the indicator of exchange patterns
#kristorder1, kristorder2, kristorder3 for the first T.
#sk1, sk2, sk3 is the indicator of exchange patterns
#kristorder1, kristorder2, kristorder3 for the second T.
#fm1 and fm2 is the indicator of  exchange patterns
#mericorder1 and metricorder2 for the first g.
#sm1 and sm2 is the indicator of exchange patterns
#mericorder1 and metricorder2 for the second g.
#tm1 and tm2 is the indicator of exchange patterns 
#mericorder1 and metricorder2 for the third g.
for (fk1 in 0:1)
 {for (fk2 in 0:1)
  {for (fk3 in 0:1)
   {for (sk1 in 0:1)
    {for (sk2 in 0:1)
     {for (sk3 in 0:1)
      {for (fm1 in 0:1)
       {for (fm2 in 0:1)
        {for (sm1 in 0:1)
         {for (sm2 in 0:1)
          {for (tm1 in 0:1)
           {for (tm2 in 0:1)
            {fkristoff<-c(1,2,3,4,5,6)  
             #(i,k)(l,s)(t,j) in the text
             skristoff<-c(7,8,9,10,11,12) 
             #(a,c)(d,e)(f,b) in the text
             fmetric <- c(1,7,6,12) #(a,i)(b,j) in the text
             smetric <- c(2,8,3,9) #(c,k)(d,l) in the text
             tmetric <- c(4,10,5,11) #(e,s)(f,t) in the text

             if (fk1==1){fkristoff <- fkristoff[kristorder1]}
             if (fk2==1){fkristoff <- fkristoff[kristorder2]}
             if (fk3==1){fkristoff <- fkristoff[kristorder3]}
             if (sk1==1){skristoff <- skristoff[kristorder1]}
             if (sk2==1){skristoff <- skristoff[kristorder2]}
             if (sk3==1){skristoff <- skristoff[kristorder3]}
             if (fm1==1){fmetric <- fmetric[metricorder1]}
             if (fm2==1){fmetric <- fmetric[metricorder2]}
             if (sm1==1){smetric <- smetric[metricorder1]}
             if (sm2==1){smetric <- smetric[metricorder2]}
             if (tm1==1){tmetric <- tmetric[metricorder1]}
             if (tm2==1){tmetric <- tmetric[metricorder2]}
             kristoff <- c(fkristoff,skristoff)
             metric <- c(fmetric,smetric,tmetric)         
             # We put original input into a matrix OS1, OS2 
             OS1 <- matrix(kristoff,ncol=2, nrow=k, byrow=T)
             OS2 <- matrix(metric,ncol=2, nrow=k, byrow=T)
             #each row of OS1,OS2 is converted into 
             #the elements of list LOS1, LOS2
             LOS1 <- list()
             LOS2 <- list()

             for (i in 1:k) 
             {
              LOS1 <- c(LOS1,list(OS1[i,])) 
              LOS2 <- c(LOS2,list(OS2[i,]))
             }

             Counter <- c(fk1,fk2,fk3,sk1,sk2,sk3,
                                fm1,fm2,sm1,sm2,tm1,tm2)
             kekka <- rbind(kekka,c(Counter,nloop(6,LOS1,LOS2)))
            }
           }
          }
         }
        }
       }
      }
     }
    }
   }
  }
 }

kekka

result <- kekka[,13]
table(result)
\end{verbatim}

\end{document}